\documentclass[12pt]{amsart}

\usepackage{amsfonts,amsthm,amsmath,amssymb,amscd}

\usepackage{upref}
\usepackage[backref,bookmarks=false]{hyperref}
\usepackage{mathrsfs}
\usepackage{url}

\numberwithin{equation}{section}

\newif\ifdraft\drafttrue
\drafttrue

\newcommand{\lab}{\label}

\newcommand{\ben}{\begin{enumerate}}
\newcommand{\een}{\end{enumerate}}

\newcommand{\bea}{\begin{eqnarray}}
\newcommand{\ba}{\begin{array}}
\newcommand{\bean}{\begin{eqnarray*}}
\newcommand{\ea}{\end{array}}
\newcommand{\eea}{\end{eqnarray}}
\newcommand{\eean}{\end{eqnarray*}}
\newcommand{\beq}{\begin{equation}}
\newcommand{\eeq}{\end{equation}}
\newcommand{\bthm}{\begin{thm}}
\newcommand{\ethm}{\end{thm}}
\newcommand{\blem}{\begin{lem}}
\newcommand{\elem}{\end{lem}}
\newcommand{\bprop}{\begin{prop}}
\newcommand{\eprop}{\end{prop}}
\newcommand{\bcor}{\begin{cor}}
\newcommand{\ecor}{\end{cor}}
\newcommand{\bdfn}{\begin{dfn}}
\newcommand{\edfn}{\end{dfn}}
\newcommand{\brem}{\begin{rem}}
\newcommand{\erem}{\end{rem}}
\newcommand{\bpf}{\begin{proof}}
\newcommand{\epf}{\end{proof}}
\newcommand{\bfact}{\begin{fact}}
\newcommand{\efact}{\end{fact}}
\newcommand{\bobs}{\begin{observation}}
\newcommand{\eobs}{\end{observation}}

\alph{enumii} \roman{enumiii}

\newtheorem{thm}{Theorem}[section]
\newtheorem{prop}[thm]{Proposition}
\newtheorem{lem}[thm]{Lemma}

\newtheorem{cor}[thm]{Corollary}

\theoremstyle{definition}
\newtheorem{dfn}[thm]{Definition}
\newtheorem{rem}[thm]{Remark}
\newtheorem{fact}[thm]{Fact}
\newtheorem{observation}[thm]{Observation}

\def\cA{\mathcal A}             \def\cB{\mathcal B}       \def\cC{\mathcal C}
             \def\cF{\mathcal F}       
\def\cL{{\mathcal L}}                  
                    \def\cJ{\mathcal J}
\def\cS{\mathcal S}             
             \def\cO{\mathcal O}

\def\endpf{\qed}

\def\N{{\mathbb N}}                  \def\R{{\mathbb R}}            \def\cR{{\mathcal R}}
\def\C{{\mathbb C}}                     \def\oc{\widehat \C}

\def\1{1\!\!\text{{\rm 1}}}
\def\and{\text{ and }}        
         
          \def\Con{\text{Con}} 
\def\Comp{\text{{\rm Comp}}}        \def\diam{\text{\rm {diam}}}
\def\dist{\text{{\rm dist}}}  
\def\Crit{\text{Crit}}        
\def\Sing{\text{Sing}}        
      
\def\F{{\mathcal F}}          \def\osc{\text{{\rm osc}}}

\def\h{{\text h}}
\def\hmu{\h_\mu}           \def\htop{{\text h_{\text{top}}}}
\def\H{\text{{\rm H}}}     \def\HD{\text{{\rm HD}}}   
            \def\PC{\text{{\rm PC}}}
\def\re{\text{{\rm Re}}}      
\def\Int{\text{{\rm Int}}}    
 
 \def\PS{\text{{\rm PS}}} \def\WBT{\text{{\rm WBT}}}
      \def\GDS{{\rm GDS}}        \def\IFS{{\rm IFS}}
         \def\P{\text{{\rm P}}}     
 \def\sign{\text{{\rm sgn}}}
\def\Leb{{\rm Leb}}
         \def\Ba{\mathcal B}       
        \def\cK{\mathcal K}
\def\cL{{\mathcal L}}

\def\a{\alpha}                \def\b{\beta}             \def\d{\delta}
\def\De{\Delta}               \def\e{\varepsilon}          \def\f{\phi}
\def\g{\gamma}                \def\Ga{\Gamma}           
\def\La{\Lambda}              \def\om{\omega}           \def\Om{\Omega}
               \def\sg{\sigma}
               \def\th{\theta}           
\def\ka{\kappa}               \def\vp{\varphi}          \def\phi{\varphi}
               \def\vep{\varepsilon}

\def\bal{\begin{aligned}}
\def\eal{\begin{aligned}}
\def\bi{\bigcap}              \def\bu{\bigcup}
\def\({\big(}                 \def\){\big)}
\def\lt{\left}                \def\rt{\right}

\def\ld{\ldots}               \def\bd{\partial}         \def\^{\widetilde}

      \def\du{\bigoplus}

\def\es{\emptyset}            \def\sms{\setminus}
\def\sbt{\subseteq}             \def\spt{\supseteq}
             \def\lek{\preceq}
\def\eqv{\Leftrightarrow}     
      
\def\comp{\asymp}
\def\upto{\nearrow}           \def\downto{\searrow}
\def\sp{\medskip}             \def\fr{\noindent}        
\def\ov{\overline}            \def\un{\underline}
\def\ess{{\rm ess}}           
\def\om{\omega}

\def\re{\text{{\rm Re}}}

\def\supp{\text{{\rm supp}}}
\def\endpf{{\hfill $\square$}}

\newcommand{\lam}{\lambda}
\newcommand{\ep}{\varepsilon}

\newcommand{\al}{\alpha}

\newcommand{\pf}{{\mathcal{L}}}

\newcommand{\Lp}{{\mathcal{L}}_\phi}

\def\J{\mathcal J}

\begin{document}
%***********************************************

\title[]
{\bf{{\large {N}}on-Escaping Sets  \\  in \\
   Conformal Dynamical Systems   
   \\ \, \, \, \,  and \\  \, \, \,  \,
   Singular Perturbations \\ of \\ Perron-Frobenius Operators
  }}
\date{\today}

\author{Mark Pollicott}
\address{University of Warwick, Institute of Mthematics, UK} 
\email{masdbl@warwick.ac.uk}

\author{Mariusz Urba\'nski}
\address{University of North Texas, Department of Mathematics, 1155
  Union Circle \#311430, Denton, TX 76203-5017, USA} 
\email{urbanski@unt.edu  \newline \hspace*{0.3cm} Web:\!\!\!
www.math.unt.edu/$\sim$urbanski}

% dedication %
%\dedicatory{}
%
% AMS information %
\thanks{The research of both authors supported in part by the NSF
  Grant DMS 0700831 and DMS 0400481.} 
\keywords{}
\subjclass{Primary:}

\begin{abstract}
 The study of  escape rates for a ball in a dynamical systems has been much studied.  Understanding the asymptotic behavior of the escape rate as the radius of the ball tends to zero is an especially  subtle problem. In the case of hyperbolic conformal systems this has been addressed by various authors
 \cite{Bun}, \cite{FP}, \cite{KL} and these results apply in the case of real one dimensional expanding maps and conformal expanding repellers, particularly hyperbolic rational maps.   
 
\sp In this paper we consider a far more general realm of conformal maps where the analysis is correspondingly more complicated. We prove the existence of escape rates and calculate them in the context of countable alphabets, either finite or infinite, uniformly contracting conformal graph directed Markov systems (see \cite{GDMS}, \cite{MU_lms}) with their special case of conformal countable alphabet iterated function systems. The reference measures are the projections of Gibbs/equilibrium states of H\"older continuous summable potentials from a countable alphabet subshifts of finite type to the limt set of the graph directed Markov system under consideration. 

This goal is achieved firstly by developing the appropriate theory of singular perturbations of Perron-Frobenius (transfer) operators associated with countable alphabet subshifts of finite type and H\"older continuous summable potentials, see \cite{MU_Israel} and \cite{GDMS} for the theory of such unperturbed operators, and \cite{KL} and \cite{FP} for singular perturbations which motivated our methods. 

In particular, this includes, as a second ingredient in its own right, the asymptotic behavior of leading eigenvalues of perturbed operators and their first and second derivatives. 

Our third ingredient is to relate the geometry and dynamics, roughly speaking to relate the case of avoiding cylinder sets and that of avoiding Euclidean geometric balls. Towards this end, in particular, we investigate in detail  thin boundary properties relating the measures of thin annuli to the measures of the balls they enclose. In particular we clarify the results in the case of expanding repellers and conformal graph directed Markov systems with finite alphabet.

\sp The setting  of conformal graph directed Markov systems is interesting in its own and moreover, in our approach, it forms the key ingredient for further results about other conformal systems. These include topological Collet-Eckmann multimodal interval maps and rational maps of the Riemann sphere (an equivalent formulation is to be uniformly hyperbolic on periodic points), and also a large class of transcendental meromorphic functions, such as those introduced and explored in \cite{MayUrbETDS} and \cite{MayUrb10}. 

Our approach here is firstly to note that all of these systems yield some sets, commonly referred to as nice ones, the first return (induced)  map to which is isomorphic to a conformal countable alphabet iterated function system with some additional properties. Secondly, with the help of appropriate large deviation results, to relate escape rates of the original system with the induced one and then to apply the results of graph directed Markov systems. The reference measures are again Gibbs/equilibrium states of some large classes of H\"older continuous potentials. 
\end{abstract}

\maketitle

\tableofcontents

%
%***********************************************
\section{Introduction}\label{intro}

The escape rate for a dynamical system is a natural concept which describes the speed at which orbits of   points first enter a small region of the space.  The size of these sets is usually measured with respect to an appropriate probability.
More precisely, given a metric space $(X,d)$, we can consider a (usually) continuous  transformation  $T: X \to X$ and a ball 
$$
B(z,\epsilon)= \{x \in X \hbox{ : } d(x,z) < \epsilon \}
$$ 
of radius $\epsilon>0$ about a given point $z$. We then obtain an open system by removing $B(x_0,\epsilon)$ and 
considering the new space $X\sms B(z,\epsilon)$ and truncating those orbits that land in the ball $B(z,\epsilon)$, which can be thought of informally  as a ``hole'' in the system.   This is the reason that many authors speak of \emph{escape rates} for the system, whereas it might be a more suitable nomenclature to call them \emph{avoidable sets}.

We can then consider for each $n > 0$ the set $R_n(z, \epsilon)$ of points $x \in X$ for which all the first $n$ terms in the orbit omit the ball, i.e., 
$x, T(x), \cdots T^{n-1}(x) \not\in B(z, \epsilon)$.
%, but $T^nx \in B(x_0, \epsilon)$ and  denote the set of such $x$ by $R_n(\epsilon)$. 
It is evident that these sets are nested in both parameters $\e$ and $n$, i.e., 
$$
R_{n+1}(z, \epsilon) \subset R_n(z, \epsilon)
$$ 
for all $n \geq 1$ and that 
$$
R_n(z, \epsilon) \subset R_n(z, \epsilon')
$$
for $\epsilon > \epsilon'$.
We can first ask about the behavior of the size of the sets $R_n(z,\epsilon)$ as $n \to +\infty$.

If we assume that $\mu$ is a $T$-invariant probability measure, say, then we can consider the measures   $\mu(R_n(z, \epsilon))$ of the sets $R_n(z, \epsilon)$ as $n \to +\infty$.  In particular, we can define the \emph{lower and upper escape rates} respectively as
$$
\un R_\mu(B(z,\ep))
 = -\varlimsup_{n \to +\infty} \frac{1}{n} \log \mu(R_n(z, \epsilon)) 
\  \  \and \  \ 
\ov R_\mu(B(z,\ep))
 = -\varliminf_{n \to +\infty} \frac{1}{n} \log \mu(R_n(z, \epsilon)).
$$
One can  further consider how the escape rate behaves as the radius of the ball 
 $\epsilon$ tends to zero. 
 An early influential result in this direction was \cite{U-non}.
 % We therefore define the  \emph{asymptotic escape rate} at $x_0$ by
% $$
% \lambda(x_0) = \lim_{\epsilon \to 0} 
%  \frac{\lambda(x_0, \epsilon)}{\mu(B(x_0, \epsilon))}
% $$
% when it is well defined.  
Perhaps the simplest case is that of  the doubling map $E_2:[0,1) \to [0,1)$ defined by 
 $E_2(x) = 2x\,(\!\!\!\mod\! 1)$ and the usual Lebesgue measure $\lam$. For this example it was Bunimovitch and Yurchenko \cite{Bun} (see also \cite{KL}) who showed the following, perhaps surprising, result showing  that
\beq\label{1_2016_06_27}
\begin{aligned}
\lim_{\ep\to 0}\frac{\un R_{\lam}(B(z,\ep))}{\lam(B(x_0,\ep))} &= 
\lim_{\ep\to 0}\frac{\ov R_{\lam}(B(z,\ep))}{\lam(B(x_0,\ep))} = \\
&=\begin{cases}
 1 & \hbox{ if } z \hbox{ is not periodic}\\
 1 - 2^{-p} & \hbox{ if } E_2^p(z)=z\hbox{ is periodic (with minimal period $p$)}.\\
 \end{cases}
 \end{aligned}
\eeq
In particular, the asymptotic escape rate can only take a certain set of values which are determined by the periods of periodic points. More results in this direction followed, particularly in \cite{KL} and \cite{FP}. We will return to generalizations of these ideas  after discussing a related problem.
 
One can also ask what is happening to full escaping/avoiding sets when $\e>0$ decreases to zero. By full escaping sets we mean the sets of the form
$$
K_z(\e)=\{x\in X: T^n(x)\notin B(z,\e)\, \, \forall n\ge 0\}
$$
Such sets are usually of measure $\mu$ zero, but there is another natural quantity to measure their size and complexity, namely their Hausdorff dimension. The second named author already addressed this question in the early 80s  by showing in \cite{U86} and \cite{U87} that in the case of the the doubling map $E_2$, or more generally, of any map $E_q(x)=qx\,(\!\!\!\mod\! 1)$, $q$ being an integer greater than $1$ in absolute value, or even more generally, in the case of any $C^{1+\eta}$ expanding map of the unit circle, the map
$$
\e\mapsto \HD(K_z(\e))
$$
is continuous.  Moreover, it was  also shown that this function is almost everywhere locally constant, in fact, the set of points where it fails to be locally constant is a closed set of Hausdorff dimension $1$ and Lebesgue measure zero. Rather curiously, the local Hausdorff dimension at each point $r$ of this set is equal to $\HD(K_z(\e))$. All of this suggests that it is interesting  to study the asymptotic properties of $\HD(K_z(\e))$ when $\e\downto 0$. Andrew Ferguson and the first named author of this paper took up the challenge by proving in \cite{FP} that 
\beq\label{2_2016_06_27} 
\lim_{\e\to 0}\frac{\HD(J)-\HD(K_z(r))}{\mu_b(B(z,r))}=
\begin{cases}
1/\chi_{\mu_b}  \!\! &\text{if } \, z\, \text{ is not a periodic point of }\, T \\
\frac{1-|(T^p)'(z)|^{-1}}{\chi_{\mu_b}}
 \!\! &\text{\!\! if } \,z\,  \text{ is a periodic point of } T \, \text{ with prime period }\\
      &p\ge 1.
\end{cases}
\eeq
in the case of any conformal expanding repeller $T:J\to J$; where $b$  here is just the Hausdorff dimension $\HD(J)$ and $\mu_b$ is the equilibrium state of the H\"older continuous potential $J\ni x\mapsto -b\log|T'(x)|$. They have also established the analogue of \eqref{1_2016_06_27} for such systems. 

The approach of \cite{FP} was based on the method of singular perturbations of the Perron--Frobenius operators determined by the open sets $B(z,\e)$. They first did this for neighborhoods of $z$ formed from finite unions of cylinders of $n$th refinements of a Markov partition and then used appropriate approximation.  This required leaving  the realm of the familiar Banach space of H\"older continuous functions, to work with a more refined space, and they applied the seminal results of Keller and Liverani from \cite{KL} to control the spectral properties of perturbed operators.
 
\sp In the current paper we want to understand the escape rates, in the sense of equations 
\eqref{1_2016_06_27} and \eqref{2_2016_06_27}, of essentially all conformal dynamical systems with an appropriate type of expanding dynamics. By this we  primarily mean all topologically exact piecewise smooth maps of the interval $[0,1]$, many rational functions of the Riemann sphere $\oc$ with degree $\ge 2$, a vast class of transcendental meromorphic functions from $\C$ to $\oc$, and last, but not least, the class of all countable alphabet conformal iterated function systems, and somewhat more generally, the class of all countable alphabet conformal graph directed Markov systems. This last class, i.e the collection of all countable alphabet conformal iterated function systems (IFSs), has a  special status for us. The reasons for this are two-folded. Firstly, this class is interesting by itself, and secondly, by means of appropriate inducing schemes (involving the first return map), it is our indispensable tool for understanding the escape rates of all other systems mentioned above. 

\sp In order to deal with escape rates for countable alphabet conformal IFSs and conformal graph directed Markov systems (GDMSs), motivated by the work \cite{FP} of Andrew Ferguson and the first named author of this paper, we first develop the singular perturbation theory for Perron-Frobenius operators associated to H\"older continuous summable potentials on countable alphabet shift of finite type symbol space. A comprehensive account of the thermodynamic formalism in the symbolic context can be found in \cite{GDMS}, cf. also \cite{MU_Israel} and \cite{MU_lms}. The general approach to control these perturbations is again based on the spectral results of Keller and Liverani from \cite{KL}. The perturbations in the case of a countable infinite alphabet require further refinement of the Banach space on which the original and perturbed Perron--Frobenius operators act. This space, $\cB_\th$, is defined already in the beginning of Section~\ref{PFOriginal}. Its definition, through the definition of the norm, involves 
the corresponding Gibbs/equilibrium measures. These measures play a further prominent role when investigating singular perturbations.
Qualitatively new difficulties here, caused by an infinite alphabet, are many fold and a great deal of them are related to the facts that the symbol space $E_A^\infty$ need not longer be compact, that there are infinitely many cylinders of given finite length, and that summable (particular geometric) potentials are unbounded in the supremum norm. Some remedy to this unboundedness issue is our repetitive use of H\"older inequalities rather than estimating by the supremum norms. 

Having analyzed  the symbolic part of the problem, we turn to escape rates for conformal GDMSs. With regard to formula \eqref{1_2016_06_27}, we consider the, already mentioned, measures on the limit set of the given conformal GDMS, that are projections of Gibbs/equilibrium states of H\"older continuous potentials from the symbol space. With respect to formula \eqref{2_2016_06_27}, we must consider geometric potentials, i.e. those of the form
$$
E_A^\infty\ni\om\longmapsto t\log\big|\phi_{\om_0}'(\pi_\cS(\sg(\om)))\big|\in\R
$$
where $\pi: E_A^\infty \to X$ is the canonical map for modelling the dynamics on $X$.
Of particular interest are those for which $t$ is close to $b_\cS$, the Bowen parameter of the system conformal GDMS$\cS$, which is defined as the only solution to the pressure equation 
$$
\P\(\sg, t\log\big|\phi_{\om_0}'(\pi_\cS(\sg(\om)))\big|\)=0,
$$
provided that such solution exists. We can then consider the projection of the Gibbs/equilibrium state $\mu_b$ for the potential $t\log\big|\phi_{\om_0}'(\pi_\cS(\sg(\om)))\big|$ on the limit set $J_\cS$ . This leads to the particularly technically involved  task of calculating the asymptotic behavior of derivatives $\lam_n'(t)$ and $\lam_n''(t)$ of leading eigenvalues of perturbed operators when the integer $n\ge 0$ diverges to infinity and the parameter $t$ approaches $b_\cS$. This is again partially due to unboundedness of the function $E_A^\infty\ni\om\longmapsto t\log\big|\phi_{\om_0}'(\pi_\cS(\sg(\om)))\big|\in\R$ in the supremum norm and partially due to lack of uniform topological mixing on the sets $K_z(\e)$.

We say that a set $J\sbt\R^d$, $d\ge 1$, is geometrically irreducible if it is not contained in any countable union of conformal images of hyperplanes or spheres of dimension $\le d-1$ (see Definition~\ref{d7_2016_07_07}). 
Our most general  results about escape rates for conformal GDMSs can now be formulated in the following four theorems. We postpone detailled definitions of the hypotheses until later.

\bthm\label{t1fp83}
Let $\cS=\{\phi_e\}_{e\in E}$ be a finitely primitive Conformal GDMS with limit set $J_{\cS}$. Let $\phi:E_A^\infty\to\R$ be a H\"older continuous summable potential with  equilibrium/Gibbs state $\mu_\phi$. Assume that the measure $\mu_\phi\circ\pi^{-1}_\cS$ is weakly boundary thein {\rm (WBT)} at a point $z\in J_\cS$. If $z$ is either 

\begin{itemize}
\item[(a)] not pseudo-periodic, 

or 

\item[(b)] uniquely periodic, it belongs to $\Int X$ (and $z=\pi(\xi^\infty)$ for a (unique) irreducible word $\xi\in E_A^*$), and $\phi$ is the amalgamated function of a summable H\"older continuous system of functions, 
\end{itemize}
then, with $\un R_{\cS,\phi}(B(z,\ep)):=\un R_{\mu_\phi}\(\pi_{\cS}^{-1}(B(z,\ep))\)$ and $\ov R_{\cS,\phi}(B(z,\ep)):=\ov R_{\mu_\phi}\(\pi_{\cS}^{-1}(B(z,\ep))\)$, we have that 
\beq\label{1fp83}
\begin{aligned}
\lim_{\ep\to 0}\frac{\un R_{\cS,\phi}(B(z,\ep))}{\mu_\phi\circ\pi_{\cS}^{-1}(B(z,\ep))}
&=\lim_{\ep\to 0}\frac{\ov R_{\cS,\phi}(B(z,\ep))}{\mu_\phi\circ\pi_{\cS}^{-1}(B(z,\ep))} =\\
&=d_\phi(z)
:=\begin{cases}
1 \  &\text{{\rm if (a) holds}}   \\
1-\exp\({S_p\phi(\xi)}-p\P(\phi)\) &\text{{\rm  if (b) holds}},
\end{cases}
\end{aligned}
\eeq
where in {\rm (b)}, $\{\xi\}=\pi_\cS^{-1}(z)$ and $p\ge 1$ is the prime period of $\xi$ under the shift map. 
\ethm

\sp
\bthm\label{t3_2016_05_27}
Assume that $\cS$ is a finitely primitive conformal {\rm GDMS} whose limit set $J_{\mathcal S}$ is geometrically irreducible.
Let $\phi:E_A^\infty\to\R$ be a H\"older continuous strongly summable potential. As usual, denote its equilibrium/Gibbs state by $\mu_\phi$. Then, with $R_{\cS,\phi}(B(z,\ep)):=R_{\mu_\phi}\(\pi_{\cS}^{-1}(B(z,\ep))\)$, we have that
\beq\label{1fp83+1}
\lim_{\ep\to 0}\frac{\un R_{\cS,\phi}(B(z,\ep))}{\mu_\phi\circ\pi_{\cS}^{-1}(B(z,\ep))}
=\lim_{\ep\to 0}\frac{\ov R_{\cS,\phi}(B(z,\ep))}{\mu_\phi\circ\pi_{\cS}^{-1}(B(z,\ep))} 
=1
\eeq
for $\mu_\phi\circ\pi_{\cS}^{-1}$--a.e. point $z$ of $\cJ_\cS$.
\ethm

\sp\fr These two theorems address the issue of \eqref{1_2016_06_27}. 
We would like to bring to the reader's attention a preprint  \cite{BDT} by H. Bruin, M.F.Demers and M.Todd, with results related to the above, which we recently received.
In regard to \eqref{2_2016_06_27}, we have proved for conformal GDMSs the following two theorems. In regard to \eqref{2_2016_06_27}, we have proved  the following two theorems for conformal GDMSs.

\bthm\label{t2had30}
Let $\cS$ be a finitely primitive strongly regular conformal {\rm GDMS}. Assume both that $\cS$ is {\rm (WBT)} and the  parameter $b_\cS$ is powering at some point $z\in J_\cS$ which is either

\begin{itemize}
\item[(a)] not pseudo-periodic
or else 
\item [(b)] uniquely periodic and belongs to $\Int X$ (and $z=\pi(\xi^\infty)$ for a (unique) irreducible word $\xi\in E_A^*$).
\end{itemize}
Then 
\beq\label{1had31}
\lim_{r\to 0}\frac{\HD(J_\cS)-\HD(K_z(r))}{\mu_b\(\pi^{-1}(B(z,r))\)}=
\begin{cases}
1/\chi_{\mu_b} \  &\text{ if } \ (a) \  \text{ holds } \\
\(1-|\phi_\xi'(z)|\)/\chi_{\mu_b}
\  &\text{ if } \ (b) \  \text{ holds }.
\end{cases}
\eeq
\ethm

\sp
\bcor\label{t2_2016_05_27}
If $\cS$ be a finitely primitive strongly regular conformal {\rm GDMS} whose limit set $J_{\mathcal S}$ is a geometrically irreducible, then 
\beq\label{1had31+2}
\lim_{r\to 0}\frac{\HD(J_\cS)-\HD(K_z(r))}{\mu_b\(\pi^{-1}(B(z,r))\)}=
\frac1{\chi_{\mu_b}}
\eeq
at $\mu_{b_\cS}\circ\pi^{-1}$--a.e. point $z$ of $J_\cS$.
\ecor

\sp\fr As we have previously remarked,  these four results are of independent interest, but they also provide a gateway to all other results on escape rates in this paper. There are necessarily several technical 
terms involved in formulations of these theorems. However, we hope that they do not obscure the overall meaning of the four theorems and all terms are carefully introduced and explained in appropriate sections dealing with them. 

We would however like to comment on one of these terms, namely on (WBT). Its meaning can be understood  as follows. Let 
$$
A(z;r,R):=B(z,R)\sms \ov{B(z,r)}
$$
be the annulus centered at $z$ with the inner radius $r$ and the outer radius $R$. 
We say that a finite Borel measure $\mu$ is weakly boundary thin (WBT) (with exponent $\b>0$) at  the point $x$ if 
$$
\lim_{r\to 0}\frac{\mu\(A_\mu^\b(x,r)\)}{\mu(B(x,r))}=0,
$$ 
where we denote
$$
A_\mu^\b(x,r):=A\(x;r-\mu(B(x,r))^\b,r+\mu(B(x,r))^\b\).
$$
This is a version of the problem of thin annuli, one that is notoriously challenging in dealling with the issue of relating  dynamical and geometric properties, and which is particularly acute in the contexts of escape rates and return rates. Due to the breakthrough of \cite{Pawelec-Urbanski-Zdunik},
where some strong versions of the thin annuli properties are proved, we have been able in the current paper to prove (WBT) for almost all points, which is reflected in both  Theorem~\ref{t3_2016_05_27} and Corollary~\ref{t2_2016_05_27}. 

\sp In the case of finite alphabets $E$ we have the following two results.

\bthm\label{t1fp83_Finite} 
Let $\cS=\{\phi_e\}_{e\in E}$ be a primitive conformal GDMS with a finite alphabet $E$ acting in the space $\R^d$, $d\ge 1$. Assume that either $d=1$ or that the system $\cS$ is geometrically irreducible. Let $\phi:E_A^\infty\to\R$ be a H\"older continuous potential. As usual, denote its equilibrium/Gibbs state by $\mu_\phi$. Let $z\in J_\cS$ be arbitrary. If either $z$ is 

\begin{itemize}
\item[(a)] not pseudo-periodic, 

or 

\item[(b)] uniquely periodic, it belongs to $\Int X$ (and $z=\pi(\xi^\infty)$ for a (unique) irreducible word $\xi\in E_A^*$), and $\phi$ is the amalgamated function of a summable H\"older continuous system of functions, 
\end{itemize}
then, 
\beq\label{1fp83B}
\begin{aligned}
\lim_{\ep\to 0}\frac{\un R_{\cS,\phi}(B(z,\ep))}{\mu_\phi\circ\pi_{\cS}^{-1}(B(z,\ep))}
&=\lim_{\ep\to 0}\frac{\ov R_{\cS,\phi}(B(z,\ep))}{\mu_\phi\circ\pi_{\cS}^{-1}(B(z,\ep))} =\\
&=d_\phi(z)
:=\begin{cases}
1 \  &\text{{\rm if (a) holds}}   \\
1-\exp\({S_p\phi(\xi)}-p\P(\phi)\) &\text{{\rm  if (b) holds}},
\end{cases}
\end{aligned}
\eeq
where in {\rm (b)}, $\{\xi\}=\pi_\cS^{-1}(z)$ and $p\ge 1$ is the prime period of $\xi$ under the shift map. 
\ethm

\bthm\label{t1fp83_Finite_B} 
Let $\cS=\{\phi_e\}_{e\in E}$ be a primitive conformal GDMS with a finite alphabet $E$ acting in the space $\R^d$, $d\ge 1$. Assume that either $d=1$ or that the system $\cS$ is geometrically irreducible. Let $z\in J_\cS$ be arbitrary. If either $z$ is 

\begin{itemize}
\item[(a)] not pseudo-periodic
or else 
\item [(b)] uniquely periodic and belongs to $\Int X$ (and $z=\pi(\xi^\infty)$ for a (unique) irreducible word $\xi\in E_A^*$).
\end{itemize}
Then 
\beq\label{1had31}
\lim_{r\to 0}\frac{\HD(J_\cS)-\HD(K_z(r))}{\mu_b\(\pi^{-1}(B(z,r))\)}=
\begin{cases}
1/\chi_{\mu_b} \  &\text{ if } \ (a) \  \text{ holds } \\
\(1-|\phi_\xi'(z)|\)/\chi_{\mu_b}
\  &\text{ if } \ (b) \  \text{ holds }.
\end{cases}
\eeq
\ethm

\sp\fr For these two theorems the two Thin Annuli Properties, Theorem~\ref{t2wbt11} and Theorem~\ref{t1wbt15}, were also instrumental. With having both Theorem~\ref{t1fp83_Finite} and Theorem~\ref{t1fp83_Finite_B} proved we have fully recovered the results of \cite{FP}.

\sp As we have already explained, our next goal in this paper is to get the existence of escape rates in the sense of \eqref{1_2016_06_27} and \eqref{2_2016_06_27} for all topologically exact piecewise smooth maps of the interval $[0,1]$, many rational functions of the Riemann sphere $\oc$ with degree $\ge 2$, and  a vast class of transcendental meromorphic functions from $\C$ to $\oc$. In order to do this we employ two principle tools. The first is formed by the  escape rates results, described above in detail,  for the class of all countable alphabet conformal graph directed Markov systems. The second is a method based on the first return (induced) map developed in Section~\ref{FRM}, Section~\ref{FRMERI}, and Section~\ref{FRMERII}. This method closely relates the escape rates of the original map and the induced map. It turns out that for the above mentioned classes of systems one can find a set of positive measure which gives rise to a first return map which is isomorphic to a countable alphabet conformal IFS or full shift map; the task being highly 
non-trivial and technically involved.  But this allows us to conclude, for suitable systems,  the existence of escape rates in the sense of \eqref{1_2016_06_27} and \eqref{2_2016_06_27}. However, in order to reach this conclusion we need to know some non-trivial properties of the original systems. Firstly, that the tails of the first return time and the first entrance time decay exponentially fast, and secondly that the Large Deviation Property (LDP) of Section~\ref{FRMERI} holds. This in turn leads to Theorem~\ref{t1ld6}, a kind of Large Deviation Theorem.

\sp We shall now describe in some detail the above mentioned applications to (quite) general conformal systems. We start with one-dimensional systems. We consider the class of topologically exact piecewise $C^3$--smooth multimodal maps $T$ of the interval $I=[0,1]$ with non-flat critical points and uniformly expanding periodic points, the property commonly referred to as Topological Collet--Eckmann. Topological exactness means that for every non-empty subset $U$ of $I$ there exists an integer $n \geq 0$ such that $T^n(U) = I$. 
Furthermore, our multimodal map $T: I \to I$  is assumed to be tame, meaning that
$$
\overline{\PC(T)}\ne I,
$$
where
$$
\Crit(T):=\{c\in I:T'(c)=0\}
$$ 
is the critical set for $T$ and 
$$
\PC(T):=\bu_{n=1}^\infty T^n(\Crit(T)),
$$
is the postcritical set of $T$.
A familiar example would be the famous unimodal map $x \mapsto \lambda x (1-x)$ with $0 < \lambda < 4$ for which the critical point $1/2$ is not in its own omega limit set, for example where $\lambda$ is a Misiurewicz point.

The class of potentials, called acceptable in the sequel, is provided by all Lipschitz continuous functions $\psi:I\to\R$ for which 
$$
\sup(\psi) - \inf (\psi) < \h_{\rm top}(T).
$$

\fr The first escape rates theorem in this setting is this.

\bthm
Let $T: I \to I$ be a tame topologically exact Topological Collet--Eckmann map.
Let $\psi:  I \to \mathbb R$ be an acceptable potential.
Let $z \in I \sms \ov{\PC(T)}$ be a recurrent point. 
Assume that the equilibrium state $\mu_\psi$ is {\rm (WBT)} at $z$. Then
$$
\begin{aligned}
\lim_{\ep\to 0}\frac{\un R_{\mu_\psi}(B(z,\ep))}{\mu_\psi(B(z,\ep))}
&=\lim_{\ep\to 0}\frac{\ov R_{\mu_\psi}(B(z,\ep))}{\mu_\psi(B(z,\ep))}= \\
&=\begin{cases}
  &{\rm if} \ z \ \text{{\rm is not any periodic point of }} T, \\
1-\exp\(S_p\psi(z)-p\P(f,\psi)\) &{\rm if} \ z  \ \text{{\rm is a periodic point of }} T.  
\end{cases}
\end{aligned}
$$
\ethm

\fr We have used here the usual notation 
$$
S_p\psi(x) = \sum_{k=0}^{p-1} \psi(T^k(x))
$$ 
of Birkhoff's sums, and  $\P(f, \psi)$ denotes the topological pressure of the potential $\psi$ with respect to the dynamical system  $T: I \to I$. We have also the following.

\bthm\label{t1mi7}
Let $T: I \to I$ be a tame topologically exact Topological Collet--Eckmann map map.
Let $\psi:  I \to \mathbb R$ be an acceptable potential. 
Then
$$
\lim_{\ep\to 0}\frac{\un R_{\mu_\psi}(B(z,\ep))}{\mu_\psi(B(z,\ep))}=
\lim_{\ep\to 0}\frac{\ov R_{\mu_\psi}(B(z,\ep))}{\mu_\psi(B(z,\ep))}=1
$$
for $\mu_\psi$--a.e. point $z\in I$. 
\ethm

\sp In order to address formula \eqref{2_2016_06_27} in this context we need a stronger assumption on the map $T:I\to I$. Our multimodal map $T: I \to I$ is said to be subexpanding if
$$
\hbox{\rm Crit}(T) \cap \overline{\PC(T)} = \emptyset.
$$
It is evident that each subexpanding map is tame and it is not hard to see that the subexpanding property entails being Topological Collet--Eckmann.
It is well known that in this case there exists a unique Borel probability $T$-invariant measure $\mu$
absolutely continuous with respect to Lebesgue measure $\lambda$.
In fact, $\mu$ is equivalent to $\lambda$ and (therefore) has full topological
support.  It is ergodic, even $K$-mixing, has Rokhlin's natural extension metrically isomorphic to some two sided Bernoulli shift. The Radon--Nikodym derivative $\frac{d\mu}{d\lam}$ is uniformly bounded above and separated from zero on the complement of every fixed neighborhood of $\ov{\PC(T)}$. We prove in this setting the following.

\begin{thm}\label{t2mi8}
Let $T: I \to I$  be a topologically exact multimodal subexpanding map.
Fix $\xi \in I \backslash \overline{\PC(T)}$. Assume that the parameter $1$ is powering at $\xi$ with respect to the conformal \GDS \ $\cS_T$ defined in Section~\ref{Interval_Maps}. Then the following limit exists, is finite, 
and positive:
$$
\lim_{r \to 0}  \frac{1 - \HD(K_\xi(r))}{\mu(B(\xi,r))}.
$$ 
\end{thm}

\sp 

\begin{thm}\label{t2mi8+1}
If $T: I \to I$  is a topologically exact multimodal subexpanding map, then for Lebesgue--a.e. point $\xi\in I\sms \ov{\PC(T)}$ the following limit exists, is finite and positive:
$$
\lim_{r \to 0}  \frac{1 - \HD(K_\xi(r))}{\mu(B(\xi,r))}.
$$ 
\end{thm}

\sp We now turn to complex one-dimensional maps. Let $f:\oc\to\oc$ be a rational map of the Riemann sphere with degree $\deg(f)\ge 2$. 
The sets $\Crit(f)$ and $\PC(f)$ have the same meaning as for the multimodal maps of the interval $I$. Let $\psi:\oc\to\R$ be a H\"older continuous function. Following \cite{DU} we say that $\psi:\oc\to\R$ has a pressure gap if
\beq\label{1_2016_07_07}
n\P(f, \psi)-\sup\(\psi_n\)>0
\eeq
for some integer $n\ge 1$. It was proved in \cite{DU} that there  exists a  unique equilibrium state $\mu_\psi$ for such  $\psi$. Some more ergodic properties of $\mu_\psi$ were established there, and a fairly extensive account of them was provided in \cite{SUZ_I}. 
For example, if $\psi=0$ then $\P(f,0)=\log\deg(f) > 0$ is the topological entropy of $f$ and the condition automatically holds. More generally, it always holds whenever
$$
\sup(\psi)-\inf(\psi)<\htop(f)\,(=\log\deg(f)).
$$
We would like to also add that \eqref{1_2016_07_07} always holds (with all $n\ge 0$ sufficiently large) if the function $f:\oc\to\oc$ restricted to its Julia set is expanding (also frequently referred to as hyperbolic). 
This is the best understood and the easiest to deal with class of rational functions. The rational map $f:\oc\to\oc$ is said to be expanding if the restriction  $f|_{J(f)}: J(f) \to J(f)$ satisfies 
\beq\label{5_2016_07_07_B}
\inf\{|f'(z)|:z\in J(f)\} > 1
\eeq
or, equivalently, 
\beq\label{6_2016_07_07_B}
|f'(z)|>1
\eeq
for all $z\in J(f)$. Another, topological, characterization of expandingness is the following.

\bfact
A rational function $f:\oc\to\oc$ is expanding if and only if 
$$
J(f)\cap\ov{\PC(f)}=\es.
$$
\efact
It is immediate from this characterization that all the polynomials $z\mapsto z^d$, $d\ge 2$, are expanding along with their small perturbations $z\mapsto z^d+\e$; in fact expanding rational functions are commonly believed to form the vast majority amongst all rational functions.

Being a tame rational function and Topological Collet--Eckmann both mean the same as in the setting of multimodal interval maps. Nowadays this property is somewhat more frequently used in its equivalent form of exponential shrinking (see \eqref{2_2016_07_07}) (ESP), and we this follow tradition. All expanding functions are tame and (ESP). Finally, as in the context of interval maps, we have the following. 

\bthm\label{t1ma7}
Let $f:\oc\to\oc$ be a tame rational function having the exponential shrinking property {\rm (ESP)}. Let $\psi:\oc\to\R$ be a H\"older continuous potential with pressure gap. Let $z\in J(f)\sms \ov{\PC(f)}$ be recurrent. Assume that the equilibrium state $\mu_\psi$ is {\rm (WBT)} at $z$. Then
$$
\begin{aligned}
\lim_{\ep\to 0}\frac{\un R_{\mu_\psi}(B(z,\ep))}{\mu_\psi(B(z,\ep))} 
&=\lim_{\ep\to 0}\frac{\ov R_{\mu_\psi}(B(z,\ep))}{\mu_\psi(B(z,\ep))} \\
&=
\begin{cases}
1 &{\rm if} \ z \ \text{{\rm is not a periodic point for }} f, \\
1-\exp\(S_p\psi(z)-p\P(f,\psi)\) &{\rm if} \ z  \ \text{{\rm is a periodic point of }} f.
\end{cases}
\end{aligned}
$$
\ethm

\sp

\bcor\label{t1ma7B}
Let $f:\oc\to\oc$ be a tame rational function having the exponential shrinking property {\rm (ESP)} whose Julia set $J(f)$ is geometrically irreducible. If $\psi:\oc\to\R$ is a H\"older continuous potential with pressure gap, then  
$$
 \lim_{\ep\to 0}\frac{\un R_{\mu_\psi}(B(z,\ep))}{\mu_\psi(B(z,\ep))}
=\lim_{\ep\to 0}\frac{\ov R_{\mu_\psi}(B(z,\ep))}{\mu_\psi(B(z,\ep))}
=1 
$$
for $\mu_\psi$--a.e. $z\in J(f)$.
\ecor

\sp As for the case of maps of an interval,  in order to establish formula \eqref{2_2016_06_27} in this context we need a stronger assumption on the rational map $f:\oc\to \oc$. Because the Julia set need not be equal to $\oc$ (and usually it is not) the definition of subexpanding rational functions is somewhat more involved, see Definition~\ref{d1ma8}.
It is evident that each subexpanding map is tame and it is not hard to see that being subexpanding entails also being Topological Collet--Eckmann. All expanding functions are necessarily  subexpanding.

\bthm\label{t2ma8}
Let $f:\oc\to\oc$ be a subexpanding rational function of degree $d\ge 2$. Fix $\xi\in J(f)\sms \ov{\PC(f)}$. Assume that the measure $\mu_h$ is {\rm (WBT)} at $\xi$ and the parameter $h:=\HD(J(f))$ is powering at $\xi$ with respect to the conformal \GDS \ $\cS_f$ defined in Section~\ref{Rational Functions}. Then the following limit exists, is finite and positive:
$$
\lim_{r\to 0}\frac{\HD(J(f))-\HD(K_\xi(r))}{\mu_h(B(\xi,r))}.
$$
\ethm

\bthm\label{t2ma8B}
If $f:\oc\to\oc$ be a subexpanding rational function of degree $d\ge 2$ whose Julia set $J(f)$ is geometrically irreducible, then for $\mu_h$--a.e. point $\xi\in J(f)\sms \ov{\PC(f)}$ the following limit exists, is finite and positive:
$$
\lim_{r\to 0}\frac{\HD(J(f))-\HD(K_\xi(r))}{\mu_h(B(\xi,r))}.
$$
\ethm

\brem 
We would like to note that if the rational function $f:\oc\to\oc$ is expanding (or hyperbolic as such functions are frequently called), then it is subexpanding and each H\"older continuous potential has a pressure gap. In particular all four theorems above pertaining to rational functions hold for it. 
\erem

\fr In both theorems $\mu_h$ is a unique (ergodic) Borel probability $f$--invariant measure on $J(f)$ equivalent to $m_h$, a unique $h$-conformal measure $m_h$ on $J(f)$ for $f$. Th was proved studied in \cite{U2}, comp. also   \cite{U1}.

\sp The last applications are in the realm of transcendental meromorphic functions. There is a large class of such systems, introduced in \cite{MayUrbETDS} and \cite{MayUrb10} for which it is possible to build (see these two papers) a fairly rich and complete account of thermodynamic formalism. Applying again  our escape rates theorems for conformal graph directed Markov systems, one prove in this setting four main theorems which are analogous of  those for the multimodal maps of an interval and rational functions of the Riemann sphere. These can be found with complete proofs in Section~\ref{Transcendental_Functions}, the last section of our manuscript. 

\

\part{Singular Perturbations of Countable Alphabet Symbol Space Classical Perron--Frobenius Operators}

\sp

\section{The classical original Perron-Frobenius Operator, \\ Gibbs and Equilibrium States, \\ Thermodynamic Formalism; Preliminaries}

In this section we present some notation and basic results on Thermodynamic Formalism as developed in \cite{GDMS}, see also \cite{MU_lms} and \cite{CTU}. It will be the base for our subsequent work.

Let $E$ be a countable, either finite or infinite, set, called in the sequel the alphabet. Let $A:E\times E\to \{0,1\}$ an arbitrary matrix. For every integer $n\ge 0$ let
$$
E_A^n:=\{\om\in E^n:A_{\om_j\om_{j+1}}=1 \ \forall\, 0\le j\le n-1\},
$$
denote the finite words of length $n$, 
let
$$
E_A^\infty:=\{\om\in E^\N:A_{\om_j\om_{j+1}}=1\ \forall\, j\ge 0\},
$$
denote the space of one-sided infinite sequences, 
and let
$$
E^*:=\bu_{n=0}^\infty E^n,  \  \text{ and }  \  E_A^*:=\bu_{n=0}^\infty E_A^n.
$$
be set of all finite strings of words, the former being without restrictions and the latter being called $A$-admissible. 
  
We call elements of $E_A^*$ and $E_A^\infty$ $A$-admissible. The matrix $A$ is called  finitely primitive
(or aperiodic) if there exist an integer $p\ge 0$ and a finite set $\La\sbt E^p$ such that for all $i,j\in E$ there exists $\om\in\La$ such that $i\om j\in E_A^*$. Denote by $\sg:E_A^\infty\to E_A^\infty$ the shift map, i. e. the map uniquely defined by the property that 
$$ 
\sg(\om)_n:=\om_{n+1}
$$
for every $n\ge 0$. Fixing $\th\in(0,1)$ endow $E_A^\infty$ with the standard metric
$$
d_\th(\om,\tau):=\th^{|\om\wedge\tau|},
$$
where for every $g\in E^*\cup E^{\N}$, $|\g|$ denotes the length of $\g$, i. e. the unique $n\in \N\cup \{\infty\}$ such that $\g\in E^n$. Given $0\le k\le |\g|$, we set 
$$
\g|_k:=\g_1\g_2\ld\g_k.
$$
We then also define
$$
[\g]:=\{\om\in E_A^\infty:\om|_n=\g\},
$$
and call $[\g]$ the (initial) cylinder generated by $\g$.
Let $\vp:E_A^\infty\to\R$ be a H\"older continuous function, called in the sequel potential. We assume that $\phi$ is summable, meaning that
$$
\sum_{e\in E}\exp\(\sup(\vp|_{[e]}\)<+\infty.
$$
It is well known (see \cite{GDMS} or \cite{MU_Israel}) that the following limit 
$$
\P(\vp):=\lim_{n\to\infty}\frac1n\log\sum_{\om\in E_A^n}\exp\(\sup(\vp|_{[\om]}\)
$$
exists. It is called the topological pressure of $\vp$. It was proved in \cite{MU_Israel} (compare \cite{GDMS}) that there exists a unique shift-invariant Gibbs/equilibrium measure $\mu_\vp$ for the potential $\vp$. The Gibbs property means that 
$$
C_\vp^{-1}\le\frac{\mu_\vp([\om|_n])}{\exp\(\vp_n(\om)-\P(\vp)n\)}\le C_\vp
$$
with some constant $C_\vp\ge 1$ for every $\om\in E_A^\infty$ and every integer $n\ge 1$, where here and in the sequel
$$
%S_n(g)=
g_n(\om):=\sum_{j=0}^{n-1}g\circ \sg^j
$$
for every function $g:E_A^\infty\to\C$. For the measure $\mu_\phi$ being an equilibrium state for the potential $\vp$ means that
$$
\h_{\mu_\phi}(\sg)+\int_{E_A^\infty}\phi d\mu_\phi=\P(\phi).
$$
It has been proved in \cite{GDMS} that
$$
\h_{\mu}(\sg)+\int_{E_A^\infty}\phi d\mu<\P(\phi)
$$
for any other Borel probability $\sg$-invariant measure $\mu$ such that $\int\phi d\mu>-\infty$. For every bounded function $g:E_A^\infty\to\R$ define $\Lp(g):E_A^\infty\to\R$ as follows
$$
\Lp(g)(\om):=\sum_{e\in E:A_{e\om_0}=1}g(e\om)\exp(\phi(e\om)).
$$
Then $\Lp(g)$ is bounded again, and we get by induction that
$$
\Lp^k(g)(\om)
:=\sum_{\tau\in E_A^k:A_{\tau_{k-1}\om_0}=1}g(\tau\om)\exp(\phi_k(\tau\om)).
$$
Let $C_b(A)$ be the Banach space of all complex-valued bounded continuous functions defined on $E_A^\infty$ endowed with the supremum norm $||\cdot||_\infty$. Let $\H_\th^b(A)$ be its vector subspace consisting of all Lipschitz continuous functions with respect to the metric $d_\th$. Equipped with the norm
\beq\label{2poll29}
\H_\th(g):=||g||_\infty+v_\th(g),
\eeq
where $v_\th(g)$ is the least constant $C\ge 0$ such that
\beq\label{1poll29}
|g(\om)-g(\tau)|\le Cd_\th(\om,\tau),
\eeq
whenever $d_\th(\om,\tau)\le \th$ (i. e. 
%$|\om\wedge\tau|\ge 1$
$\omega_0 = \tau_0$), the vector space $\H_\th^b(A)$ becomes a Banach space. It is easy to see that the operator $\Lp$ preserves both Banach spaces $C_b(A)$ (as we have observed some half-page ago) and $\H_\th^b(A)$ and also  acts continuously on each of them. The adjective ``original'' indicates that we do not deal with its perturbations while ``classical'' refers to standard Banach spaces $C_b(A)$ and $\H_\th^b(A)$.
The following theorem, describing fully the spectral properties of $\pf_\phi$,  has been proved in \cite{GDMS} and \cite{MU_Israel}.

\bthm\label{t1poll31}
If $A:E\times E\to\{0,1\}$ is finitely primitive and $\phi\in \H_\th^b(A)$, then
\begin{itemize}
\item[(a)] The spectral radius of the operator $\Lp$ considered as acting either on $C_b(A)$ or $\H_\th^b(A)$ is in both cases equal to $e^{\P(\phi)}$.

\item[(b)] In both cases of (a) the number $e^{\P(\phi)}$ is a simple eigenvalue of $\Lp$ and there exists corresponding to it an everywhere positive eigenfunction $\rho_\phi\in \H_\th^b(A)$ such that $\log\rho_\phi$ is a bounded function.

\item[(c)] The reminder of the spectrum of the operator $\Lp:\H_\th^b(A)\to\H_\th^b(A)$ is contained in a closed disk centered at $0$ with radius strictly smaller than $e^{\P(\phi)}$. In particular, the operator $\Lp:\H_\th^b(A)\to\H_\th^b(A)$ is quasi-compact.

\item[(d)] There exists a unique Borel probability measure $m_\phi$ on $E_A^\infty$ such that 
$$
\Lp^*m_\phi=e^{\P(\phi)}m_\phi,
$$
where $\Lp^*:C_b^*(A)\to C_b^*(A)$, is the operator dual to $\Lp$ acting on the space of all bounded linear functionals from $C_b(A)$ to $\C$.

\item[(e)] If $\rho_\phi:E_A^\infty\to (0,\infty)$ is normalized so that $m_\phi(\rho_\phi)=1$, then $\rho_\phi m_\phi=\mu_\phi$, where, we recall, the latter is the unique shift-invariant Gibbs/equilibrium measure for the potential $\vp$.

\item[(e)] The Riesz projector $Q_1:\H_\th^b(A)\to\H_\th^b(A)$, corresponding to the eigenvalue $e^{\P(\phi)}$, is given by the formula 
$$
Q_1(g)=e^{\P(\phi)}m_\phi(g)\rho_\phi.
$$
\end{itemize}
\ethm

\fr If we multiply the operator $\Lp:\H_\th^b(A)\to\H_\th^b(A)$ by $e^{-\P(\phi)}$ and conjugate it via the linear homeomorphism 
$$
g\mapsto \rho_\phi^{-1}g,
$$
then the resulting operator $T:\H_\th^b(A)\to\H_\th^b(A)$ has the same properties, described above, as the operator $\Lp$, with $e^{\P(\phi)}$ replaced by $1$, 
$\rho_\phi$ by the function $\1$ which is identically equal to $1$, and $m_\phi$ replaced by $\mu_\phi$. Since in addition it is equal to $\pf_{\tilde\phi}:\H_\th^b(A)\to\H_\th^b(A)$ with 
$$
\tilde\phi:=\phi-\P(\phi)+\log\rho_\phi,
$$
we will frequently deal with the operator $\pf_{\tilde\phi}$ instead of $\Lp$, exploiting its useful  property
$$
\pf_{\tilde\phi}\1=\1.
$$
We will occasionally refer to $\pf_{\tilde\phi}$ as fully normalized. Sometimes, we will only need the  semi-normalized operator $\Lp$ given by the formula
$$
\hat{\pf}_\phi:=e^{-\P(\phi)}\Lp.
$$
It essentially differs from only by having $e^{\P(\phi)}$ replaced by $1$. Now we bring up two standard well-known technical facts about the above concepts. These can be found for example in \cite{GDMS}.
 
\blem\label{l1_2014_08_29}
There exists a constant $M_\phi\in (0,+\infty)$ such that 
$$
|\phi_k(\om)-\phi_k(\tau)|\le M_\phi\th^m
$$
for all integers $k, m\ge 1$, and all words $\om, \tau\in E_A^\infty$ such that 
$\om|_{k+m}=\tau|_{k+m}$.
\elem 

\blem\label{l2_2014_08_29}
With the hypotheses of Lemma~\ref{l1_2014_08_29} and increasing the constant $M_\phi$ if necessary, we have that
$$
\big|1-\exp\(\phi_k(\g\om)-\phi_k(\g\tau)\)\big|
\le M_\phi|\phi_k(\om)-\phi_k(\tau)|.
$$
\elem

\section{Non-standard original Perron-Frobenius Operator; \\ Definition and first technical Results}\label{PFOriginal}

We keep the setting of the previous section. We still deal with the original operator $\pf_\phi$ but we let it act on a different non-standard Banach space $\Ba_\th$ defined below. This space is more suitable for consideration of perturbations of $\pf_\phi$.
 
Given a function $g\in L^1(\mu_\phi)$ and an integer $m\ge 0$, we define the function $\osc_m(g):E_A^\infty\to [0,\infty)$ by the following formula:
\beq\label{1fp61}
\osc_m(g)(\om):=\ess\sup\{|g(\a)-g(\b)|:\a,\b\in[\om|_m]\}
\eeq
and 
$$
\osc_0(g):= \hbox{\rm esssup}(g)-\hbox{\rm essinf}(g).
$$
We further define:
\beq\label{2fp61}
|g|_\th:=\sup_{m\ge 0}\{\th^{-m}||\osc_m(g)||_1\},
\eeq
where $|\cdot|$ denotes the $L^1$-norm with respect to the measure $\mu_\phi$. Note the subtle difference between this definition and the  analogous one, which motivated us, from \cite{FP}. Therein in the analogue of formula \eqref{2fp61} the supremum is taken over integers $m\ge 1$ only. Including $m=0$ causes some technical difficulties, particularly the (tedious) part of the  proof of Lemma~\ref{l1fp63} for the integer $m=0$.  However, without the case $m=0$ we would not be able to prove Lemma~\ref{l1fp61}, in contrast to  the finite alphabet case of \cite{FP}, which is indispensable for our entire approach. The, previously  announced, non-standard (it even depends on the dynamics -- via $\mu_\phi$) Banach space is defined as follows:
$$
\Ba_\th:=\{g\in L^1(\mu_\phi):|g|_\th<+\infty\}
$$
and we denote
\beq\label{3fp61}
||g||_\th:=||g||_1+|g|_\th.
\eeq
Of course $\Ba_\th$ is a vector space and the function 
\beq\label{4fp61}
\Ba_\th\ni g\mapsto ||g||_\th
\eeq
is a norm on $\Ba_\th$. This is the non-standard Banach space we will be working with throughout the whole manuscript. We shall prove the following.

\blem\label{l1fp61}
If $g\in\Ba_\th$, then $g$ is essentially bounded and 
$$
||g||_\infty\le  ||g||_\th.
$$
\elem
\begin{proof}
For all $\om\in E_A^\infty$, we have
$$
\begin{aligned}
|g(\om)|
&\le \Big|\int_{E_A^\infty}g\, d\mu_\phi+\osc_0(g)(\om)\Big|
=   \Big|\int_{E_A^\infty}g\,  d\mu_\phi+\int_{E_A^\infty}\osc_0(g)\, d\mu_\phi\Big| \\
&\le \int_{E_A^\infty}|g|\,  d\mu_\phi+||\osc_0(g)||_1 \\
&\le ||g||_\th.
\end{aligned}
$$
The proof is complete.
\end{proof}

\fr From now on, unless otherwise stated, we assume that the potential $\vp:E_A^\infty\to\R$ is normalized (by adding a constant and a coboundary) so that 
$$
\pf_\phi\1=\1.
$$
For ease of notation we also abbreviate $\pf_\phi$ to $\pf$. We shall prove the following.

\blem\label{l1fp63}
There exists a constant $C>0$ for every integer $k\ge 0$ and every $g\in\Ba_\th$, we have 
$$
|\pf^kg|_\th\le C(\th^k|g|_\th+||g||_1).
$$
\elem
\begin{proof}
For every $e\in E$ let
$$
E_A^k(e):=\{\g\in E_A^k:A_{\g_k e}=1\}.
$$
Fix first an integer $m\ge 1$ and then $\om,\tau\in E_A^\infty$ such that $\om|_m=\tau|_m$. Using Lemmas~\ref{l1_2014_08_29} and \ref{l2_2014_08_29}, we then get
$$
\begin{aligned}
|\pf^kg(\om)-\pf^kg(\tau)|
&\le \sum_{\g\in E_A^k(\om_1)}e^{\phi_k(\g\om)}|g(\g\om)-g(\g\tau)| +
     \sum_{\g\in E_A^k(\om_1)}|g(\g\tau)|\Big|e^{\phi_k(\g\om)}-e^{\phi_k(\g\tau)}\Big|\\
&\le \sum_{\g\in E_A^k(\om_1)}\osc_{k+m}(g)(\g\om)e^{\phi_k(\g\om)}\ + \\
&  \  \  \  \  \  \  \  \  \  \  + \sum_{\g\in E_A^k(\om_1)}|g(\g\tau)|e^{\phi_k(\g\tau)}\big|1-\exp\(\phi_k(\g\om)-\phi_k(\g\tau)\)\big| \\
&\le \sum_{\g\in E_A^k(\om_1)}\osc_{k+m}(g)(\g\om)e^{\phi_k(\g\om)}\ + \\
&  \  \  \  \  \  \  \  \  \  \  + \sum_{\g\in E_A^k(\om_1)}|g(\g\tau)|e^{\phi_k(\g\tau)}
   M_\phi|\phi_k(\g\om)-\phi_k(\g\tau)| \\
&\le \pf^k(\osc_{k+m}(g))(\om)+
     M_\phi^2\th^m\sum_{\g\in E_A^k(\om_1)}\(|g(\g\om)|+
       \osc_{k+m}(g)(\g\om)\)e^{\phi_k(\g\om)} \\
&\le \pf^k(\osc_{k+m}(g))(\om)+M_\phi^2\th^m\pf^k(|g|)(\om)+
     M_\phi^2\th^m\pf^k(\osc_{k+m}(g))(\om) \\
&\le (1+M_\phi^2)\pf^k(\osc_{k+m}(g))(\om)+M_\phi^2\th^m\pf^k(|g|)(\om)
\end{aligned}
$$
Hence, 
$$
\osc_m(\pf^kg)(\om)
\le (1+M_\phi^2)\pf^k(\osc_{k+m}(g))(\om)+M_\phi^2\th^m\pf^k(|g|)(\om)
$$
Integrating against the measure $\mu_\phi$, this yields
\beq\label{1fp65}
\begin{aligned}
\th^{-m}||\osc_m(\pf^kg)||_1
&\le (1+M_\phi^2)\th^{-m}\int_{E_A^\infty}\pf^k(\osc_{k+m}(g))\,d\mu_\phi +
       M_\phi^2\int_{E_A^\infty}\pf^k(|g|)\,d\mu_\phi \\
&=   (1+M_\phi^2)\th^{-m}\int_{E_A^\infty}\osc_{k+m}(g)\,d\mu_\phi +
       M_\phi^2\int_{E_A^\infty}|g|\,d\mu_\phi. \\
&\le (1+M_\phi^2)\th^k|g|_\th+ M_\phi^2||g||_1 \\
&\le (1+M_\phi^2)(\th^k|g|_\th+||g||_1).
\end{aligned}       
\eeq
Some separate considerations are needed if $m=0$. However, we note that it would require no special treatment in the case of a full shift, i. e. when the incidence matrix $A$ consists of $1$s only. Let $p\ge 1$ be the value in the definition of finite primitivity of the matrix $A$. Replacing $p$ by a sufficiently large integral multiple, we will have that the set
$$
E_A^p(a,b):=\{\a\in E_A^p:a\a b\in E_A^*\}
$$
consisting of words of length $p$ prefixed by $a$ and suffixed by $b$
is non-empty for all $a, b\in E$ and it is countable infinite if the alphabet $E$ is infinite. For every function $h:E_A^\infty\to\R$ and every finite word $\g\in E_A^*$ with associated cylinder $[\g]$ consisting of all infinite sequences beginning with $\gamma$ let $\hat h(\g)\in\R$ be a number with the following two properties:
\sp
\begin{itemize}
\item[(a)] $\hat h(\g)\in\ov{h([\g])}$
and
\item[(b)] $|\hat h(\g)|=\inf\{|h(\rho)|:\rho\in[\g]\}$.
\end{itemize}

\sp\fr Let us introduce the following two functions:
$$
\De_1\pf^{k+p}(g)(\rho)
:=\sum_{|\g|=k}\sum_{\a\in E_A^p(\g_k,\rho_1)}
  \(g(\g\a\rho)e^{\phi_k(\g\a\rho)}e^{\phi_p(\a\rho)}
   -\hat g(\g)e^{\hat\phi_k(\g)}e^{\phi_p(\a\rho)}\)
$$
and 
$$
\De_2\pf^{k+p}(g)(\om,\tau)
:=\sum_{|\g|=k}\hat g(\g)e^{\hat\phi_k(\g)}
  \lt(\sum_{\a\in E_A^p(\g_k,\om_1)}e^{\phi_p(\a\om)}
  -\sum_{\b\in E_A^p(\g_k,\tau_1)}e^{\phi_p(\b\tau)}\rt).
$$
We then have 
\beq\label{1fp64}
\pf^{k+p}(g)(\om)-\pf^{k+p}(g)(\tau)
=\De_1\pf^{k+p}(g)(\om)+\De_2\pf^{k+p}(g)(\om,\tau)-\De_1\pf^{k+p}(g)(\tau).
\eeq
We will estimate the absolute value of each of these three summands in terms of $\om$ only (i. e. independently of $\tau$) and then we will integrate against the measure $\mu_\phi$. First:
\beq\label{1fp64A}
\begin{aligned}
|\De_1\pf^{k+p}(g)(\rho)|
&\le \sum_{|\g|=k}\sum_{\a\in E_A^p(\g_k,\rho_1)}
|g(\g\a\rho)e^{\phi_k(\g\a\rho)}-\hat g(\g)e^{\hat\phi_k(\g)}|e^{\phi_p(\a\rho)}\\
&\le \sum_{|\g|=k}\sum_{\a\in E_A^p(\g_k,\rho_1)}\!\!\!\!\!
     \lt(|g(\g\a\rho)-\hat g(\g)|e^{\phi_{k+p}(\g\a\rho)}
    +|e^{\phi_k(\g\a\rho)}-e^{\hat\phi_k(\g)}|
     \cdot|\hat g(\g)|e^{\phi_p(\a\rho)}\rt) \\
&\le \sum_{|\g|=k}\sum_{\a\in E_A^p(\g_k,\rho_1)}
\lt(\osc_k\(g|_{[\g]}\)e^{\phi_{k+p}(\g\a\rho)}
    +M_\phi e^{\phi_k(\g\a\rho)}e^{\phi_p(\a\rho)}|\hat g(\g)|\rt) \\
&\le \sum_{|\g|=k}\sum_{\a\in E_A^p(\g_k,\rho_1)}\!\!\!\!\!
     \osc_k\(g|_{[\g]}\)e^{\phi_{k+p}(\g\a\rho)}
     +M_\phi\sum_{|\g|=k}\sum_{\a\in E_A^p(\g_k,\rho_1)}
     |g(\g\a\rho)|e^{\phi_{k+p}(\g\a\rho)} \\
&=\pf^{k+p}(\osc_k(g))(\rho)+M_\phi\pf^{k+p}(|g|)(\rho),
\end{aligned}
\eeq
with some appropriately large constant $M_\phi$. Plugging into the above inequality, 
$\rho=\om$, this gives
\beq\label{2fp64A}
|\De_1\pf^{k+p}(g)(\om)|
\le \pf^{k+p}(\osc_k(g))(\om)+M_\phi\pf^{k+p}(|g|)(\om).
\eeq
Now notice that because of our choice of $p\ge 1$ there exists a number $Q\ge 1$ and for every $e\in E$ there exists an at most $Q$-to-$1$ function $f_e:E_A^p(e,\tau_1)\to E_A^p(e,\om_1)$ (can be chosen to be a bijection if the alphabet $E$ is infinite). So, plugging in turn $\rho=\tau$ to \eqref{1fp64A}, we get
\beq\label{2fp64B}
\begin{aligned}
|\De_1&\pf^{k+p}(g)(\tau)| \le \\
&\le \sum_{|\g|=k}\sum_{\b\in E_A^p(\g_k,\tau_1)}\!\!\!\!\!
     \osc_k\(g|_{[\g]}\)e^{\phi_{k+p}(\g\b\tau)}
     +M_\phi\sum_{|\g|=k}\sum_{\b\in E_A^p(\g_k,\tau_1)}
     |g(\g\b\tau)|e^{\phi_{k+p}(\g\b\tau)} \\
&\le M_\phi\sum_{|\g|=k}\sum_{\b\in E_A^p(\g_k,\tau_1)}\!\!\!\!\!
     \lt(\osc_k(g)(\g f_e(\b)\om)e^{\phi_{k+p}(\g f_e(\b)\om)}
    +M_\phi|\hat g(\g)|e^{\phi_{k+p}(\g f_e(\b)\om)}\rt)\\
&\le M_\phi\!\!\!\sum_{|\g|=k}\sum_{\b\in E_A^p(\g_k,\tau_1)}\!\!\!\!\!\!\!\!
     \osc_k(g)(\g f_e(\b)\om)e^{\phi_{k+p}(\g f_e(\b)\om)}
    +M_\phi\!\!\!\sum_{|\g|=k}\sum_{\b\in E_A^p(\g_k,\tau_1)}\!\!\!\!\!\!\!\!
      |g(\g f_e(\b)\om)|e^{\phi_{k+p}(\g f_e(\b)\om)}\\
&\le QM_\phi\lt(\sum_{|\g|=k}\sum_{\a\in E_A^p(\g_k,\om_1)}\!\!\!\!\!
     \osc_k(g)(\g\a\om)e^{\phi_{k+p}(\g\a\om)} 
     +M_\phi\sum_{|\g|=k}\sum_{\a\in E_A^p(\g_k,\om_1)}\!\!\!\!\!
     |g(\g\a\om)|e^{\phi_{k+p}(\g\a\om)}\rt)\\
&=QM_\phi\(\pf^{k+p}(\osc_k(g))(\om)+M_\phi\pf^{k+p}(|g|)(\om)\)
\end{aligned}
\eeq
with some appropriate constant $Q>0$. Turning to $\De_2\pf^{k+p}(g)$, we get
\beq\label{1fp64B}
\begin{aligned}
|\De_2\pf^{k+p}(g)(\om,\tau)|
&\le \sum_{|\g|=k}|\hat g(\g)|e^{\hat\phi_k(\g)}\lt(
    \sum_{\a\in E_A^p(\g_k,\om_1)}e^{\phi_p(\a\om)}
   +\sum_{\b\in E_A^p(\g_k,\tau_1)}e^{\phi_p(\b\tau)}\rt)\\
&\le \sum_{|\g|=k}|\hat g(\g)|e^{\hat\phi_k(\g)}\(\pf^p\1(\om)+\pf^p\1(\tau)\)\\
&=2\sum_{|\g|=k}|\hat g(\g)|e^{\hat\phi_k(\g)} \\
&\le 2M_\phi\sum_{|\g|=k}|g(\g\a(\g_k,\om_1)\om)|
     e^{\phi_{k+p}(\g\a(\g_k,\om_1)\om)}e^{-\phi_p(\a(\g_k,\om_1)\om)}\\
&\le 2M_\phi e^{-C_p}\sum_{|\g|=k}|g(\g\a(\g_k,\om_1)\om)|
     e^{\phi_{k+p}(\g\a(\g_k,\om_1)\om)}\\
&\le 2M_\phi e^{-C_p}\pf^{k+p}(|g|)(\om),
\end{aligned}
\eeq  
where $\a(\g_k,\om_1)$ is one, arbitrarily chosen, element from $\La$, a finite set witnessing finite primitivity of $A$, such that $\g\a(\g_k,\om_1)\in E_A^*$, and $C_p:=\min\{\inf\{\phi_p|_{[\a]}:\a\in\La\}>0$. Inserting now \eqref{1fp64B}, \eqref{2fp64B}, and \eqref{2fp64A} to \eqref{1fp64}, we get for all $\om, \tau\in E_A^\infty$ that 
$$
\big|\pf^{k+p}(g)(\om)-\pf^{k+p}(g)(\tau)\big|
\le C(\pf^{k+p}(\osc_k(g))(\om)+\pf^{k+p}(|g|)(\om))
$$
with some universal constant $C>0$. Integrating against the measure $\mu_\phi$, this gives
\beq\label{2fp64B+1}
\begin{aligned}
\th^{-0}||\osc_0(\pf^{k+p}(g))||_1
&\le C\lt(\int_{E_A^\infty}\pf^{k+p}(\osc_k(g))\,d\mu_\phi +\int_{E_A^\infty}\pf^{k+p}(|g|)\,d\mu_\phi\rt) \\
&=C\lt(\int_{E_A^\infty}\osc_k(g)\,d\mu_\phi+\int_{E_A^\infty}|g|\,d\mu_\phi\rt) \\
&\le C(\th^k|g|_\th+||g||_1) \\
&\le C\th^{-p}(\th^{k+p}|g|_\th+||g||_1).
\end{aligned}
\eeq 
Along with \eqref{1fp65} this gives that 
\beq\label{2fp65}
|\pf^kg|_\th\le C(\th^k|g|_\th+||g||_1)
\eeq
for all $k\ge p$ with some suitable constant $C>0$. Also, for every $0\le k\le p$ we have
$$
\begin{aligned}
|\pf^kg|_\th
&\le ||\pf^kg||_\th
\le \max\{||\pf||_\th^j:0\le j\le p\}||g||_\th \\
&\le \th^{-p}\max\{||\pf||_\th^j:0\le j\le p\}||g||_\th(\th^k||g||_\th) \\
&\le \th^{-p}\max\{||\pf||_\th^j:0\le j\le p\}||g||_\th(\th^k|g|_\th+||g||_1).
\end{aligned}
$$
Along with \eqref{2fp65} this finishes the proof.
\end{proof}

\sp

\section{Singular Perturbations of (original) Perron--Frobenius Operators I: Fundamental Inequalities}\label{SecSingPerturb}
This is the first section in which we deal with singular perturbations of the operator $\pf_\phi$. We work in the quite general setting  described below. We keep the same
non-standard Banach space $\Ba_\th$ but, motivated by \cite{FP}, we introduce an even more exotic norm $||\cdot||_*$, which depends even more on dynamics than $||\cdot||_\th$.

Passing to details, in this section we assume that $(U_n)_{n=0}^\infty$, a nested sequence of open subsets of $E_A^\infty$ is given, with the following properties:
\begin{itemize}
\item[(U0)] $U_0=E_A^\infty$,
\item[(U1)] For every $n\ge 0$ the open set $U_n$ is a (disjoint) union of cylinders all of which are of length $n$,
\item[(U2)] There exists $\rho\in (0,1)$  such that such that 
$$
\mu_\phi(U_n)\le \rho^n
$$
for all $n\ge 0$.
%\footnote{Originally $C=1$?}
\end{itemize}
Let $|\cdot|_*,\, ||\cdot||_*:\Ba_\th\to[0,+\infty]$ be the functions defined by respective formulas
$$
|g|_*:=\sup_{j\ge 0}\sup_{m\ge 0}\lt\{\th^{-m}\int_{\sg^{-j}(U_m)}|g|\,d\mu_\phi\rt\}
$$
and 
$$
||g||_*:=||g||_1+|g|_*.
$$
Without loss of generality assume from now on that $\th\in (\rho,1)$. We shall prove the following.

\blem\label{l1fp65}
For all $g\in \Ba_\th$, we have that
$$
||g||_*\le2||g||_\infty\le 2||g||_\th.
$$
\elem

\begin{proof}
By virtue of (U2), we get
$$
|g|_*
\le\sup_{m\ge 0}\big\{\th^{-m}\mu_\phi(U_m)||g||_\infty\big\}
\le \sup_{m\ge 0}\big\{\th^{-m}\rho^m||g||_\infty\big\}
=\sup_{m\ge 0}\{(\rho/\th)^m||g||_\infty\}
=||g||_\infty.
$$
Hence,
$$
||g||_*
=||g||_1+|g|_*
\le ||g||_\infty+||g||_\infty
=2||g||_\infty.
$$
Combining this  with Lemma~\ref{l1fp61} completes the proof.
\end{proof}

\fr In particular, this lemma assures us that $|\cdot|_*$ and $||\cdot||_*$, respectively, are a semi-norm and a norm on $\Ba_\th$. It is straightforward to check that $\Ba_\th$ endowed with the norm $||\cdot||_*$ becomes a Banach space. For all integers $k\ge 1$ and $n\ge 0$ let
\beq\label{1fp67}
\1_n^k:=\prod_{j=0}^{k-1}\1_{\sg^{-j}(U_n^c)}
=\prod_{j=0}^{k-1}\1_{U_n^c}\circ\sg^j.
\eeq
We also abbreviate 
$$
\1_n:=\1_n^1
$$
and set
$$
\1_n^c:=\1_{U_n}=\1-\1_n.
$$
Let $\pf_n:\Ba_\th\to\Ba_\th$ be defined by the formula
$$
\pf_n(g):=\pf(\1_n^1g).
$$
These, for $n\ge 0$, are our perturbations of the operator $\pf$. The difference 
$\pf-\pf_n$ in the supremum, or even $||\cdot||_\th$, norm can be quite large even for arbitrarily large $n$, however, as  Lemma~\ref{l2fp129} shows, the incorporation of the $||\cdot||_*$ norm makes this difference kind of small. The main result of this section is Proposition~\ref{p1fp80}, complemented by Proposition~\ref{p1fp131}, which describes in detail how well the spectral properties of the operator $\pf$ are preserved under perturbations $\pf_n$. Note that for every $k\ge 1$, we then have 
$$
\pf_n^k(g):=\pf^k(\1_n^kg).
$$
The results we now obtain, leading ultimately to Proposition~\ref{p1fp80} and Proposition~\ref{p1fp131}, stem from  Lemma~3.9 and Lemma~3.10 in \cite{FP}. We develop these and extend them to the case of infinite alphabets. Since the sets $U_n$ may, and in applications, will, consist of infinitely many cylinders (of the same length), we are cannot take advantage of good mixing properties of the symbol dynamical system $(\sg:E_A^\infty\to E_A^\infty, \mu\phi)$. We use instead the H\"older inequality, which also, as a by-product, simplifies some of the reasoning of \cite{FP}. In what follows, the last fragment, directly preceding Proposition~\ref{p1fp80}, and leading to verifying the requirements from Remark~3 in \cite{KL}, is particularly delicate and entirely different to \cite{FP}.

\blem\label{l1fp127}
For all integers $k\ge 1$ and $n\ge 0$ , we have
$$
||\pf_n^k||_*\le 1.
$$
\elem

\begin{proof}
Let $g\in L^1(\mu_\phi)$. Then,
\beq\label{FP3}
||\pf_n^k(g)||_1
=\int|\pf^k(\1_n^kg))|\,d\mu_\phi
\le \int\pf^k(|\1_n^kg|)\,d\mu_\phi
= \int|\1_n^kg|\,d\mu_\phi
\le ||g||_1.
\eeq
Also, for all integers $j, m\ge 0$, we have
$$
\begin{aligned}
\th^{-m}\int_{\sg^{-j}(U_m)}|\pf^k(\1_n^kg))|\,d\mu_\phi
&\le \th^{-m}\int_{\sg^{-j}(U_m)}\pf^k(|\1_n^kg|)\,d\mu_\phi
= \th^{-m}\int_{\sg^{-(j+1)}(U_m)}|\1_n^kg|\,d\mu_\phi \\
&\le \th^{-m}\int_{\sg^{-(j+1)}(U_m)}|g|\,d\mu_\phi \\
&\le |g|_*.
\end{aligned}
$$
Taking the supremum over $j$ and $m$ yields
$$
|\pf_n^k(g)|_*\le|g|_*.
$$
Combining this and \eqref{FP3} completes the proof.
\end{proof}

\blem\label{l2fp127}
For all integers $j, n\ge 0$ and for  $g\in \Ba_\th$, we have that 
$$
|g\1_{\sg^{-j}(U_n^c)}|_\th\le |g|_\th+\th^{-j}||g||_*
$$
\elem

\begin{proof}
Fix an integer $m\ge 1$. We consider two cases. Namely: $j+n\le m$ and $m<j+n$. Suppose first that $j+n\le m$. Then, $\osc_m\(g\1_{\sg^{-j}(U_n^c)}\)(\om)\le \osc_m(g)(\om)$ for all $\om\in E_A^\infty$. Thus 
\beq\label{fp11.6}
\th^{-m}\int\osc_m\(g\1_{\sg^{-j}(U_n^c)}\)\,d\mu_\phi
\le\th^{-m}\int\osc_m(g)\,d\mu_\phi
\le |g|_\th.
\eeq
On the other hand, if $m<j+n$, then it is easy to see that if $[\om|_m]\sbt\sg^{-j}(U_n^c)$, then 
\beq\label{1mu_2014_10_15}
\osc_m\(g\1_{\sg^{-j}(U_n^c)}\)(\om)=\osc_m(g)(\om).
\eeq
On the other hand, if $[\om|_m]\cap\sg^{-j}(U_n)\ne\es$, then 
$$
\osc_m\(g\1_{\sg^{-j}(U_n^c)}\)(\om)=\max\{\osc_m(g)(\om),||g\1_{[\om|_m]}||_\infty\}.
$$
In this latter case
$$
\osc_m\(g\1_{\sg^{-j}(U_n^c)}\)
\le \max\{\osc_m(g)(\om),||g\1_{[\om|m]}||_\infty\}
\le \osc_m(g)(\om)+\frac1{\mu_\phi\([\om|_m]\)}\int_{[\om|_m]}|g|\,d\mu_\phi.
$$
Together with \eqref{1mu_2014_10_15} this implies that
\beq\label{fp11.8}
\th^{-m}\int\osc_m\(g\1_{\sg^{-j}(U_n^c)}\)\,d\mu_\phi
\le |g|_\th+\th^{-m}\int_{\{\om\in E_A^\infty:[\om|_m]\cap\sg^{-j}(U_n)\ne\es\}} |g|\,d\mu_\phi.
\eeq
We now consider two further sub-cases. If $m\le j$, then we see that
\beq\label{fp11.8+1}
\th^{-m}\int_{\{\om\in E_A^\infty:[\om|_m]\cap\sg^{-j}(U_n)\ne\es\}} |g|\,d\mu_\phi
\le \th^{-j}\int_{\{\om\in E_A^\infty:[\om|_m]\cap\sg^{-j}(U_n)\ne\es\}} |g|\,d\mu_\phi
\le \th^{-j}||g||_1.
\eeq
If $j<m<j+n$, the descending property of the sequence $\(U_n\)_{n=0}^\infty$ yields
$$
\{\om\in E_A^\infty:[\om|_m]\cap\sg^{-j}(U_n)\ne\es\}\sbt \sg^{-j}(U_{m-j}).
$$
In this case
\beq\label{fp11.9}
\th^{-m}\int_{\{\om\in E_A^\infty:[\om|_m]\cap\sg^{-j}(U_n)\ne\es\}} |g|\,d\mu_\phi
\le \th^{-j}\th^{-(m-j)}\int_{\sg^{-j}(U_{m-j})}|g|\,d\mu_\phi
\le \th^{-j}|g|_*.
\eeq
Combining \eqref{fp11.6}, \eqref{fp11.8}, and \eqref{fp11.9} yields the desired inequality, and completes the proof.
\end{proof}

\fr As a fairly straightforward  inductive argument using  Lemma~\ref{l2fp127}, we shall prove the following.

\blem\label{l1fp129}
For all integers $k\ge 1$ and $n\ge 0$, and all functions $g\in \Ba_\th$ , we have that
\beq\label{1fp129}
|\1_n^kg|_\th\le |g|_\th+\th(1-\th)^{-1}\th^{-k}||g||_*.
\eeq
\elem

\begin{proof}
Keeping $n\ge 0$ fixed, we will proceed by induction with respect to the integer $k\ge 1$. The case of $k=1$ follows directly from Lemma~\ref{l2fp127}. Assuming for the inductive step that \eqref{1fp129} for some integer $k\ge 1$ and applying again Lemma~\ref{l2fp127}, we get
$$
\begin{aligned}
|\1_n^{k+1}g|_\th
&=\big|\1_{\sg^{-k}(U_n^c)}(\1_n^kg)\big|_\th
\le |\1_n^kg|_\th+\th^{-k}||\1_n^kg||_* \\
&\le |\1_n^kg|_\th+\th^{-k}||g||_* \\
&\le |g|_\th+\th(1-\th)^{-1}\th^{-k}||g||_*+\th^{-k}||g||_* \\
&=|g|_\th+\th(1-\th)^{-1}\th^{-(k+1)}||g||_*.
\end{aligned}
$$
The proof is complete.
\end{proof}

\fr As a fairly immediate consequence of Lemma~\ref{l1fp129} and Lemma~\ref{l1fp63}, we get the following.

\bcor\label{c1fp67}
There exists a constant $c>0$ such that
$$
||\pf_n^kg||_\th\le c\(\th^k||g||_\th+||g||_*\)
$$
for all $g\in\Ba_\th$ and all integers $k, n\ge 0$.
\ecor

\begin{proof}
Substituting $\1_n^kg$ for $g$ into the statement of Lemma~\ref{l1fp63} and then applying Lemma~\ref{l2fp127}, we get
$$
\begin{aligned}
|\pf_n^kg|_\th 
&=|\pf^k(\1_n^kg)|_\th
\le C\(\th^k|\1_n^kg|_\th+||g||_1\) \\
&\le C\(\th^k(|g|_\th+\th(1-\th)^{-1}\th^{-k}||g||_*)+||g||_1\) \\
&\le C\(\th^k(|g|_\th+\th(1-\th)^{-1}||g||_*+||g||_1\).
\end{aligned}
$$
Hence, 
$$
\begin{aligned}
||\pf_n^kg||_\th 
&=|\pf_n^kg|_\th+||\pf_n^kg||_1
\le |\pf_n^kg|_\th+||g||_1 \\
&\le(C+1)\(\th^k|g|_\th+\th(1-\th)^{-1}||g||_*+||g||_1\)\\
&\le\tilde C\(\th^k||g||_\th+||g||_*\),
\end{aligned}
$$
for some sufficiently large $\tilde C>0$ depending only on $C$ and $\th$. The proof is complete.
\end{proof}

\sp

\section{Singular Perturbations of (original) Perron--Frobenius Operators II: Stability of the Spectrum}
For a linear operator $Q:\Ba_\th\to\Ba_\th$ define
$$
|||Q|||:=\sup\{||Qg||_*:||g||_\th\le 1\}.
$$
\fr From now on fix $p, q>1$ such that $\frac1p+\frac1q=1$ and, by taking $0 < \rho < 1$ coming from (U2), sufficiently close to $1$, assume without loss of generality that
$$
\th\in(\rho^{1/p},1).
$$
We shall prove the following.

\blem\label{l2fp129}
For every $n\ge 0$ we have 
$$
|||\pf-\pf_n|||\le 2(\rho^{1/q})^n.
$$
\elem

\begin{proof}
Fix an arbitrary $g\in\Ba_\th$ with $||g||_\th\le 1$. Using Lemma~\ref{l1fp61} we then get
\beq\label{1FP140628}
\begin{aligned}
||(\pf-\pf_n)g||_1
&=||\pf(\1_n^1g)||_1
=||\1_n^1g||_1
\le \mu_\phi(U_n)||g||_\infty \\
&\le \mu_\phi(U_n)||g||_\th
\le \mu_\phi(U_n)\le \rho^n \\
&\le (\rho^{1/q})^n\end{aligned}
\eeq
Now fix also two integers $m, j\ge 0$. Using the H\"older Inequality, we get 
\beq\label{1fp89}
\begin{aligned}
\th^{-m}\int_{\sg^{-j}(U_m)}
&|(\pf-\pf_n)g|\,d\mu_\phi 
\le \th^{-m}\mu_\phi(\sg^{-(j+1)}(U_m)\cap U_n)||g||_\th= \\
&=\th^{-m}||g||_\th\int\1_{\sg^{-(j+1)}(U_m)}\1_{U_n}\,d\mu_\phi \\
&\le ||g||_\th\th^{-m}\lt(\int \1_{\sg^{-(j+1)}(U_m)}\,d\mu_\phi\rt)^{1/p}
\lt(\int \1_{U_n}\,d\mu_\phi\rt)^{1/q}  \\
&=||g||_\th\th^{-m}\mu_\phi(U_m)^{1/p}\mu_\phi(U_n)^{1/q} \\
&\le ||g||_\th\(\rho^{1/p}/\th\)^m\rho^{n/q} 
 \le (\rho^{1/q})^n||g||_\th
 \le (\rho^{1/q})^n,
\end{aligned}
\eeq
where the second to the last inequality follows from the fact that $\th\in (\rho^{1/p},1)$. Along with \eqref{1FP140628} this implies that $||\pf-\pf_n||_*\le 2(\rho^{1/q})^n$. So, taking the supremum over all $g\in\Ba_\th$ with $||g||_\th\le 1$, we get that $|||\pf-\pf_n|||\le 2(\rho^{1/q})^n$. The proof is complete.
\end{proof}

\sp\fr With Lemma~\ref{l1fp127}, Corollary~\ref{c1fp67}, and Lemma~\ref{l2fp127}, we have checked that the respective conditions (2), (3), and (5), from \cite{KL} are satisfied. We shall now check that condition (4) from there also holds. We will do this by showing that the requirements from Remark~3 in \cite{KL} hold.

\sp\fr For every integer $k\ge 1$ let $\cA^k$ be the partition of $E_A^\infty$ into cylinders of length $k$. Let $\pi_k^*:L^1(\mu_\phi)\to L^1(\mu_\phi)$ be the operator of expected value with respect to the probability measure $\mu_\phi$ and the $\sg$-algebra $\sg(\cA^k)$ generated by the elements of $\cA^k$; i. e. 
$$
\pi_k^*(g)=E_{\mu_\phi}(g|\sg(\cA^k)).
$$
If $g\in\Ba_\th$ then $|\pi_k^*(g)-g|\le\osc_k(g)$, and therefore
\beq\label{1fp77}
||\pi_k^*(g)-g||_1
=\int_{E_A^\infty}|\pi_k^*(g)-g|\,d\mu_\phi
\le \int_{E_A^\infty}\osc_k(g)\,d\mu_\phi
\le \th^k|g|_\th.
\eeq
Let now $\cA_0^k$ be a finite subset of $\cA^k$ such that
\beq\label{2fp77}
\mu_\phi\(A_c^k\)\le\th^k,
\eeq
where 
$$
A_c^k:=\bu_{A\in \cA^k\sms \cA_0^k}A.
$$
Let also 
$$
A_0^k:=\bu_{A\in \cA_0^k}A.
$$
Let $\hat\cA^k$ be the partition of $E_A^\infty$ consisting of $A_c^k$ and all elements of $\cA_0^k$. Similarly as above, let $\pi_k:L^1(\mu_\phi)\to L^1(\mu_\phi)$ be defined by the formula
$$
\pi_k(g)=E_{\mu_\phi}(g|\sg(\hat\cA^k)).
$$
We then have that
\beq\label{3fp77}
||\pi_k||_1\le 1,
\eeq
and for every $g\in\Ba_\th$, because of \eqref{1fp77} and Lemma~\ref{l1fp61}, and \eqref{2fp77}:
\beq\label{1fp79}
\begin{aligned}
||\pi_k(g)-g||_1
&=\int_{E_A^\infty}|\pi_k(g)-g|\,d\mu_\phi
 =\int_{A_0^k}|\pi_k(g)-g|\,d\mu_\phi + \int_{A_c^k}|\pi_k(g)-g|\,d\mu_\phi \\
&=\int_{A_0^k}|\pi_k^*(g)-g|\,d\mu_\phi + \int_{A_c^k}|\pi_k(g)-g|\,d\mu_\phi \\
&\le \int_{E_A^\infty}|\pi_k^*(g)-g|\,d\mu_\phi +2||g||_\infty\mu_\phi(A_c^k) \\
&\le \th^k|g|_\th +2||g||_\infty \theta^k  \\
&\le 3\th^k||g||_\th.
\end{aligned}
\eeq
Now, for all $m$ and $k$ we have that
$$
\osc_m(\pi_k(g))=
\begin{cases}
0 &\text{ if } m\ge k \\
\le \osc_0(g)\le 2||g||_\infty\le 2||g||_\th \  &\text{ if } m< k.
\end{cases}
$$
Moreover, if $\om\in A_0^k$ and $m<k$, then
$$
\osc_m(\pi_k(g))(\om)
=\osc_m(\pi_k^*(g))(\om)
\le \osc_m(g)(\om).
$$
Thus,
$$
\begin{aligned}
\th^{-m}||\osc_m(\pi_k(g))||_1
&=\th^{-m}\int_{E_A^\infty}\osc_m(\pi_k(g)\,d\mu_\phi \\
&=\th^{-m}\int_{A_0^k}\osc_m(\pi_k(g))\,d\mu_\phi +
 \th^{-m}\int_{A_c^k}\osc_m(\pi_k(g))\,d\mu_\phi \\
&\le \th^{-m}\int_{A_0^k}\osc_m(g)\,d\mu_\phi + 2\th^{-k}||g||_\th\mu_\phi(A_c^k) \\
&\le |g|_\th + 2||g||_\th  \\
&\le 3||g||_\th.
\end{aligned}
$$
Therefore $|\pi_k(g)|_\th\le 3||g||_\th$. Together with \eqref{3fp77}, this gives $||\pi_k||_\th\le 4$. In other words:
\beq\label{1fp81}
\sup_{k\ge 1}\{||\pi_k||_\th\}\le 4<+\infty.
\eeq
Now assume that $||g||_\th\le 1$. Recall that we have fixed $p, q>1$ such that $(1/p)+(1/q)=1$. Using H\"older's Inequality and \eqref{1fp79} we then get for all integers $k\ge 1$, $j\ge 0$, and $n\ge 0$, that
$$
\begin{aligned}
\int_{\sg^{-j}(U_n)}|\pi_k(g)-g|\,d\mu_\phi
&=\int_{E_A^\infty}\1_{\sg^{-j}(U_n)}|\pi_k(g)-g|\,d\mu_\phi \\
&\le \lt(\int_{E_A^\infty}\1_{\sg^{-j}(U_n)}\,d\mu_\phi\rt)^{1/p}
\lt(\int_{E_A^\infty}|\pi_k(g)-g|^q\,d\mu_\phi\rt)^{1/q} \\
&\le \mu_\phi(U_n)^{1/p}2^{\frac{q-1}q}\lt(\int_{E_A^\infty}|\pi_k(g)-g|\,d\mu_\phi\rt)^{1/q}\\ 
&\le  \mu_\phi(U_n)^{1/p}\(3\th^k||g||_\th\)^{1/q} \\
&\le 3\rho^{n/p}\th^{k/q}.
\end{aligned}
$$
Recall that $\th\in(0,1)$ was fixed so large that $\th>\rho^{1/p}$. In other words $\rho^{1/p}/\th<1$, and we get
$$
\th^{-n}\int_{\sg^{-j}(U_n)}|\pi_k(g)-g|\,d\mu_\phi
\le 3\(\rho^{1/p}/\th\)^n\th^{k/q}
\le 3 \th^{k/q}.
$$
In other words $\|\pi_k(g)-g\|_*\le 3 \th^{k/q}$. Together with \eqref{1fp79} this gives
\beq\label{2fp81}
|\pi_k(g)-g|_*\le 2\th^k + 3\th^{k/q}
\le 5\(\th^{1/q}\)^k.
\eeq
Since all the operators $\pi_k:\cB_\th\to\cB_\th$ have finite--dimensional ranges, all the operators $\pf_n\circ\pi_k:\cB_\th\to\cB_\th$ are compact. As $\th\le \th^{1/q}$, in conjunction with \eqref{2fp81} and \eqref{1fp81}, this shows that all the requirements of Remark~3 in \cite{KL} are satisfied and (4) (as well as (3)) hold with $\a=\th^{1/q}$ and $M=1$. All the hypotheses of Theorem~1 in \cite{KL} have been thus verified. Note also that the number $1$ is a simple eigenvalue of the operator $\pf:\cB_\th\to\cB_\th$ as there exists exactly one Borel probability $\sg$-invariant measure absolutely continuous with respect to the Gibbs measure $\mu_\phi$. Applying Theorem~1 in \cite{KL} and the appropriate corollaries therein, we get the following fundamental perturbative result which extends Propositions~3.17, 3.19, and 3.7 from \cite{FP} to the case of infinite alphabet.

\bprop\label{p1fp80} 
For all $n\ge 0$ sufficiently large there exist two bounded linear operators $Q_n,\De_n:\cB_\th\to\cB_\th$ and complex numbers $\lam_n\ne 0$ with the following properties:
\begin{itemize}
\item[(a)] $\lam_n$ is a simple eigenvalue of the operator $\pf_n:\cB_\th\to\cB_\th$.

\sp\item[(b)] $Q_n:\cB_\th\to\cB_\th$ is a projector ($Q_n^2=Q_n$) onto the $1$--dimensional eigenspace of $\lam_n$.

\sp\item[(c)] $\pf_n=\lam_nQ_n+\De_n$.

\sp\item[(d)] $Q_n\circ\De_n=\De_n\circ Q_n=0$.

\sp\item[(e)] There exist $\ka\in (0,1)$ and $C>0$ such that
$$
||\De_n^k||_\th\le C\ka^k
$$
for all $n\ge 0$ sufficiently large and all $k\ge 0$. In particular,
$$
||\De_n^kg||_\infty\le ||\De_n^kg||_\th\le C\ka^k||g||_\th
$$
for all $g\in\Ba_\th$.

\sp\item[(f)] $\lim_{n\to\infty}\lam_n=1$.

\sp\item[(g)] Enlarging the above constant $C>0$ if necessary, we have
$$
||Q_n||_\th\le C.
$$
In particular, 
$$
||Q_ng||_\infty\le ||Q_ng||_\th\le C||g||_\th
$$
for all $g\in\Ba_\th$.

\sp\item[(h)] $\lim_{n\to\infty}|||Q_n-Q|||=0$. 
\end{itemize}
\eprop

\fr The proof of the next proposition is fairly standard. We provide it here for the sake of completeness. 

\bprop\label{p1fp131}
All eigenvalues $\lam_n$ produced in Proposition~\ref{p1fp80} are real and positive, and all operators $Q_n:\cB_\th\to\cB_\th$ preserve $\Ba_\th(\R)$ and $\Ba_\th^+(\R)$, the subsets of $\Ba_\th$ consisting, respectively, of real--valued functions and positive real--valued functions.
\eprop

\begin{proof}
Let $\rho_n\in\Ba_\th$ be an eigenfunction of the eigenvalue $\lam_n$. Write $\lam_n=|\lam_n|e^{i\g_n}$, with $\g_n\in [0,2\pi)$. It follows from (b), (c), and (d) of Proposition~\ref{p1fp80} that
\beq\label{1fp131}
|\lam_n|^ke^{-ik\g_n}\pf_n^k\1=Q_n\1+\lam_n^{-k}\De_n^k\1.
\eeq
By (1) of Corollary~1 in \cite{KL} we have that $Q_n\1\ne0$ for all $n\ge 0$ large enough (so after disregarding finitely many terms, we can assume this for all $n\ge 0$) and $|\lam_n|>(1+\ka)/2$. Since also $\pf_n^k\1$ is a real--valued function, it therefore follows from \eqref{1fp131} and (e) that the arguments of $Q_n\1(\om)$ are the same (mod $2\pi$) whenever $Q_n\1\ne0$. This in turn implies that the set of accumulation points of the sequence $(k\g_n)_{k=0}^\infty$ is a singleton (mod $2\pi$). This yields $\g_n=0$ (mod $2\pi$). Thus $\lam_n\in\R$, and, as $\lam_n$ is close to $1$ (by Proposition~\ref{p1fp80}), it is positive. Knowing this and assuming $g\ge 0$, the equality
$$
Q_ng=\lam_n^{-k}\pf_n^kg-\lam_n^{-k}\De_n^k(g),
$$
along with (e) of Proposition~\ref{p1fp80}, non-negativity of $\pf_n^kg$, and inequality $|\lam_n|>(1+\ka)/2$, yield $Q_ng\ge 0$. Finally, for $g\in \Ba_\th(\R)$, write canonically $g=g_+-g_-$ with $g_+,g_-\in \Ba_\th^+(\R)$ and apply the invariance of $\Ba_\th^+(\R)$ under the action of $\pf_n$. The proof is complete.
\end{proof}

\brem\label{r2fp131}
We would like to note that unlike \cite{FP}, we did not use the dynamics (i.e., the interpretation of $\log\lam_n$ as some topological pressure) to demonstrate item (f) of 
Proposition~\ref{p1fp80} and to prove Proposition~\ref{p1fp131}. We instead used the full power of the perturbation results from \cite{KL}. The dynamical interpretation will eventually emerge, and will be important for us, but not until  Section~\ref{derivatives_of_eigenvalues}. Therein Lemma~\ref{l1had28.1} will provide, at least in part, a dynamical interpretation.
\erem

\section{An Asymptotic Formula for $\lam_n$s, \\  the Leading Eigenvalues of Perturbed Operators}\label{Asymptotic_formula_for_lam_n}

\fr In this section we keep the setting of the previous sections. Our goal here is to estblish the asymptotic behavior of eigenvalues $\lam_n$ as  $n$ diverges to $+\infty$. Let 
$$
U_\infty:=\bi_{n=0}^\infty\ov U_n.
$$
In addition to (U0), (U1), and (U2), we now also  assume that:
\begin{itemize}
\item[(U3)] $U_\infty$ is a finite set.

\sp\item[(U4)] 
Either

\begin{itemize}
\item[(U4A)] 
$$
U_\infty\cap\bu_{n=1}^\infty\sg^n(U_\infty)=\es
$$ 
or 

\sp\item[(U4B)] 
$U_\infty=\{\xi\}$, where $\xi$ is a periodic point of $\sg$ of prime period equal to some integer $p\ge 1$,
the pre-concatenation by the first $p$ terms of $\xi$ with elements of $U_n$ satisfy
\beq\label{1_2014_12_19}
[\xi|_p]U_n\sbt U_n
\eeq
for all $n\ge 0$, and 
\beq\label{2_2014_12_19}
\lim_{n\to\infty}\sup\{|\phi(\om)-\phi(\xi)|:\om\in U_n\}=0.
\eeq
\end{itemize}

\sp\item[(U5)]
There are no integer $l\ge 1$, no sequence $\(\om^{(n)}\)_{n=0}^\infty$ of points in $E_A^\infty$, and no increasing sequence $\(s_n\)_{n=0}^\infty$ of positive integers with the following properties:

\begin{itemize}
\sp\item[(U5A)] 
$$
\om^{(n)}, \sg^l(\om^{(n)})\in U_{s_n}
$$
for all $n\ge 0$,

\sp\item[(U5B)]
$$
\varliminf_{n\to\infty}d_\th(\om^{(n)},U_\infty) 
\begin{cases}
\, >0  &\text{ if (U4A) holds, } \\
\, >\th^l  &\text{ if (U4B) holds, }
\end{cases}
$$
\item[(U5C)]
$$
\varlimsup_{n\to\infty}\sum_{i=1}^l\om_i^{(n)}<+\infty,
$$
for fixed $l$,
where we identify $E$ with the natural numbers to give $\om_i^{(n)}$ their numerical values.
\end{itemize} 
\end{itemize}

\fr Having proved all the perturbation results of the previous section, we shall now prove the following analogue of Proposition~4.1 in \cite{FP},  which is our main result concerning the asymptotic behavior of eigenvalues $\lam_n$ as $n\to+\infty$. 

\bprop\label{p1fp133} With the setting of Sections~\ref{PFOriginal} and \ref{SecSingPerturb}, assume that {\rm (U0)--(U5)} hold. Then
$$
\lim_{n\to\infty}\frac{\lam-\lam_n}{\mu_\phi(U_n)}=
\begin{cases}
\lam &\text{{\rm if (U4A) holds,}} \\
\lam(1-\lam^{-p}e^{\phi_p(\xi)}) &\text{{\rm if (U4B) holds,}}
\end{cases}
$$
where $\lam$ and $\lam_n$ are respective eigenvalues of original (i. e., we recall, not normalized) operators $\pf$ and $\pf_n$. 
\eprop

\fr For every integer $n\ge 0$ let $\nu_n$ be $\mu_\phi$-conditional measure on $U_n$, i. e.:
$$
\nu_n:=\frac{\mu_\phi|_{U_n}}{\mu_\phi(U_n)}.
$$
We denote
$$
\1_n^c:=\1_{U_n}=\1-\1_n^1.
$$
We start with the following.
\blem\label{l1fp132}
If {\rm (U0)--(U4A)} and {\rm (U5)} hold, then
$$
\lim_{n\to\infty}\frac{\int Q_n\(\pf\1_n^c\)\,d\nu_n}{\lam-\lam_n} 
=\lim_{n\to\infty}\int_{E_A^\infty}Q_n\1\,d\nu_n
=1.
$$
\elem

\begin{proof}
Assume without loss of generality that $\pf$ is normalized so that $\lam=1$ and $\pf\1=\1$. With an aim to prove the first equality, we note that
$$
\begin{aligned}
\int Q_n(\pf\1_n^c)\,d\nu_n
&=\int Q_n(\pf\1-\pf\1_n^1)\,d\nu_n
 =\int Q_n(\1-\pf_n\1)\,d\nu_n \\
&= \int Q_n\1\,d\nu_n-\int Q_n\pf_n\1\,d\nu_n 
 =\int Q_n\1\,d\nu_n-\int \pf_nQ_n\1\,d\nu_n \\
&=\int Q_n\1\,d\nu_n-\lam_n\int Q_n\1\,d\nu_n \\
&=(1-\lam_n)\int Q_n\1\,d\nu_n,
\end{aligned}
$$
using Proposition \ref{p1fp80}. 
and the first equality is established. Now, fix an arbitrary integer $k\ge 1$. For every $\om\in U_n$ let
\beq\label{1fp132}
\sg_0^{-k}(\om):=\{\tau\in\sg^{-k}(\om): \exists\,_{0\le j\le k-1}\, \sg^j(\tau)\in U_n\}
\eeq
and 
\beq\label{2fp132}
\sg_c^{-k}(\om):=\sg^{-k}(\om)\sms \sg^{-k}(\om).
\eeq
If $\tau\in \sg_0^{-k}(\om)$, then $\sg^j(\tau)\in U_n$ for some $0\le j\le k-1$. Denote  $\sg^j(\tau)$ by $\g$. Then
$$
\g\in U_n \ \text{ and }  \  \sg^{k-j}(\g)\in U_n; \  1\le k-j\le k.
$$
Fix an arbitrary $M>0$. We claim that for all $n\ge 0$ sufficiently large, say $n\ge N:=N_k(M)$, we have that
\beq\label{2fp132a}
\sum_{i=1}^{k-j}\g_i\ge Mk
\eeq
for any $\gamma = \sigma^j(\tau)$ for any $\tau \in \sigma_0^{-k}(\omega)$
Indeed, seeking a contradiction we assume that there exist an increasing sequence $(s_n)_0^\infty$ of positive integers, a sequence $\(\g^{(n)}\)_0^\infty\sbt E_A^\infty$, and  an integer $l\in[1,k]$ such that
\beq\label{1fp132a}
\g^{(n)}, \sg^l(\g^{(n)})\in U_{s_n},
\eeq
and 
$$
\sum_{i=1}^l \g_i^{(n)}<Mk
$$
for all $n\ge 0$. It then follows from conditions (U4A) and (U5) that the contrapositive of (U5B) holds, i.e.:
$$
\varliminf_{n\to\infty}d_\th(\g^{(n)},U_\infty)=0.
$$
Hence, from continuity of the shift map $\sg:E_A^\infty\to E_A^\infty$ and from the finiteness of the set $U_\infty$ (by (U3)), 
$$
\varliminf_{n\to\infty}d_\th\(\sg^l(\g^{(n)}),\sg^l(U_\infty))\)=0.
$$
So, passing to a subsequence, and invoking finitenes of the set $\sg^l(U_\infty)$, we may assume without loss of generality that the sequence $\(\sg^l(\g^{(n)})\)_0^\infty$  has a limit, call it $\b$, and then $\b\in 
\sg^l(U_\infty)$. But, since the sequence $\(\ov U_n\)_0^\infty$ is descending, it follows from \eqref{1fp132a} that $\b\in \ov{U}_q$ for every $q\ge 0$. Thus $\b\in \bi_{q=0}^\infty\ov{U}_q=U_\infty$. We have therefore obtained that $U_\infty\cap \sg^l(U_\infty)\ne\es$ as this set contains $\b$. This contradicts (U4A) and finishes the proof of \eqref{2fp132a}. So, letting $n\ge N_k(M)$ and $\om \in U_n$, we get
\beq\label{2fp132b}
\begin{aligned}
\pf_n^k\1(\om)
& = \pf^k(\1_n^1)(\om) \\
& =  \sum_{\tau\in \sg_c^{-k}(\om)}\1_n^1(\tau)e^{\phi_k(\tau)}+ 
     \sum_{\tau\in \sg_0^{-k}(\om)}\1_n^1(\tau)e^{\phi_k(\tau)}  \\
& =  \sum_{\tau\in \sg_c^{-k}(\om)}e^{\phi_k(\tau)}
  = \pf^k\1(\om)-\sum_{\tau\in \sg_0^{-k}(\om)}e^{\phi_k(\tau)}  \\
&=  \1(\om)-\sum_{\tau\in \sg_0^{-k}(\om)}e^{\phi_k(\tau)}.
\end{aligned}
\eeq
Now, if $\tau\in \sg_0^{-k}(\om)$, then $\g:=\sg^{j_\tau}(\tau)\in U_n$ with some $0\le j_\tau\le k-1$, and using \eqref{2fp132a}, we get
\beq\label{1fp132b}
\aligned
S_0(\om)
:&=\sum_{\tau\in \sg_0^{-k}(\om)}e^{\phi_k(\tau)}
 \lek \sum_{\tau\in \sg_0^{-k}(\om)}\mu_\phi([\tau]) 
=\mu_\phi\lt(\sum_{\tau\in \sg_0^{-k}(\om)}[\tau]\rt) \\
&\le \mu_\phi\lt(\bu_{j=0}^{k-1}\sg^{-j}\lt(\bu_{e\ge M}[e]\rt)\rt) \\
&\le \sum_{j=0}^{k-1}\mu_\phi\lt(\sg^{-j}\lt(\bu_{e\ge M}[e]\rt)\rt)   =\sum_{j=0}^{k-1}\mu_\phi\lt(\bu_{e\ge M}[e]\rt)\\
&=k\mu_\phi\lt(\bu_{e\ge M}[e]\rt).
\endaligned
\eeq
This means that there exists a constant $C>0$ such that
$$
S_0(\om)\le Ck\mu_\phi\lt(\bu_{e\ge M}[e]\rt).
$$
Denote the number $C\mu_\phi\lt(\bu_{e\ge M}[e]\rt)$ by $\eta_M$.
Using \eqref{1fp132b}, \eqref{2fp132b}, and Proposition~\ref{p1fp80}, we get the following.
$$
\begin{aligned}
\lt|1-\int Q_n\1\,\nu_n\rt|
&=\lt|\int\1\,d\nu_n-\int Q_n\1\,d\nu_n\rt|
 =\lt|\int\(\pf_n^k\1+S_0\)\,d\nu_n-\int Q_n\1\,\nu_n\rt| \\
&=\lt|\int\(\pf_n^k-\lam_n^kQ_n\)\1\,d\nu_n+\int(\lam_n^k-1)Q_n\1\,d\nu_n+
     \int S_0\,d\nu_n \rt| \\
&\le \int|\De_n^k\1|\,d\nu_n+|\lam_n^k-1|\cdot||Q_n\1||_\infty+  
     \int S_0\,d\nu_n\\
&\le C\ka^n+C|\lam_n^k-1|+k\eta_M.
\end{aligned}
$$
Now, fix $\e>0$. Take then $n\ge 1$ so large that $C\ka^n<\e/3$. Next, take $M\ge 1$ so large that $k\eta_M<\e/3$. Finally take any $n\ge N_k(M)$ so large that $C|\lam_n^k-1|<\e/3$. Then $\lt|1-\int Q_n\1\,\nu_n\rt|<\e$, and the proof is complete. 
\end{proof}

\fr The proof of the next lemma, corresponding to Lemma~4.3 in \cite{FP}, goes through unaltered in the case of an infinite alphabet. We include it here for the sake of completeness and for the convenience of the reader.

\blem\label{l1fp132b}
If \emph{(U1)-(U4A)} and \emph{(U5)} hold, then
$$
\lim_{n\to\infty}\frac{\int Q_n\(\pf\1_n^c\)\,d\nu_n}{\mu_\phi(U_n)}=\lam.
$$
\elem

\bpf
We assume without loss of generality that $\lam=1$. Let $\tau_n:U_n\to U_n$ be the first return time from $U_n$ to $U_n$ under the shift map $\sg:E_A^\infty\to E_A^\infty$. It is defined as
$$
\tau_n(\om):=\inf\{k\ge 1:\sg^k(\om)\in U_n\}.
$$
By Poincar\'e's Recurrence Theorem, $\tau_n(\om)<+\infty$ for $\mu_\phi$--a.e. $\om\in E_A^\infty$. We deal with the concept of first return time and first return time more thoroughly in Sections~\ref{FRM}, \ref{FRMERI}, and \ref{FRMERII}. We have 
$$
\begin{aligned}
\int_{U_n}\tau_n\,d\nu_n
&=\sum_{i=1}^\infty i\nu_n(\tau_n^{-1}(i))
 =\sum_{i=1}^\infty i\nu_n\(\1_{\tau_n^{-1}(i)}\)
 =\nu_n(\tau_n^{-1}(1))
     +\sum_{i=2}^\infty i\nu_n\(\1_n^{i-1}\circ\sg\cdot\1_n^c\circ\sg^i\) \\
&=\nu_n(\tau_n^{-1}(1))
     +\sum_{i=2}^\infty\frac{i}{\mu_\phi(U_n)}\mu_\phi\(\1_n^{i-1}\circ\sg\cdot\1_n^c\circ\sg^i\) \\
&=\nu_n(\tau_n^{-1}(1))
     +\sum_{i=2}^\infty\frac{i}{\mu_\phi(U_n)}\mu_\phi\(\pf^i\(\(\1_n^{i-1}\circ\sg\cdot\1_n^c\circ\sg^i\)\). \\     
\end{aligned}
$$
Now using several times the property $\pf^j(f\cdot g\circ\sg^j)=g\pf^j(f)$, a formal calculation leads to
$$
\int_{U_n}\tau_n\,d\nu_n
=\nu_n(\tau_n^{-1}(1))+\sum_{i=2}^\infty i\nu_n\(\pf_n^{i-1}(\pf(\1_n^c))\).
$$
Invoking at this point Proposition~\ref{p1fp80}, we further get
$$
\begin{aligned}
\int_{U_n}\tau_n\,d\nu_n
&=\nu_n(\tau_n^{-1}(1))+\sum_{i=2}^\infty i\nu_n\(\lam_n^{i-1}Q_n\pf(\1_n^c) +
\De_n^{i-1}\pf(\1_n^c)\) \\
&=\nu_n(\tau_n^{-1}(1))+\nu_n\(Q_n\pf(\1_n^c)\)\sum_{i=2}^\infty i\lam_n^{i-1}+
    \sum_{i=2}^\infty i\nu_n\(\De_n^{i-1}\pf(\1_n^c)\) \\
&=\nu_n(\tau_n^{-1}(1))+\nu_n\(Q_n\pf(\1_n^c)\)\lt(\frac1{(1-\lam_n)^2}-1\rt)
    +\sum_{i=2}^\infty i\nu_n\(\De_n^{i-1}(\pf\1-\pf\1_n)\)\\
&=\nu_n(\tau_n^{-1}(1))+\nu_n\(Q_n\pf(\1_n^c)\)\lt(\frac1{(1-\lam_n)^2}-1\rt)
    +\sum_{i=2}^\infty i\nu_n\(\De_n^{i-1}(\pf\1-\pf_n\1)\) \\
&=\nu_n(\tau_n^{-1}(1))+\nu_n\(Q_n\pf(\1_n^c)\)\lt(\frac1{(1-\lam_n)^2}-1\rt)+ \\
& \  \  \  \  \  \   \  \  \  \  \  \  \   \  \  \  \  \  \  \   \  \  \ +\sum_{i=2}^\infty i\nu_n\(\De_n^{i-1}(\pf\1)\)-
     \sum_{i=2}^\infty i\nu_n\(\De_n^i\1)\).  
\end{aligned}
$$
Since, By Kac's Theorem, $\int_{U_n}\tau_n\,d\nu_n=1/\mu_\phi(U_n)$, multiplying both sides of this formula by $\nu_n\(Q_n\pf(\1_n^c)\)$, we thus get
$$
\begin{aligned}
\frac{\nu_n\(Q_n\pf(\1_n^c)\)}{\mu_\phi(U_n)}
=\lt(\frac{\nu_n\(Q_n\pf(\1_n^c)\)}{1-\lam_n}\rt)^2+\nu_n\(Q_n\pf&(\1_n^c)\)  
    \Big(\nu_n(\tau_n^{-1}(1))-\nu_n\(Q_n\pf(\1_n^c)\)+\\
    &+\sum_{i=2}^\infty i\nu_n\(\De_n^{i-1}(\pf\1)\)-
      \sum_{i=2}^\infty i\nu_n\(\De_n\1)\Big).  
\end{aligned}
$$
Since, by Lemma~\ref{l1fp132}, 
$$
\lim_{n\to\infty}\frac{\nu_n\(Q_n\pf(\1_n^c)\)}{1-\lam_n}=1,
$$
we have that $\lim_{n\to\infty}\nu_n\(Q_n\pf(\1_n^c)\)=0$, and since, applying Proposition~\ref{p1fp80} again, we deduce that the four terms in the big parentheses above are bounded, we get that 
$$
\lim_{n\to\infty}\frac{\nu_n\(Q_n\pf(\1_n^c)\)}{\mu_\phi(U_n)}=1.
$$
The proof is complete.
\epf

\fr We shall prove the following.

\blem\label{l1fp132c}
If \emph{(U1)-(U4B)} and \emph{(U5)} hold, then
$$
\lim_{n\to\infty}\frac{\int Q_n\(\pf\1_n^c\)\,d\nu_n}{\lam-\lam_n}
=\lim_{n\to\infty}\int_{E_A^\infty}Q_n\1\,d\nu_n
=1-\lam^{-p}e^{\phi_p(\xi)}.
$$
\elem

\begin{proof}
Assume again without loss of generality that $\pf$ is normalized so that $\lam=1$ and $\pf\1=\1$. 
The first equality is general and has been established at the beginning of the proof of Lemma~\ref{l1fp132}. We will thus concentrate on the second one. So, fix $\om\in U_n$ and $k$, an integral multiple of $p$, say $k=qp$ with $q\ge 0$. Define the sets $\sg_0^{-k}(\om)$ and $\sg_c^{-k}(\om)$ exactly as in the proof of Lemma~\ref{l1fp132}, i.e. by formulae \eqref{1fp132} and \eqref{2fp132}. We further repeat the proof of Lemma~\ref{l1fp132} verbatim until formula \eqref{2fp132a}, which now takes on the form:
$$
\text{{\rm Either both } } \ k-j\ge p \ \text{ {\rm and } }  \g|_{k-j}=\xi|_{k-j} \   \text{ {\rm or else }  }  \sum_{i=1}^{k-j}\g_i\ge Mk.
$$
Indeed, this is an immediate consequence of (U4B) and (U5). In other words
$$
\sg_0^{-k}(\om)=\sg_1^{-k}(\om)\cup\sg_2^{-k}(\om),
$$
where
$$
\sg_1^{-k}(\om):=\Big\{\tau\in\sg_0^{-k}(\om): \exists\, (0\le j\le q-1)\,\, \sg^{pj}(\tau)\in U_n \text{ and } \sg^{pj}(\tau)|_{p(q-j)}=(\xi|_p)^{q-j}\Big\}
$$
and 
$$
\begin{aligned}
\sg_2^{-k}(\om)
&=\sg_0^{-k}(\om)\sms\sg_1^{-k}(\om) \\
&\sbt \Big\{\tau\in\sg_0^{-k}(\om): \exists\, (0\le j\le k-1)\,\, \sg^j(\tau)\in U_n
 \text{ and } \sum_{i=j+1}^k\tau_i\ge Mk\Big\}.
\end{aligned}
$$
Now, we shall prove that
\beq\label{2fp132c}
\sg_1^{-k}(\om)= Z:= \Big\{\tau\in\sg_0^{-k}(\om): \sg^{k-p}(\tau)\in [\xi|_p]\Big\}.
\eeq
Indeed, denote the set on the right-hand side of this equality by $Z$. If $\tau\in Z$, then $\sg^{p(q-1)}(\tau)|_p=\xi|_p$ and 
$$
\sg^{p(q-1)}(\tau)
=\(\sg^{p(q-1)}(\tau)\)|_p\sg^{pq}(\tau)=\xi|_p\om\in\xi|_pU_n\sbt U_n,
$$
where the last inclusion is due to (U4B). Thus, taking $j=q-1$, we see that $\tau\in \sg_1^{-k}(\om)$. So, the inclusion
\beq\label{1fp132c}
Z\sbt \sg_1^{-k}(\om)
\eeq
has been established. In order to prove the opposite inclusion, let $\tau\in \sg_1^{-k}(\om)$. Then there exists $j\in\{0,1,\ld,q-1\}$ such that $\sg^{pj}(\tau)\in U_n$ and $\sg^{pj}(\tau)|_{p(q-j)}=(\xi|_p)^{q-j}$. Then
$$
\sg^{k-p}(\tau)|_p
=\(\sg^{p(q-j-1)}\circ\sg^{pj}(\tau)\)|_p
=\sg^{pj}(\tau)|_{p(q-j-1)+1}^{p(q-j)+1}
=\xi|_p,
$$
and so, $\tau\in Z$. This establishes the inclusion $\sg_1^{-k}(\om)\sbt Z$, and, together with \eqref{1fp132c} completes the proof of \eqref{2fp132c}.

\sp\fr Therefore, keeping $\om\in U_n$ and using \eqref{2fp132c} and \eqref{2fp132b}we can write
\beq\label{1fp132d}
\begin{aligned}
\pf_n^k\1(\om)
& = \pf^k(\1_n^1)(\om) \\
%& =  \sum_{\tau\in \sg_c^{-k}(\om)}\1_n^1(\tau)e^{\phi_k(\tau)}+ 
%     \sum_{\tau\in \sg_0^{-k}(\om)}\1_n^1(\tau)e^{\phi_k(\tau)}  \\
%& =  \sum_{\tau\in \sg_c^{-k}(\om)}e^{\phi_k(\tau)}
%  = \pf^k\1(\om)-\sum_{\tau\in \sg_0^{-k}(\om)}e^{\phi_k(\tau)}  \\
&=  \1(\om)-\sum_{\tau\in \sg_1^{-k}(\om)}e^{\phi_k(\tau)}
           -\sum_{\tau\in \sg_2^{-k}(\om)}e^{\phi_k(\tau)} \\
&=  \1(\om)-\sum_{\tau\in \sg^{-k}(\om)}\1_{[\xi|
           _p]}\circ\sg^{p(q-1)}(\tau)e^{\phi_k(\tau)}
           -\sum_{\tau\in \sg_2^{-k}(\om)}e^{\phi_k(\tau)} \\
&=  \1(\om)-\pf^{pq}\(\1_{[\xi|_p]}\circ\sg^{p(q-1)}\)(\om)
           -\sum_{\tau\in \sg_2^{-k}(\om)}e^{\phi_k(\tau)} \\
&=  \1(\om)-\pf^p\(\1_{[\xi|_p]}\)(\om)
           -\sum_{\tau\in \sg_2^{-k}(\om)}e^{\phi_k(\tau)}.
\end{aligned}
\eeq
Putting
$$
S_2(\om):=\sum_{\tau\in \sg_2^{-k}(\om)}e^{\phi_k(\tau)}
$$
and keeping $\eta_M$ the same as in the proof of Lemma~\ref{l1fp132}, the same estimates as in \eqref{1fp132b}, give us
$$
S_2(\om)\le k\eta_M.
$$
Hence, using also \eqref{1fp132d}, we get
$$
\begin{aligned}
\Big|1-e^{\varphi_p(\xi)} &-\int\pf_n^k\1\,d\nu_n\Big|
=\lt|\int\pf^p\(\1_{[\xi|_p]}\)\,d\nu_n-e^{\vp_p}(\xi)+\int S_2\,d\nu_n\rt|= \\
&=\lt|\int\(e^{\vp_p}(\xi|_p\om)-e^{\vp_p}(\xi)\)d\nu_n+\int S_2\,d\nu_n\rt|\\
&\le \int\big|e^{\vp_p}(\xi|_p\om)-e^{\vp_p}(\xi)\big|d\nu_n+\int S_2\,d\nu_n\\
&\le \vep_n+k\eta_M,
\end{aligned}
$$
with some $\vep_n\to 0$ resulting from the last item of (U4B). Hence, keeping $k$ fixed and letting $M$ and then $n$ to infinity, we obtain
\beq\label{1fp132e}
\lim_{n\to\infty}\int\pf_n^k\1\,d\nu_n=1-e^{\vp_p}(\xi)
\eeq
for every $k=qp\ge 1$. Using Proposition~\ref{p1fp80}, we get
$$
\begin{aligned}
\Big|\int\pf_n^k\1\,d\nu_n-\int Q_n\1\,d\nu_n\Big|
&=\Big|\int\(\pf_n^k-\lam_n^kQ_n\)\1\,d\nu_n+\int\(\lam_n^k-1\)Q_n\1\,d\nu_n\Big|\\
&\le \Big\|\(\pf_n^k-\lam_n^kQ_n\)\1\Big\|_\infty\,+\big|\lam_n^k-1\big|\cdot\big\|Q_n\1\big\|_\infty \\
&\le\big\|\De_n^k\big\|_\infty\,+C\big|\lam_n^k-1\big| \\
&\le C\ka^k+C\big|\lam_n^k-1\big|.
\end{aligned}
$$So, fixing $\vep>0$, we first take and fix $k\ge 1$ large enough so that $C\ka^k<\vep/2$, and then using Proposition~\ref{p1fp80}, we take $n\ge 1$ large enough so that $C\big|\lam_n^k-1\big|<\vep/2$. Combining this with \eqref{1fp132e}, we finally get the desired equality
$$
\lim_{n\to\infty}\int Q_n\1\,d\nu_n=1-e^{\vp^{(p)}(\xi)}.
$$
The proof is complete.
\end{proof}

\fr Applying Lemma~\ref{l1fp132c} and proceeding along the lines of the proof of Lemma~\ref{l1fp132b} (or Lemma 4.3 in \cite{FP}), we get the following analogue of Lemma~4.5 from \cite{FP}.

\blem\label{l1fp132da}
If \emph{(U1)-(U4B)} and \emph{(U5)} hold, then
$$
\lim_{n\to\infty}\frac{\int Q_n\(\pf\1_n^c\)\,d\nu_n}{\mu_\phi(U_n)}
=\lam\(1-\lam^{-p}e^{\phi_p(\xi)}\)^2
$$
\elem

\fr Having proved Lemmas~\ref{l1fp132}, \ref{l1fp132b}, \ref{l1fp132c}, and \ref{l1fp132da}, Proposition~\ref{p1fp133} follows.

\sp We now recall a basic escape rates definitions. Let $G$ be an arbitrary subset of $E_A^\infty$. We set
\beq\label{11_2016_07_14}
\un R_{\mu_\phi}(G) := -\varlimsup_{k\to +\infty} \frac{1}{k}\log \mu_\phi\Big(\big\{\om\in E_A^\infty: \sigma^i(\om)\not\in G\, \hbox{ for all } i = 1, \cdots, k\big\}\Big)
\eeq
and 
\beq\label{12_2016_07_14}
\ov R_{\mu_\phi}(G) := -\varliminf_{k\to +\infty} \frac{1}{k}\log \mu_\phi\Big(\big\{\om\in E_A^\infty: \sigma^i(\om)\not\in G\, \hbox{ for all } i = 1, \cdots, k\big\}\Big).
\eeq
We call $\un R_{\mu_\phi}(G)$ and $\ov R_{\mu_\phi}(G)$ respectively the lower and the upper escape rate of $G$. Of course
$$
\un R_{\mu_\phi}(G)\le \ov R_{\mu_\phi}(G),
$$ 
and if these two numbers happen to be equal, we denote their common value by
$$
R_{\mu_\phi}(G)
$$ 
and call it the escape rate of $G$. We provide here for the sake of completeness and convenience of the reader the short elegant proof, entirely taken from \cite{FP}, of the following.

\blem\label{l11_2016_07_14}
If \emph{(U0)-(U5)} hold, then for all integers $n\ge 0$ large enough the escape rates $R_{\mu_\phi}(U_n)$ exist, and moreover
$$
R_{\mu_\phi}(U_n)=\log\lam-\log\lam_n.
$$
\elem

\bpf
Assume without loss of generality that the Perron-Frobenius operator $\pf:\cB_\th\to\cB_\th$ is fully normalized so that $\lam=1$ and $\pf\1=\1$. By virtue of Proposition~\ref {p1fp80} (b), (c), and (d), we have for every $n\ge 0$ large enough and for all $k\ge 1$ that
\beq\label{1cl5}
\begin{aligned}
\mu_\phi\Big(\big\{\om\in E_A^\infty: &\sigma^i(\om)\not\in U_n\, \hbox{ for all } i = 1, \cdots, k\big\}\Big)=\\
&=\mu_\phi\lt(\bi_{j=0}^{k-1}\sg^{-j}(U_n^c)\rt)
=\int_{E_A^\infty}\1_n^k\,d\mu_\phi
=\int_{E_A^\infty}\pf^k\(\1_n^k\)\,d\mu_\phi \\
&=\int_{E_A^\infty}\pf_n^k(\1)\,d\mu_\phi
=\int_{E_A^\infty}\(\lam_n^kQ_n\1+\De_n^k\1\)\,d\mu_\phi \\
&=\lam_n^k\int_{E_A^\infty}Q_n\1\,d\mu_\phi + \int_{E_A^\infty}\De_n^k\1\,d\mu_\phi.
\end{aligned}
\eeq
So, employing Proposition~\ref{p1fp80} (b) and Proposition~\ref{p1fp131}, the latter to make sure that $\lam_n\in (0,+\infty)$ and $\int_{E_A^\infty}Q_n\1\,d\mu_\phi\in (0,+\infty)$, we conclude from \eqref{1cl5} with the help of Proposition~\ref{p1fp80} (e) and (g), that the limit
$$
\lim_{k\to +\infty} \frac{1}{k}\log \mu_\phi\Big(\big\{\om\in E_A^\infty: \sigma^i(\om)\not\in U_n\, \hbox{ for all } i = 1, \cdots, k\big\}\Big)
$$
exists and is equal to $\log\lam_n$. The proof is complete.
\epf

\fr Now we are in position to prove the following main result of this section.

\bprop\label{p2fp133}
With the setting of Sections~\ref{PFOriginal} and \ref{SecSingPerturb}, assume that {\rm (U0)--(U5)} hold. Then
$$
\lim_{n\to\infty}\frac{R_{\mu_\phi}(U_n)}{\mu_\phi(U_n)}=
\begin{cases}
1 &\text{{\rm if (U4A) holds,}} \\
1-\exp\(\phi_p(\xi)-p\P(\phi)\) &\text{{\rm if (U4B) holds.}}
\end{cases}
$$
\eprop

\bpf
By Lemma~\ref{l11_2016_07_14} we have
$$
R_{\mu_\phi}(U_n)
=\frac{\log\lam-\log\lam_n}{\mu_\phi(U_n)}
=-\frac{\log\lam_n-\log\lam}{\lam_n-\lam}\cdot \frac{\lam_n-\lam}{\mu_\phi(U_n)}.
$$
Therefore, invoking Proposition~\ref{p1fp133}, we get that
$$
\begin{aligned}
\lim_{n\to\infty}\frac{R_{\mu_\phi}(U_n)}{\mu_\phi(U_n)}
&=\lim_{n\to\infty}\frac{\log\lam_n-\log\lam}{\lam_n-\lam}\cdot
 \lim_{n\to\infty}\frac{\lam-\lam_n}{\mu_\phi(U_n)} \\
&=\frac1{\lam}
\begin{cases}
\lam &\text{{\rm if (U4A) holds,}} \\
\lam(1-\lam^{-p}\exp\(\phi_p(\xi)\) &\text{{\rm if (U4B) holds.}}
\end{cases} \\
&=\begin{cases}
1 &\text{{\rm if (U4A) holds,}} \\
1-\exp\(\phi_p(\xi)-p\P(\phi)\) &\text{{\rm if (U4B) holds.}}
\end{cases}
\end{aligned}
$$
The proof is complete.
\epf

\

\part{Escape Rates for Conformal GDMSs and IFSs}

\sp

Our approach to proving results on  escape rates for conformal graph directed Markov systems and conformal iterated function systems is based on the symbolic dynamics, more precisely, the symbolic thermodynamic formalism, developed in the preceding sections.  

\section{Preliminaries on Conformal GDMSs}\label{Attracting_GDMS_Prel}

A Graph Directed Markov System (abbr. GDMS) consists of a directed
multigraph and an associated incidence matrix, $(V,E,i,t,A)$. As earlier $A$ is the incidence matrix, i. e.
$$
A:E\times E\to \{0,1\}
$$
The multigraph consists of a finite set $V$ of
vertices and a countable (either finite or infinite) set of directed
edges $E$ and two functions $i,t:E\to V$ together with a set of nonempty compact metric
spaces $\{X_v\}_{v\in V}$, a number $s$, $0<s<1$, and for every $e\in
E$, a 1-to-1 contraction $\phi_e:X_{t(e)}\to X_{i(e)}$ with a Lipschitz constant
$\le s$. For brevity, the set 
$$
S=\{\phi_e:X_{t(e)}\to X_{i(e)}\}_{e\in E}
$$
is called a Graph Directed Markov System (abbr. GDMS). The main object
of interest in this book will be the limit set of the system $S$ and
objects associated to this set. We now describe the limit set.
For each $\om \in E^*_A$, say $\om\in E^n_A$, we consider the map coded
by $\om$:
$$
\phi_\om=\phi_{\om_1}\circ\ld\circ\phi_{\om_n}:X_{t(\om_n)}\to X_{i(\om_1)}.
$$
For $\om \in E^\infty_A$, the sets
$\{\phi_{\om|_n}\(X_{t(\om_n)}\)\}_{n\ge 1}$ form a descending 
sequence of non-empty compact sets and therefore $\bi_{n\ge
1}\phi_{\om|_n}\(X_{t(\om_n)}\)\ne\es$. Since for every $n\ge 1$,
$\diam\(\phi_{\om|_n}\(X_{t(\om_n)}\)\)\le s^n\diam\(X_{t(\om_n)}\)\le
s^n\max\{\diam(X_v):v\in V\}$, we conclude that the intersection 
$$
\bi_{n\ge1}\phi_{\om|_n}\(X_{t(\om_n)}\)
$$
is a singleton and we denote its only element by $\pi(\om)$. In this 
way we have defined the map 
$$
\pi:E^\infty_A\longrightarrow X:=\du_{v\in V}X_v
$$
from $E_A^\infty$ to $\du_{v\in V}X_v$, the disjoint union of the compact
sets $X_v$. The set
$$
J=J_S=\pi(E^\infty_A)
$$
will be called the limit set of the GDMS $S$. 

In order  to pass to geometry, we call a GDMS conformal (CGDMS) if the 
following conditions are satisfied.

\sp
\begin{itemize}
\item[(a)] For every vertex $v\in V$, $X_v$ is a compact connected
subset of a euclidean space $\R^d$ (the dimension $d$ common for all
$v\in V$) and $X_v=\ov{\Int(X_v)}$.

\sp\item[(b)] (Strong Open Set Condition) This  consists of two parts:

\sp\begin{itemize}
\item[(b1)] (Open Set Condition) For all $a,b\in E$, $a\ne b$,
$$
\phi_a(\Int(X_{t(a)}))\cap \phi_b(\Int(X_{t(b)}))=\es,
$$
and 
%the second part:
\item[(b2)] 
$$
J_\cS\cap \Int(X)=\es.
$$
\end{itemize}
\item[(c)]  For every vertex $v\in V$ there exists an open connected set
$W_v\spt X_v$ (where $X= \cup_{v\in V}X_v$)
 such that for every $e\in I$ with $t(e)=v$, the map
$\f_e$ extends to a $C^1$ 
conformal diffeomorphism of $W_v$ into $W_{i(e)}$. 

\sp\item[(d)] There are two constants $L\ge 1$ and $\a>0$ such that
$$
\bigl||\f_e'(y)|-|\f_e'(x)|\bigr|\le L||(\f_e')^{-1}||^{-1}||y-x||^\a.
$$ 
for every $e\in E$ and every pair of points $x,y\in X_{t(e)}$, where 
$|\f_\om'(x)|$ means the norm of the derivative.
\end{itemize}

\sp
\brem\lab{p1.033101}
If $d\ge 2$ and a family $S=\{\phi_e\}_{e\in E}$ satisfies the conditions (a) and (c), then, due to Koebe's Distortion Theorem in dimension $d=2$ and the  Loiuville Representation Theorem (stating  that if $d\ge 3$ then each conformal map is necessarily a M\"obius transformation)  it also satisfies condition (d) with $\a=1$.
\erem

\brem\lab{Rem: Cone Condition}
In the papers \cite{MU_lms} and \cite{GDMS} there appeared also the so called Cone Condition. This condition was however exclusively needed only to prove the following.
\bthm\lab{t4.3.1} 
If $\mu$ is a Borel shift-invariant ergodic
probability measure on $E_A^\infty$, then
\begin{equation}\lab{4.3.1}
\mu\circ\pi^{-1}\(\phi_\om(X_{t(\om)})\cap\phi_\tau\(X_{t(\tau)})\)=0 
\end{equation}
for all incomparable words $\om,\tau\in E^*$.
\ethm

\fr This theorem is of particular importance if measure $\mu$ is a Gibbs state of a H\"older continuous function. The following slight strengthening of Theorem~\ref{t4.3.1} however immediately follows from the Strong Open Set Condition.
\bthm\lab{t4.3.1B} 
If $\mu$ is a Borel shift-invariant ergodic probability measure on $E_A^\infty$ with full topological support, then
\begin{equation}\lab{4.3.1+1}
\mu\circ\pi^{-1}\(\phi_\om(X_{t(\om)})\cap\phi_\tau\(X_{t(\tau)})\)=0 
\end{equation}
for all incomparable words $\om,\tau\in E^*$.
\ethm

\fr Indeed, the Strong Open Set Condition ensures that for such measures $\mu$ 
$$
\mu(\Int(X))>0
$$
and, since we clearly have,
$$
\sg^{-1}(\pi^{-1}(\Int(X))\sbt\Int(X),
$$
we thus conclude from ergodicity of $\mu$ that $\mu(\Int(X))=0$. The assertion of Theorem~\ref{t4.3.1B} thus follows. Note also that all Gibbs states are of full support.

\sp\fr We would like however to complete  this comment by saying that in the case of finite alphabet $E$ the Open set Condition alone suffices, and the item (b2) is not needed at all. It is not needed in the case of infinite alphabet either as long as we are only interested in the Hausdorff dimension of the limit set, i. e. as long as we only want prove Bowen's Formula.
\erem

\fr Let $F=\{f^{(e)}:X_{t(e)}\to\R:e\in E\}$ be a family of real-valued
functions. For every $n\ge 1$ and $\b>0$ let
$$
V_n(F)=\sup_{\om\in E^n}\sup_{x,y\in X_{t(\om)}}\{|f^{(\om_1)}(\phi_{\sg(\om)}
(x))-f^{(\om_1)}(\phi_{\sg(\om)}(y))|\}\ep^{\b(n-1)},
$$
We have made the conventions that the empty
word $\es$ is the only word of length $0$ and $\phi_\es=\text{Id}_X$.
Thus, $V_1(F) < \infty$ simply means the diameters of the sets $f^{(e)}(X)$
are uniformly bounded. The collection $F$ is called a H\"older family of 
functions (of order $\b$) if 
\begin{equation}\lab{3.1.1}
V_\b(F)=\sup_{n\ge 1}\{V_n(F)\}<\infty.  
\end{equation}
We call the H\"older family $F$, summable (of order $\b$) 
if (\ref{3.1.1}) is satisfied and 
\begin{equation}\lab{3.1.2}
\sum_{e\in E}\exp\(\sup\(f|_{[e]}\)\)<+\infty. 
\end{equation}
In order to get the link with the previous sections on thermodynamic formalism on  symbol spaces, we introduce now a potential function or amalgamated function, $f:E^\infty\to\R$, induced by the family of functions $F$ as follows. 
$$
f(\om)= f^{(\om_1)}(\pi(\sg(\om))).
$$
Our convention will be to use lower case letters for the potential
function corresponding to a given H\"older system of functions. The following
lemma is a straightforward, see \cite{GDMS} for a proof.

\blem\lab{l3.1.3} 
If $F$ is a H\"older family (of order $\b$) then the 
amalgamated function $f$ is H\"older continuous (of order $\b$). If $F$ is summable, then so is $f$.
\elem

\fr We recall from \cite{MU_lms} and \cite{GDMS} the following definitions:
$$
\th_\cS:=\inf\Ga_\cS \hbox{ where } \Ga_\cS=\inf\lt\{s\ge 0: \sum_{e\in E}||\phi_e'||_\infty^s<+\infty\rt\}.
$$
The proofs of the following two statements can be found in \cite{GDMS}.

\bprop\label{p1_2016_01_12}
If $\cS$ is an irreducible conformal {\rm GDMS}, then for every $s\ge 0$ we have that 
$$
\Ga_\cS=\{s\ge 0: \P(s)<+\infty\}.
$$
In particular,
$$
\th_\cS:=\inf\lt\{s\ge 0: \P(s)<+\infty\rt\}.
$$
\eprop

\bthm\label{t3_2016_01_12}
If $\cS$ is a finitely irreducible conformal {\rm GDMS}, then the function $\Ga_\cS\ni s\mapsto \P(s)$ is 

\begin{itemize}
\sp\item strictly decreasing, 

\sp\item real-analytic, 

\sp\item convex, and 

\sp\item $\lim_{s\to+\infty}\P(s)=-\infty$.
\end{itemize}
\ethm

\fr We also introduce the following important characteristic of the system $\cS$.
$$
b_\cS:=\inf\{s\ge 0: \P(s)\le 0\}\ge \th_\cS.
$$
We call $b_\cS$ the Bowen's parameter of the system $\cS$. The following theorem, providing a geometrical interpretation of this parameter has been proved in \cite{GDMS}.

\bthm\label{t2_2016_01_12}
If $\cS$ is an finitely irreducible conformal {\rm GDMS}, then 
$$
\HD(J_\cS)=b_\cS\ge \th_\cS.
$$
\ethm

\fr Following \cite{MU_lms} and \cite{GDMS} we call the system $\cS$ regular if there exists $s\in (0,+\infty)$ such that 
$$
\P(s)=0.
$$
Then by Theorems~\ref{t2_2016_01_12} and \ref{t3_2016_01_12}, such a zero is unique and is equal to $b_\cS$. 

\sp\fr We call the system $\cS$ strongly regular if there exists $s\in [0,+\infty)$ (in fact in $(\g_\cS,+\infty)$) such that 
$$
0<\P(s)<+\infty. 
$$
By Theorem ~\ref{t3_2016_01_12} each strongly regular conformal GDMS is regular. 

\sp Let $\zeta: E^\infty_A \to \mathbb{R}$ be defined by the formula
\beq\label{1MU_2014_09_10}
\zeta(\om)= \log|\vp'_{\om_1}(\pi(\sg(\om))|.
\eeq
Let us record the following obvious observation.

\bobs\label{o1_2016_02_12}
For every $t\ge 0$, $t\zeta$ is the amalgamated function of the following family of functions:
$$
t\Xi:=\{X_{t(e)}\ni x\mapsto t\log|\phi_e'(x)|\in \R\}_{e\in E}.
$$
\eobs

\fr The following proposition is easy to prove; see \cite[Proposition 3.1.4]{GDMS} for complete details. 

\bprop\label{l1j85}
For every real $t\geq 0$ the function $t\zeta:E^\infty_A \to
\mathbb{R}$ is H\"older continuous and $t\Xi$ is a H\"older continuous family of functions. 
\eprop

\bobs\label{o5_2016_0212}
For every $t\ge 0$ we have that $t\in \Ga_\cS$ if and only if the H\"older continuous potential $t\zeta$ is summable if and only if the H\"older continuous family of functions $t\Xi$ is summable.
\eobs

\fr We denote:
$$
\P(\sg,t\zeta):=\P(t).
$$
for every $t\ge 0$.

\section{More Technicalities on Conformal GDMSs}

We keep the setting and notation from the previous section. 

\sp\begin{itemize}
\item We call a point $z\in X$ pseudo-periodic for $\cS$ if there exists $\om\in E_A^*$ such that $z\in X_{t(\om_0)}$ and $\phi_\om(z)=z$. 

\sp\item We call a point $z\in\cS$ periodic for $\cS$ if $z=\pi(\om)$ for some periodic element $\om\in E_A^\infty$. 

\sp\item Of course every periodic point is pseudo-periodic. Also obviously, for maximal graph directed Markov systems, in particular for conformal iterated function systems, periodic points and pseudo-periodic points coincide. 

\sp\item We call a periodic point $z\in J_\cS$ uniquely periodic if $\pi^{-1}(z)$ is a singleton and there is exactly one $\xi\in E_A^*$ such that the infinite concatenation $\xi^\infty\in E_A^\infty$, $\phi_\xi(z)=z$, and if $\phi_\a(z)=z$ for some $\a\in E_A^*$, then $\a=\xi^q$ for some integer $q\ge 1$. 
\end{itemize}

\sp We shall prove the following.

\blem\label{l1fp82}
If $z\in J_\cS$ is not pseudo-periodic for $\cS$, then
$$
\pi^{-1}(z)\cap \bu_{k=1}\sg^n(\pi^{-1}(z))=\es.
$$
\elem

\begin{proof}
Assume for a contradiction that there exists $\om\in\pi^{-1}(z)$ such that $\sg^n(\om)\in\pi^{-1}(z)$ for some $n\ge 1$. We then have 
$$
\phi_{\om|_n}(z)
=\phi_{\om|_n}\(\pi(\sg^n(\om))\)=\pi\(\om|_n\sg^n(\om)\)
=\pi(\om)
=z.
$$
So, $z$ is pseudo-periodic, and this contradiction finishes the proof. 
\end{proof}

\fr In fact, we will need more:

\blem\label{l2fp82}
Assume that $z\in J_\cS$ is not pseudo-periodic for the system $\cS$. If $k\ge 1$ is an integer, $(l_n)_{n=1}^\infty$ is a sequence of integers in $\{1,2,\ld,k\}$, and $\(\tau^{(n)}\)_{n=1}^\infty$ is a sequence of points in $E_A^\infty$ such that
$$
\lim_{n\to\infty}\pi\(\tau^{(n)}\)
=\lim_{n\to\infty}\pi\(\sg^{l_n}(\tau^{(n)})\)
=z,
$$
then
$$
\lim_{n\to\infty}\sum_{i=0}^{l_n}\tau_i^{(n)}=+\infty.
$$
\elem

\begin{proof}
Seeking a  contradiction suppose that 
$$
\varliminf_{n\to\infty}\sum_{i=0}^{l_n}\tau_i^{(n)}<+\infty.
$$
Passing to a subsequence, we may assume without loss of generality 
$$
\varlimsup_{n\to\infty}\sum_{i=0}^{l_n}\tau_i^{(n)}<+\infty.
$$
There then exists $M\in (0,+\infty)$ such that 
$$
\sum_{i=0}^{l_n}\tau_i^{(n)}\le M
$$
for all $n\ge 1$. Hence,
$$
\tau_i^{(n)}\le M
$$
for all $n\ge 1$ and all $i=0,1,2,\ld,l_n$. So, passing to yet another subsequence, we may further assume that the sequence $(l_n)_{n=1}^\infty$ is constant, say $l_n=l$ for all $n\ge 1$, and that for every $i=0,1,2,\ld,l_n$ the sequence $\(\tau^{(n)}\)_{n=1}^\infty$ is constant, say $\tau_i^{(n)}=\tau_i\le M$ for all $n\ge 1$. Let 
$$
\tau:=\tau_1\tau_2\ld\tau_l.
$$
it then follows from our hypothesis that
$$
\begin{aligned}
\phi_\tau(z)
&=\phi_\tau\(\lim_{n\to\infty}\pi\(\sg^l(\tau^{(n)})\)\)
=\lim_{n\to\infty}\phi\(\tau\(\pi\(\sg^l(\tau^{(n)})\)\)\)\\
&=\lim_{n\to\infty}\pi\(\sg^l(\tau^{(n)})\) \\
&=\lim_{n\to\infty}\pi\(\tau^{(n)}\)
=z.
\end{aligned}
$$
Thus $z$ is a pseudo-periodic point for the graph directed Markov system $\cS$, and this contradiction finishes the proof.
\end{proof}

\fr A statement corresponding to Lemma~\ref{l2fp82} in the case of periodic points is the following.

\blem\label{l2fp82a}
Assume that $z\in J_\cS$ is uniquely periodic for the system $\cS$
%Recall that this condition means that
(i.e.,  $\pi^{-1}(z)$ is a singleton and that there exists a unique point $\xi
\in E_A^*$ such that $\xi^\infty\in E_A^\infty$, $\phi_\xi(z)=z$, and if $\phi_\a(z)=z$ for some $\a\in E_A^*$, then $\a=\xi^q$ for some integer $q\ge 1$). Then if $k\ge 1$ is an integer, 
$(l_n)_{n=1}^\infty$ is a sequence of integers in $\{1,2,\ld,k\}$, and $\(\tau^{(n)}\)_{n=1}^\infty$ is a sequence of points in $E_A^\infty$ such that

\begin{itemize}
\item[(a)]
$$
\lim_{n\to\infty}\pi\(\tau^{(n)}\)
=\lim_{n\to\infty}\pi\(\sg^{l_n}(\tau^{(n)})\)
=z,
$$
\fr\fr\fr\fr\fr and

\sp\item[(b)]
$$
\varlimsup_{n\to\infty}\sum_{i=0}^{l_n}\tau_i^{(n)}<+\infty,
$$
\end{itemize}
then $l_n$ is an integral multiple of $|\xi|$, say $l_n=q_n|\xi|$, and 
$$
\tau^{(n)}|_{l_n}=\xi^{q_n}
$$
for all $n\ge 1$ large enough.

\begin{proof}
It follows from item (b) that there exists $M\ge 1$ such that 
$\tau_i^{(n)}\le M$ for all $n\ge 1$ and all $1\le i\le l_n$. Assuming the contrapositive  statement to our claim and passing to a subsequence, we may assume without loss of generality that the sequence $(l_n)_{n=1}^\infty$ is constant,
say $l_n=l$ for all $n\ge 1$, and we may further assume that for every $1\le i\le l_n$ the sequence $\(\tau^{(n)}\)_{n=1}^\infty$ is constant, say $\tau_i^{(n)}=\tau_i\in\{1,2,\ld,M\}$ for all $n\ge 1$ and
\beq\label{1fp82b}
\tau_j^{(n)}\ne\xi_j
\eeq
for all $n\ge 1$ and some $1\le j\le l$. Let 
$$
\tau:=\tau_1\tau_2\ld\tau_l.
$$
We now conclude, in exactly the same way as in the proof of Lemma~\ref{l2fp82} that $\phi_\tau(z)=z$. Therefore, since $z$ is uniquely pseudo-periodic, we get $\tau=\xi^q$ with some $q\ge 1$. In particular $q|\xi|=l$, and so, using \eqref{1fp82b}, we deduce that $\tau\ne\xi^q$. This contradiction finishes the proof.
\end{proof}

\elem

\section{Weakly Boundary Thin (WBT) Measures and Conformal GDMSs}\label{Section: WBT}

In this section we first introduce the concept of Weakly Bounded Thin (WBT) measures. Roughly speaking, this notion relates the measure of an annulus to the measure of the ball it encloses.  We prove some basic properties of (WBT) and provide some sufficient conditions for (WBT) to hold for a large class of measures on the limit set of a CGDMS.  We were able to establish these properties, mainly due to the progress achieved in \cite{Pawelec-Urbanski-Zdunik},
%Let $\mu$ be a Borel probability measure on a metric space $X$. Let $\a>0$. We say that $\mu$ is $\a%$-upper powering if there exists a constant $C>0$ such that 
%$$
%\mu(B(x,r))\le Cr^\a
%$$
%for all $x\in X$ and all radii $r>0$ sufficiently small. Keep $\mu$, a Borel probability measure on %a separable metric space $(X,d)$.

Let $\mu$ be a Borel probability measure on a separable metric space $(X,d)$. For all $\b>0$, $x\in X$ and $r>0$, let
$$
A_\mu^\b(x,r):=A\(x;r-\mu(B(x,r))^\b,r+\mu(B(x,r))^\b\),
$$
where, in general, 
$$
A(z;r,R):=B(z,R)\sms B(z,r)
$$
is the annulus centered at $z$ with the inner radius $r$ and the outer radius $R$. We say that the measure $\mu$ is weakly boundary thin (WBT) (with exponent $\b$) at  the point $x$ if 
$$
\lim_{r\to 0}\frac{\mu\(A_\mu^\b(x,r)\)}{\mu(B(x,r))}=0.
$$ 
Given $\a>0$, we further define:
$$
A_\mu^{\b,\a}(x,r):=A\(x;r-\a\mu(B(x,r))^\b,r+\a\mu(B(x,r))^\b\).
$$
The following proposition is obvious.

\bprop\label{p2fp132eba}
If $\mu$ is  a Borel probability measure on a separable metric space $X$, then for every point $x\in\supp(\mu)$, the following are equivalent.

\sp\begin{itemize}

\sp\item[(a)] $\mu$ is {\rm (WBT)} at $x$.

\sp\item[(b)] There exists $\b>0$ such that the measure $\mu$ is {\rm (WBT)} at $x$ with exponent $\gamma >0$ either if and only if $\g\in(\b,+\infty)$ or  if and only if $\g\in[\b,+\infty)$. Denote this $\b$ by $\b_\mu(x)$.

\sp\item[(c)] There exist $\a, \b>0$ such that 
$$
\lim_{r\to 0}\frac{\mu\(A_\mu^{\b,\a}(x,r)\)}{\mu(B(x,r))}=0.
$$ 
\sp\item[(d)] For every $\b\in (\b_\mu(x),+\infty)$, 
$$
%\lim_{r\to 0}\frac{\mu\(A_\mu^{\b,\a}(x,r)\)}{\mu(B(x,r))}=0.
$$ 
\end{itemize}
\eprop

\fr We say that a measure is weakly boundary thin (WBT) if it is (WBT) at every point of its topological support. We also say that a measure is weakly boundary thin almost everywhere (WBTAE) if it is (WBT) at almost every point. Of course (WBT) implies (WBTAE).

Now we aim to provide sufficient conditions for a Borel probability measure to be (WBT) and (WBTAE). Let $\mu$ be an arbitrary Borel probability measure on a separable metric space.   Let $\alpha > 0$.  We say that $\mu$ is $\alpha$-upper Ahlfors ($\alpha$-up) at a point $x \in X$ if there exists a constant $C>0$ (which may depend on $x$) such that
$$
\mu(B(x,r)) \leq C r^\alpha
$$
for  all radii $r >0$.  Equivalently, for all radii $r>0$ sufficiently small.
Following \cite{PUbook}, the  measure $\mu$ is said to have the Thin Annuli Property (TAP) at a point $x\in X$ if there exists $\kappa > 0$ (which may depend on $x$) such that 
$$
\lim_{r \to 0} \frac{\mu(A(x; r, r + r^\kappa))}{\mu(B(x,r))}=0
$$
We shall easily show the following.

\begin{prop}\label{p1wbt4}
Let $(X,d)$ be a separable metric space, let $\mu$ be a Borel probability measure on $X$ and let $\alpha>0$.  If $\mu$ is $\alpha$-upper Ahlfors with the Thin Annuli Property {\rm(TAP)} at some $x\in X$, then $\mu$ is {\rm (WBT)} at $x$.
\end{prop}
\begin{proof}

Taking $\beta >0$ so large that $C^\beta r^{\beta \alpha} \leq r^\kappa$ for $r > 0$ small enough.  Then for each radii $r>0$ we have that $A_\mu^\beta(x; r, r+ r^\kappa) \subset A(x, r, r + r^\kappa)$ and thus
$$
0 \leq \limsup_{r\to 0} \frac{\mu(A_\mu^\beta(x,r))}{\mu(B(x,r))}
\leq \limsup_{r\to 0} \frac{\mu(A(x,r, r + r^\kappa))}{\mu(B(x,r))} = 0
$$
The proof is then complete.
\end{proof}

We recall from the book \cite{PUbook} that 
$$
\HD_*(\mu) = \inf \{\HD(Y)  \hbox{ : } Y \subset X \hbox{ is Borel and } \mu(Y) > 0\}.
$$

We call $\hbox{\rm HD}_*(\mu)$ the lower Hausdorff dimension of $\mu$.   The Hausdorff Dimension of $\mu$ is commonly defined to be 
$$
\HD(\mu) = \inf \{\HD(Y)  \hbox{ : } Y \subset X \hbox{ is Borel and }\mu(Y) =1 \}
$$
The reader should be aware that in \cite{PUbook} the above infimum is denoted $\HD^*(\mu)$
and is called the upper Hausdorff Dimension of $\mu$.   We however, will use  the more commonly accepted tradition rather than the point of view taken in \cite{PUbook}.  Referring to the well-known fact (see \cite{PUbook} for instance) that if $\mu(B(x,t)) \geq C r^\gamma$ for the points $x$ belonging to some Borel set $F \subset X$ then $\hbox{\rm HD}(F) \leq \gamma$, we immediately obtain the following.   

\begin{lem}\label{l1wbt5}  If $\hbox{\rm HD}_*(\mu) > 0$ then $\mu$ is $\alpha$-upper Ahlfors
for every $\alpha \in (0, \hbox{\rm HD}_*(\mu))$ and $\mu$--a.e. $x\in X$ with some constant $C\in(0,+\infty)$ and every $r>0$ small enough. 
\end{lem}

\bdfn\label{d7_2016_07_07}
We say that a set $J\sbt\R^d$, $d\ge 1$, is geometrically irreducible if it is not contained in any countable union of conformal images of hyperplanes or spheres of dimension $\le d-1$. 
\edfn 

\bobs
Every set $J\sbt\R^d$, $d\ge 1$, with $\HD(J)>d-1$ is geometrically irreducible.
\eobs

\bobs
If a set $J\sbt\C$ is not contained in any countable union of real analytic curves, then $J$ is geometrically irreducible.
\eobs

Now we can apply the results obtained above, in the context of CGMS. We shall prove the following.

\begin{thm}\label{t1wbt6}
Let $\mathcal S = \{\phi_e\}_{e\in E}$ be a finitely primitive {\rm CGDS}
satisfying the {\rm SOSC} with a phase space $X \subset \mathbb R^d$. Let $\psi: E_A^{\mathbb N} \to \mathbb R$ be a H\"older continuous strongly summable potential, the latter meaning that 
\beq\label{1wbt6.1}
\sum_{e\in E} \exp \left( \inf (\phi|_{[e]})\right)\|\phi_e'\|^{-\beta} < +\infty
\eeq
for some $\beta >0$. 
As usual, let $\mu_\psi$ denote its unique equilibrium state. If the limit set of $J_{\mathcal S}$ is geometrically irreducible, then 
\begin{enumerate}
\item[(a)] $\HD_*(\mu_\psi \circ \pi_{\mathcal S}^{-1}) = \HD(\mu_\psi\circ \pi_{\mathcal S}^{-1}) > 0$;
\item[(b)] The measure $\mu_\psi \circ \pi_{\mathcal S}^{-1}$ satisfies the Thin Annuli Property {\rm (TAP)} at 
$\mu_\psi \circ \pi_{\mathcal S}^{-1}$ a.e. point of $\mathcal S$
\item[(c)]
$\mu_\psi \circ \pi_{\mathcal S}^{-1}$ is {\rm (\WBT)} at $\mu_{\psi}$ a.e. point of 
$J_{\mathcal S}$.
\end{enumerate}
\end{thm}

\begin{proof}
The proof of Theorem 4.4.2 in \cite{MU_lms} \cite{GDMS} gives in fact that the measure $\mu_\psi$ is dimensionally exact, i.e., that
$$
\lim_{r \to 0}  \frac{\log \mu_\psi \circ \pi_{\mathcal S}^{-1}(B(x,r))}{\log r}
$$ 
exists for $\mu_\psi \circ \pi_{\mathcal S}^{-1}$ for a.e. $x \in  J_{\mathcal S}$ and is equal to $\h_{\mu_{\psi}}(\sigma)/\chi_{\mu_\psi}>0$.  A complete proof with all the details can be found in the last section of \cite{CTU}.   Therefore, property (a) is established.  Property (b) follows now immediately from Theorem 30 in  \cite{Pawelec-Urbanski-Zdunik}.  Condition (c) is now an immediate consequence of (a),(b), Lemma \ref{l1wbt5} and Proposition \ref{p1wbt4}.
\end{proof}

\begin{rem}\label{r1wbt7} Condition \ref{1wbt6.1} is satisfied for instance for all potentials of the form $E_A^{\mathbb N} \ni \omega \mapsto t \log |(\phi_{\mathcal S}'(\pi_{\mathcal S}\sigma(\omega)))| \in \mathbb R$, where $t > \theta_{\mathcal S}$.   It also holds for $t = \theta_{\mathcal S}$ if the system $\mathcal S$ is strongly regular.
\end{rem}

\sp Now we shall deal with the case of a finite alphabet. We shall show that in the case of a finite alphabet (under a mild geometric condition in dimension $d\ge 2$) the equilibrium states of all H\"older continuous potentials satisfy (WBT) at every point of the limit set. Thus our approach is complete in the case of the finite alphabet and  present paper entirely covers the case of conformal IFSs (even GDMSs) with finite alphabet.  We shall prove the following.

\begin{thm}\label{t2wbt11}
Let $E$ be a finite set and let $\mathcal S = \{\phi_e\}_{e \in E}$ be a primitive conformal {\rm GDMS} acting in the space $\R$. If $\psi: E_A^{\infty} \to \mathbb R$ is an arbitrary H\"older continuous (with the phase space sets $X_v \subset W_v \subset \mathbb R$, $v \in V$) and  $\mu_\phi$ is the corresponding equilibrium state on $E_A^{\infty}$ then the projection measure $\mu_\psi \circ \pi_{\mathcal S}^{-1}$ is {\rm (WBT)} at every point of ${ J}_{\mathcal S}$.
\end{thm}
\begin{proof}
Put 
$$
u  := K^{-1} \min\big\{ \|\phi_e'\| \hbox{ : } e \in E\big\}
$$
so that 
$$
|\phi_e'(x)| \geq u
$$
for all $e \in E$ and all $x \in X_{t(e)}$. For ease of notation we denote 
$$
\widehat \mu_{\psi} :=  \mu_{\psi} \circ \pi_{\mathcal S}^{-1}.
$$
Fix $\beta > 0$. Consider the family 
$$
\mathcal F_{\psi}^\beta(z,r):= \{ \omega \in E_A^k \hbox{ : } A^\beta_{\mu_\psi}(z,r)
\cap \phi_{\omega} (J_{\omega_{|\omega|-1}}) \neq \emptyset
\hbox{ and }\|\phi_{\omega}'\| \geq \mu_\psi(B(z,r)^\beta)
\}
$$
Now consider $\widehat{\mathcal F}_\psi^\beta(z,r)$, the family of  all words in 
${\mathcal F}_\psi^\beta(z,r)$  that have no extensions to elements in 
${\mathcal F}_\psi^\beta(z,r)$, where we don't consider a finite word to be an extension of itself.    Note that then:

\sp\begin{enumerate}
\item[(a)]
$\widehat{\mathcal F}_\psi^\beta(z,r)$ consists of mutually incomparable  words;

\sp\item[(b)]
$\bu_{\omega  \in \widetilde {\mathcal F}_\psi^\beta(z,r)}[w] 
 \supset \pi_{\mathcal S}^{-1} (A_{\mu_\psi}^\beta (z,r))$; and  

\sp\item[(c)]
 $\forall \omega \in \widehat {\mathcal F}_\psi^\beta(z,r)$, 
 $\mu_\psi (B(z,r))^\beta \leq \| \phi_\omega'\| \leq u^{-1} 
 \mu_\psi (B(z,r))^\beta$
\end{enumerate}

\sp\fr Therefore the family 
$$
\{\phi_{\omega} (\Int (X_{t(\omega)}))  \hbox{ : } \omega \in \widehat {\mathcal F}_\psi^\beta(z,r) \}
$$
consists of mutually disjoint open sets each of which contains a ball of radius
$K^{-1} R \mu_\psi (B(z,r))^\beta$ where $R$ is as in the proof of Lemma \ref{l1wbt6}.  Since also 
$$
\bigcup_{\omega \in \widehat {\mathcal F}_\psi^\beta(z,r)} \phi_{\omega}(X_{t(\omega)})
\subset A(z, r - (1+ D M^{-1}) \mu_{\psi} (B(z,r))^\beta, r - (1+ D M^{-1}) \mu_{\psi}(B(z,r))^\beta   )
$$
we obtain that
\begin{equation}\label{2wbt14}
\begin{aligned}
\# \widehat {\mathcal F}_\psi^\beta(z,r)
&\leq
 \frac{\Leb_1(A(z, r - (1+ D M^{-1}) \mu_{\psi} (B(z,r)^\beta), r - (1+ D u^{-1}) \mu_{\psi} (B(z,r))^\beta  )}{ 2 K^{-1} R \mu_{\psi}(B(z,r))^\beta}\cr
&\approx \frac{\mu_\psi^\beta (B(z,r))}{\mu_\psi^\beta (B(z,r))} = 1.
\cr
\end{aligned}
\end{equation} 
So we have shown that the number of elements of $\widetilde {\mathcal F}_\psi^\beta(z,r)$ is uniformly bounded above, and in order to estimate
$\widehat \mu_\psi (A_{\mu_\psi}^\beta(z,r))$. i.e. in order to complete the proof 
we now only  need a sufficiently good upper bound on $\mu_\psi ([\omega])$ for all $\omega \in \widehat {\mathcal F}_\psi^\beta(z,r)$.  We will do so now.  It is well known (see \cite{MU_lms}, \cite{GDMS}) that there are two constants $\eta \in (0, +\infty)$ and $C \in (0, +\infty)$ such that
\beq\label{1wbt14}
\mu_\psi([\tau]) \leq C \exp (- \eta (|\tau| +1)) 
\eeq
 for all 
$\tau \in E_A^*$.   Fix $\omega \in \widehat{\mathcal F}_\psi^\beta(z,r)$.  
Denote $k:= |\omega|$.  Invoking (c) we get that $u^k \leq \|\phi_\omega'\|
\leq u^{-1}\mu_\psi^\beta (B(z,r))$, whence 
$$
k+1 \geq \frac{\beta \log \mu_\psi(B(z,r))}{\log u}.
$$
Inserting this into  (\ref{1wbt14}) we get that 
$$
\mu_\psi([\omega]) \leq C \exp \left( - \beta \eta \frac{\log \mu_\psi(B(z,r))}{\log u} \right) = C \mu^{\gamma \beta}(B(z, r))
$$
 where $\gamma = \frac{\eta}{\log(1/u)} \in (0, +\infty)$.  Having this and invoking 
 (b) and (\ref{2wbt14}) we obtain that 
$$
\frac{\widehat\mu_\psi (A_{\mu_\psi}^\beta(z, r))}{\mu_\psi(B(z,r))}
\leq 
\frac{\widehat \mu_\psi (B(z,r))^{\gamma \beta}}{\widehat \mu_\psi(B(z,r))}
\leq\mu_\psi (B(z,r))^{\gamma \beta-1}
$$ 
and the proof is complete by noting that $\lim_{r \to 0}\mu_\psi (B(z,r))^{\gamma \beta-1} = 0$ provide that we take $\gamma > 1/\beta$.
 \end{proof}
 
\sp Now we pass to the case of $ d \geq 2$.   We get the same full result as in the case of $d=1$ but with a small additional assumption 
 that the conformal system $\mathcal S$  is geometrically irreducible.
 
 \begin{thm}\label{t1wbt15}
 Let $E$  be a finite set and let $\mathcal S = \{ \phi_e\}_{e \in E}$ be a primitive geometrically irreducible conformal {\rm GDMS}
 with the phase space sets $X_v \subset W_v \subset \mathbb R^d$. 
  If $\psi:  E_A^{\mathbb N} \to \mathbb R$
 is an arbitrary H\"older continuous potential and $\mu_\phi$ is the corresponding equilibrium state then the 
 projection measure $\mu_\psi\circ \pi_{\mathcal S}^{-1}$ is {\rm (WBT)} at every point of $J_{\mathcal S}$
\end{thm} 
 
\begin{proof}
The meaning of $\widehat \mu_\psi$ is exactly the same as in the proof of the previous theorem. The proof of the current theorem is entirely based on the following.
 
\sp{\bf Claim~1:} 
 There are a constant $\alpha > 0$ and $C \in (0, +\infty)$ such that 
 $$
 \widehat \mu_\psi\(B\(\bd B(z,R),r\)\) \leq C r^\alpha
 $$
 for all $z \in \mathbb R^d$ and all radii $r, R>0$.
 
\sp\fr This claim is actually a sub-statement of formula (2.19) from \cite{MU-JNT} in a more specific setting. In particular: 

\sp\begin{enumerate}
\item[(a)]
\cite{MU-JNT} deals with finite IFSs rather than finite alphabet CGDMS; 

\sp\item[(b)]
\cite{MU-JNT} deals with H\"older continuous families of functions and their corresponding equilibrium states 
rather than H\"older continuous potentials on the symbol space $E_A^{\mathbb N}$ and their projections; and 

\sp\item[(c)] 
with the restrictions of (a) and (b) Claim 1 is a sub-statement of formula (2.19) 
from \cite{MU-JNT} only in the case of $d \ge 3$.
\end{enumerate} 
 However, a  close inspection of arguments leading to (2.19) of \cite{MU-JNT} indicates that the difference of (a) is 
 inessential for these arguments,  and for (b) that the only property of equilibrium states of H\"older continuous families of functions was that of being projections of H\"older continuous potentials from the symbol space $E_A^{\infty}$.  
 Concerning (c) it only remains to consider the case $d=2$.  We then redefine the family $\mathcal  F_0$ from section 2, 
 page 225 of \cite{MU-JNT} to conclude that also all the intersections of the form $X \cap L$, where $L \subset \mathbb C$
 where is a round circle (of arbitrary center and radius).   The argument in \cite{MU-JNT} leading to (2.19) goes  through 
 with obvious minor modifications.  Claim 1 is then established.
 
 Using this claim, we obtain
 $$
 \frac{\widehat \mu_\psi (A_{\mu_\psi}^\beta (z,r))}{\widehat \mu_\psi (B(x,r))}
 \leq \frac{C\widehat \mu_\psi^{\alpha \beta} (B(z,r))}{\widehat \mu_\psi (B(x,r))}
 = C \widehat \mu_\psi^{\alpha \beta -1} (B(z,r))
 $$
and the proof is complete and by noting that $\lim_{r \to 0} \mu_\psi^{\alpha \beta -1} (B(z,r))$ for every $\beta > 1/\alpha$.
\end{proof}

%\begin{rem}\label{r2wbt6}
%If $\mathcal S$ is a conformal Iterated Function System, then the hypothesis for $ J_{\mathcal S}$ not being contained in any countable union of proper real analytic submanifolds (in Theorem \ref{t1wbt6}) can be replaced by the weaker one that $ J_{\mathcal S}$ is not contained in a real analytic proper submanifold.  If $\mathcal S$ is a maximal IFS then a finite union of real analytic proper submanfolds is the right assumption.
%\end{rem}

%\begin{rem}\label{r2wbt7}
%The hypothesis for ${ J}_{\mathcal S}$ not being contained in any countable union of proper real analytic submanifolds holds for the ``majority'' of conformal CDMS and IFSs, for instance it always holds if $\HD({ J}_{\mathcal S}) > d-1$
%\end{rem}

Fixing a $\ka>0$ let
$$
N_\ka(x,r):=
%E
%\lt(
\left[
-\frac1\ka\log\mu(B(x,r)),
\right]
%\rt)
\in\N\cup\{+\infty\}
$$
where $[t]$, $t\in\R$, denotes the integer part of $t$. Let us make right away an immediately evident, but extremely important, observation.

\bobs\label{o1fp132eba}
If $\mu$ is a Borel probability measure on $X$, then for every $r>0$, we have that
$$
e^{-\ka N_\ka(x,r)}\le \mu(B(x,r))\le e^\ka e^{-\ka N_\ka(x,r)}.
$$
\eobs

\sp Now, let in addition $\cS=\{\phi_e\}_{e\in E}$ be a finitely primitive CGDMS with a phase space $X\sbt\R^d$. For every $x\in X$ and $r>0$ and an integer $n\ge 0$, let 
$$
A_n^*(x,r):=\bu\big\{\phi_\om(J):\om\in E_A^n, \  \  \phi_\om(J)\cap B(x,r)\ne\es 
\  \,  \text{ and } \  \phi_\om(J)\cap B^c(x,r)\ne\es\big\}.
$$
 
We say that the measure $\mu$ is dynamically boundary thin (DBT) at the point $x\in\ov J_\cS$ if for some $\ka>0$
\beq\label{1fp132ec}
\lim_{r\to 0}\frac{\mu\(A_{N_\ka(x,r)}^*(x,r)\)}{\mu(B(x,r))}
=0.
\eeq
We say that the measure $\mu$ is Dynamically Boundary Thin (DBT) almost everywhere if the set of points where it fails to be (DBT) is measure zero, and that the measure $\mu$ is
Dynamically Boundary Thin (DBT) if it is (DBT) at every point of its topological support. 
\fr We shall prove the following.

\bprop\label{p1fp132ec}
If a Borel probability measure $\mu$ on $\ov J_\cS$ is {\rm (WBT)} at some point $x\in \ov J_\cS$, then it is {\rm (DBT)} at $x$.
\eprop

\begin{proof}
Let $\b>0$ be as in the definition of (WBT) of $\mu$ at $x$. Since $\cS$ is a conformal GDMS, there exist constants $\eta>0$ and $D\ge 1$ such that 
$$
\diam(\phi_\om(X))\le D\eta^{-\eta|\om|}
$$
for all $\om\in E_A^*$. Therefore, if $\ka>0$ is sufficiently small, then for every $x\in \ov J_\cS$ and every $r>0$ we have
$$
\begin{aligned}
A_{N_\ka(x,r)}^*(x,r)
&\sbt A\(x;r-De^{-\ka N_\ka(x,r)},r+De^{-\ka N_\ka(x,r)}\) \\
&\sbt A\(x;r-D(\mu(B(x,r))^{\eta/\ka},r+D(\mu(B(x,r))^{\eta/\ka}\)\\
&=A_\mu^{\eta/\ka,De}(x,r).
\end{aligned}
$$
For every $r>0$, sufficiently small, we then have
$$
\frac{\mu\(A_{N_\ka(x,r)}^*(x,r)\)}{\mu(B(x,r))}
\le \frac{\mu\(A_\mu^{\eta/\ka,De}(x,r)\)}{\mu(B(x,r))}.
$$
Now, if $\ka>0$ is sufficiently small, then $\eta/\ka>\b$ and, in consequence,
$$
0
\le\lim_{r\to 0}\frac{\mu\(A_{N_\ka(x,r)}^*(x,r)\)}{\mu(B(x,r))}
\le\lim_{r\to 0}\frac{\mu\(A_\mu^{\eta/\ka,De}(x,r)\)}{\mu(B(x,r))}=0.
$$
This means that $\mu$ is (DBT) at $x$ and the proof is complete.
\end{proof}

\sp Now we shall provide some sufficient conditions, different than (WBT), for (DBT) to hold at every point of ${ J}_{\mathcal S}$. We will do it by developing the reasoning of Lemma~5.2 in \cite{BSTV}. We will not really make use of these conditions in the current manuscript but these are very close to the subject matter of the current section and will not occupy too much space. These may be needed in some future. We shall first prove the following.

\begin{lem}\label{l1wbt6}
Let $\mathcal S = \{ \phi_e\}_{e \in E}$ be a finitely primitive 
\hbox{\rm CGDMS} satisfying \hbox{\rm SOSC}. 
  Assume that a number $t > \max\{ \theta_{\mathcal S}, d-1\}$ satisfies 
\beq\label{lwbt9}
t > d-1 + \frac{\P(t)}{\log s}
\eeq
Denote by $\mu_t$ the unique equilibrium state of the potential $E_A^{\infty } \ni \omega \mapsto t \log |\phi_{\omega_1}(\pi(\sigma \omega))|$.  Then there exists constants $\alpha > 0$ and $C \in [0, +\infty]$ such that 
$$
\mu_t \circ \pi_{\mathcal S}^{-1}(A_k^*(z,r)) \leq C e^{-\alpha k}
$$
for all $z \in \overline{ J_{\mathcal S}}$, all radii $r>0$ and all integers $n \geq 0$.
\end{lem}

\begin{proof}
For all $a \in E$, let $r>0$.  Set 
$$
J_a := \bu_{b \in E \hbox{:} A_{ab}=1}  \pi_\mathcal S([b]).
$$
%Fix $z \in \overline {\mathcal J}_{\mathcal }$ and
 $r \in (0,1]$.  For  $k \geq 0$ consider the set 
$$
E^k_A(z,r) := \Big\{ 
\omega \in E_A^k \hbox{ : } \phi_{\omega}({J}_{w_{{|n|-1}}}) \cap B(z,r) \neq \emptyset
\  \hbox{ and }  \  \phi_{\omega}({J}_{w_{|w|-1}}) \cap B(z,r)^c \neq \emptyset
\Big\}.
$$
Furthermore, for every $k \geq 0$ let 
$$
E^k_A(z,r; n) :=\{ \omega \in E_A^k(z,r) \hbox{ : } s^{n+1} < \|\phi_w'\| \leq s^n  \}
$$
Then the family 
$$
\F_k(z,r; n) :=\big\{ \phi_{\omega}(\Int(X)) \hbox{ : } \omega \in E_A^k(z,r;n)\big\}
$$
consists of mutually disjoint open sets contained in 
$$
A(z; r-Ds^n, r + Ds^n)
$$
each of which contains a ball of radius $K^{-1} R s^{n+1}$ where $R > 0$ is the radius of an open ball entirely contained in $\Int X_v$ for all $v \in V$.   So, then 
$$
\# \F_k(z,r;n) 
\leq \frac{\Leb_d(A(z; r-Ds^n, r + Ds^n))}{\Leb_d(0, k^{-1} R s^{n+1})}
\leq C_1 \frac{r^{d-1} s^n}{s^{nd}} 
= C_1 r^{d-1} s^{(1-d)n} \leq C_1 s^{(1-d)n}
$$
with the same universal constant $C_1 \in (0, +\infty)$.  Since $E_A^k(z,r,n) =  \emptyset$ for every $n <  k$, then knowing that $t > \max\{\theta_{\mathcal S}, d-1\}$ gives that 
$$
\begin{aligned}
\mu_t\circ \pi_{\mathcal S}^{-1}(A_k^*(z,v))
 &= \sum_{n=k}^\infty \mu_t\left(\bu_{\omega\in E_A^k(z,r;k)} \phi_{\omega}(\J_{\omega_{|\omega|-1}})\right)\cr
&\leq 
\sum_{n=k}^\infty \# \F_k(z,r;) \sup \{ \mu_t(\phi_{\omega}(X)) \hbox{ : } \omega \in E_A(z,r; n)\}\cr
&\leq C_1 \sum_{n=k}^\infty s^{(1-d)n}  e^{-P(t) k}  s^{tn}\cr
&= C_1 e^{-\P(t)k} \sum_{n=k}^\infty s^{(t+1-d)n}\cr
&= C_1 (1- s^{t+1-d})^{-1} e^{-P(t) k} s^{(t+ d-1)k}\cr
&= C_1 (1-s^{t+d-1})^{-1} \exp(((t+1 -d) \log s - \P(t))k)\cr
\end{aligned}
$$

But $(t+1-d) \log s - \P(t) < 0$ by virtue of (\ref{lwbt9}) and the proof is complete.
\end{proof}

As an immediate consequence of this lemma we get the following.

\begin{thm}\label{t1wbt10}
Let $\mathcal S = \{\phi_e \}_{e \in E}$ be a finitely primitive {\rm CGDMS} satisfying {\rm SOSC}.  If a number $t > \max \{\theta_{\mathcal S}, d-1 \}$ satisfies 
\beq\label{1wbt10}
t > d -1 + \frac{\P(t)}{\log s}  
\eeq
then $\mu_t \circ \pi_{\mathcal S}^{-1}$,  the projection of the corresponding equilibrium state $\mu_t$ on $E_A^\infty$,  is {\rm DBT} at every point of $\overline{J_{\mathcal S}}$
\end{thm}

\begin{proof}
Because of the Lemma \ref{l1wbt6} for all $z \in \overline{J_{\mathcal S}}$ and all radii $r > 0$, we have that 
$$
\begin{aligned}
\mu_t \circ\pi _{\mathcal S}^{-1} \left( A^*_{N_\kappa(z,r)}(z, r)\right)
&\leq C\exp(- \a N_{\kappa}(z,k))\cr
&\leq C \exp\left(-\a\left( \frac{1}{\kappa}
\log \mu_t \circ \pi_{\mathcal S}^{-1}(B(z,r))-1\right)\right)\cr
&=Ce^\a (\mu_t \circ \pi_{\mathcal S}(B(z,r)))^{\a/\kappa}\cr
&= Ce^{\a} \left(\mu_t \circ \pi_{\mathcal S}^{-1}(B(z,r))\right)^{\frac{\a}{\kappa} - 1}
\mu_t \circ \pi_{\mathcal S}^{-1}(B(z,r))
\end{aligned}
$$
Equivalently, 
$$
\frac{\mu_t \circ \pi_{\mathcal S}^{-1}(A_{N_k(z,r)}^*(z,r))}{\mu_t \circ \pi_{\mathcal S}^{-1}(B(z,r))} 
\leq  Ce^{\a}  \left(\mu_t \circ \pi_{\mathcal S}^{-1}(B(z,r))\right)^{\frac{\a}{\kappa} - 1}
$$
and the proof is complete since the right hand-side of this inequality
converges to $0$ as $r \to 0$ for every $\kappa \in (0,\a)$.

\end{proof}

\fr As an immediate consequence of this theorem we get the following.

\begin{cor}\label{c1wbt11}
Let $\mathcal S$   be a finitely primitive strongly regular {\rm CGDMS} satisfying {\rm SOSC}.
Then there exists $\eta > 0$ such that if $t \in (\max\{\theta_{\mathcal S}, d-1\}, \HD(J_{\mathcal S})+\eta)$, then $\mu_t\circ \pi_{\mathcal S}^{-1}$, the projection of the corresponding equilibrium state $\mu_t$ on $E_A^{\infty}$,  is {\rm DBT} at every point of $\overline { J_{\mathcal S}}$.
\end{cor}

\begin{proof}
We only need to check that if $t \in (\max\{\theta_{\mathcal S}, d-1\}, \HD( J_{\mathcal S})+\eta)$ for some $\eta > 0$  sufficiently small then (\ref{1wbt10}) holds. Indeed, since $\P(b_\cS)=0$ (by strong regularity of $\mathcal S$) this is an immediate consequence of continuity of the function $(\theta_{\mathcal S}, +\infty) \ni t \mapsto \P(t) \in \mathbb R$.
\end{proof}

\section{Escape Rates for Conformal GDMSs; Measures}\label{Escape Rates for Conformal GDMSs; Measures}

In this section we continue the analysis from the previous section and we prove our first main results concerning escape rates; the one for conformal GDMSs and equilibrium/Gibbs measures. We first work for a while in  full generality. Indeed,
let $\mu$ be an arbitrary Borel probability measure on a metric space $(X,d)$. Fix $\ka>0$. Fix $z\in X$. Let
$$
\Ga:=\Ga_\ka(z):=\{N_\ka(z,r):0<r\le 2\diam(X)\}.
$$
Represent $\Ga$ as a strictly increasing sequence $(l_n)_{n=0}^\infty$ of non-negative integers. Let us record the following.

\bobs\label{o1fp156}
If $z\in\supp(\mu)$, then $\Ga_\ka(z)\sbt\N$. Moreover, the set $\Ga$ is infinite if and only $z$ is not an atom of $\mu$.
\eobs

\fr We shall prove the following.

\blem\label{l1fp157}
If $\mu$ is a Borel probability measure on $X$ which is {\rm(WBT)} at some point $z\in X$, then the set $\Ga_\ka(z)$ has bounded gaps, precisely meaning that
$$
\De l(z):=\sup_{n\ge 0}\{l_{n+1}-l_n\}<+\infty
$$
\elem

\begin{proof}
Fix $n\ge 0$. There then exists $r_n>0$ such that
$$
N_\ka(z,r)\le l_n+1
$$
for all $r>r_n$, and 
$$
N_\ka(z,r)\ge l_{n+1}
$$
for all $r<r_n$. Therefore, by Observation~\ref{o1fp132eba}
$$
\mu(B(z,r_n))\le e^\ka\exp\(-\ka l_{n+1}\)
$$
and 
$$
\mu(B(z,r_n))\ge \exp\(-\ka(l_n+1)\).
$$
Hence, 
$$
\frac{\mu\(B\(z,r_n+\mu^\b\(B(z,r_n)\)\)\)}{\mu(B(z,r_n))}
\ge e^{-\ka}\exp\(\ka(l_{n+1}-(l_n+1))\).
$$
for all $\b>0$, in particular for $\b>\b_\mu(z)$. But since the measure $\mu$ is (WBT) at $z$, we therefore have that
$$
\varlimsup_{n\to\infty}\exp\(\ka(l_{n+1}-(l_n+1))\)
\le e^\ka\lim_{n\to\infty}
\frac{\mu\(B\(z,r_n+\mu^\b\(B(z,r_n)\)\)\)}{\mu(B(z,r_n))}
\le e^\ka.
$$
Thus
$$
\varlimsup_{n\to\infty}(l_{n+1}-(l_n+1))<+\infty
$$ 
and the proof is complete.
\end{proof}

\fr For every $n\ge 0$ let
$$
\cR_n:=\{r\in (0,2\diam(X)):N_\ka(z,r)=l_n\},
$$
and, given in addition $0\le m\le n$, let 
$$
\cR(m,n):=\bu_{k=m}^n\cR_k.
$$ 

Now we make an additional substantial assumption that 
$$
\cS=\{\phi_e\}_{e\in E},
$$
a conformal GDMS is given, and $\supp(\mu)=\ov J_\cS$. For any $z\in J_\cS$ and $r\in (0,2\diam(X))$, define
\beq\label{1fp159}
W^-(z,r):=B_{J_\cS}(z,r)\sms A_{N_\ka(z,r)}^*(z,r) \
\text{ and }  \
W^+(z,r):=B_{J_\cS}(z,r)\cup A_{N_\ka(z,r)}^*(z,r).
\eeq
Let us record the following two immediate consequences of this definition.

\bobs\label{o1fp159}
For every $z\in J_\cS$ and $r\in (0,2\diam(X))$, we have
$$
W^-(z,r)\sbt B(z,r)\sbt W^+(z,r).
$$
\eobs

\bobs\label{o2fp159}
For every $z\in J_\cS$ and $r\in (0,2\diam(X))$ both sets $W^-(z,r)$ and $W^+(z,r)$ can be represented as unions of cylinders of length $N_\ka(z,r)$. 
\eobs

\fr Fix $\ka>0$ so small that \eqref{1fp132ec} holds and so that $\eta/\ka>\b_\mu(z)$. We shall prove the following.

\blem\label{l3fp159}
For all $k\ge 0$ large enough, if $n-k\ge 2$, then
$$
W^+(z,s)\sbt W^-(z,r)
$$
for all $r\in\cR_k$ and all $s\in\cR_n$.
\elem

\begin{proof}
The assertion of our lemma is equivalent to the statement that
$$
W^+(z,s)\cap A_{l_k}^*(z,r)=\es.
$$
Assume for a contradiction  that there are sequences $\(n_j\)_{j=0}^\infty$ and $\(k_j\)_{j=0}^\infty$ of positive integers such that $\lim_{j\to\infty}k_j=+\infty$ and $n_j-k_j\ge 2$ for all $j\ge 0$, and also there are radii $r_j\in \cR_{k_j}$ and $s_j\in\cR_{n_j}$ such that
$$
W^+(z,s_j)\cap A_{l_{k_j}}^*(z,r_j)=\es,
$$
for all $j\ge 0$. Since we know that for each $\om\in E_A^*$,
$$
\diam\(\phi_\om(J)\)\le De^{-\eta|\om|},
$$
using Observation~\ref{o1fp132eba}, we therefore conclude that
$$
\begin{aligned}
s_j+D\mu^{\eta/\ka}(B(z,r_j))
&\ge s_j+D\exp\(-\eta N_\ka(z,r_j)\)
 \ge s_j+D\e^{-\eta l_{k_j}} \\
&\ge s_j+D\e^{-\eta l_{n_j}}
 \ge r_j-D\e^{-\eta l_{k_j}} \\
&=r_j-D\exp\(-\eta N_\ka(z,r_j)\) \\
&\ge r_j-D\mu^{\eta/\ka}(B(z,r_j)).
\end{aligned}
$$
Hence, $s_j\ge r_j-2D\mu^{\eta/\ka}(B(z,r_j))$, and therefore,
\beq\label{1fp153}
\frac{\mu(B(z,s_j))}{\mu(B(z,r_j))}
\ge \frac{\mu(B(z,r_j))-\mu\(A_\mu^{\eta/\ka,2D}(z,r_j)\)}{\mu(B(z,r_j))}
=1-\frac{\mu\(A_\mu^{\eta/\ka,2D}(z,r_j)\)}{\mu(B(z,r_j))}.
\eeq
On the other hand,it follows from Observation~\ref{o1fp132eba} that 
$$
\mu(B(z,s_j))\le e^\ka e^{-\ka l_{n_j}} \  \text{ and }  \
\mu(B(z,r_j))\ge e^{-\ka l_{k_j}}.
$$
This yields
$$
\frac{\mu(B(z,s_j))}{\mu(B(z,r_j))}
\le e^{\ka}\exp\(-\ka(l_{n_j}-l_{k_j})\)
\le e^{\ka}\exp\(-\ka(n_j-k_j)\)
\le e^\ka e^{-2\ka}
=e^{-\ka}.
$$
Along with \eqref{1fp153} this implies that
\beq\label{1fp155}
\frac{\mu\(A_\mu^{\eta/\ka,2D}(z,r_j)\)}{\mu(B(z,r_j))}
\ge 1-e^{-\ka}.
\eeq
However, since $\lim_{j\to\infty}r_j=0$, since the measure $\mu$ is (WBT), and since $\ka>0$ was taken so small that $\eta/\ka>\b_\mu(z)$, we conclude that \eqref{1fp155} may hold for finitely many integers $j\ge 0$ only, and the proof of Lemma~\ref{l3fp159} is complete.
\end{proof}

\fr As an immediate consequence of this lemma and Observation~\ref{o1fp159}, we get the following.

\blem\label{l4fp159}
For all integers $k\ge 0$ large enough, if $n-k\ge 2$, then
$$
W^{-}(z,s)\sbt W^{-}(z,r) \  \text{ and } \
W^{+}(z,s)\sbt W^{+}(z,r)
$$
for all $r\in\cR_k$ and all $s\in\cR_n$.
\elem

\fr Now we shall prove the following.

\bprop\label{p2fp161}
Let $\cS$ be a conformal {\rm GDMS}. Let $\mu$ be a Borel probability 
measure supported on $\ov J_\cS$. Suppose that $\mu$ is {\rm(WBT)} at some point $z\in J_\cS$ which is not an atom of $\mu$. Let $\cR$ be an arbitrary countable set of positive reals containing $0$ in its closure. Then there exists $(n_j)_{j=0}^\infty$, a strictly increasing sequence of 
non-negative integers with the following properties.

\sp\begin{itemize}
\item[(a)] $n_{j+1}-n_j\le 4$,

\sp\item[(b)] $n_{j+1}-n_j\ge 2$,

\sp\item[(c)] The set $\cR\cap\cR_{n_j}\ne\es$ for infinitely many $j$s.
\end{itemize}
\eprop

\begin{proof}
We construct the sequence $(n_j)_{j=0}^\infty$ inductively. Assume without loss of generality that $r_0=2\diam\(\ov J_\cS\)$ and set $n_0:=0$. For the inductive step suppose that $n_j\ge 0$ with some $j\ge 0$ has been constructed. Look at the set $\cR(n_j+2,n_j+4)$. If 
$$
\{l_k:n_j+2\le k\le n_j+4\}\cap\{N_\ka(z,r):r\in\cR\}\ne\es,
$$
take $n_{j+1}$ to be an arbitrary number from $\{n_j+2,n_j+3,n_j+4\}$ such that
$$
l_{n_{j+1}}\in\{N_\ka(z,r):r\in\cR\}.
$$
If, on the other hand,
$$
\{l_k:n_j+2\le k\le n_j+4\}\cap\{N_\ka(z,r):r\in\cR\}=\es,
$$
set
$$
n_{j+1}=n_j+2.
$$
Properties (a) and (b) are immediate from our construction. In order to prove (c) suppose on the contrary that 
$$
\cR\cap\bu_{j=p}^\infty\cR_{n_j}=\es
$$
with some $p\ge 0$. This yields $n_{j+1}=n_j+2$ for all $j\ge p$, i.e. $n_j=n_p+2(j-p)$ and
$$
\bu_{j=p}^\infty\{l_k:n_p+2(j+1-p)\le k\le n_p+2(j+2-p)\} \cap 
\{N_\ka(z,r):r\in\cR\}=\es
$$
But 
$$
\bu_{j=p}^\infty\{l_k:n_p+2(j+1-p)\le k\le n_p+2(j+2-p)\}=
\{l_k:k\ge n_p+2\}.
$$
Thus, $N_\ka(z,r)\le n_p+1$ for all $r\in\cR$. By Observation~\ref{o1fp132eba} this gives that $\mu(B(z,r))\ge \exp\(-\ka(n_p+1)\)$ for all $r\in\cR$, contrary to the facts that $0\in\ov\cR$ and that $z$ is not an atom of $\mu$. We are done. 
\end{proof}

Now, for every $j\ge 0$ fix arbitrarily $r_j\in \cR_{n_j}$ requiring in addition that $r_j\in\cR$ if $\cR\cap\cR_{
n_j}\ne\es$. Set 
\beq\label{2fp163}
U_{l_{n_j}}^-(z):=\pi^{-1}\(W^-(z,r_j)\) \  \text{ and } \
U_{l_{n_j}}^+(z):=\pi^{-1}\(W^+(z,r_j)\).
\eeq
These sets are well defined as the function $l:\N\to\N$ is $1$-to-$1$ and, by (b), the function $j\mapsto n_j$ is also 
$1$-to-$1$. Furthermore, for every $j\ge 0$ and every $l_{n_j}\le k< l_{n_{j+1}}$, define
\beq\label{1fp163}
U_k^{\pm}(z):=U_{l_{n_j}}^{\pm}(z).
\eeq
In this way we have well-defined two sequences of open neighborhoods of $\pi^{-1}(z)$. We shall prove the following.

\bprop\label{p1fp163}
With hypotheses exactly as in Proposition~\ref{p2fp161}, both 
$\(U_k^{\pm}(z)\)_{k=0}^\infty$ are descending sequences of open subsets of $E_A^\infty$ satisfying conditions {\rm (U0)--(U2)}.
\eprop

\begin{proof}
(U0) is immediate from the very definition. If $k\ge 0$ and then $j=j_k\ge 0$ is uniquely chosen so that $l_{n_j}\le k< l_{n_{j+1}}$, then $U_k^{\pm}(z):=U_{l_{n_j}}^{\pm}(z)$, and both sets are disjoint unions of cylinders of length $n_j$ by Observation~\ref{o2fp159} and since $r_j\in\cR_{n_j}$, so also of length $k$ as $k\ge l_{n_j}$. Thus (U1) holds. That both sequences $\(U_k^{\pm}(z)\)_{k=0}^\infty$ are descending follows immediately from Lemmas~\ref{l4fp159}, property (b) of Proposition~\ref{p2fp161}, and formulas \eqref{1fp163} and 
\eqref{2fp163}. Applying formulas \eqref{1fp163} and 
\eqref{2fp163} along with Proposition~\ref{p2fp161} (b), Lemma~\ref{l3fp159}, Observation~\ref{o1fp159}, Observation~\ref{o1fp132eba}, Lemma~\ref{l1fp157}, and Proposition~\ref{p2fp161} (a), we get
\beq\label{1fp165}
\begin{aligned}
\mu\(U_k^{\pm}(z)\)
&\le \mu\(U_k^+(z)\)
 =   \mu\(U_{l_{n_j}}(z)\)
 =  \mu\(\pi^{-1}\(W^+(z,r_j)\)\)
 \le\mu\(\pi^{-1}\(W^+(z,r_{j-1})\)\) \\
&=\mu\(\pi^{-1}\(W^-(z,r_{j-1})\)\)
 \le\mu\(\pi^{-1}\(B(z,r_{j-1})\)\) 
 \le e^\ka\exp\(-\ka N_\ka(z,r_{j-1})\) \\
&=   e^\ka e^{-l_{n_{j-1}}}
 =   e^\ka e^{-l_{n_{j+1}}}
     \exp\(\ka(l_{n_{j+1}}-l_{n_{j-1}})\) \\
&\le e^\ka e^{-l_{n_{j+1}}}
     \exp\(\ka\De l(z)\(n_{j+1}-n_{j-1}\)\)
 \le e^\ka e^{8\ka\De l(z)}\exp\(-\ka l_{n_{j+1}}\) \\
&\le \exp\(\ka((1+8\De l(z))\)e^{-\ka k},
\end{aligned}
\eeq
and thus condition (U2) is satisfied with any $\rho\in (e^{-\ka},1)$ sufficiently close to $1$. The proof is complete.
\end{proof}

\bprop\label{p1fp165}
With hypotheses exactly as in Proposition~\ref{p2fp161}, both 
$\(U_k^{\pm}(z)\)_{k=0}^\infty$ satisfy condition {\rm (U3)}. If in addition either $z$ is not pseudo-periodic for $\cS$ or it is uniquely periodic and $z\in\Int X$, then {\rm (U5)} holds. In the former case also {\rm (U4)} holds while in the latter case it holds  if in addition $\mu$ is an equilibrium state of the amalgamated function of a summable H\"older continuous system of functions. 
\eprop

\begin{proof}

With the same arguments as in \eqref{1fp165} we get that
\beq\label{2fp165}
\pi^{-1}(z)
\sbt\bi_{k=0}^\infty\ov{U_n^-(z)}
\sbt\bi_{k=0}^\infty\ov{U_n^+(z)}
\sbt\bi_{j=1}^\infty\pi^{-1}\(\ov{B\(z,r_{j-1}\)}\)
\sbt \pi^{-1}(z).
\eeq
So (U3) holds as $\pi^{-1}(z)$ is a finite set. Assume now that $z$ is not 
pseudo-periodic. Then condition (U4A) holds because of Lemma~\ref{l1fp82} and \eqref{2fp165}, while (U5) directly follows from Lemma~\ref{l2fp82} and the inclusion $U_{l_{n_j}}^{\pm}(z)\sbt \pi^{-1}\(B(z,r_{j-1})\)$. 

\sp Assume in turn that $z\in \Int X$ is uniquely periodic point of $\cS$ with prime period $p$. Then $U_\infty$ consists of a periodic point, call it $\xi$, of period $p$ because of \eqref{2fp165}. So, $\xi=\tau^\infty$ for a unique point $\tau\in E_A^\infty$. Condition (U5) directly follows from Lemma~\ref{l2fp82a}.
Now we shall show that the sequence $\(U_i^+(z)\)_{i=0}^\infty$ satisfies the property (U4B). Indeed, without loss of generality we may assume that $i=l_k$, where $k=n_j$, $j\ge 0$. Take an arbitrary $\om\in U_{l_k}^+(z)$. This means that $\om|_{l_k}\in E_A^{l_k}$ and $\phi_{\om|_{l_k}}(J)\cap B(z,r_j)\ne\es$. Then
$$
\phi_{\tau\om|_{l_k}}(J)\cap B(z,r_j)
\spt \phi_{\tau\om|_{l_k}}(J)\cap \phi_\tau(B(z,r_j))
=\phi_{\tau}\(\phi_{\om|_{l_k}}(J)\cap B(z,r_j)\)
\ne\es.
$$
Hence, $\phi_{\om|_{l_k}}(J)\cap B(z,r_j)\ne\es$, meaning that $\tau\om\in  U_{l_k}^+(z)$. So, the inclusion $\tau U_{l_k}^+(z)\sbt U_{l_k}^+(z)$ has been proved and \eqref{1_2014_12_19} of (U4B) holds for the sequence $\(U_i^+(z)\)_{i=0}^\infty$. 

\sp In order to establish \eqref{1_2014_12_19} of (U4B) for the sequence $\(U_i^-(z)\)_{i=0}^\infty$, recall that $\eta>0$ is so small that $||\phi_\om'||\le e^{-\eta|\om|}$ for all $\om\in E_A^*$. Take now $\ka>0$ as small as previously  required and furthermore so small that $\b\eta\ka^{-1}>2$. On the other hand, for every $k:=n_j$, $j\ge 1$, we have
\beq\label{1fp167}
\phi_\tau(W^-(z,r_j))
\sbt\phi_\tau(B(z,r_j))
\sbt B\(z,|\phi_\tau'(z)|r_j)\)
\sbt B\(z,e^{-\eta|\tau|}r_j\).
\eeq
On the other hand, by \eqref{1fp159} and the definition of $l_{n_j}$, we have that
$$
\begin{aligned}
W^-(z,r_j)
&\spt B\(z,r_j-De^{-\eta l_j}\)
\spt B\(z,r_j-D\mu^{\eta/\ka}(B(z,r_j))\) \\
&\spt B\(z,r_j-DC^{\eta/\ka}r_j^{\b\eta/\ka}\) \\
&\spt B\(z,e^{-\eta|\tau|}r_j\)
\end{aligned}
$$
provided that $\ka>0$ is taken sufficiently small (independently of $j$). Along with \eqref{1fp167} this gives,
$$
\phi_\tau\(W^-(z,r_j)\)\sbt W^-(z,r_j)).
$$
Hence, 
$$
\begin{aligned}
\pi\(\tau U_{l_k}^-(z)\)
&=\pi\(\tau\pi^{-1}(W^-(z,r_j))\)
 =\phi_\tau\(\pi\(\pi^{-1}(W^-(z,r_j))\)\) \\
&=\phi_\tau\(W^-(z,r_j)\).
\end{aligned}
$$
Thus
$$
\tau U_{l_k}^-(z)
\sbt \pi^{-1}(W^-(z,r_j))
= U_{l_k}^-(z).
$$
Thus, the part \eqref{1_2014_12_19} of (U4B) is established. In order to prove \eqref{2_2014_12_19} of (U4B), let $k\ge 0$ and $j_k\ge 0$ be as in the proof of Proposition~\ref{p1fp163}. The proof of this proposition gives that
$$
U_k^{\pm}(z)\sbt \pi^{-1}\(W^-(z,r_{j_k-1})\).
$$
Since we now assume that $\phi(\om)=f^{\om_0}(\pi(\sg(\om)))$, $\om\in E_A^\infty$, where $\(f^e\)_{e\in E}$ is a H\"older continuous summable system of functions, condition \eqref{2_2014_12_19} of (U4B) follows from continuity of the function $f^{\tau_0}:X_{t(\tau_0)}\to \R$ and the fact that $\lim_{k\to\infty}j_k=+\infty$.
The proof of our proposition is complete.
\end{proof}

\sp\fr Now, we are in position to prove the following main result of this section, which is also one of the main results of the entire paper. Recall that the lower and upper escape rates $\un R_{\mu}$ and $\ov R_{\mu}$ have been defined by formulas \eqref{11_2016_07_14} and \eqref{12_2016_07_14}.

\bthm\label{t1fp83}
Let $\cS=\{\phi_e\}_{e\in E}$ be a finitely primitive Conformal Graph Directed Markov System. Let $\phi:E_A^\infty\to\R$ be a H\"older continuous summable potential. As usual, denote its equilibrium/Gibbs state by $\mu_\phi$. Assume that the measure $\mu_\phi\circ\pi^{-1}_\cS$ is {\rm (WBT)} at a point $z\in J_\cS$. If $z$ is either 

\begin{itemize}
\item[(a)] not pseudo-periodic, 

or 

\item[(b)] uniquely periodic, it belongs to $\Int X$ (and $z=\pi(\xi^\infty)$ for a (unique) irreducible word $\xi\in E_A^*$), and $\phi$ is the amalgamated function of a summable H\"older continuous system of functions, 
\end{itemize}
then, with $\un R_{\cS,\phi}(B(z,\ep)):=\un R_{\mu_\phi}\(\pi_{\cS}^{-1}(B(z,\ep))\)$ and $\ov R_{\cS,\phi}(B(z,\ep)):=\ov R_{\mu_\phi}\(\pi_{\cS}^{-1}(B(z,\ep))\)$, we have that 
\beq\label{1fp83}
\begin{aligned}
\lim_{\ep\to 0}\frac{\un R_{\cS,\phi}(B(z,\ep))}{\mu_\phi\circ\pi_{\cS}^{-1}(B(z,\ep))}
&=\lim_{\ep\to 0}\frac{\ov R_{\cS,\phi}(B(z,\ep))}{\mu_\phi\circ\pi_{\cS}^{-1}(B(z,\ep))} =\\
&=d_\phi(z)
:=\begin{cases}
1 \  &\text{{\rm if (a) holds}}   \\
1-\exp\({S_p\phi(\xi)}-p\P(\phi)\) &\text{{\rm  if (b) holds}},
\end{cases}
\end{aligned}
\eeq
where in {\rm (b)}, $\{\xi\}=\pi_\cS^{-1}(z)$ and $p\ge 1$ is the prime period of $\xi$ under the shift map. 
\ethm

\begin{proof}
Assume for a contradiction  that \eqref{1fp83} does not hold. This means that there exists a strictly decreasing sequence $s_n(z)\to 0$ of positive reals such that at least one of the sequences
$$
\lt(\frac{\un R_{\cS,\phi}(B(z,s_n(z)))}{\mu_\phi\circ\pi_{\cS}^{-1}(B(z,s_n(z)))}\rt)_{n=0}^\infty 
\  \  \  {\rm or} \  \  \
\lt(\frac{\ov R_{\cS,\phi}(B(z,s_n(z)))}{\mu_\phi\circ\pi_{\cS}^{-1}(B(z,s_n(z)))}\rt)_{n=0}^\infty 
$$
does not have $d_\phi(z)$ as its accumulation point. Let
$$
\cR:=\{s_n(z):n\ge 0\}.
$$
Let $\(U_n^{\pm}(z)\)_{n=0}^\infty$ be the corresponding sequence of open subsets of $E_A^\infty$ produced in formula \eqref{1fp163}. Then, because of both Proposition~\ref{p1fp163} and Proposition~\ref{p1fp165}, Proposition~\ref{p2fp133} applies to give
\beq\label{1fp169}
\lim_{n\to\infty}\frac{R_{\mu_\phi}(U_n^\pm(z))}{\mu_\phi(U_n^\pm(z))}=d_\phi(z).
\eeq
Let $(n_j)_{j=0}^\infty$ be the sequence produced in Proposition~\ref{p2fp161} with the help of $\cR$ defined above. By this proposition there exists an increasing sequence $(j_k)_{k=0}^\infty$ such that $\cR\cap \cR_{n_{j_k}}\ne\es$ for all $k\ge 1$. For every $k\ge 1$ pick one element $r_k\in \cR\cap \cR_{n_{j_k}}$. Set $q_k:=l_{n_{j_k}}$. By Observation~\ref{o1fp159} and formula \eqref{2fp163}, we then have
\beq\label{2fp169}
\begin{aligned}
\frac{R_{\mu_\phi}(U_{q_k}^-(z))}{\mu_\phi(U_{q_k}^-(z))}\cdot
\frac{\mu_\phi(U_{q_k}^-(z))}{\mu_\phi(B(z,r_k))} \
&\le \frac{\un R_{\cS,\phi}(B(z,r_k))}{\mu_\phi\circ\pi_{\cS}^{-1}(B(z,r_k))}
 \le \frac{\ov R_{\cS,\phi}(B(z,r_k))}{\mu_\phi\circ\pi_{\cS}^{-1}(B(z,r_k))}\le \\
&\le \frac{R_{\mu_\phi}(U_{q_k}^+(z))}{\mu_\phi(U_{q_k}^+(z))}\cdot
\frac{\mu_\phi(U_{q_k}^+(z))}{\mu_\phi(B(z,r_k))}.
\end{aligned}
\eeq
But, since $\mu_\phi\circ\pi_{\cS}^{-1}$ is (WBT) at $z$, it is (DBT) at $z$ by Proposition~\ref{p1fp132ec}, and it therefore follows from \eqref{1fp132ec} along with formulas \eqref{1fp159} and \eqref{2fp163} that
$$
\lim_{k\to\infty} \frac{\mu_\phi(U_{q_k}^-(z))}{\mu_\phi(B(z,r_k))}
=1
=\lim_{k\to\infty} \frac{\mu_\phi(U_{q_k}^+(z))}{\mu_\phi(B(z,r_k))}.
$$
Inserting this to \eqref{1fp169} and \eqref{2fp169}, yields:
$$
\lim_{k\to\infty} \frac{\un R_{\cS,\phi}(B(z,r_k))}{\mu_\phi\circ\pi_{\cS}^{-1}(B(z,r_k))}
=\lim_{k\to\infty} \frac{\ov R_{\cS,\phi}(B(z,r_k))}{\mu_\phi\circ\pi_{\cS}^{-1}(B(z,r_k))}
=d_\phi(z).
$$
Since $r_k\in \cR$ for all $k\ge 1$, this implies that $d_\phi(z)$ is an accumulation point of both sequences $\(\un R_{\cS,\phi}(B(z,r_k))\big/\mu_\phi\circ\pi_{\cS}^{-1}(B(z,r_k))\)_{n=1}^\infty$, $\(\ov R_{\cS,\phi}(B(z,r_k))\big/\mu_\phi\circ\pi_{\cS}^{-1}(B(z,r_k))\)_{n=1}^\infty$, and this contradiction finishes the proof of Theorem~\ref{t1fp83}.
\end{proof}

\sp\fr Now, as an immediate consequence of Theorem~\ref{t1fp83} and Theorem~\ref{t1wbt6}, we get the following.

\bthm\label{t3_2016_05_27} 
Assume that $\cS$ is a finitely primitive conformal GDMS whose limit set $ J_{\mathcal S}$ is geometrically irreducible.
Let $\phi:E_A^\infty\to\R$ be a H\"older continuous strongly summable potential. As usual, denote its equilibrium/Gibbs state by $\mu_\phi$. Then
\beq\label{1fp83+1}
\lim_{\ep\to 0}\frac{\un R_{\cS,\phi}(B(z,\ep))}{\mu_\phi\circ\pi_{\cS}^{-1}(B(z,\ep))}
=\lim_{\ep\to 0}\frac{\ov R_{\cS,\phi}(B(z,\ep))}{\mu_\phi\circ\pi_{\cS}^{-1}(B(z,\ep))} 
=1
\eeq
for $\mu_\phi\circ\pi_{\cS}^{-1}$--a.e. point $z$ of $\cJ_\cS$.
\ethm

\sp In the realm of finite alphabets $E$, by virtue of  Theorem~\ref{t1fp83} and both Theorem~\ref{t2wbt11} and Theorem~\ref{t1wbt15}, we get the following stronger result.

\bthm\label{t1fp83_Finite}
Let $\cS=\{\phi_e\}_{e\in E}$ be a primitive Conformal Graph Directed Markov System with a finite alphabet $E$ acting in the space $\R^d$, $d\ge 1$. Assume that either $d=1$ or that the system $\cS$ is geometrically irreducible. Let $\phi:E_A^\infty\to\R$ be a H\"older continuous potential. As usual, denote its equilibrium/Gibbs state by $\mu_\phi$. Let $z\in J_\cS$ be arbitrary. If either $z$ is 

\begin{itemize}
\item[(a)] not pseudo-periodic, 

or 

\item[(b)] uniquely periodic, it belongs to $\Int X$ (and $z=\pi(\xi^\infty)$ for a (unique) irreducible word $\xi\in E_A^*$), and $\phi$ is the amalgamated function of a summable H\"older continuous system of functions, 
\end{itemize}
then
\beq\label{1fp83B}
\begin{aligned}
\lim_{\ep\to 0}\frac{\un R_{\cS,\phi}(B(z,\ep))}{\mu_\phi\circ\pi_{\cS}^{-1}(B(z,\ep))}
&=\lim_{\ep\to 0}\frac{\ov R_{\cS,\phi}(B(z,\ep))}{\mu_\phi\circ\pi_{\cS}^{-1}(B(z,\ep))} =\\
&=d_\phi(z)
:=\begin{cases}
1 \  &\text{{\rm if (a) holds}}   \\
1-\exp\({S_p\phi(\xi)}-p\P(\phi)\) &\text{{\rm  if (b) holds}},
\end{cases}
\end{aligned}
\eeq
where in {\rm (b)}, $\{\xi\}=\pi_\cS^{-1}(z)$ and $p\ge 1$ is the prime period of $\xi$ under the shift map. 
\ethm

\section{The derivatives $\lam_n'(t)$ and $\lam_n''(t)$ of Leading Eigenvalues}\label{derivatives_of_eigenvalues}
In this section we have  $\cS=\{\phi_e\}_{e\in E}$, a finitely primitive strongly regular conformal graph directed Markov system. We keep a parameter $t>\th_\cS$ and consider the H\"older continuous summable potential $\phi_t:E_A^\infty\to\R$ given by the formula
$$
\phi_t(\om):=t\log|\phi_{\om_0}'(\pi(\sg(\om)))|.
$$
We further assume that a sequence $(U_n)_{n=0}^\infty$ of open subsets of $E_A^\infty$ is given satisfying the conditions (U0)-(U5). The eigenvalues $\lam$ and $\lam_n$ along with other objects associated to the potential $\phi_t$ are now indicated with the letter/number $t$. 

\sp Our goal in this section is to calculate the asymptotic behavior of derivatives $\lam_n'(t)$ and $\lam_n''(t)$ of leading eigenvalues of perturbed operators $\pf_{t,n}$ when the integer $n\ge 0$ diverges to infinity and the parameter $t$ approaches $b_\cS$.
This is a particularly tedious and technically involved task, partially due to unboundedness of the function $\phi_t$ in the supremum norm and partially due to lack of uniform topological mixing on the sets $K_z(\e)$ introduced below. 

The main theorems of this section form the crucial ingredients in the escape rates considerations of the next section, i.e. Section~\ref{ERCGDMSHD}.

\sp We start with the following.

\bthm\label{t1had2}
For every $0\le n \le +\infty$, the function $(\th_\cS,+\infty)\ni t\mapsto \lam_n(t)\in(0,+\infty)$ is real analytic and 
\beq\label{1had2}
\lam'(t)=\lim_{n\to\infty}\lam_n'(t).
\eeq
\ethm

\begin{proof}

By extending the transfer operators $\pf_{t,n}:\Ba_\th\to\Ba_\th$ in the natural way to complex operators for all $t\in\C$ with $\re(t)<\th_\cS$, and applying Kato-Rellich Perturbation Theorem (see \cite{Zinsmeister}), along with Proposition~\ref{p1fp80}, we see that for every  $0\le n \le +\infty$ there exists $V_n$, an open neighborhood of  $(\th_\cS,+\infty)$, such that each function $(\th_\cS,+\infty)\ni t\mapsto \lam_n(t)\in(0,+\infty)$ extends (and we keep the same symbol $\lam_n$ for this extension) to a holomorphic function from $V_n$ to $\C$, and also each function $(\th_\cS,+\infty)\ni t\mapsto Q_n^{(t)}\1\in \Ba_\th$ extends to a holomorphic function from $V_n$ to $\C$ belonging to $\Ba_\th$. Denote these latter extensions by
$$
g_n:V_n\to\C, \ n \ge 0.
$$
It is also a part of Kato-Rellich Theorem that 
\beq\label{2had2}
\pf_{t,n}g_n(t)=\lam_n(t)g_n(t)
\eeq
for all $0\le n \le +\infty$ and all $t\in V_n$. In particular, all the functions 
$(\th_\cS,+\infty)\ni t\mapsto \lam_n(t)\in(0,+\infty)$, $0\le n \le +\infty$, are real analytic. In order to prove \eqref{1had2}, we shall derive first a "thermodynamical" formula for $\lam_n'(t)$. Differentiating both sides of \eqref{2had2}, we obtain
\beq\label{3had2}
\pf_{t,n}'g_n(t)+\pf_{t,n}g_n'(t)
=\lam_n'(t)g_n(t)+\lam_n(t)g_n'(t),
\eeq
where
\beq\label{4had2}
\pf_{t,n}'(u)(\om)
:=\sum_{E:A_{e\om_0}=1}\!\!\!\1_{U_n^c}(e\om)u(e\om)\log|\phi_e'(\pi(u))|\cdot|\phi_e'(\pi(u))|^t,
\eeq
and all four terms involved in \eqref{3had2} belong to $\Ba_\th$. Applying the operator 
$Q_n^{(t)}$ to both sides of this equation, we get
$$
Q_n^{(t)}\(\pf_{t,n}'g_n(t)\)+Q_n^{(t)}\pf_{t,n}\(g_n'(t)\)
=\lam_n'(t)Q_n^{(t)}\(g_n(t)\)+\lam_n(t)Q_n^{(t)}\(g_n'(t)\).
$$
Since
$$
Q_n^{(t)}\(g_n(t)\)=g_n(t)
$$
and
$$
Q_n^{(t)}\pf_{t,n}\(g_n'(t)\)
=\pf_{t,n}Q_n^{(t)}\(g_n'(t)\)
=\lam_n(t)Q_n^{(t)}\(g_n'(t)\),
$$
we thus get
\beq\label{1had3}
\lam_n'(t)g_n(t)=Q_n^{(t)}\(\pf_{t,n}'g_n(t)\).
\eeq
Since in addition $Q_n^{(t)}$ is a projector onto the $1$-dimensional space $\C g_n(t)$, this operator gives rise to a unique bounded linear functional 
$$
\nu_{t,n}:\Ba_\th\to\C
$$
determined by the property that
$$
Q_n^{(t)}(u)=\nu_{t,n}(u)g_n(t)
$$
for every $u\in \Ba_\th$. So, we can write \eqref{1had3} in the form
\beq\label{3had3}
\lam_n'(t)=\nu_{t,n}\(\pf_{t,n}'g_n(t)\).
\eeq
Now, writing
$$
\ell(\om):=\log|\phi_{\om_0}(\pi(\sg(\om))|,
$$
formula \eqref{4had2} readily gives
$$
\pf_{t,n}'(u)=\pf_{t,n}(u\ell),
$$
so that \eqref{1had3} takes on the form
\beq\label{4had3}
\lam_n'(t)=\nu_{t,n}\(\pf_{t,n}(\ell g_n(t))\).
\eeq
Keeping $t\in(\th_\cS,+\infty)$ for the rest of the proof set
$$
\psi_n:=\pf_{t,n}(\ell g_n(t))
$$
but remember that $\psi_n$ depends on $t$ too. Now, we have
$$
\begin{aligned}
Q_n^{(t)}(\psi_n)-Q^{(t)}(\psi)
& =\nu_{t,n}(\psi_n)g_n(t)- \nu_t(\psi)g(t) \\
& =\(\nu_{t,n}(\psi_n)- \nu_t(\psi)\)g(t)+(g_n(t)- g(t))\nu_{t,n}(\psi_n).
\end{aligned}
$$
Hence, recalling that $g(t)\equiv\1$, we get
$$
\(\nu_{t,n}(\psi_n) - \nu_t(\psi)\)\1
=Q_n^{(t)}(\psi_n)-Q^{(t)}(\psi)+(g(t)- g_n(t))\nu_{t,n}(\psi_n).
$$
Therefore,
$$
\begin{aligned}
\Big|\nu_{t,n}(\psi_n) - \nu_t(\psi)\Big|
&=\int\Big|Q_n^{(t)}(\psi_n)-Q^{(t)}(\psi)+\nu_{t,n}(\psi_n)(g(t)- g_n(t))\Big|d\,\nu_t\\ 
&\le \int\Big|Q_n^{(t)}(\psi_n)-Q^{(t)}(\psi)\Big|d\,\nu_t+\nu_{t,n}(-\psi_n)\int|g_n(t)- g(t)|d\,\nu_t.
\end{aligned}
$$
But, because of Proposition~\ref{p1fp80} (h),
\beq\label{3had4}
\begin{aligned}
\int|g_n(t)- g(t)|d\,\nu_t
&\le ||g_n(t)- g(t)||_*
=\Big\|Q_n^{(t)}(\1)-Q^{(t)}(\1)\Big\|_* \\
&\le\Big|\Big|\Big| Q_n^{(t)}(\1)-Q^{(t)}\Big|\Big|\Big|\|\1\|_\th \\
&=\Big|\Big|\Big| Q_n^{(t)}(\1)-Q^{(t)}\Big|\Big|\Big|\longrightarrow 0
\end{aligned}
\eeq
as $n\to 0$. Hence, in view \eqref{4had3}, in order to conclude the theorem, it suffices to show that
\beq\label{1had4}
\lim_{n\to\infty}\int\Big|Q_n^{(t)}(\psi_n)-Q^{(t)}(\psi)\Big|d\,\nu_t=0
\eeq
and 
\beq\label{2had4}
M:=\sup_{n\ge 1}\{\nu_{t,n}(-\psi_n)\}<+\infty.
\eeq
We first deal with the latter. It follows from Proposition~\ref{p1fp80} (f) that $|\lam_n(t)|\ge 1/2$ for all $n\ge 1$ large enough. It therefore follows from Proposition~\ref{p1fp80} (c) and (e) along with Lemma~\ref{l1fp61} that
\beq\label{1had5}
||g_n(t)||_\infty
\le ||g_n(t)||_\th=||Q_n^{(t)}(\1)||_\th
\le 2\(||\pf_{t,n}\1||_\th+||\De_{t,n}\1||_\th\)
\le 2(1+C)<+\infty.
\eeq
Now, since $t>\th_\cS$, it directly follows from the inequality 
\beq\label{3had5}
\big|\log|\phi_e'|\big|\le||\phi_e'||^{-\e}
\eeq
for every $e>0$ and all $e\in\N$ large enough that
\beq\label{2had5}
\begin{aligned}
||\psi_n||_\infty
&=||\pf_{t,n}\(\ell g_n(t)\)||_\infty
\le||g_n(t)||_\infty||\pf_{t,n}\ell||_\infty \\
&\le 2(1+C)||\pf_{t,n}\ell||_\infty \\
&\le 2(1+C)||\pf_t\ell||_\infty<+\infty
\end{aligned}
\eeq
for all $n\ge 1$ (including infinity) large enough. In fact we will need a somewhat more general result, namely that for every $\g\in\Ba_\th$,
\beq\label{1had5.1}
\begin{aligned}
||\pf_{t,n}\(\ell\g\)||_1
&\le||\pf_{t,n}\(\ell\g\)||_\infty
\le||\pf_{t,n}\(\ell\cdot||\g||_\infty\)||_\infty \\
&=||\g||_\infty||\pf_{t,n}\(\ell\)||_\infty \\
&\le||\pf_{t,n}\ell||_\infty||\g||_\th \\
&\le||\pf_t\ell||_\infty||\g||_\th.
\end{aligned}
\eeq
Let us now estimate $|\psi_n|_\th$. Write
\beq\label{4had5}
\g_n:=g_n(t)\1_{U_n^c}.
\eeq
Then
\beq\label{5had5}
\psi_n:=\pf_{t,n}(\ell g_n(t))=\pf_t(\ell\g_n).
\eeq
We now will also proceed more generally than merely estimating $|\psi_n|_\th$. We shall prove the following.

\blem\label{l1had5.1}
There exists a constant $C>0$ such that for every $\g\in\Ba_\th$, we have that
$$
||\pf_t(\ell\g)||_\th\le C||\g||_\th.
$$
\elem

\begin{proof}
By virtue of \ref{1had5.1} it suffices to estimate $|\pf_t(\ell\g)|_\th$. Fix an integer $m\ge 0$, $\om\in E_A^\infty$, and $\a, \b\in[\om|_m]$. Let $e\in E$ be such that $A_{e\a_0}=A_{e\b_0}=1$. Then
$$
\begin{aligned}
\Big|\ell(e\b)&\g(e\b)|\phi_e'(\pi(\b))|^t - \ell(e\a)\g(e\a)|\phi_e'(\pi(\a))|^t\Big|\\
&=\Big|\ell(e\b)\(\g(e\b)|\phi_e'(\pi(\b))|^t-\g(e\a)|\phi_e'(\pi(\a))|^t\)
   +\g(e\a)|\phi_e'(\pi(\a))|^t\(\ell(e\b) - \ell(e\a)\)\Big| \\
&\le \Big|\ell(e\b)\(\g(e\b)\(|\phi_e'(\pi(\b))|^t-|\phi_e'(\pi(\a))|^t\)+
    |\phi_e'(\pi(\a))|^t(\g(e\b)-\g(e\a))\)\Big|\\ 
&\  \  \  \  \   \   \   \   \   \  \  \  \  \  \   \   \   \   \   \ +\osc_{m+1}(\ell) 
      (e\om)\g(e\a)|\phi_e'(\pi(\a))|^t \\
&\le A\th^{2m}|\ell(e\b)\g(e\b)|\cdot|\phi_e'(\pi(\b))|^t+|\phi_e'(\pi(\a))|^t\osc_{m+1} 
    (\g)(e\om)+ \\
&\  \  \  \  \   \   \   \   \   \  \  \  \  \  \   \   \   \   \   \ +A\th^{2m}\g(e\a)|\phi_e'(\pi(\a))|^t
\end{aligned}
$$
with some constant $A\in (0,+\infty)$ and some constant $\th\in(0,1)$ sufficiently close to $1$. Hence, using also \eqref{1had5.1} and Lemma~\ref{l1fp61}, we get
$$
\begin{aligned}
\big|\pf_t(\ell\g)(\b)&-\pf_t(\ell\g)(\a)\big| \\
&\le A\th^{2m}\(|\pf_t(|\ell|\g)(\b)+\pf_t(\g)(\a)\)
     +K^t\pf_t\(\osc_{m+1}(\g\))(\om)\\
&\le A\th^{2m}\(||\pf_t(\ell)||_\infty||\g||_\th+||\pf_t(\g)||_\infty\)
     +K^t\pf_t\(\osc_{m+1}(\g\))(\om)\\
&\le A\th^{2m}\(||\pf_t(\ell)||_\infty||\g||_\th+||\pf_t(\g)||_\th\)
     +K^t\pf_t\(\osc_{m+1}(\g\))(\om) \\
&\le A\th^{2m}\(||\pf_t(\ell)||_\infty||\g||_\th+||\pf_t||_\th||\g||_\th\)
     +K^t\pf_t\(\osc_{m+1}(\g\))(\om) \\    
&\le A\th^{2m}\(C +1 +||\pf_t(\ell)||_\infty\)||\g||_\th+K^t\pf_t\(\osc_{m+1}(\g\))(\om)
\end{aligned}
$$
where we know that $||\pf_t||_\th \leq C+1$.
Therefore,
$$
\osc\(\pf_t(\ell\g)\)(\om)
\le A\th^{2m}\(1+||\pf_t(\ell)||_\infty\)||\g||_\th+K^t\pf_t\(\osc_{m+1}(\g\))(\om).
$$
Thus, after integrating against measure $\nu_t$, we get
$$
\begin{aligned}
||\osc_m\(\pf_t(\ell\g)\)||_{L^1(\nu_t)}
&\le A\th^{2m}\(C + 1+||\pf_t(\ell)||_\infty\)||\g||_\th
      +K^t\int\pf_t\(\osc_{m+1}(\g)\) \,d\nu_t \\
&=   A\th^{2m}\(C + 1+||\pf_t(\ell)||_\infty\)||\g||_\th
      +K^t\int\osc_{m+1}(\g)\,d\nu_t \\
&\le A\th^{2m}\(C + 1+||\pf_t(\ell)||_\infty\)||\g||_\th
      +K^t\th^{-(m+1)}|\g|_\th.
\end{aligned}
$$
Therefore,
$$
\theta^{-2m}||\osc_m\(\pf_t(\ell\g)\)||_{L^1(\nu_t)}\le 
\(A(C + 1+||\pf_t(\ell)||_\infty)+K^t\th^{-1}\)||\g||_\th.
$$
Combining this with \eqref{1had5.1} we finally get
$$
||\pf_t(\ell\g)||_\th 
\le\(||\pf_t\ell||_\infty+A(C + 1+||\pf_t(\ell)||_\infty)+K^t\th^{-1}\)||\g||_\th.
$$
So, the proof is complete.
\end{proof}

\fr As a fairly straightforward consequence of this lemma we get the following.

\bcor\label{c1had6.1}
There exists a constant $C'>0$ such that for every $\g\in\Ba_\th$ and all $n\ge 1$, we have that
$$
||\pf_{t,n}(\ell\g)||_\th\le C'||\g||_\th.
$$
\ecor

\begin{proof}
By virtue of Lemma~\ref{l1fp129} (with $k=1$) and Lemma~\ref{l1fp65} we get
$$
|\1_n\g|_\th
\le|\g|_\th+(1-\th)^{-1}||\g||_*
\le(1+2(1-\th)^{-1})||\g||_\th.
$$
Of course,
$$
||\1_n\g||_{L^1(\nu_t)}
\le ||\g||_{L^1(\nu_t)}
\le ||\g||_\th.
$$
Hence,
$$
||\1_n\g||_\th\le 2(1+(1-\th)^{-1})||\g||_\th.
$$
As 
$$
\pf_{t,n}(\ell\g)
=\pf_t(\ell\g\1_n)
=\pf_t(\ell(\1_n\g)),
$$
applying Lemma~\ref{l1had5.1}, we thus get
$$
||\pf_{t,n}(\ell\g)||_\th
=||\pf_t(\ell(\1_n\g))||_\th
\le 2C(1+(1-\th)^{-1})||\g||_\th.
$$
The proof is complete.
\end{proof}

\fr It immediately follows from this corollary, along with \eqref{5had5} and \eqref{1had5}, that 
\beq\label{1had7}
||\psi_n||_\th\le M_1
\eeq
with some constant $M_1\in(0,\infty)$ and all integers $n\ge 0$. But then by Proposition~\ref{p1fp80} (g) ,
$$
||Q_n^{(t)}(-\psi_n)||_\th\le C||\psi_n||_\th\le C M_1.
$$
Since, on the other hand, $Q_n^{(t)}(-\psi_n)=\nu_{t,n}(-\psi_n)g_n(t)$, and also, by \eqref{3had4},  
$$
\int g_n(t)\, d\nu_t\ge \frac12\int g(t)\, d\nu_t=1/2
$$
for all $n\ge 1$ sufficiently large, we thus conclude that 
$$
\begin{aligned}
CM_1
&\ge||Q_n^{(t)}(-\psi_n)||_{L^1(\nu_t)}
\ge \int|\nu_{t,n}(-\psi_n)g_n(t)|\, d\nu_t \\
&=|\nu_{t,n}(-\psi_n)|\int|g_n(t)|\, d\nu_t \\
&\ge \frac12|\nu_{t,n}(-\psi_n)|.
\end{aligned}
$$
So, 
$$
|\nu_{t,n}(-\psi_n)|\le 2CM_1,
$$
and formula \eqref{2had4} is established.

\sp Now we shall prove that \eqref{1had4} holds. Write, as usually, $||h||_1
:=||h||_{L^1(\nu_t)}$ for all $h\in L^1(\nu_t)$. With the use of \eqref{1had7} we then estimate 
$$
\begin{aligned}
||Q_n^{(t)}(\psi_n)-Q^{(t)}(\psi)||_1
&=   ||(Q_n^{(t)}-Q^{(t)})\psi_n+Q^{(t)}(\psi_n-\psi)||_1 \\
&\le ||(Q_n^{(t)}-Q^{(t)})\psi_n||_1+||Q^{(t)}(\psi_n-\psi)||_1 \\
&\le ||(Q_n^{(t)}-Q^{(t)})\psi_n||_*+||Q^{(t)}||_1||\psi_n-\psi||_1 \\
&\le |||(Q_n^{(t)}-Q^{(t)})|||\cdot||\psi_n||_\th+||Q^{(t)}||_1||\psi_n-\psi||_1 \\
&\le M_1|||(Q_n^{(t)}-Q^{(t)})|||+||Q^{(t)}||_1||\psi_n-\psi||_1.
\end{aligned}
$$
Hence, applying Proposition~\ref{p1fp80} (h), we see that in order to prove that \eqref{1had4} holds, and by having done this, to conclude the proof of Theorem~\ref{t1had2}, it suffices to show that 
\beq\label{1had8}
\lim_{n\to\infty}||\psi_n-\psi||_1=0.
\eeq
It is well-known, and follows easily from \eqref{3had5} that $\ell\in L^p(\nu_t)$ for all real $p>0$. Using Cauchy-Schwarz inequality we then estimate:
$$
\begin{aligned}
||\psi_n-\psi||_1
&=\big\|\pf_t(\ell\g_n)-\pf_t(\ell\g)\big\|_1
 =\big\|\pf_t(\ell\g_n-\ell\g)\big\|_1
 =||\ell(\g_n-\g)||_1 \\
&\le ||\ell||_2||\g_n-\g||_2
 =||\ell||_2||g_n(t)\1-g(t)||_2 \\
&=||\ell||_2||\1_n(\g_n(t)-\g(t))+\g(t)(\1_n-\1)||_2 \\
&\le||\ell||_2\(||\1_n(\g_n(t)-\g(t))||_2+||\g(t)\1_{U_n}||_2 \) \\
&\le||\ell||_2\(||\g_n(t)-\g(t)||_2+||\1_{U_n}||_2 \) \\
&\le||\ell||_2\(||\g_n(t)-\g(t)||_2+\sqrt{\nu_t(U_n)}\) \\
&\le||\ell||_2\(||\g_n(t)-\g(t)||_4+\sqrt{\nu_t(U_n)}\) 
\end{aligned}
$$
But $\lim_{n\to\infty}\nu_t(U_n)=0$ and $\lim_{n\to\infty}||g_n(t)-g(t)||_4=0$ because of 
\eqref{3had4} and \eqref{1had5}. Hence, the formula \eqref{1had8} holds and the proof of 
Theorem~\ref{t1had2} is complete.
\end{proof}

Now our goal is to show that the derivatives $\lam_n''(t)$ are uniformly bounded above in appropriate domains of $t$ and $n$. In order to do this we will need several auxiliary results. Our strategy is to apply the results of \cite{KL} for the family of operators
$$
\(\pf_{t,n}:t\in (s-\d,s+\d),n\ge 0\),
$$
where $s>\th_\cS$ and $\d>0$ is small enough. At the beginning the only normalization we assume is that $\lam_s=1$ and $g(s)\equiv\1$. Later on for ease of expression we will also assume that $\lam_t=1$ and $g(s)\equiv\1$ for all $t\in (s-\d,s+\d)$ and appropriate $\d>0$ small enough. It is evident from the form of our potentials $\phi_t(\om):=t\log|\phi_{\om_0}'(\pi(\sg(\om)))|$ that the distortion constants $M_\phi$ of Lemma~\ref{l1_2014_08_29} and Lemma~\ref{l2_2014_08_29} can be taken of common value for all $t\in(0,2s-\th_\cS]$. Denote this common constant by $M_s$. An inspection of the proof of Lemma~\ref{l1fp63} leads to the following. 

\blem\label{l1had9.1}
For every $\d\in (0,s-\th_\cS)$ there exists a constant $C_\d\in(0,+\infty)$ such that for every $t\in [s-\d,s+\d]$, every integer $k\ge 0$, and every $g\in\Ba_\th$, we have 
$$
|\pf_t^kg|_\th\le C_\d(\th\lam_t)^k|g|_\th+\lam_t^k||g||_1.
$$
\elem

\fr Since the function $(\th_\cS,+\infty)\ni t\mapsto\lam_t$ is strictly decreasing, denoting $\lam_{s-\d}$ by $M$, as an immediate consequence of Lemma~\ref{l1had9.1} we get the following.

\blem\label{l2had9.1}
For every $\d\in (0,s-\th_\cS)$, every $t\in [s-\d,s+\d]$, every integer $k\ge 0$, and every $g\in\Ba_\th$, we have 
$$
|\pf_t^kg|_\th\le C_\d(\th M)^k|g|_\th+M^k||g||_1.
$$
\elem

\fr Lemma~\ref{l1fp127} directly translates into the following.

\blem\label{l3had9.1}
For every $\d\in (0,s-\th_\cS)$, every $t\in [s-\d,s+\d]$, every integer $k\ge 0$, and every $n\ge 0$, we have
$$
||\pf_{t,n}^k||_*\le\lam_t^k\le M^k.
$$
\elem

\fr The proof of  Corollary~\ref{c1fp67} provides exact estimates of constants, and gives  the following. 

\blem\label{l1had9.2}
For every $\d\in (0,s-\th_\cS)$, every $t\in [s-\d,s+\d]$, every integer $k\ge 0$, every integer $n\ge 0$, and every $g\in\Ba_\th$, we have 
$$
||\pf_{t,n}^kg||_\th\le (C_\d+1)(\th M)^k||g||_\th+ (C_\d+1)(1+\th(1-\th)^{-1})M^k||g||_*.
$$
\elem

\sp\fr From now on throughout the entire section we assume that condition (U2) holds in the following uniform version:

\sp\begin{itemize}
\item[(U2*)] There exists $\rho\in(0,1)$ such that for some $\d>0$ and for all integers $n\ge 0$ we have
$$
\sup\big\{\mu_t(U_n):t\in [s-\d,s+\d]\big\}\le\rho^n.
$$
\end{itemize}

\sp\fr 
%With all the above lemma~\ref{l2fp129}
We now have the following.

\sp

\blem\label{l2had9.2}
For every $\d\in (0,s-\th_\cS)$, every $t\in [s-\d,s+\d]$ and every integer $n\ge 0$, we have 
$$
|||\pf-\pf_n|||\le 2\lam_t(\rho^{1/q})^n\le 2M\rho^{n/q}.
$$
\elem

\sp\fr Now, Lemmas~\ref{l3had9.1}, \ref{l1had9.1}, and \ref{l2had9.2}, along with formula 
\eqref{2fp81}, and compactness (in fact finite dimensionality) of the operators $\pi_k:\cB_\th\to\cB_\th$ imply that Theorem~1 in \cite{KL} along with all corollaries therein, applies to the family of operators 
$$
\(\pf_{t,n}:t\in (s-\d,s+\d),n\ge 0\),
$$
(i. e. $\pf_s$ corresponds to $P_0$ and $\pf_{t,n}$ correspond to operators $P_\e$) with 
$$
(t,n)\to s \  \  \eqv \  \  t\to s \  \text{ and } \ n\to +\infty
$$
to give the following extension of Proposition~\ref{p1fp80}.

\sp
\bprop\label{p1had9.3} Fix $s>\th_\cS$. Let the Perron-Frobenius operator $\pf_s:\Ba_\th\to\Ba_\th$ be normalized so that $\lam_s=1$. Then there exist $\d\in (0,s-\th_\cS)$ sufficiently small and an integer $n_s\ge 0$ sufficiently large such that for all $(t,n)\in (s-\d,s+\d)\times \{n_s,n_s+1,\ld,\}$ there exist bounded operators $Q_n^{(t)},\De_n^{(t)}:\Ba_\th\to\Ba_\th$ and complex numbers $\lam_n(t)\ne 0$ with the following properties: 
\begin{itemize}
\item[(a)] $\lam_n(t)$ is a simple eigenvalue of the operator $\pf_{t,n}:\cB_\th\to\cB_\th$.

\sp\item[(b)] $Q_t^{(n)}:\cB_\th\to\cB_\th$ is a projector ($Q_t^{(n)2}=Q_n$) onto the $1$--dimensional eigenspace of $\lam_n(t)$.

\sp\item[(c)] $\pf_{t,n}=\lam_n(t)Q_t^{(n)}+\De_{t,n}$.

\sp\item[(d)] $Q_t^{(n)}\circ\De_{t,n}=\De_{t,n}\circ Q_t^{(n)}=0$.

\sp\item[(e)] There exist $\ka\in (0,1)$ and $C>0$ (independent of $(t,n)\in (s-\d,s+\d)\times \{n_s,n_s+1,\ld,\}$) such that
$$
||\De_{t,n}^k||_\th\le C\ka^k
$$
for all $k\ge 0$. In particular,
$$
||\De_{t,n}^kg||_\infty\le ||\De_{t,n}^kg||_\th\le C\ka^k||g||_\th
$$
for all $g\in\Ba_\th$.

\sp\item[(f)] $\lim_{(t,n)\to s}\lam_n(t)=1$.

\sp\item[(g)] Enlarging the above constant $C>0$ if necessary, we have
$$
||Q_t^{(n)}||_\th\le C.
$$
In particular, 
$$
||Q_t^{(n)}g||_\infty\le ||Q_t^{(n)}g||_\th\le C||g||_\th
$$
for all $g\in\Ba_\th$.

\sp\item[(h)] $\lim_{(t,n)\to s}|||Q_t^{(n)}-Q_s|||=0$.
\end{itemize}
\eprop

\sp\fr Now we are ready to prove the following.

\blem\label{l1had9}
For every $s>\th_\cS$ there exists $\eta\in(0,1)$ such that
$$
M:=\sup_{n\ge n_s}\sup\{\lam_n''(t):t\in (s-\eta,s+\eta)\}<+\infty.
$$
\elem

\begin{proof}
Throughout the whole proof we always assume that $t\in (s-\d,s+\d)$ and $n\ge n_s$, where $\d>0$ is the one produced in Proposition~\ref{p1had9.3}. Fix an integer $N\ge 1$ and differentiate the eigenvalue equation 
$$
\pf_{t,n}^Ng_n(t)=\lam_n^N(t)g_n(t)
$$
with respect the variable $t$ two times. This gives in turn
$$
(\pf_{t,n}^N)'(g_n(t))+\pf_{t,n}^N(g_n'(t))
=N\pf_{t,n}^N\lam_n'(t)\lam_n^{N-1}(t)g_n(t)+\lam_n^N(t)g_n'(t)
$$
and
$$
\begin{aligned}
(\pf_{t,n}^N)''(g_n(t))&+(\pf_{t,n}^N)'(g_n'(t))+(\pf_{t,n}^N)'(g_n'(t))
       +\pf_{t,n}^N(g_n''(t))= \\
&=N(N-1)\lam_n^{N-2}(t)(\lam_n'(t))^2g_n(t)+N\lam_n^{N-1}(t)\lam_n'(t)g_n'(t)+\\
       &\   \  \  \  \  \  \  \   \ +N\lam_n^{N-1}(t)\lam_n''(t)g_n(t) +N\lam_n^{N-1} 
        (t)\lam_n'(t)g_n'(t)+\lam_n^N(t)g_n''(t).
\end{aligned}
$$
Equivalently:
$$
\begin{aligned}
(\pf_{t,n}^N)''(g_n(t))&+2(\pf_{t,n}^N)'(g_n'(t))+\pf_{t,n}^N(g_n''(t))= \\
&=N(N-1)\lam_n^{N-2}(t)(\lam_n'(t))^2g_n(t)+2N\lam_n^{N-1}(t)\lam_n'(t)g_n'(t)+\\
&\   \  \  \  \  \  \  \   \ +N\lam_n^{N-1}(t)\lam_n''(t)g_n(t)+\lam_n^N(t)g_n''(t).
\end{aligned}
$$
Noting that 
$$
Q_t^{(n)}\pf_{t,n}^N(g_n''(t))
=\pf_{t,n}^NQ_t^{(n)}(g_n''(t))
=\lam_n^N(t)Q_t^{(n)}(g_n''(t))
$$
and applying to both sides of this equality the linear operator $Q_t^{(n)}$, we get
$$
\begin{aligned}
Q_t^{(n)}&(\pf_{t,n}^N)''(g_n(t))+2Q_t^{(n)}(\pf_{t,n}^N)'(g_n'(t))=\\
&=N(N-1)\lam_n^{N-2}(t)(\lam_n'(t))^2g_n(t)+2N\lam_n^{N-1}(t)\lam_n'(t)Q_t^{(n)}(g_n'(t))+N\lam_n^{N-1}(t)\lam_n''(t)g_n(t).
\end{aligned}
$$
Now since $(\pf_{t,n}^N)'(g_n(t))=\pf_{t,n}^N(g_n(t)S_N\ell)$, and so $(\pf_{t,n}^N)''(g_n(t))=\pf_{t,n}^N(g_n(t)(S_N\ell)^2)$, and since also $(\pf_{t,n}^N)'(g_n'(t))=\pf_{t,n}^N(g_n'(t)S_N\ell)$, we thus get
\beq\label{1had10}
\begin{aligned}
Q_t^{(n)}&\pf_{t,n}^N(g_n(t)(S_N\ell)^2)+2Q_t^{(n)}\pf_{t,n}^N(g_n'(t)S_N\ell)=\\
&=N(N-1)\lam_n^{N-2}(t)(\lam_n'(t))^2g_n(t)+2N\lam_n^{N-1}(t)\lam_n'(t)Q_t^{(n)}(g_n'(t))+N\lam_n^{N-1}(t)\lam_n''(t)g_n(t).
\end{aligned}
\eeq
We first deal with the term $Q_t^{(n)}\pf_{t,n}^N(g_n'(t)S_N\ell)$. We have
\beq\label{2had10}
\begin{aligned}
Q_t^{(n)}\pf_{t,n}^N(g_n'(t)S_N\ell)
&=Q_t^{(n)}\pf_{t,n}^N\lt(g_n'(t)\sum_{j=0}^{N-1}\ell\circ\sg^j\rt)
 =\sum_{j=0}^{N-1}Q_t^{(n)}\pf_{t,n}^N\(g_n'(t)\ell\circ\sg^j\) \\
&=\sum_{j=0}^{N-1}Q_t^{(n)}\pf_{t,n}^{N-j}\(\ell\pf_{t,n}^j)'(g_n'(t))\)\\
&=\sum_{j=0}^{N-1}\pf_{t,n}^{N-j-1}Q_t^{(n)}\pf_{t,n}\(\ell\pf_{t,n}^j)'(g_n'(t))\)\\
&=\sum_{j=0}^{N-1}\lam_n(t)^{N-j-1}Q_t^{(n)}\pf_{t,n}\(\ell\pf_{t,n}^j)'(g_n'(t))\).
\end{aligned}
\eeq
Now, by virtue of Proposition~\ref{p1had9.3}, particularly by its parts (c) and (e) and by Corollary~\ref{c1had6.1}, we get
$$
\begin{aligned}
\Big\|\pf_{t,n}\(\ell\pf_{t,n}^j)'(g_n'(t)\) &- \lam_n(t)^j\pf_{t,n}\(\ell Q_t^{(n)})'(g_n'(t)\)\Big\|_\infty \le \\
&\le \Big\|\pf_{t,n}\(\ell\pf_{t,n}^j)'(g_n'(t)\)- \lam_n(t)^j\pf_{t,n}\(\ell Q_t^{(n)})'(g_n'(t)\)\Big\|_\th \\
&= \Big\|\pf_{t,n}\(\ell\(\pf_{t,n}^j)'(g_n'(t))- \lam_n(t)^jQ_t^{(n)})'(g_n'(t))\)\)\Big\|_\th\\
&= \big\|\pf_{t,n}\(\ell\De_n^j)'(g_n'(t))\)\big\|_theta
\le C'||(%\De_n^j)'(g_n'(t))||_\th \\
&\le C'C\kappa^j||g_n'(t)||_\th.
\end{aligned}
$$
Therefore, by item (g) of Proposition~\ref{p1fp80} we get
\beq\label{3had10}
\Big\|Q_t^{(n)}\pf_{t,n}\(\ell\pf_{t,n}^j(g_n'(t))\) - \lam_n^j(t)Q_t^{(n)}\pf_{t,n}\(\ell Q_t^{(n)}(g_n'(t))\)\Big\|_\infty
\le C'C^2\kappa^j||g_n'(t)||_\th.
\eeq 
\end{proof}
On the other hand, because of \eqref{4had3}, we get
$$
\begin{aligned}
 \lam_n'(t)Q_t^{(n)}(g_n'(t))
&=\nu_{t,n}(g_n'(t))\lam_n'(t)g_n(t)
=\nu_{t,n}(g_n'(t))Q_t^{(n)}\pf_{t,n}\(\ell g_n(t)\) \\
&=Q_t^{(n)}\pf_{t,n}\(\ell \nu_{t,n}(g_n'(t)) g_n(t)\) \\
&=Q_t^{(n)}\pf_{t,n}\(\ell Q_t^{(n)}(g_n'(t))\) \\
\end{aligned}
$$
Therefore, using \eqref{2had10} and \eqref{3had10}, we get
$$
\begin{aligned}
\Big\|\lam_n^{1-N}(t)N^{-1}&\Big(Q_t^{(n)}\pf_{t,n}\(g_n'(t)S_N(\ell)\) - N\lam_n^{N-1}(t)\lam_n'(t)Q_t^{(n)}(g_n'(t))\Big)\Big\|_\infty=\\
&=\Big\|\frac1N\sum_{j=0}^{N-1}\lam_n(t)^{-j}Q_t^{(n)}\pf_{t,n} \(\ell\pf_{t,n}^j(g_n'(t))\) - Q_t^{(n)}\pf_{t,n}\(\ell Q_t^{(n)}(g_n'(t)\))\Big\|_\infty\\
&=\Big\|\frac1N\sum_{j=0}^{N-1}\Big[\lam_n(t)^{-j}Q_t^{(n)}\pf_{t,n} \(\ell\pf_{t,n}^j(g_n'(t))\)-Q_t^{(n)}\pf_{t,n}\(\ell
   Q_t^{(n)}(g_n'(t))\) \Big]\Big\|_\infty\\
&\le\frac1N\sum_{j=0}^{N-1}\Big\|\lam_n(t)^{-j}Q_t^{(n)}\pf_{t,n} \(\ell\pf_{t,n}^j(g_n'(t))\)-Q_t^{(n)}\pf_{t,n}\(\ell
   Q_t^{(n)}(g_n'(t))\)\Big\|_\infty\\
&\le\frac1N\sum_{j=0}^{N-1} C'C^2||g_n'(t)||_\th(\kappa/\lam_n(t))^j.
\end{aligned}
$$
But, by Proposition~\ref{p1had9.3}, we have $\kappa/\lam_n(t)\le \frac{2\ka}{1+\ka}<1$ for all $(t,n)$ sufficiently close to $s$. Therefore, for all such pairs $(t,n)$, we have 
$$
\lim_{N\to\infty}\Big\|2N^{-1}\lam_n^{1-N}(t)\Big(Q_t^{(n)}\pf_{t,n}^N\(g_n'(t)S_N(\ell)\) - N\lam_n^{N-1}(t)\lam_n'(t)Q_t^{(n)}(g_n'(t))\Big)\Big\|_\infty=0
$$
where the convergence, in the supremum norm $||\cdot||_\infty$ is uniform with respect to all $t$ sufficiently close to $s$. Inserting this to \eqref{1had10}, we thus get
\beq\label{1had13}
\lam_n''(t)g_n(t)=\lim_{N\to\infty}\Big[N^{-1}\lam_n^{1-N}(t) Q_t^{(n)}\pf_{t,n}^N\(g_n(t)(S_N(\ell))^2\)\Big]-\lam_n^{-1}(t)(N-1)g_n(t)(\lam_n'(t))^2g_n(t),
\eeq
with the same meaning of convergence as above. 

\sp Let us first deal with the term $Q_t^{(n)}\pf_{t,n}^N(g_n(t)(S_N\ell)^2)$. We have
\beq\label{1had14}
\begin{aligned}
Q_t^{(n)}&\pf_{t,n}^N(g_n(t)(S_N\ell)^2)= \\
&=  2\sum_{i=0}^{N-1}\sum_{j=i+1}^{N-1}Q_t^{(n)}\pf_{t,n}^N\(g_n(t)\ell\circ\sg^i\cdot
      \ell\circ\sg^j\)+\sum_{j=0}^{N-1}Q_t^{(n)}\pf_{t,n}^N\(g_n(t)\ell^2\circ\sg^j\)\\
&=\sum_{j=0}^{N-1}Q_t^{(n)}\pf_{t,n}^N\(g_n(t)\ell^2\circ\sg^j\)+
    2\sum_{i=0}^{N-1}\sum_{j=i+1}^{N-1}Q_t^{(n)}\pf_{t,n}^N\(g_n(t)(\ell\cdot\ell 
      \circ\sg^{j-i})\circ\sg^i\) \\
&=\sum_{j=0}^{N-1}Q_t^{(n)}\pf_{t,n}^{N-j}\(\ell^2\pf_{t,n}^j(g_n(t))\)+
    2\sum_{i=0}^{N-1}\sum_{j=i+1}^{N-1}Q_t^{(n)}\pf_{t,n}^{N-i}\(\ell\cdot\ell 
      \circ\sg^{j-i}\pf_{t,n}^i(g_n(t))\) \\
&=\lam_n^j(t)\sum_{j=0}^{N-1}Q_t^{(n)}\pf_{t,n}^{N-j}\(g_n(t)\ell^2\)+
    2\lam_n^i(t)\sum_{i=0}^{N-1}\sum_{j=i+1}^{N-1}Q_t^{(n)}\pf_{t,n}^{N-
    i}\(g_n(t)\ell\cdot\ell\circ\sg^{j-i})\) \\
&=\lam_n^j(t)\sum_{j=0}^{N-1}Q_t^{(n)}\pf_{t,n}^{N-j}\(g_n(t)\ell^2\)+
    2\lam_n^i(t)\sum_{i=0}^{N-1}\sum_{j=i+1}^{N-1}Q_t^{(n)}\pf_{t,n}^{N-
    j}\(\ell\pf_{t,n}^{j-i}(\ell g_n(t))\) \\
&=\lam_n^{N-1}(t)\sum_{j=0}^{N-1}Q_t^{(n)}\pf_{t,n}\(g_n(t)\ell^2\)+
    2\sum_{i=0}^{N-1}\sum_{j=i+1}^{N-1}   \lam_n^{N+i-(j+1)}(t) Q_t^{(n)}\pf_{t,n}
    \(\ell\pf_{t,n}^{j-i}(\ell g_n(t))\) \\
&=\lam_n^{N-1}(t)Q_t^{(n)}\pf_{t,n}\(g_n(t)\ell^2\)+
    2\sum_{k=1}^{N-1}\lam_n^{N-k-1}(t)(N-k)Q_t^{(n)}\pf_{t,n}
    \(\ell\pf_{t,n}^k(\ell g_n(t))\). \\
\end{aligned}
\eeq 
Now, using Proposition~\ref{p1had9.3} (c) (and (d)) and denoting $\psi_n = \pf_{t,n}\ell g_n(t))$, we get
$$
\begin{aligned}
Q_t^{(n)}\pf_{t,n}\(\ell\pf_{t,n}^k(\ell g_n(t))\)
&= Q_t^{(n)}\pf_{t,n}\(\ell\pf_{t,n}^{k-1}\(\pf_{t,n}\ell g_n(t))\)\) 
 = Q_t^{(n)}\pf_{t,n}\(\ell\pf_{t,n}^{k-1}(\psi_n)\) \\
&= Q_t^{(n)}\pf_{t,n}\(\ell(\lam_n^{k-1}(t)Q_t^{(n)}(\psi_n)+\De_n^{k-1}(\psi_n))\)\\
&= \lam_n^{k-1}(t)Q_t^{(n)}\pf_{t,n}\(\ell Q_t^{(n)}(\psi_n)\)+
   Q_t^{(n)}\pf_{t,n}\(\De_n^{k-1}(\psi_n)\) \\
&= \lam_n^{k-1}(t)Q_t^{(n)}\pf_{t,n}\(\ell\nu_{t,n}(\psi_n)g_n(t)\)+
   Q_t^{(n)}\pf_{t,n}\(\De_n^{k-1}(\psi_n)\) \\
&= \lam_n^{k-1}(t)\nu_{t,n}(\psi_n)Q_t^{(n)}\pf_{t,n}\(\ell
   g_n(t)\)+Q_t^{(n)}\pf_{t,n}\(\De_n^{k-1}(\psi_n)\) \\
&= \lam_n^{k-1}(t)\nu_{t,n}(\psi_n)Q_t^{(n)}(\psi_n)
   +Q_t^{(n)}\pf_{t,n}\(\De_n^{k-1}(\psi_n)\) \\
&= \lam_n^{k-1}(t)\(\nu_{t,n}(\psi_n)\)^2g_n(t)
   +Q_t^{(n)}\pf_{t,n}\(\De_n^{k-1}(\psi_n)\) \\
&= \lam_n^{k-1}(t)\(\lam_n'(t)\)^2g_n(t)
   +Q_t^{(n)}\pf_{t,n}\(\De_n^{k-1}(\psi_n)\). \\
\end{aligned}
$$
Therefore, using Proposition~\ref{p1had9.3}, we get
$$
\begin{aligned}
2\sum_{k=1}^{N-1}&\lam_n^{N-k-1}(t)(N-k)Q_t^{(n)}\pf_{t,n}
    \(\ell\pf_{t,n}^k(\ell g_n(t))\)=\\
&= \lam_n^{N-2}(t)N(N-1)\(\lam_n'(t)\)^2g_n(t)+2\sum_{k=1}^{N-1}\lam_n^{N-k-1}(t)
    Q_t^{(n)}\pf_{t,n}\(\De_n^{k-1}(\psi_n)\) \\
&= \lam_n^{N-2}(t)N(N-1)\(\lam_n'(t)\)^2g_n(t)+2\sum_{k=1}^{N-1}\lam_n^{N-k-1}(t)
    \pf_{t,n}Q_t^{(n)}\(\De_n^{k-1}(\psi_n)\) \\
&= \lam_n^{N-2}(t)N(N-1)\(\lam_n'(t)\)^2g_n(t)+2\lam_n^{N-2}(t)\pf_{t,n}Q_t^{(n)}(\psi_n)  
    \\
&= \lam_n^{N-2}(t)N(N-1)\(\lam_n'(t)\)^2g_n(t)+2\lam_n^{N-2}(t)\pf_{t,n}\(\nu_{t,n}(\psi_n)
    g_n(t)\) \\
&= \lam_n^{N-2}(t)N(N-1)\(\lam_n'(t)\)^2g_n(t)+2\lam_n^{N-2}(t)\nu_{t,n}(\psi_n)\pf_{t,n}\(
    g_n(t)\) \\
&= \lam_n^{N-2}(t)N(N-1)\(\lam_n'(t)\)^2g_n(t)+2\lam_n^{N-1}(t)\lam_n'(t)g_n(t).
\end{aligned}
$$
In consequence, denoting by $T_N(t,n)$ the function whose limit (as $n\to\infty$) is calculated in \eqref{1had13}, and utilizing \eqref{1had14}, we get
$$
T_N(t,n)=Q_t^{(n)}\pf_{t,n}\(g_n(t)\ell^2\)+\frac2N\lam_n'(t)g_n(t)
$$
It thus follows from \eqref{1had13} that 
\beq\label{1had15}
\lam_n''(t)g_n(t)=Q_t^{(n)}\pf_{t,n}\(g_n(t)\ell^2\).
\eeq
Since, by Proposition~\ref{p1fp131}, all the operators $Q_t^{(n)}:\Ba_\th\to\Ba_\th$ are positive, and because of this also non-decreasing, and $g_n(t)=Q_t^{(n)}\1$ is non-negative, the formula \eqref{1had15} yields the following.
$$
\begin{aligned}
\lam_n''(t)g_n(t)
&\le ||g_n(t)||_\infty Q_t^{(n)}\(\pf_{t,n}(\ell^2)\) 
 \le ||g_n(t)||_\infty ||\pf_{t,n}(\ell^2)||_\infty Q_t^{(n)}\1 \\
&\le ||g_n(t)||_\infty ||\pf_t(\ell^2)||_\infty g_n(t) \\
&\le 2||g_n(t)||_\infty||\pf_s(\ell^2)||_\infty g_n(t),
\end{aligned}
$$
where the last inequality was written for $t$ sufficiently close to $s$. Canceling out $g_n(t)$ and noticing that by Proposition~\ref{p1had9.3} (g), $||g_n(t)||_\infty=||Q_t^{(n)}\1||_\infty\le C$, we now finally obtain that 
$$
\lam_n''(t)\le 2C||\pf_s(\ell^2)||_\infty,
$$
and the proof is complete.
\endpf

\fr Now we shall prove the following. 

\blem\label{l1had17}
We have

\sp
\begin{itemize}
\item[(a)] For every $n\ge 1$ the function $(\th_\cS,+\infty)\ni t\mapsto\lam_n(t)$ is decreasing.

\sp\item[(b)] For every $s\in (\th_\cS,+\infty)$ and for every $n\ge 1$ large enough there exists $\d>0$ such that the function $\lam_n|_{(s-\d,s+\d)}$ is strictly decreasing, in fact $\lam_n'\le\frac14\lam'(s)$ on $(s-\d,s+\d)$.

\sp\item[(c)] For every $t\in (\th_\cS,+\infty)$ and for every $n\ge 1$, $\lam_n(t)\le\lam(t)$.

\sp\item[(d)] For every $n\ge 1$, $\lim_{t\to+\infty}\lam_n(t)=0$.

\sp\item[(e)] For every $n\ge 1$ large enough there exists a unique $b_n>0$ such that $\lam_n(b_n)=1$.
\end{itemize}
\elem

\begin{proof}
For part (a), Proposition~\ref{p1had9.3} implies that
\beq\label{1had17}
\lam_n(t)=\lim_{k\to\infty}\big|\big|\pf_{t,n}^k\1\big|\big|_\infty^{1/k},
\eeq
and since for each $n\ge 1$ the function $t\mapsto\big|\big|\pf_{t,n}^k\1\big|\big|_\infty$ is decreasing, item (a) follows immediately. For part (b) note that $\lam'(s)<0$. Hence, by Theorem~\ref{t1had2}, $\lam_n'(s)<\frac12\lam'(s)<0$ for all 
$n\ge 1$ large enough, say $n\ge N_1$. Take now $\d\in(0,\eta)$ so small that $M\d\le -
\frac14\lam'(s)$, where $M\ge 0$ is the constant coming from Lemma~\ref{l1had9}. By the Mean Value Theorem $\lam_n'(t)=\lam_n'(s)+\lam_n''(u)(t-s)$ for every $t\in (s-\d,s+\d)$ and some $u\in (s-\d,s+\d)$ depending on $t$. Hence, applying Lemma~\ref{l1had9}, we get for all $n\ge N_1$ and all $t\in (s-\d,s+\d)$ that
$$
\lam_n'(t)<\frac12\lam'(s)+M\d<0.
$$
Thus item (b) is proved. Similarly as in item (a), item (c) immediately follows from \eqref{1had17} and inequality $\pf_{t,n}^k\1\le \pf_t^k\1$. Item (d) is an immediate consequence of item (c) and the well-known fact (see \cite{GDMS}) that $\lim_{t\to+\infty}\lam(t)=0$. Proving (e), it is well-known (see again \cite{GDMS}) that there exists a unique $b\in (\th_\cS,+\infty)$ such that
$$
\lam(b)=1.
$$
Let $\d>0$ be the value produced in item (b) for $s=b$. We know that
$$
\lam\(b-\frac12\d\)>0 \  \text{ and } \  \lam\(b+\frac12\d\)<0.
$$
It the follows from Proposition~\ref{p1had9.3} (f) that
$$
\lam_n\(b-\frac12\d\)\ge \frac12\lam\(b-\frac12\d\)>0 \  \  \text{ and } \  \  \lam_n\(b+\frac12\d\)\le \frac12\lam\(b+\frac12\d\)<0
$$
for all 
$n\ge 1$ large enough, say $n\ge N_2$. Because of the choice of $\d>0$ and because of item (b), we may also have $N_2\ge 1$ so large that the function $\lam_n\big|_{[b-\frac12\d,b+\frac12\d]}$ is strictly decreasing for every $n\ge N_2$. Therefore, for every $n\ge N_2$ the function $\lam_n\big|_{[b-\frac12\d,b+\frac12\d]}$ has a unique zero. Along with item (a) this finishes the proof of item (e). The proof of Lemma~\ref{l1had17} is complete.
\end{proof}

\brem\label{r1had18}
With the help of Proposition~\ref{p1had9.3} we could have strengthened Theorem~\ref{t1had2} to show uniform convergence with respect to $t$ ranging over compact subsets of $(\th_\cS,+\infty)$. However, we really do not need this in the current paper.
\erem

\fr By analogy to the unperturbed case, we call the numbers $b_n$ produced in this lemma Bowen's parameters. Now we can prove the following. 

\bprop\label{p1had18}
With the settings of the current section (in particular with the stronger condition {\rm (U2*)} replacing {\rm (U2)}), we have
$$
\lim_{n\to\infty}\frac{b-b_n}{\mu_b(U_n)}=
\begin{cases}
1/\chi_{\mu_b} &\text{{\rm if (U4A) holds}} \\
\(1-|\phi_\xi'(\pi(\xi^\infty))|\)/\chi_{\mu_b} &\text{{\rm if (U4B) holds}}
\end{cases}
$$
\eprop

\begin{proof}
Since the functions $(\th_\cS,+\infty)\ni t\mapsto\lam_n(t)$, $n\ge 1$, are all real-analytic by the  Kato-Rellich Perturbation Theorem, making use of Lemma~\ref{l1had9}, we can apply Taylor's Theorem to get
$$
1=\lam_n(b_n)=\lam_n(b)+\lam'(b)+\cO((b-b_n)^2).
$$
Equivalently,
$$
\frac{1-\lam_n(b)}{b-b_n}=-\lam'(b)+\cO(b-b_n).
$$
Denoting by $d(\xi)$ the right-hand side of the formula  appearing in  Proposition~\ref{p1fp133}, and using this proposition along with the fact that $\lam(b)=1$, we thus get
$$
\lim_{n\to\infty}\frac{\mu_b(U_n)}{b-b_n}=-\lam'(b)d^{-1}(\xi).
$$
Equivalently,
\beq\label{1had19}
\lim_{n\to\infty}\frac{b-b_n}{\mu_b(U_n)}=-\frac1{\lam'(b)}d(\xi).
\eeq
But expanding \eqref{4had3} with $n=\infty$, we get $\lam'(b)=-\lam(b)\chi_{\mu_b}=-\chi_{\mu_b}$, and inserting this into \eqref{1had19} completes the proof.
\end{proof}

\sp\fr Now we shall link Bowen's parameters $b_n$ to geometry. We shall prove the following.

\bthm\label{t1had19}
Let $\cS=\{\phi_e\}_{e\in E}$ be a finitely primitive strongly regular conformal graph directed Markov system. Let $(U_n)_{n=0}^\infty$ be a nested sequence of open subsets of $E_A^\infty$ satisfying conditions {\rm (U0), (U1)}, and {\rm (U2*)} with $s=b_{\cS}$. If for $n\ge 0$
$$
\tilde K_n:=\bi_{k=0}^\infty\sg^{-k}(U_n^c)=\{\om\in E_A^\infty:\forall_{(k\ge 0)}\ \sg^k(\om)\notin U_n\}
$$
and
$$
K_n:=\pi_\cS(\tilde K_n).
$$
Then 
$$
\HD(K_n)=b_n
$$
for all $n\ge 0$ large enough.
\ethm

\begin{proof}
Put
$$
h_n:=\HD(K_n).
$$
We first shall prove that
$$
h_n\le b_n
$$
for all $n\ge 0$ large enough. Assume that $\d>0$ is chosen so small that the conclusion  of 
 Lemma~\ref{l1had17} (b)holds. Take then an arbitrary $t>b_n$. Fix any $q\ge 1$. Define
$$
\tilde K_n(q):=\{\om\in \tilde K_n: \om_n=q \  \text{{\rm for infinitely many }} n\}
$$
and 
$$
K_n(q):=\pi(\tilde K_n(q)).
$$
Our first goal is to show that
\beq\label{1had21}
\HD(K_n(q))\le b_n
\eeq
for all $n\ge 0$ large enough. Indeed, for every $k\ge 1$ let
$$
\tilde E_k(q):=\{\om|_k:\om\in \tilde K_n(q)  \  \text{{\rm  \ and  \ }} \om_{k+1}=q\}.
$$
Fix an arbitrary $\a\in E_A^\infty$ such that $q\a\in E_A^\infty$. Then, using (BDP),Proposition~\ref{p1had9.3} (c), (e), and (g), along with Lemma~\ref{l1fp61}, we get
$$
\begin{aligned}
\sum_{\tau\in\tilde E_k(q)}\diam^t\(\phi_\tau(X_{t(\tau)})\)
&=     \sum_{\tau\in\tilde E_k(q)}\diam^t\(\phi_\tau(X_{t(q)})\)
 \comp \sum_{\tau\in\tilde E_k(q)}||\phi_\tau'||^t \\
&\comp \sum_{\tau\in\tilde E_k(q)}|\phi_\tau'(\pi(q\a))|^t 
 \le \pf_t^k(\1_n^k)(q\a)=\pf_{t,n}^k(\1)(q\a) \\
&=\lam_n^k(t)Q_t^{(n)}(\1)(q\a)+S^k(\1)(q\a) \\
&\le \lam_n^k(t)||Q_t^{(n)}(\1)||_\infty+||S^k(\1)||_\infty \\
&\le C\lam_n^k(t)+C\ka^k \\
&=C(\lam_n^k(t)+\ka^k).
\end{aligned}
$$
Therefore, for every $k\ge 0$, using the facts that $\lam_n(t)<1$ (Lemma~\ref{l1had17}
(b)) and that $\ka<1$, we get
\beq\label{1had22}
\sum_{k=l}^\infty\sum_{\tau\in\tilde E_k(q)}\diam^t\(\phi_\tau(X_{t(\tau)})\)
\le C\sum_{k=l}^\infty(\lam_n^k(t)+\ka^k)
\le C\(1-\lam_n(t))^{-1}\lam_n^l(t)+(1-\ka)^{-1}\ka^l\).
\eeq
Since $\bu_{k=l}^\infty\bu_{\tau\in\tilde E_k(q)}\phi_\tau(X_{t(\tau)})$ is a cover of $K_n(q)$ whose diameters converge (exponentially fast) to zero as $l\to\infty$, formula 
\eqref{1had22} yields $\H_t(K_n(q))=0$.Therefore, $\HD(K_n(q))\le t$. As $t>b_n$ was arbitrary, this gives  formula \eqref{1had21}. Let
$$
\begin{aligned}
\tilde K_n(\infty)
:&=\big\{\om\in \tilde K_n: \text{{\rm at least one }}\, q\in\N \  \, \text{{\rm appears in }}  \  \om  \ \, \text{{\rm infinitely many times}}\big\} \\
&=\bu_{q=1}^\infty\tilde K_n(q)
\end{aligned}
$$
and let
$$
K_n(\infty):=\pi\(\tilde K_n(\infty)\)=\bu_{q=1}^\infty K_n(q).
$$
Formula \eqref{1had21} and $\sg$-stability of Hausdorff dimension then imply that
\beq\label{2had22}
\HD(K_n(\infty))\le b_n
\eeq
Now, for every integer $l\ge 1$ let
$$
\tilde K_n^*(l)
:=\big\{\om\in E_A^\infty: \text{{\rm the letters }} \ 1,2,\ld,l \ \,  \text{{\rm appear in }} \ \om \  \, \text{{\rm only finitely many times}}\big\}
$$
and 
$$
\tilde K_n^0(l)
:=\big\{\om\in E_A^\infty: \text{{\rm the letters }} \ 1,2,\ld,l \ \,  \text{{\rm do not appear in }} \ \om \  \, \text{{\rm at all}}\big\}
$$
Furthermore,
$$
K_n^*(l):=\pi\(\tilde K_n^*(l)\) \  \text{{\rm and }} \  K_n^0(l):=\pi\(\tilde K_n^0(l)\).
$$
But 
$$
K_n^*(l)\sbt \bu_{\om\in E_A^*}\phi_\om(K_n^0(l)),
$$
and therefore
$$
\HD(K_n^*(l))=\HD(K_n^0(l)).
$$
But $K_n\sms K_n(\infty)\sbt \bi_{l=1}^\infty K_n^*(l)$. Hence, applying Theorem~4.3.6 in 
\cite{GDMS}, we get
$$
\HD(K_n\sms K_n(\infty))
\le\inf_{l\ge 1}\{\HD(K_n^*(l))\}
=\inf_{l\ge 1}\{\HD(K_n^0(l))\}
=\th_\cS
<b_\cS(=b).
$$
Since $\lim_{n\to\infty}b_n=b$, this implies that for all $n\ge 1$ large enough $\HD(K_n\sms K_n(\infty))<b_n$. Along with \eqref{2had22} this yields 
\beq\label{1had23}
\HD(K_n)\le b_n.
\eeq
The opposite inequality is even more involved. The main difficulty is caused by the fact that the Variational Principle (for topological pressure) holds in the case of infinite alphabet only for topologically mixing subshifts. Our idea is, given $n\ge 0$, to make smaller and smaller perturbations of the operators $\pf_n$, such that the difference is an operator acting essentially on a finite alphabet symbol space. So, given an integer $l\ge 1$ we denote
$$
\N_l:=\{1,2,\ld,l\}.
$$
Given also $n\ge 0$ we set
$$
U_n^{(l)}:=U_n\cup \N_l^c.
$$
For a time being $\phi:E_A^\infty\to\R$ is again an arbitrary summable H\"older continuous potential and $\mu_\phi$ is the corresponding $\sg$-invariant Gibbs/equilibrium state. Later on we will need $\phi$ to be of the form $\phi_t(\om)=t\log|\phi_{\om_0}'(\pi(\sg(\om)))|$. Given $q\ge n$ let $l_q\ge 1$ be the least integer such that
\beq\label{1had26.1}
\mu_\phi(\N_{l_q}^c)\le \rho^n.
\eeq
Set
$$
U_n(q):=U_n^{l_q}.
$$
Of course each open set $U_n(q)$ is a disjoint union of cylinders of length $q$ so that condition (U1) is satisfied for the sequence $(U_n(q))_{q=n}^\infty$. $\pf:=\pf_\phi$ is the fully normalized transfer operator associated to $\phi$. As in Section~\ref{SecSingPerturb} we define the operators
$$
\pf_{n,q}(g):=\pf\(\1_{U_n^c(q)}g\).
$$
The space $\Ba_\th$ and the norm $||\cdot||_\th$ remain unchanged. We however naturally adjust the seminorm $|\cdot|_*$ to depend on our sequence $(U_n(q))_{q=n}^\infty$. We set for $g\in\Ba$:
$$
|g|_n^*:=\sup_{j\ge 0}\sup_{m\ge 1}\Big\{\th^{-m}\int_{\sg^{-j}\(\N_{l_m}^c\)}|g|\, d\mu_\phi\Big\}
$$
and 
$$
||g||_n^*:=||g||_1+|g|_n^*.
$$
We intend to apply Keller and Liverani (see \cite{KL}) perturbation results. Because of \eqref{1had26.1}, Lemma~\ref{l1fp65} goes through for the norm $||\cdot||_n^*$. We put
$$
\begin{aligned}
\1_{n,q}^k
&:=\prod_{j=0}^{k-1}\1_{\sg^{-j}(U_n^c(q))}
 =\prod_{j=0}^{k-1}\1_{U_n^c(q)}\circ\sg^j, \\
\1_{n,q}^{k,*}
&:=\prod_{j=0}^{k-1}\1_{\sg^{-j}(\N_{l_q})}
 =\prod_{j=0}^{k-1}\1_{\N_{l_q}}\circ\sg^j,
\end{aligned}
$$
and note that
$$
\1_{n,q}^k=\1_n^k\cdot\1_{n,q}^{k,*}.
$$
The proof of Lemma~\ref{l1fp127} goes the same way for the operators $\pf_{n,q}^k$ with only formal change of $\1_n^k$ by $\1_{n,q}^k$ and $U_m$ by $U_m(q)$. It gives:

\blem\label{l1had26}
For every $k\ge 1$ and for every $q\ge n$, we have that
$$
||\pf_{n,q}^k||_n^*\le 1.
$$
\elem

\fr Lemmas~\ref{l2fp127}, \ref{l1fp129}, and Corollary~\ref{c1fp67} used only the (U1) property of the sequence $(U_n)_{n=0}^\infty$, and therefore these apply to the sets $U_n(q)$, $q\ge n$, and the operators $\pf_{n,q}^k$ (to be clear, the role of $n$ is in these three results is now played by the pair $(n,q)$). Fix $a, b>1$ such that $\frac1a+\frac1b=1$. We shall prove the following analogue of Lemma~\ref{l2fp129}.

\blem\label{l1had27}
For every $n\ge 0$ we have 
$$
|||\pf_n-\pf_{n,q}|||\le 2(\rho^{1/b})^q.
$$
\elem

\bpf
Fix an arbitrary $g\in\Ba_\th$ with $||g||_\th\le 1$. Using Lemma~\ref{l1fp61} and \eqref{1had26.1}, we get
\beq\label{1had27}
\begin{aligned}
||(\pf_n-\pf_{n,q})g||_1
&=||\pf(\1_{U_n^c\sms U_n^c(q)}g)||_1
 =||\1_{U_n^c\sms U_n^c(q)}g||_1 
 \le \mu_\phi(U_n^c\sms U_n^c(q))||g||_\infty \\
&=  \mu_\phi(U_n^c\cap U_n(q))||g||_\infty
 =  \mu_\phi\(U_n^c\cap (U_n\cup\N_{l_q}^c)\)||g||_\infty \\
&=  \mu_\phi\(U_n^c\cap\N_{l_q}^c)\)||g||_\infty 
\le \mu_\phi(\N_{l_q}^c)||g||_\th \\
&\le \rho^q||g||_\th 
\le \rho^q 
\le (\rho^{1/b})^q.
\end{aligned}
\eeq
Also, using Cauchy-Schwarz Inequality, we get
$$
\begin{aligned}
    \th^{-m}\int_{\sg^{-j}\(\N_{l_m}^c\)}&\big|(\pf_n-\pf_{n,q})g\big|\,d\mu_\phi = \\
&=  \th^{-m}\int_{E_A^\infty}\1_{\N_{l_m}^c}\circ\sg^j\big|\pf\(\1_{U_n^c\sms  
            U_n^c(q)}g\)\big|\,d\mu_\phi \\
&\le\th^{-m}||g||_\infty\int_{E_A^\infty}\1_{\N_{l_m}^c}\circ\sg^j\pf\(\1_{U_n^c\sms  
            U_n^c(q)}\)\,d\mu_\phi \\
&=   \th^{-m}||g||_\infty\int_{E_A^\infty}\pf\(\1_{\N_{l_m}^c}\circ\sg^{j+1} 
            \1_{U_n^c\sms U_n^c(q)}\)\,d\mu_\phi \\
&=   \th^{-m}||g||_\infty\int_{E_A^\infty}\1_{\N_{l_m}^c}\circ\sg^{j+1} 
            \1_{U_n^c\sms U_n^c(q)}\,d\mu_\phi \\
&=   \th^{-m}||g||_\infty\int_{E_A^\infty}\1_{\N_{l_m}^c}\circ\sg^{j+1} 
            \1_{\N_{l_q}^c}\,d\mu_\phi \\
&=   \th^{-m}||g||_\infty\int_{E_A^\infty}\1_{\sg^{-(j+1)}(\N_{l_m}^c)}
            \1_{\N_{l_q}^c}\,d\mu_\phi \\
&\le \th^{-m}||g||_\th\mu_\phi^{1/a}(\N_{l_m}^c)\mu_\phi^{1/b}(\N_{l_q}^c) \\
&\le \th^{-m}||g||_\th(\rho^{1/a}/\th)^m\rho^{q/b} 
 \le \rho^{q/b}||g||_\th
 \le \rho^{q/b}.
\end{aligned}
$$
Therefore, $\big|(\pf_n-\pf_{n,q})g\big|_n^*\le \rho^{q/b}$, and together with \eqref{1had27}, this completes the proof of our lemma.
\epf

\sp\fr Having all of this, particularly the last lemma, and taking into account the considerations between the end of the proof of Lemma~\ref{l2fp129} and Proposition~\ref{p1fp80}, we get the following analogue of the latter for the  $\pf_n$ replaced by $\pf_{n,q}$.

\blem\label{l1had28}
For all integers $n\ge 0$ large enough and for all $q\ge n$ large enough there exist two bounded linear operators $Q_{n,q},\De_{n,q}:\cB_\th\to\cB_\th$ and complex numbers $\lam_{n,q}\ne 0$ with the following properties:
\begin{itemize}
\item[(a)] $\lam_{n,q}$ is a simple eigenvalue of the operator $\pf_{n,q}:\cB_\th\to\cB_\th$.

\sp\item[(b)] $Q_{n,q}:\cB_\th\to\cB_\th$ is a projector ($Q_{n,2}^2=Q_{n,q}$) onto the $1$--dimensional eigenspace of $\lam_{n,q}$.

\sp\item[(c)] $\pf_{n,q}=\lam_{n,q}Q_{n,q}+\De_{n,q}$.

\sp\item[(d)] $Q_{n,q}\circ\De_{n,q}=\De_{n,q}\circ Q_{n,q}=0$.

\sp\item[(e)] There exist $\ka_n\in (0,1)$ and $C>0$ such that
$$
||\De_{n,n}^k||_\th\le C\ka_n^k.
$$
In particular,
$$
||\De_{n,q}^kg||_\infty\le ||\De_{n,q}^kg||_\th\le C\ka_n^k||g||_\th
$$
for all $g\in\Ba_\th$.

\sp\item[(f)] $\lim_{q\to\infty}\lam_{n,q}=\lam_n$.

\sp\item[(g)] Enlarging the above constant $C>0$ if necessary, we have
$$
||Q_{n,q}||_\th\le C.
$$
In particular, 
$$
||Q_{n,q}g||_\infty\le ||Q_{n,q}g||_\th\le C||g||_\th
$$
for all $g\in\Ba_\th$.

\sp\item[(h)] $\lim_{q\to\infty}|||Q_{n,q}-Q_n|||=0$. 
\end{itemize}
\elem 

\sp\fr The following lemma can be proved in exactly the same way as was  Proposition~\ref{p1fp131}.

\blem\label{l2had28}
All eigenvalues $\lam_{n,q}$ produced in Lemma~\ref{l1had28} are real and positive, and all operators $Q_{n,q}:\cB_\th\to\cB_\th$ preserve $\Ba_\th(\R)$ and $\Ba_\th^+(\R)$, the subspaces of $\Ba_\th$ consisting, respectively, of real--valued functions and positive real--valued functions.
\elem

\brem\label{r1had28.1}
How large $n$ needs to be in Lemmas~\ref{l1had28} and \ref{l2had28} is determined by the requirement that the assertions of Proposition~\ref{p1fp80} hold for such $n$. 
\erem

\sp\fr Now, let us consider the dynamical systems $\sg:\tilde K_n(q)\to \tilde K_n(q)$, where
$$
\tilde K_n(q):=\bi_{k=0}^\infty\sg^{-k}(U_n^c(q))\  \and \ K_n(q):=\pi(\tilde K_n(q)).
$$
We shall prove the following.

\blem\label{l1had28.1}
If $n\ge 0$ is so large as required in Lemma~\ref{l1had28}, then for all $q\ge n$ large enough we have that
$$
\P\(\sg|_{\tilde K_n(q)},\phi|_{\tilde K_n(q)}\)\ge \log\lam_{n,q}.
$$
\elem
\bpf 
A straightforward elementary calculation shows that if $f,g\in\cB_\th$, then 
$||fg||_\th\le 3||f||_\th||g||_\th$; hence in particular $fg\in\cB_\th$. This allows us to define a linear functional $\mu_{n,q}:\cB_\th\to\R$ by the requirement that
$$
Q_{n,q}(gg_{n,q})=\mu_{n,q}(g)g_{n,q}.
$$
Since, by Lemma~\ref{l2had28}, $Q_{n,q}$ is a positive ($Q_{n,q}(\Ba_\th^+(\R))\sbt \Ba_\th^+(\R)$) operator and $Q_{n,q}\ne 0$ all $q\ge n$ large enough, it follows that $\mu_{n,q}$ is a positive $(\mu_{n,q}(\Ba_\th^+(\R))\sbt [0,+\infty)$) functional and
\beq\label{1had28.1}
\mu_{n,q}(\1)=1.
\eeq
Positivity of $\mu_{n,q}$ immediately implies  its monotonicity in the sense that if $f,g\in\cB_\th$ and $f(x)\le g(x)$ $\mu_\phi$-a.e. in $E_A^\infty$, then
\beq\label{4had28.2}
\mu_{n,q}(f)\le \mu_{n,q}(g)
\eeq
Now, let $C_b^u(E_A^\infty)$ be the vector subspace of $C_b(E_A^\infty)$ consisting of all functions that are uniformly continuous with respect to the metric $d_\th$. Let us define a function $\mu:C_b^u(E_A^\infty)\to[0,+\infty)$ by the following formula:
\beq\label{3had28.2}
\mu(g):=\sup\big\{\mu_{n,q}(f):f\le g \ \and\, f\in \H_\th^b(A)\big\}.
\eeq
Of course by \eqref{4had28.2} we get that
\beq\label{1had28.2.1}
\mu|_{\H_\th^b(A)}=\mu_{n,q}|_{\H_\th^b(A)}.
\eeq
Given $g\in C_b^u(E_A^\infty)$ and $k\ge 1$ define two functions
$$
\un g_k(\om):=\inf\{g(\tau):\tau\in[\om|k]\} \  \and  \  
\ov g_k(\om):=\sup\{g(\tau):\tau\in[\om|k]\}.
$$
Of course
$$
\un g_k\le g \le \ov g_k
$$
and
\beq\label{1had28.2}
 \lim_{k\to\infty}||g-\un g_k||_\infty
=\lim_{k\to\infty}||g-\ov g_k||_\infty
=0.
\eeq
We shall prove that for every $g\in C_b^u(E_A^\infty)$ we have that
\beq\label{2had28.2}
\mu(g)= {\overline \mu}(g) := \inf\big\{\mu_{n,q}(f):f\ge g \ \and\, f\in \H_\th^b(A)\big\}.
\eeq
%Indeed, denote this infimum temporarily by $\ov \mu(g)$. 
Then for every $k\ge 1$ we have that
$$
\begin{aligned}
\mu(g)
&\le \mu_{n,q}(\ov g_k)
 =   \mu_{n,q}\(\un g_k+(\ov g_k-\un g_k)\) 
 =   \mu_{n,q}(\un g_k)+\mu_{n,q}(\ov g_k-\un g_k) \\
&\le \mu(g)+\mu(||\ov g_k-\un g_k||_\infty) \\
&=   \mu(g)+||\ov g_k-\un g_k||_\infty,
\end{aligned}
$$
and invoking \eqref{1had28.2}, we obtain $\mu(g)\le\ov\mu(g)\le\mu(g)$, completing the proof of \eqref{2had28.2}. We now can prove the following.

\blem\label{l1had28.2}
The function $\mu:C_b^u(E_A^\infty)\to\R$ is a positive linear functional such that $\mu(\1)=1$ and $\mu|_{\H_\th^b(A)}=\mu_{n,q}|_{\H_\th^b(A)}$.
\elem
\bpf
Positivity is immediate from formula \eqref{3had28.2}. It is also immediate from this formula that
\beq\label{2had28.3}
\mu(\a g)=\a\mu(g)
\eeq
for every $\a\ge 0$. Employing also \eqref{2had28.2}, we get that
$$
\begin{aligned}
\mu(-g)
&=\inf\big\{\mu_{n,q}(f):f\ge-g \ \and\, f\in \H_\th^b(A)\big\} \\
&=\inf\big\{-\mu_{n,q}(-f):-f\le g \ \and\, f\in \H_\th^b(A)\big\} \\
&=\inf\big\{-\mu_{n,q}(f):f\le g \ \and\, f\in \H_\th^b(A)\big\} \\
&=-\sup\big\{\mu_{n,q}(f):f\le g \ \and\, f\in \H_\th^b(A)\big\} \\
&=-\mu(g).
\end{aligned}
$$
Along with \eqref{1had28.3} this implies that
\beq\label{1had28.3}
\mu(\a g)=\a\mu(g)
\eeq
for every $g\in C_b^u(E_A^\infty)$ and all $\a\in\R$. Now fix two functions $f, g\in C_b^u(E_A^\infty)$. Because of \eqref{2had28.2} and \eqref{3had28.2} there exist four sequences $(f_k^-)_1^\infty$, $(f_k^+)_1^\infty$, $(g_k^-)_1^\infty$, and $(g_k^+)_1^\infty$ of elements of $\H_\th^b(A)$ such that
$$
f_k^-\le f\le f_k^+, \  \  \  g_k^-\le g\le g_k^+,
$$
and
$$
 \lim_{k\to\infty}\mu_{n,q}(f_k^-)
=\lim_{k\to\infty}\mu_{n,q}(f_k^+)
=\mu(f) \  \  \and \,  \ 
 \lim_{k\to\infty}\mu_{n,q}(g_k^-)
=\lim_{k\to\infty}\mu_{n,q}(g_k^+)
=\mu(g).
$$
Therefore, applying again \eqref{2had28.2} and \eqref{3had28.2}, we obtain
$$
\mu(f+g)
\ge \varlimsup_{k\to\infty}\mu_{n,q}(f_k^-+g_k^-)
=\lim_{k\to\infty}\mu_{n,q}(f_k^-)+\lim_{k\to\infty}\mu_{n,q}(g_k^-)
=\mu(f)+\mu(g)
$$
and 
$$
\mu(f+g)
\le \varliminf_{k\to\infty}\mu_{n,q}(f_k^++g_k^+)
=\lim_{k\to\infty}\mu_{n,q}(f_k^+)+\lim_{k\to\infty}\mu_{n,q}(g_k^+)
=\mu(f)+\mu(g).
$$
Hence,
$$
\mu(f+g)=\mu(f)+\mu(g),
$$
and along with \eqref{1had28.3} this finishes the proof of Lemma~\ref{l1had28.2} (the last two assertions of this lemma are immediate consequences of \eqref{1had28.1} and \eqref{1had28.2.1}.
\epf

\fr Now we shall prove the following auxiliary fact. 

\blem\label{l1had28.4}
If $g\in C_b^u(E_A^\infty)$ and $g|_{\tilde K_{n,q}}=0$, then $\mu(g)=0$.
\elem

\bpf
Let
$$
\cF_{n,q}=\big\{\om\in E_A^n:[\om]\sbt U_{n,q}^c\big\},
$$
and note that $\cF_{n,q}$ is a finite set. For every $k\ge 1$ let
$$
U_{n,k}^{ck}:=\bi_{j=0}^{k-1}\sg^{-j}(U_n^c).
$$
We shall prove the following.

\sp\fr{\bf Claim 1:} There exists $p\ge 1$ such that if $\om\in E_A^{kn}$ and $[\om]\sbt U_{n,k}^{ck}$, then $[\om|_{kn-pn}]\cap\tilde K_{n,q}\ne\es$.

\bpf
Let 
\beq\label{1had28.4}
p:=\#\cF_{n,q}+1<+\infty.
\eeq
Seeking a contradiction suppose that $k>p$ and 
\beq\label{2had28.4}
[\om|_{(k-p)n}]\cap\tilde K_{n,q}=\es
\eeq
for some $\om\in E_A^{kn}$ with $[\om]\sbt U_{n,q}^{ck}$. Because $|\om|_{(k-p)n+1}^{kn}|=pn$ and because $\om|_{(k-p)n+1}^{kn}$ is a concatenation of non-overlapping blocks from 
$\cF_{n,q}$, it follows from \eqref{1had28.4} that there are two non-overlapping subblocks of $\om|_{(k-p)n+1}^{kn}$ forming the same element of $\cF_{n,q}$. Let $\om|_{(l-1)n+1}^{ln}$, $k-p\le l-1\le k-1$ be the latter of these two blocks, and let the former, denote it by $\tau$, have the last coordinate $j$ ($j\le (l-1)n$). But then the infinite word $\om|_{ln}(\om|_{j+1}^{ln})^\infty$ is an element of $E_A^\infty$ and $\om|_{ln}(\om|_{j+1}^{ln})^\infty$ is a concatenation of non-overlapping blocks from $\cF_{n,q}$. But this means that $\om|_{ln}(\om|_{j+1}^{ln})^\infty\in\tilde K_{n,q}$. Thus, $[\om|_{ln}]\cap \tilde K_{n,q}\ne\es$. As $l\ge k-p$, this contradicts \eqref{2had28.4} and finishes the proof of Claim~1.
\epf

\fr Now passing to the direct proof of our lemma, fix $\e>0$ arbitrary. Since $g|_{\tilde K_{n,q}}=0$ and $g\in C_b^u(E_A^\infty)$, there exists $l\ge 1$ sufficiently large that
\beq\label{1had28.5}
|\textbf{e}|g||_{[\om]}\le \e/2
\eeq
if $|\om|\ge l$ $(\om\in E_A^l$) and $[\om]\cap\tilde K_{n,q}\ne\es$. Take any $k\ge l+p$ so large that $||\ov g_{kn}-g||_\infty \le\e/2$. Employing Claim~1, \eqref{1had28.5}, Lemma~\ref{l1had28.2}, and \eqref{1had28.2.1}, we get
$$
\begin{aligned}
\mu(g)g_{n,q}
&\le \mu(\ov g_{kn})g_{n,q}
 =   \mu_{n,q}(\ov g_{kn})g_{n,q}
 =   Q_{n,q}(\ov g_{kn}g_{n,q})
 =\lam_{n,q}^{-kn}\pf_{n,q}^{kn}Q_{n,q}(\ov g_{kn}g_{n,q}) \\
&=\lam_{n,q}^{-kn}Q_{n,q}\pf_{n,q}^{kn}(\ov g_{kn}g_{n,q}) \\
&=\lam_{n,q}^{-kn}Q_{n,q}\(\tau\mapsto\sum_{[\om]\sbt U_{n,q}^{ck}:A_{\om_{kn}\tau_0}=1}\ov g_{kn}(\om\tau)g_{n,q}(\om\tau)e^{\phi_{kn}(\om\tau)}\) \\
&\le \lam_{n,q}^{-kn}Q_{n,q}\(\tau\mapsto\e\sum_{A_{\om_{kn}\tau_0}=1}\1_n^{kn}(\om\tau)g_{n,q}(\om\tau)e^{\phi_{kn}(\om\tau)}\) \\
&=\e\lam_{n,q}^{-kn}Q_{n,q}\pf_{n,q}^{kn}(g_{n,q})
 =\e Q_{n,q}(g_{n,q})
   \le\e||g_{n,q}||_\infty Q_{n,q}(\1)\\
&= \e||g_{n,q}||_\infty g_{n,q}
\le\e||g_{n,q}||_\th g_{n,q}.
\end{aligned}
$$
Hence,
$$
\mu(g)\le ||g_{n,q}||_\th\e.
$$
Likewise, -$\mu(g)=\mu(-g)\le ||g_{n,q}||_\th\e$, and in consequence.
$$
|\mu(g)|\le ||g_{n,q}||_\th\e.
$$
Letting $\e\downto 0$ we thus get that $\mu(g)=0$ finishing the proof of Lemma~\ref{l1had28.4}
\epf
\fr Since every function $g\in C\(\tilde K_{n,q}\)$ is uniformly continuous, it extends to some uniformly continuous function $\tilde g:E_A^\infty\to\R$. The value
$$
\mu(g):=\mu(\tilde g)
$$
is then, by virtue of, Lemma~\ref{l1had28.4}, independent of the choice of extension $\tilde g\in C_b^u(E_A^\infty)$ of $g$. By Lemma~\ref{l1had28.2}., we get the following.

\blem\label{l1had28.6}
The function $\tilde\mu:C(\tilde K_{n,q})\to\R$ is a positive linear functional such that $\mu(\1)=1$. Thus by the Riesz Representation Theorem $\tilde\mu$ represents a Borel probability measure on $\tilde K_{n,q}$.
\elem

\fr We shall prove the following.

\blem\label{l2had28.6}
The measure $\tilde\mu$ on $\tilde K_{n,q}$ is $\sg$-invariant.
\elem

\bpf
Let $g\in C\(\tilde K_{n,q}\)$. Let $\tilde g\in C_b^u(E_A^\infty)$ be an extension of $g$. Then $\tilde g\circ\sg\in C_b^u(E_A^\infty)$ and it extends $g\circ\sg$. Fix $\e>0$ and take $\tilde g_+$ and $\tilde g_-$ both in $\H_\th^b(A)$, such that $\tilde g_-\le\tilde g\le\tilde g_+$ and
$$
\mu_{n,q}(\tilde g_+)-\e\le\mu(\tilde g)\le \mu_{n,q}(\tilde g_-)+\e.
$$
Of course then we also have $\tilde g_+\circ\sg, \tilde g_-\circ\sg\in \H_\th^b(A)$ and $\tilde g_-\circ\sg\le\tilde g\circ\sg\le\tilde g_+\circ\sg$. We thus get
$$
\begin{aligned}
\tilde\mu(g\circ\sg)g_{n,q}
&=  \mu(\tilde g\circ\sg)g_{n,q}
\le \mu(\tilde g_+\circ\sg)g_{n,q}
 =  Q_{n,q}(g_{n,q}\tilde g_+\circ\sg) \\
&=  \lam_{n,q}^{-kn}\pf_{n,q}^{kn}Q_{n,q}(g_{n,q}\tilde g_+\circ\sg) \\
&=  \lam_{n,q}^{-kn}Q_{n,q}\pf_{n,q}^{kn}(g_{n,q}\tilde g_+\circ\sg) \\
&=  \lam_{n,q}^{-kn}Q_{n,q}\(\tilde g_+\pf_{n,q}^{kn}(g_{n,q})\) \\
&=  Q_{n,q}(\tilde g_+g_{n,q}) \\
&=  \mu_{n,q}(\tilde g_+)g_{n,q} \\
&\le(\mu(\tilde g)+\e)g_{n,q}
 =  (\mu(\tilde g)+\e)g_{n,q}.
\end{aligned}
$$
Hence, $\tilde\mu(g\circ\sg)\le \mu(\tilde g)+\e$. By letting $\e\downto 0$ this yields 
$\tilde\mu(g\circ\sg)\le \mu(\tilde g)$. Likewise, working with $\tilde g_-$ instead of 
$\tilde g_+$, we get $\tilde\mu(g\circ\sg)\ge \mu(\tilde g)$. Thus $\tilde\mu(g\circ\sg)=\mu(\tilde g)$ and the proof is complete.
\epf

\fr We now pass to the direct proof of the inequality being the assertion of  Lemma~\ref{l1had28.1}. Given any $\om\in \tilde K_{n,q}$, we have for every $k\ge 1$ that
$$
\begin{aligned}
\tilde\mu\(\1_{[\om|_{kn}]}\)g_{n,q}
&=  \mu_{n,q}\(\1_{[\om|_{kn}]}\)g_{n,q}
 =  Q_{n,q}\(\1_{[\om|_{kn}]}g_{n,q}\) \\
&=  \lam_{n,q}^{-kn}\pf_{n,q}^{kn}Q_{n,q}\(\1_{[\om|_{kn}]}g_{n,q}\) 
 =  \lam_{n,q}^{-kn}Q_{n,q}\pf_{n,q}^{kn}\(\1_{[\om|_{kn}]}g_{n,q}\) \\
&=  \lam_{n,q}^{-kn}Q_{n,q}\(\tau\mapsto e^{\phi_{kn}(\om|_{kn}\tau)}
         g_{n,q}(\om|_{kn}\tau)\).
\end{aligned}
$$
Now, because of Lemma~\ref{l1_2014_08_29}, $M_\phi^{-1}e^{\phi_{kn}(\om|_{kn}\tau)} 
\le e^{\phi_{kn}(\om)} \le M_\phi e^{\phi_{kn}(\om|_{kn}\tau)}$. Therefore, the monotonicity of $Q_{n,q}$ yields
$$
\begin{aligned}
M_\phi^{-1}\lam_{n,q}^{-kn}e^{\phi_{kn}(\om)}
  &Q_{n,q}\(\tau\mapsto g_{n,q}(\om|_{kn}\tau)\)  \\
&\le \tilde\mu\(\1_{[\om|_{kn}]}g_{n,q}\) \\
\le M_\phi\lam_{n,q}^{-kn} &e^{\phi_{kn}(\om)}
Q_{n,q}\(\tau\mapsto g_{n,q}(\om|_{kn}\tau)\).
\end{aligned}
$$
We are only interested in the right-hand side of this inequality. We further have
$$
\tilde\mu\(\1_{[\om|_{kn}]}g_{n,q}\)g_{n,q}
\le M_\phi|| g_{n,q}||_\infty\lam_{n,q}^{-kn} e^{\phi_{kn}(\om)}Q_{n,q}(\1)
=M_\phi||g_{n,q}||_\infty\lam_{n,q}^{-kn} e^{\phi_{kn}(\om)}g_{n,q}.
$$
Hence
$$
\tilde\mu\(\1_{[\om|_{kn}]}g_{n,q}\) 
\le M_\phi||g_{n,q}||_\infty\lam_{n,q}^{-kn} e^{\phi_{kn}(\om)}.
$$
Denoting by $\a$ the partition of $E_A^\infty$ into cylinders of length one, i. e. the partition $\{[e]\}_{e\in E}$, we therefore get
$$
\begin{aligned}
\H_{\tilde\mu}(\a^{kn})+kn\tilde\mu(\phi)
&= \H_{\tilde\mu}(\a^k)+\tilde\mu(\phi_{kn})
 = \int_{\tilde K_{n,q}}\!\!\!\!\!\!
   -\log\tilde\mu([\om|_k])\,d\tilde\mu(\om)+\tilde\mu(\phi_{kn}) \\
&\ge -\log\(M_\phi||g_{n,q}||_\infty\)+kn\log\lam_{n,q}-\tilde\mu(\phi_{kn})  
     +\tilde\mu(\phi_{kn}) \\
&=-\log\(M_\phi||g_{n,q}||_\infty\)+kn\log\lam_{n,q}.
\end{aligned}
$$
Since $\a|_{\tilde K_{n,q}}$ is a finite generating partition, we thus get that
$$
\h_{\tilde\mu}(\sg)+\tilde\mu(\phi) 
=\lim_{k\to\infty}\frac1k\(\H_{\tilde\mu}(\a^{kn})+kn\tilde\mu(\phi)\)
\ge \log\lam_{n,q}.
$$
Hence, invoking the Variational Principle we get
$$
\P\(\sg|_{\tilde K_n(q)},\phi|_{\tilde K_n(q)}\)\ge \log\lam_{n,q}.
$$
and the proof of Lemma~\ref{l1had28.1} is complete.
\epf

Aiming now directly towards proving the inequality
\beq\label{1had29}
\HD(K_n)\ge b_n
\eeq
for all $n\ge 1$ large enough, we apply Proposition~\ref{p1had9.3} with $s=b(=b_\cS)$. Let $n_b\ge 1$ be the integer produced in this proposition, let $\d>0$ be the minimum of both $\d$s, the one produced in Proposition~\ref{p1had9.3} and the one coming from Lemma~\ref{l1had17} (b). Let $N_b\ge n_b$ be so large (depending only on $s=b$) that the assertions of Lemma~\ref{l1had17} are true for all $n\ge N_b$. By Proposition~\ref{p1had18} there exists $N_b^*\ge N_b$ so large that $b_n\in(b-\d,b+\d)$ for all $n\ge N_b^*$. Take an arbitrary integer $n$ with this property, i. e. $n\ge N_b^*$. Fix any $t\in (b-\d,b_n)$. By Lemma~\ref{l1had17} (b) and (e), we have that $\lam_n(t)>1$. Since $n\ge N_b^*\ge n_b$, it follows from Proposition~\ref{p1had9.3} that the assertions of Proposition~\ref{p1fp80} hold for this $n$. In turn, it therefore follows from Remark~\ref{r1had28.1} that Lemma~\ref{l1had28} holds for this $n$. Its item (f) yields some $q\ge n$ such that $\lam_{n,q}>1$ By virtue of Lemma~\ref{l1had28.1} and the Variational Principle, for all $q\ge n$ large enough there exists a Borel probability $\sg$-invariant measure $\mu$ on $\tilde K_n(q)$ such that $\hmu(\sg)-t\chi_\mu>0$. But then invoking Theorem~4.4.2 from \cite{GDMS}, we get that
$$
\HD(K_n)\ge \HD(K_{n,q})\ge\HD(\mu\circ\pi^{-1})=\frac{\hmu(\sg)}{\chi_\mu}>t.
$$
So, letting $t\upto b_n$, inequality \eqref{1had29} follows. Along with \eqref{1had21} this completes the proof of Theorem~\ref{t1had19}.
\epf

As a direct consequence of this theorem and Proposition~\ref{p1had18}, we get the following.

\bprop\label{p1had30}
With the hypotheses of Theorem~\ref{t1had19} we have that
\beq\label{1had30}
\lim_{n\to\infty}\frac{\HD(J_\cS)-\HD(K_n)}{\mu_b(U_n)}=
\begin{cases}
1/\chi_{\mu_b} \  &\text{ if } \ (U4A) \  \text{ holds } \\
\(1-|\phi_\xi'(\pi(\xi^\infty))|\)/\chi_{\mu_b} &\text{ if } \ (U4B) \  \text{ holds }.
\end{cases}
\eeq
\eprop

\section{Escape Rates for Conformal GDMSs; Hausdorff Dimension}\label{ERCGDMSHD}

\fr This mini-section is the main fruit of the labor in the previous section. It pertains to the rate of decay of Hausdorff dimension of escaping points. It contains,  in particular, Theorem~\ref{t2had30}, the second main result of this manuscript. Given $z\in E_A^\infty$ and $r>0$ let 
$$
K_z(r):=\pi(\tilde K_z(r)),
$$
where 
$$
\tilde K_z(r)
:=\big\{\om\in E_A^\infty:\forall_{n\ge 0}\, \sg^n(\om)\notin \pi^{-1}(B(\pi(z),r)\big\}
 =\bi_{n=0}^\infty\sg^{-n}\(\pi^{-1}(B(\pi(z),r))\).
$$
We say that a parameter $t>\th_\cS$ is powering at a point $\xi\in X$ if there exist $\a>0$, $C>0$, and $\d>0$ such that
\beq\label{1uwbt2}
\mu_s\circ\pi^{-1}\(B(\xi,r)\)\le C\(\mu_t\circ\pi^{-1}\(B(\xi,r)\)\)^\a
\eeq
for every $s\in(t-\d,t+\d)$ and for all radii $r>0$ small enough. The constant $\a$ is called the powering exponent of $t$ and $\xi$. The following is one of the main results of our paper. 

\bthm\label{t2had30}
Let $\cS$ be a finitely primitive strongly regular conformal GDMS. Assume that both $\cS$ is (\WBT) and parameter $b_\cS$ is powering at some point $z\in J_\cS$ which is either

\begin{itemize}
\item[(a)] not pseudo-periodic
or else 
\item [(b)] uniquely periodic and belongs to $\Int X$ (and $z=\pi(\xi^\infty)$ for a (unique) irreducible word $\xi\in E_A^*$).
\end{itemize}
Then 
\beq\label{1had31}
\lim_{r\to 0}\frac{\HD(J_\cS)-\HD(K_z(r))}{\mu_b\(\pi^{-1}(B(z,r))\)}=
\begin{cases}
1/\chi_{\mu_b} \  &\text{ if } \ (a) \  \text{ holds } \\
\(1-|\phi_\xi'(z)|\)/\chi_{\mu_b}
\  &\text{ if } \ (b) \  \text{ holds }.
\end{cases}
\eeq
\ethm

\bpf
Denote the right-hand side of \eqref{1had30} by $\xi(z)$. Put
$$
h:=\HD(J_\cS)=b_\cS \  \   \   \and \  \  \  h_r:=\HD(K_z(r)).
$$
Seeking contradiction assume that \eqref{1had31} fails to hold at some point $z\in E_A^\infty$. This means that there exists a strictly decreasing sequence $(s_n(z))_{n=0}^\infty$ of positive reals such that the sequence
$$
\lt(\frac{h-h_{s_n(z)}}{\mu_b\(\pi^{-1}(B(\pi(z),s_n(z)))\)}\rt)_{n=0}^\infty
$$
does not have $\xi(z)$ as its accumulation point. Let
$$
\cR:=\{s_n(z):n\ge 0\}.
$$
Let $(U_n^{\pm}(z))_{n=0}^\infty$ be the corresponding sequence of open subsets of $E_A^\infty$ produced in formula \eqref{1fp163}. We shall prove the following.

\sp{\bf Claim~$1^0$:} Both sequences $(U_n^{\pm}(z))_{n=0}^\infty$ satisfy the (U2*) condition for the parameter $h$. 

\bpf
Let $\a>0$ be a powering exponent of $h=b_\cS$ at $z$ and let $\d>0$ come from this powering property. Let $s\in(h-\d,h+\d)$. Applying then formula \eqref{1fp165} to the measure $\mu_h$, we get, with notation used in this formula, that
$$
\mu_s\(U_k^{\pm}(z)\)
\le \mu_s\circ\pi^{-1}(B(z,r_{j-1}))
\le C\(\mu_h\circ\pi^{-1}(B(z,r_{j-1}))\)^\a
\le C\exp^\a\(\ka(1+8\De l(z))e^{-\a\ka k}.
$$
The claim is proved.
\epf 

\sp\fr By this claim and because of Propositions~\ref{p1fp163} and \ref{p1fp165}, Proposition~\ref{p1had30} applies to give
\beq\label{1had31+1}
\lim_{n\to\infty}\frac{h-h_n^{\pm}}{\mu_b(U_n^{\pm}(z))}=\xi(z),
\eeq
where $h_n^{\pm}:=\HD(K(U_n^{\pm}(z))$. Let $(n_j)_{j=0}^\infty$ be the sequence produced in Proposition~\ref{p2fp161} with the help of $\cR$. By virtue of this proposition there exists an increasing sequence $(j_k)_{k=1}^\infty$ such that $\cR\cap\cR_{n_{j_k}}\ne\es$ for all $k\ge 1$. For every $k\ge 1$ pick one element $r_k\in \cR\cap\cR_{n_{j_k}}$. Set $q_k:=l_{n_{j_k}}$. By Observation~\ref{o1fp159} and formula \eqref{2fp163}, we have
\beq\label{2had31}
\begin{aligned}
\frac{h-h_{q_k}^-}{\mu_b(U_{q_k}^-(z))}
\cdot 
\frac{\mu_b(U_{q_k}^-(z))}{\mu_b\(\pi^{-1}(B(\pi(z),r_k))\)} 
&\le 
\frac{h-h_{r_k}}{\mu_b\(\pi^{-1}(B(\pi(z),r_k))\)}\le \\
&\le
\frac{h-h_{q_k}^+}{\mu_b(U_{q_k}^+(z))}
\cdot 
\frac{\mu_b(U_{q_k}^+(z))}{\mu_b\(\pi^{-1}(B(\pi(z),r_k))\)}
\end{aligned}
\eeq
But since $\mu_b\circ\pi^{-1}$ is WBT, it is DBT by Proposition~\ref{p1fp132ec}, and it therefore follows from \eqref{1fp132ec} along with formulas \eqref{1fp159} and \eqref{2fp163} that 
$$
\lim_{k\to\infty}\frac{\mu_b(U_{q_k}^-(z))}{\mu_b\(\pi^{-1}(B(\pi(z),r_k))\)}
=1=
\lim_{k\to\infty}\frac{\mu_b(U_{q_k}^+(z))}{\mu_b\(\pi^{-1}(B(\pi(z),r_k))\)}.
$$
Inserting this and \eqref{1had31+1} to \eqref{2had31} yields
$$
\lim_{k\to\infty}\frac{h-h_{r_k}}{\mu_b\(\pi^{-1}(B(\pi(z),r_k))\)}=\xi(z).
$$
Since $r_k\in\cR$ for all $k\ge 1$, this implies that $\xi(z)$ is an accumulation point of the sequence`
$$
\lt(\frac{h-h_{s_n(z)}}{\mu_b\(\pi^{-1}(B(\pi(z),s_n(z)))\)}\rt)_{n=0}^\infty,
$$
and this contradiction finishes the proof of Theorem~\ref{t2had30}.
\epf

\sp\fr We have discussed at length the (WBT) condition in Section~\ref{Section: WBT}, particularly in Theorem~\ref{t1wbt6}; we now would like also to note that since any two measures $\mu_t$, $t>\th_\cS$, are either equal or mutually singular, the standard covering argument gives the following simple but remarkable result.

\bprop\label{p1_2016_05_27}
If $\cS$ is a finitely primitive regular conformal {\rm GDMS}, then every parameter $t>\th_\cS$ is powering with exponent $1$ at $\mu_t\circ\pi^{-1}$--a.e. point of $J_\cS$.
\eprop

\sp\fr Now, as an immediate consequence of Theorem~\ref{t2had30}, Theorem~\ref{t1wbt6}, and Proposition~\ref{p1_2016_05_27}, we get the following result, also one of our main.

\bcor\label{t2_2016_05_27}
If $\cS$ be a finitely primitive strongly regular conformal {\rm GDMS} whose limit set $ J_{\mathcal S}$ is geometrically irreducible, then 
\beq\label{1had31+2}
\lim_{r\to 0}\frac{\HD(J_\cS)-\HD(K_z(r))}{\mu_b\(\pi^{-1}(B(z,r))\)}=
\frac1{\chi_{\mu_b}}
\eeq
at $\mu_{b_\cS}\circ\pi^{-1}$--a.e. point $z$ of $J_\cS$.
\ecor

\fr In the case of finite alphabet $E$, we can say much more for the parameter $b_\cS$ than established in Proposition~\ref{p1_2016_05_27}. Namely, we shall prove the following.

\bprop\label{p1cl2}
If $\cS$ is a finite alphabet primitive conformal {\rm GDMS}, then $\cS$ is powering at the parameter $b_\cS$ at each point $\xi\in J_\cS$.
\eprop

\bpf
The proof of Theorem~7.20 in \cite{CTU} (see also Theorem~7.17 therein for the main geometric ingredient of this proof) produces for every radius $r\in\(0,\frac12\min\{\diam(X_v):v\in V\}$ a family $Z(r)\sbt E_A^*$ consisting of mutually incomparable words with the following properties.

\sp\begin{itemize}
\item [(1)] $C_1^{-1}r\le \|\phi_\om'\|_\infty, \diam\(\phi_\om(X_{t(\om)})\)\le C_1r$\, for all $\om\in Z(r)$

\sp \item [(2)] $\phi_\om(X_{t(\om)})\cap B(\xi,r)\ne\es$\, for all $\om\in Z(r)$

\sp \item [(3)] $\pi_\cS^{-1}(B(\xi,r))\sbt \bu_{\om\in Z(r)}[\om]$,

\sp \item [(4)] $\#Z(r)\le C_2$,
\end{itemize}

\sp\fr where $C_1$ and $C_2$ are some finite positive constants  independent of $\xi$ and $r$. Abbreviate 
$$
b:=b_\cS.
$$
It easily follows from \cite{GDMS} that there exist a constant $\d\in (0,b_\cS/4)$ and a constant $Q\in (1,+\infty)$ such that
$$
Q^{-1}\le \frac{\mu_s([\tau])}{e^{-\P(s)|\tau|}\|\phi_\tau'\|_\infty^s}\le Q
$$
for every $s\in (b-\d,b+\d)$ and for all $\tau\in E_A^*$. We therefore get for every $s\in (b-\d,b+\d)$ and all $\om\in Z(r)$ that
\beq\label{1cl2}
\mu_s([\om])
\le Qe^{-\P(s)|\om|}\|\phi_\om'\|_\infty^s
\le QC_1^se^{-\P(s)|\om|}r^s
\eeq
and 
\beq\label{2cl2}
\mu_b([\om])
\ge QC_1^{-b}r^b.
\eeq
It is also known from \cite{CTU} that, with perhaps larger $Q\ge 1$:
\beq\label{3cl2}
\mu_b\circ\pi_\cS^{-1}\(B(\xi,r)\)\ge Q^{-1}r^b
\eeq
This formula follows for example from \eqref{2cl2} applied to a sufficiently small fixed fraction of $r$. If $b/2\le s\le b$, then $\P(s)\ge 0$, and we get
\beq\label{1cl4}
\begin{aligned}
\mu_s([\om])
&\le QC_1^sr^s
\le QC_1^br^s
= QC_1^b\(r^b\)^{s/b} \\
&\le QC_1^bQ^{\frac{s}{b}}\mu_b^{\frac{s}{b}}\circ\pi_\cS^{-1}\(B(\xi,r)\) \\
&\le Q^2C_1^b\mu_b^{\frac{s}{b}}\circ\pi_\cS^{-1}\(B(\xi,r)\) \\
&\le Q^2C_1^b\mu_b^{\frac12}\circ\pi_\cS^{-1}\(B(\xi,r)\).
\end{aligned}
\eeq
Now we assume that $s\ge b$. We set
$$
\ka:=\max\{\|\phi_e'\|_\infty:e\in E\}<1,
$$
and we recal that
$$
\chi_b:=\chi_{\mu_b}=-\int_{E_A^\infty}\log|\phi_{\om_0}'(\pi_\cS(\sg(\om)))|\,d\mu_b(\om)>0.
$$
By taking $\d\in(0,b/4)$ small enough, we will have
$$
\frac{s-\frac{b}{2}}{s-b}\ge \frac{2\chi_b}{log(1/\ka)} 
\, \, \and \, \, 
\P(s)\ge -2\chi_b(s-b)
$$
for all $s\in (b,b+\d)$. Hence
$$
\lt(s-\frac{b}{2}\rt)\log\ka\le -2\chi_b(s-b)\le \P(s).
$$
Equivalently $\ka^{\lt(s-\frac{b}{2}\rt)}\le e^{\P(s)}$. Thus
$$
\ka^{\lt(s-\frac{b}{2}\rt)|\om|}\le e^{\P(s)|\om|}.
$$
As $\|\phi_\om'\|_\infty\le \ka^{\om|}$ and $s\ge b$, we therefore get
$$
\begin{aligned}
\mu_s([\om])
&\le Qe^{-\P(s)|\om|}\|\phi_\om'\|_\infty^s
\le Q\|\phi_\om'\|_\infty^{\frac{b}{2}}
\le QC_1^{\frac{b}{2}}r^{\frac{b}{2}} \\
&\le Q^2C_1^{\frac{b}{2}}Q^{\frac{b}{2}}\mu_b^{\frac12}\circ\pi_\cS^{-1}\(B(\xi,r)\)\\
&=   Q^{3/2}C_1^{\frac{b}{2}}\mu_b^{\frac12}\circ\pi_\cS^{-1}\(B(\xi,r)\) \\
&\le Q^2C_1^b\mu_b^{\frac12}\circ\pi_\cS^{-1}\(B(\xi,r)\).
\end{aligned}
$$
Combining this along with \eqref{1cl4} we get that 
$$
\mu_s([\om])\le Q^2C_1^b\mu_b^{\frac12}\circ\pi_\cS^{-1}\(B(\xi,r)\).
$$
for all $s\in (b-\d,b+\d)$ and all $\om\in Z(r)$. Thus, looking also up at (4) and (3), this yields
$$
\mu_s\circ\pi_\cS^{-1}\(B(\xi,r)\)\le C_2Q^2C_1^b\mu_b^{\frac12}\circ\pi_\cS^{-1}\(B(\xi,r)\)
$$ 
for all $s\in (b-\d,b+\d)$ and all radii $r\in\(0,\frac12\min\{\diam(X_v):v\in V\}\)$. The proof of Proposition~\ref{p1cl2} is complete.
\epf

\fr As an immediate consequence of Theorem~\ref{t2had30}, Theorem~\ref{t2wbt11}, and Proposition~\ref{p1cl2}, we get the following considerably stronger/fuller result.

\sp
\bthm\label{t1fp83_Finite_B} 
Let $\cS=\{\phi_e\}_{e\in E}$ be a primitive Conformal Graph Directed Markov System with a finite alphabet $E$ acting in the space $\R^d$, $d\ge 1$. Assume that either $d=1$ or that the system $\cS$ is geometrically irreducible. Let $z\in J_\cS$ be arbitrary. If either $z$ is 

\begin{itemize}
\item[(a)] not pseudo-periodic
or else 
\item [(b)] uniquely periodic and belongs to $\Int X$ (and $z=\pi(\xi^\infty)$ for a (unique) irreducible word $\xi\in E_A^*$).
\end{itemize}
Then 
\beq\label{1had31}
\lim_{r\to 0}\frac{\HD(J_\cS)-\HD(K_z(r))}{\mu_b\(\pi^{-1}(B(z,r))\)}=
\begin{cases}
1/\chi_{\mu_b} \  &\text{ if } \ (a) \  \text{ holds } \\
\(1-|\phi_\xi'(z)|\)/\chi_{\mu_b}
\  &\text{ if } \ (b) \  \text{ holds }.
\end{cases}
\eeq
\ethm

\section{Escape Rates for Conformal Parabolic GDMSs} 

In this section, following  
%\cite{MU_Parabolic_1} 
\cite{MU_Parabolic_1}
and \cite{GDMS}, we first shall provide the appropriate setting and basic properties of conformal parabolic iterated function systems, and more generally of parabolic graph directed Markov systems. We then prove for them the appropriate theorems on escaping rates.

\sp\fr As in Section~\ref{Attracting_GDMS_Prel} there are given a directed multigraph $(V,E,i,t)$ ($E$ countable, $V$ finite), an incidence matrix $A:E\times E\to \{0,1\}$, and two functions $i,t:E\to V$ such that $A_{ab} = 1$ implies $t(b) = i(a)$. Also, we have nonempty compact metric spaces $\{X_v\}_{v\in V}$. Suppose further that we have  a collection of
conformal maps $\phi_e:X_{t(e)}\to X_{i(e)}$, $e\in E$, satisfying the following conditions:

\sp\begin{itemize}
\item[(1)](Open Set Condition)
 $\phi_i(\Int(X))\cap \phi_j(\Int(X))=\es$ for all $i\ne j$.

\sp\item[(2)] $|\phi_i'(x)|<1$ everywhere except for finitely many
pairs $(i,x_i)$, $i\in E$, for which $x_i$ is the unique fixed point
of $\phi_i$ and $|\phi_i'(x_i)|
=1$. Such pairs and indices $i$ will be called parabolic and the set of
parabolic indices will be denoted by $\Om$. All other indices will be 
called hyperbolic. We assume that $A_{ii}=1$ for all $i\in\Om$.

\sp\item[(3)]  $\forall n\ge 1 \  \forall \om = (\om_1,...,\om_n)\in E^n$
if $\om_n$ is a hyperbolic
index or $\om_{n-1}\ne \om_n$, then $\phi_{\om}$ extends conformally to
an open connected set $W_{t(\om_n)}\sbt\R^d$ and maps $W_{t(\om_n)}$ into $W_{i(\om_n)}$.

\sp\item[(4)] If $i$ is a parabolic index, then $\bi_{n\ge
0}\phi_{i^n}(X)
=\{x_i\}$ and the diameters of the sets $\phi_{i^n}(X)$ converge
to 0.

\sp\item[(5)] (Bounded Distortion Property) $\exists K\ge 1 \  \forall
n\ge 1
 \  \forall \om = (\om_1,...,\om_n)\in I^n  \  \forall x,y\in V$ if
$\om_n$ 
is a hyperbolic
index or $\om_{n-1}\ne \om_n$, then
$$
{|\phi_\om'(y)| \over |\phi_\om'(x)| } \le K.
$$
\item[(6)] $\exists s<1 \  \forall n\ge 1  \  \forall \om\in E_A^n$ if 
$\om_n$ is a hyperbolic index or $\om_{n-1}\ne \om_n$, then
$||\phi_\om'||\le s$.

\sp\item[(7)] (Cone Condition)  There exist $\a,l>0$ such that for
 every $x\in\bd X \sbt\R^d$ there exists an open cone 
$\Con(x,\a,l)\sbt \Int(X)$ with vertex $x$, central
angle of Lebesgue measure $\a$, and altitude $l$.

\sp\item[(8)] There exists a constant $L\ge 1$ such that
$$
\bigl||\f_i'(y)|-|\f_i'(x)|\bigr|\le L||\f_i'|||y-x|
$$ 
for every $i\in I$ and every pair of points $x,y\in V$.
\end{itemize}

\sp\fr We call such a system of maps 
$$
\cS=\{\phi_i:i\in E\}
$$ 
a subparabolic
iterated function system. Let us note that conditions (1),(3),(5)-(7)
are modeled on similar  conditions which were used to examine hyperbolic
conformal systems. If $\Om\ne\es$, we call the system  
$\{\phi_i:i\in E\}$ parabolic. As
declared in (2) the elements of the set $E\sms \Om$ are called
hyperbolic.
We extend this name to all the words appearing in (5) and (6). It follows
from (3) that for every hyperbolic word $\om$,
$$
\phi_\om(W_{t(\om)})\sbt W_{t(\om)}.
$$
Note that our conditions ensure that $\f_i'(x) \neq 0$ 
for all $i\in E$ and all $x \in X_{t(i)}$. It was proved (though only for IFSs but the case of GDMSs can be treated completely similarly) in 
\cite{MU_Parabolic_1} (comp. \cite{GDMS}) that
\beq\label{1_2016_03_15}
\lim_{n\to\infty}\sup_{|\om|=n}\{\diam(\phi_\om(X_{t(\om)}))\}=0.
\eeq
As its immediate consequence, we record the following.

\bcor\lab{p1c2.3} 
The map $\pi:E_A^\infty\to X:=\du_{v\in V}X_v$, $\{\pi(\om)\}:=
\bi_{n\ge 0}\phi_{\om|_n}(X)$, is well defined, i.e. this intersection is always a singleton, and the map $\pi$ is uniformly continuous.
\ecor

\fr As for hyperbolic (attracting) systems the limit set $J = J_\cS$ of the system $\cS = \{\f_e\}_{e\in e}$ is defined to be
$$
J_\cS:=\pi(E_A^\infty)
$$
and it enjoys the following self-reproducing property:
$$
J = \bu_{e\in E} \f_e(J).
$$
We now, following still \cite{MU-JNT} and \cite{GDMS}, want to associate to the parabolic system $\cS$ a canonical hyperbolic system $\cS^*$. The set of edges is this.
$$
E_*:= \big\{i^nj: n\ge 1, \  i\in\Om, \ i\ne j\in E, \ A_{ij}=1\big\} \cup 
(E\sms \Om)\sbt E_A^*.
$$ 
We set
$$
V_*=t(E_*)\cup i(E_*)
$$
and keep the functions $t$ and $i$ on $E_*$ as the restrictions of $t$ and $i$ from $E_A^*$. The incidence matrix $A_*:E_*\times E_*\to\{0,1\}$ is defined in the natural (the only reasonable) way by declaring that $A^*_{ab}=1$ if and only if $ab\in E_A^*$. Finally 
$$
\cS^*:=\{\phi_e:X_{t(e)}\to X_{t(e)}:\, e\in E_*\}.
$$
It immediately follows from our assumptions (see \cite{MU-JNT} and \cite{GDMS} for details) that the following is true.

\bthm\lab{p1t5.2} 
The system $\cS^*$ is a hyperbolic conformal GDMS and the limit sets $J_\cS$ and $J_{\cS^*}$ differ only by a countable set.
\ethm

\fr We have the following quantitative result, whose complete proof can be found in \cite{SUZ_I}. 

\bprop\lab{p1c5.13} 
Let $\cS$ be a conformal parabolic GDMS. Then there exists a constant $C\in(0,+\infty)$ and for every $i\in\Om$ there exists some constant
$\beta_i\in(0,+\infty) $ such that for all $n\ge 1$ and for all $z\in X_i:=
\bu_{j\in I\sms\{i\}}\phi_j(X)$,
$$
C^{-1}n^{-{\beta_i+1\over \beta_i}}\le |\phi_{i^n}'(z)|
\le Cn^{-{\beta_i+1\over \beta_i}}.
$$
In fact we know more: if $d=2$ then all constants $\b_i$ are integers $\ge 1$ and if $d\ge 3$ then all constants $\b_i$ are equal to $1$. 
\eprop

\fr Let 
$$
\b=\b_\cS:=\min\{\b_i:\in\in\Om\}
$$
Passing to equilibrium/Gibbs states and their escape rates, we now describe the class of potentials we want to deal with. This class is somewhat narrow as we restrict ourselves to geometric potentials only. There is no obvious natural larger class of potentials for which our methods would work and trying to identified such classes would be of dubious value and unclear benefits. We thus only consider potentials of the form
$$
E_A^\infty\ni\om\mapsto\zeta_t(\om)
:=t\log\big|\phi_{\om_0}'(\pi_\cS(\sg(\om)))\big|\in\R, \  \  t\ge 0.
$$
We then define the potential $\zeta_t^*:E_{*A^*}^\infty\to\R$ as
$$
\zeta_t^*(i^nj\om)=\sum_{k=0}^n\zeta_t(\sg^k(i^nj\om)), \  \  \ i\in\Om, \  n\ge 0, \  j\ne i \  \and  \  i^nj\om\in  E_{*A^*}^\infty.
$$
We shall prove the following.

\bprop\label{p1ps1}
If $\cS$ is a conformal parabolic GDMS, then the potential $\zeta_t^*$ is H\"older continuous for each $t\ge 0$ it is summable if and only if 
$$
t>\frac{\b}{\b+1}
$$
\eprop

\bpf
H\"older continuity of potentials $\zeta_t^*$, $t\ge 0$, follows from the fact that the system $\cS^*$ is hyperbolic, particularly from its distortion property, while the summability statement immediately follows from Proposition~\ref{p1c5.13}.
\epf

\sp\fr So, for every $t>\frac{\b}{\b+1}$ we can define $\mu_t^*$ to be the unique equilibrium/Gibbs state for the potential $\zeta_t^*$ with respect to the shift map $\sg_*:
E_{*A^*}^\infty\to E_{*A^*}^\infty$. We will not use this information in the current paper but we would like to note that $\mu_t^*$ gives rise to a Borel $\sg$-finite, unique up to multiplicative constant, $\sg$-invariant measure $\mu_t$ on $E_A^\infty$, absolutely continuous, in fact equivalent, with respect to $\mu_t^*$; see \cite{GDMS} for details in the case of $t=b_\cS=b_{\cS^*}$, the Bowen's parameter of the systems $\cS$ and $\cS^*$ alike. The case of all other $t>\frac{\b}{\b+1}$ can be treated similarly. It follows from
% Theorem~??? in 
\cite{GDMS} that the measure $\mu_t$ is finite if and only if either

\sp\begin{itemize}
\item[(a)] $t\in\lt(\frac{\b}{\b+1},b_\cS\rt)$ or

\sp\item[(b)] $t=b_\cS$ \and $b_\cS>\frac{2\b}{\b+1}$. 
\end{itemize}

\sp\fr Now having all of this, as an immediate consequence of theorems Theorem~\ref{t1fp83} 
and Theorem~\ref{t3_2016_05_27} we get the following two results.

\bthm\label{t1_2016_05_25} 
Let $\cS=\{\phi_e\}_{e\in E}$ be a parabolic Conformal Graph Directed Markov System. Fix  $t>\frac{\b}{\b+1}$ and assume that the measure $\mu_t^*\circ\pi^{-1}_{\cS^*}$ is {\rm (WBT)} at a point $z\in J_{\cS^*}$. If $z$ is either 

\begin{itemize}
\item[(a)] not pseudo-periodic with respect to the system $\cS^*$, 

or 

\item[(b)] uniquely periodic with respect to $\cS^*$, it belongs to $\Int X$ (and $z=\pi_{\cS^*}(\xi^\infty)$ for a (unique) irreducible word $\xi\in E_{*A^*}^*$),
\end{itemize}

then, with $\un R_{\cS^*,\mu_t^*}(B(z,\ep):=\un R_{\mu_t^*}\(\pi_{\cS^*}^{-1}(B(z,\ep))\)$ and $\ov R_{\cS^*,\mu_t^*}(B(z,\ep):=\ov R_{\mu_t^*}\(\pi_{\cS^*}^{-1}(B(z,\ep))\)$, we have 
\beq\label{1fp83_2016_05_31}
\begin{aligned}
\lim_{\ep\to 0}\frac{\un R_{\cS^*,\mu_t^*}(B(z,\ep))}{\mu_t^*\circ\pi_{\cS^*}^{-1}(B(z,\ep))}
&=\lim_{\ep\to 0}\frac{\ov R_{\cS^*,\mu_t^*}(B(z,\ep))}{\mu_t^*\circ\pi_{\cS^*}^{-1}(B(z,\ep))} \\
&=d_\phi(z)
:=\begin{cases}
1 \  &\text{{\rm if (a) holds}}   \\
1-\big|\phi_\xi'(z)\big|e^{-p\P_{\cS^*}(t)} &\text{{\rm  if (b) holds}},
\end{cases}
\end{aligned}
\eeq
where in {\rm (b)}, $\{\xi\}=\pi_{\cS^*}^{-1}(z)$ and $p\ge 1$ is the prime period of $\xi$ under the shift map $\sg_*:E_{*A^*}^\infty\to E_{*A^*}^\infty$. 
\ethm

\bthm\label{t1_2016_05_31}
Let $\cS=\{\phi_e\}_{e\in E}$ be a parabolic Conformal Graph Directed Markov System whose limit set $J_{\mathcal S}$ is geometrically irreducible. If $t>\frac{\b}{\b+1}$ then
\beq\label{1fp83B_2016_05_31}
\lim_{\ep\to 0}\frac{\un R_{\cS^*,\mu_t^*}(B(z,\ep))}{\mu_t^*\circ\pi_{\cS^*}^{-1}(B(z,\ep))}
=\lim_{\ep\to 0}\frac{\ov R_{\cS^*,\mu_t^*}(B(z,\ep))}{\mu_t^*\circ\pi_{\cS^*}^{-1}(B(z,\ep))}
=1 
\eeq
for $\mu_t\circ\pi^{-1}$--a.e. point of $J_\cS$.
\ethm

Sticking to notation of Section~\ref{ERCGDMSHD}, given $z\in E_{*A^*}^\infty$ and $r>0$ let 
$$
K_z^*(r):=\pi_{\cS^*}(\tilde K_z^*(r)),
$$
where 
$$
\tilde K_z^*(r)
:=\big\{\om\in E_{*A^*}^\infty:\forall_{n\ge 0}\, \sg_*^n(\om)\notin \pi_{\cS^*}^{-1
  }(B(\pi_{\cS^*}(z),r)\big\}
 =\bi_{n=0}^\infty\sg_*^{-n}\(\pi_{\cS^*}^{-1}(B(\pi_{\cS^*}(z),r))\).
$$
As immediate consequences respectively of Theorem~\ref{t2had30} and Corollary~\ref{t2_2016_05_27}, we get the following two results.

\bthm\label{t2_2016_05_31} 
Let $\cS=\{\phi_e\}_{e\in E}$ be a parabolic Conformal Graph Directed Markov System.  Assume that both $\cS^*$ is {\rm (WBT)} and parameter $b_\cS$ is powering at some point $z\in J_{\cS^*}$. If $z$ is either 

\begin{itemize}
\item[(a)] not pseudo-periodic with respect to the system $\cS^*$, 

or 

\item[(b)] uniquely periodic with respect to $\cS^*$, it belongs to $\Int X$ (and $z=\pi_{\cS^*}(\xi^\infty)$ for a (unique) irreducible word $\xi\in E_{*A^*}^*$),
\end{itemize}

\sp then, 
\beq\label{1had31D}
\lim_{r\to 0}\frac{\HD(J_\cS)-\HD(K_z^*(r))}
{\mu_b^*\(\pi_{\cS^*}^{-1}(B(z,r))\)}=
\begin{cases}
1/\chi_{\mu_b^*} \  &\text{ if } \ (a) \  \text{ holds } \\
\(1-|\phi_\xi'(z)|\)/\chi_{\mu_b^*}
\  &\text{ if } \ (b) \  \text{ holds }.
\end{cases}
\eeq
\ethm 

\bthm\label{t2had30BB}
Let $\cS=\{\phi_e\}_{e\in E}$ be a parabolic Conformal Graph Directed Markov System whose limit set $ J_{\mathcal S}$ is geometrically irreducible. Then 
\beq\label{1had31+3}
\lim_{r\to 0}\frac{\HD(J_\cS)-\HD(K_z^*(r))}
{\mu_b^*\(\pi_{\cS^*}^{-1}(B(z,r))\)}=
\frac1{\chi_{\mu_b^*} }
\eeq
for $\mu_{b_\cS^*}\circ\pi^{-1}$--a.e. point $z$ of $J_\cS$.
\ethm

\

\part{Applications: Escape Rates for Multimodal Interval Maps and One--Dimensional Complex Dynamics}

\sp Our goal in this part of the manuscript is to get the existence of escape rates in the sense of \eqref{1_2016_06_27} and \eqref{2_2016_06_27} for all topologically exact piecewise smooth multimodal maps of the interval $[0,1]$, all rational functions of the Riemann sphere $\oc$ with degree $\ge 2$, and a vast class of transcendental meromorphic functions from $\C$ to $\oc$. In order to do this we employ two primary tools. The first one is formed by the escape rates results for the class of all countable alphabet conformal graph directed Markov systems obtained in Sections~\ref{Escape Rates for Conformal GDMSs; Measures} and \ref{ERCGDMSHD}. The other one is the method based on the first return (induced) map developed in Section~\ref{FRM}, Section~\ref{FRMERI}, and Section~\ref{FRMERII} of this part. This method closely relates the escape rates of the original map and the induced one. It turns out that for the above mentioned class of systems one can find a set of positive measure which gives rise to the first returned map which is isomorphic to a countable alphabet conformal IFS or full shift map; the task highly non-trivial and technically involved in general.  In conclusion, the existence of escape rates in the sense of \eqref{1_2016_06_27} and \eqref{2_2016_06_27} follows.

\section{First Return Maps}\label{FRM}

Let $(X,\rho)$ be a metric space and let $F\sbt X$ be a Borel set. Let $T:X\to X$ be a Borel map. Define
$$
F_\infty:=F\cap \bi_{k=0}^\infty\bu_{n=k}^\infty T^{-k}(F),
$$
i. e. $F_\infty$ is the set of all those points in $F$ that return to $F$ infinitely often under the iteration of the map $T$. Then for every $x\in F_\infty$ the number
$$
\tau_F(x)
:=\min\{n\ge 1:T^n(x)\in F\}
 =\min\{n\ge 1:T^n(x)\in F_\infty\}
$$
is well-defined, i. e. it is finite. The number $\tau_F(x)$ is called the first return of $x$ to $F$ under the map $F$. Having the function $\tau_F:F_\infty\to\N_1$ defined one defines the first return map $T_F:F_\infty\to F_\infty$ by the formula
\beq\label{2rm1}
T_F(x):T^{\tau_F(x)}(x)\in F_\infty\sbt F.
\eeq
Let $B$ be a Borel subsets of $F$. As in the previous section let
$$
K(B)=K_T(B):=\bi_{n=0}^\infty T^{-n}(X\sms B) \  \  {\rm and } \  \
      K_F(B):=\bi_{n=0}^\infty T_F^{-n}(F_\infty\sms B).
$$
A straightforward observation is that $K_F(B)=F_\infty\cap K(B)$, so that we have the following.
\beq\label{1rm1}
K_F(B)=F_\infty\cap K(B)\sbt F\cap K(B).
\eeq
We shall prove the following.

\bthm\label{t1rm1}
If the map $T:X\to X$ is locally bi-Lipschitz and $B\sbt F$ are Borel subsets of $X$, then
$$
\HD(K(B))=\max\{\HD(K_F(B)), \HD(K(F))\}.
$$
\ethm

\bpf Since $K(F)\sbt K(B)$, we have that $\HD(K(F))\le \HD(K(B))$, and by \eqref{1rm1} we have $\HD(K_F(B))\le \HD(K(B))$. We are thus let to show only that
$$
\HD(K(B))\le \max\{\HD(K_F(B)), \HD(K(F))\}.
$$
Indeed, fix $x\in K(B)$. Let 
$$
N_x:=\min\{n\ge 0:T^n(x)\in F\}.
$$
Consider two cases:

\sp\fr Case~$1^0$: The set $N_x$ is finite. Denote then by $n_x$ its largest element. Then $T^{n_x+1}(x)\in K(F)$. Hence
$$
x\in \bu_{n=0}^\infty T^{-n}(K(F));
$$
note that this relation holds even if $N_x=\es$. 

\sp\fr Case~$1^0$: The set $N_x$ is infinite. Then there exists $m_x\ge 0$ such that $T^{m_x}(x)\in F_\infty$. Hence,
$$
x\in \bu_{n=0}^\infty T^{-n}(F_\infty).
$$
In conclusion
$$
K(B)\sbt \bu_{n=0}^\infty T^{-n}(K(F))\, \cup \, \bu_{n=0}^\infty T^{-n}(F_\infty).
$$
But then, using \eqref{1rm1}, we get 
$$
\begin{aligned}
K(B)
&\sbt \lt(\bu_{n=0}^\infty T^{-n}(K(F))\rt)\, \cup \, \lt(K(B)\cap\bu_{n=0}^\infty  
         T^{-n}(F_\infty)\rt) \\
&\sbt \bu_{n=0}^\infty T^{-n}(K(B)\cap K(F))\, \bu \, \bu_{n=0}^\infty T^{-n}(K(B)\cap 
         F_\infty) \\
&=    \bu_{n=0}^\infty T^{-n}(K(F))\, \bu \, \bu_{n=0}^\infty T^{-n}(K_F(B))
\end{aligned}
$$
Therefore, using $\sg$-stability of Hausdorff dimension and local bi-Lipschitzness of $T$, we get
$$
\begin{aligned}
\HD(K(B))
&\le \sup_{n\ge 0}\big\{\!\!\max\{\HD(T^{-n}(K(F))),\HD(T^{-n}(K_F(B)))\}\big\} \\
&\le \sup_{n\ge 0}\big\{\!\!\max\{\HD(T^{-n}(K(F))),\HD(T^{-n}(K_F(B)))\}\big\} \\
&=\max\{\HD(K(F)),\HD(K_F(B))\}
\end{aligned}
$$
The proof of Theorem~\ref{t1rm1} is complete.
\epf

\fr As an immediate consequence of this theorem we get the following.

\bcor\label{c1rm3}
If the map $T:X\to X$ is locally bi-Lipschitz, $B\sbt F$ are Borel subsets of $X$, and $\HD(K(F))<\HD(K(B))$, then
$$
\HD(K(B))=\HD(K_F(B)).
$$
\ecor

\section{First Return Maps and Escaping Rates, I}\label{FRMERI}

As in Section~\ref{FRM} $(X,\rho)$ is a metric space, $F\sbt X$ be a Borel set and $T:X\to X$ is a Borel map. The symbols $F_\infty$, $\tau_F$, and $T_F$ have the same meaning as in Section~\ref{FRM}. Now in addition we also assume that the system $T:X\to X$ preserves a Borel probability measure $\mu$ on $X$. It is well-known that then the first return map $T_F:F_\infty\to F_\infty$ preserves the conditional measure $\mu_F$ on $F$ (or $F_\infty$ alike). This measure is given by the formula
$$
\mu_F(A)=\frac{\mu(A)}{\mu(F)}
$$
for every Borel set $A\sbt F$. The famous Kac's Formula tells us that
$$
\int_F\tau_F\, d\mu_F=\frac1{\mu(F)}.
$$
For every $n\ge 1$ denote 
$$
\tau_F^{(n)}:=\sum_{j=0}^{n-1}\tau_F\circ T_F^J,
$$
so that
$$
T_F^n(x)=T^{\tau_F^{(n)}(x)}(x).
$$
If $B$, as in Section~\ref{FRM}, is a Borel subset of $F$, then for every $n\ge 1$ we denote
$$
B_n^c:=\bi_{j=0}^{n-1}T^{-j}(X\sms B), \  \
B_n^c(F):=F_\infty\cap B_n^c, \  \  {\rm and } \  \
B_n^c(T_F):=\bi_{j=0}^{n-1}T_F^{-j}(X\sms B).
$$
For every $\eta\in (0,1)$ and every integer $k\ge 1$ denote
$$
F_{k-1}(\eta):=\lt\{x\in F_\infty:\lt(\frac1{\mu(F)}-\eta\rt)k\le \tau_F^{(k)}(x)
    \le \lt(\frac1{\mu(F)}+\eta\rt)k\rt\}.
$$
Let us record the following straightforward observation.
\beq\label{1rmr3}
F_{n-1}(\eta)\cap B_{\lt(\frac1{\mu(F)}+\eta\rt)n}^c
\sbt F_{n-1}(\eta)\cap B_n^c(T_F)
\sbt F_{n-1}(\eta)\cap B_{\lt(\frac1{\mu(F)}-\eta\rt)n}^c.
\eeq
This simple relation will be however our starting point for relating the escape rates of $B$ with respect to the map $T$ and the first return map $T_F:F_\infty\to F_\infty$.

\bdfn\label{d1rmr3}
We say that the pair $(T,F)$ satisfies the large deviation property $($LDP$)$ if for all $\eta\in (0,1)$ there exist two constants $\hat\eta>0$ and $C_\eta\in[1,+\infty)$ such that
$$
\mu(F_n^c(\eta))\le C_\eta e^{-\hat\eta n}
$$
for all integers $n\ge 1$.
\edfn

\fr In what follows we will need one (standard) concept more. We define for every $x\in X$ the number
$$
E_F(x):=\min\big\{n\in\{0,1,2,\ld,\infty\}:T^n(x)\in F\big\}.
$$
This number is called the first entrance time to $F$ under the map $T$ and it is closely related to $\tau_F$, 
$$
\tau_F(x)=E_F(T(x))+1
$$ 
if $x\in F$, but of course it is different.

\bdfn\label{d1_2016_07_14}
We say that the pair $(T,F)$ has exponential tail decay (ETD) if 
$$
\mu\(E_F^{-1}([n,+\infty]\)\le Ce^{-\a n}
$$
for all integers $n\ge 0$ and some constants $C, \a\in (0,+\infty)$.
\edfn

\fr Let $B$ be a Borel subset of $F$. Following the previous sections denote respectively by $R_{T,\mu}(B)$ and $R_{T_F,\mu}(B)$ the respective escape rates of $B$ by the maps $T:X\to X$ and $T_F:F_\infty\to F_\infty$, i. e.
$$
\un R_{T,\mu}(B):=-\varlimsup_{n\to\infty}\frac1n\log\mu(B_n^c) 
\  \  \  \le \  \  \  
\ov R_{T,\mu}(B):=-\varliminf_{n\to\infty}\frac1n\log\mu(B_n^c),
$$
and 
$$
\un R_{T_F,\mu}(B)
:=-\varlimsup_{n\to\infty}\frac1n\log\mu_F(B_n^c(T_F))
 =-\varlimsup_{n\to\infty}\frac1n\log\mu(B_n^c(T_F)),
$$
$$
\ov R_{T_F,\mu}(B)
:=-\varliminf_{n\to\infty}\frac1n\log\mu_F(B_n^c(T_F))
 =-\varliminf_{n\to\infty}\frac1n\log\mu(B_n^c(T_F)),
$$
with obvious inequality
$$
\un R_{T_F,\mu}(B) \le \ov R_{T_F,\mu}(B).
$$
We shall prove the following.

\bthm\label{t2rmr3}
Assume that a pair $(T,F)$ satisfies the large deviation property {\rm (LDP)} and has exponential tail decay {\rm (ETD)}. Let $(B_k)_{k=0}^\infty$ be a sequence of Borel subsets of $F$ such that
\begin{itemize}

\sp
\item[(a)] $\lim_{k\to\infty}\mu(B_k)=0$,

\sp\item[(b)] The limits 
$$
\lim_{k\to\infty}\frac{\un R_{T_F,\mu}(B_k)}{\mu_F(B_k)}
\  \  \  \and \  \  \
\lim_{k\to\infty}\frac{\ov R_{T_F,\mu}(B_k)}{\mu_F(B_k)}
$$
exist, are equal, and belong to $(0,+\infty)$; denote their common value by $R_F(\mu)$.
\end{itemize}
Then the limits
$$
\lim_{k\to\infty}\frac{\un R_{T,\mu}(B_k)}{\mu(B_k)}
\  \  \  \and \  \  \
\lim_{k\to\infty}\frac{\ov R_{T,\mu}(B_k)}{\mu(B_k)}
$$
also exist, and, denoting their common value by $R_T(\mu)$, we have that
$$
R_T(\mu)=R_F(\mu).
$$
\ethm

\bpf
Fix $\eta,\e\in (0,1)$. Fix two integers $k, n\ge 1$. Denote the sets 
$(B_k)_{\lt(\frac1{\mu(F)}+\eta\rt)n}^c$ and $(B_k)_{\lt(\frac1{\mu(F)}-\eta\rt)n}^c$ respectively by $B_{k,n}^-(\eta)$ and $B_{k,n}^+(\eta)$. Because of \eqref{1rmr3}, we have
\beq\label{1rmr4}
\mu\(F_{n-1}(\eta)\cap (B_k)_n^c(T_F)\)\le \mu\(B_{k,n}^+(\eta)\).
\eeq
Fix $M_1\ge 1$ so large that
\beq\label{2rmr4}
(1-\e)R_F(\mu)
\le \frac{\un R_{T_F,\mu}(B_k)}{\mu_F(B_k)}
\le \frac{\ov R_{T_F,\mu}(B_k)}{\mu_F(B_k)}
\le (1+\e)R_F(\mu)
\eeq
and 
\beq\label{2rmr4.1}
4R_F(\mu)\mu_F(B_k)\le \min\{\e,\hat\eta/2\}
\eeq
for all $k\ge M_1$. Fix such a $k$. Fix then $N_k\ge 1$ so large that
$$
\exp\(-(1+\e)\ov R_{T_F,\mu}(B_k)n\)
\le \mu\((B_k)_n^c(T_F)\)
\le \exp\(-(1-\e)\un R_{T_F,\mu}(B_k)n\)
$$
for all $n\ge N_k^{(1)}$. Along with \eqref{2rmr4} this gives
\beq\label{1rmr4.1}
\exp\(-(1+e)^2R_F(\mu)\mu_F(B_k)n\)
\le \mu\((B_k)_n^c(T_F)\)
\le \exp\(-(1-\e)^2R_F(\mu)\mu_F(B_k)n\).
\eeq
Therefore, using also Definition~\ref{d1rmr3}, we get for all $k\ge M_1$ and all $n\ge N_k^{(1)}$ that
$$
\begin{aligned}
\frac{\mu\(F_{n-1}(\eta)\cap (B_k)_n^c(T_F)\)}{\mu\((B_k)_n^c(T_F)\)}
&=\frac{\mu\((B_k)_n^c(T_F)\)-\mu\(F_{n-1}^c(\eta)\cap (B_k)_n^c(T_F)\)}{\mu\((B_k)_n^c(T_F)\)} \\
&\ge \frac{\mu\((B_k)_n^c(T_F)\)-\mu\(F_{n-1}^c(\eta)\)}{\mu\((B_k)_n^c(T_F)\)} 
 =1-\frac{\mu\(F_{n-1}^c(\eta)\)}{\mu\((B_k)_n^c(T_F)\)} \\
&\ge 1-\frac{C_\eta e^{-\hat\eta(n-1)}}{\exp\(-4R_F(\mu)\mu_F(B_k)n\)} \\
&=1-C_\eta e^{\hat\eta}\exp\((4R_F(\mu)\mu_F(B_k)-\hat\eta)n\) \\
&\ge 1-C_\eta e^{\hat\eta}\exp\(-\frac12\hat\eta n\)
\ge 1/2,
\end{aligned}
$$
where the last inequality holds for all  $n\ge N_k^{(1)}$ large enough, say  $n\ge N_k^{(2)}\ge N_k^{(1)}$. Along with  \eqref{1rmr4} this gives
$$
\mu\(B_{k,n}^+(\eta)\)\ge \frac12\mu\((B_k)_n^c(T_F)\).
$$
Hence
$$
-\varliminf_{n\to\infty}\frac1n\log\mu\(B_{k,n}^+(\eta)\)
\le \ov R_{T_F,\mu}(B_k).
$$
Since 
$$
(B_k)_n^c\spt B_{k,\lt[\frac{n}{\frac1{\mu(F)}-\eta}\rt]+1}^+\,,
$$
we get
$$
\varliminf_{n\to\infty}\frac1n\log\mu\(B_{k,n}^+(\eta)\)
\le \lt(\frac1{\mu(F)}-\eta\rt)\varliminf_{n\to\infty}\frac1n\log\mu\((B_k)_n^c\).
$$
Therefore,
$$
\ov R_{T_F,\mu}(B_k)\ge \lt(\frac1{\mu(F)}-\eta\rt)\ov R_{T,\mu}(B_k).
$$
Dividing both sides of this inequality by $\mu(B_k)$ and passing to the limit with $k\to\infty$, this entails
$$
R_F(\mu)\ge (1-\eta\mu(F))\varlimsup_{k\to\infty}\frac{\ov R_{T,\mu}(B_k)}{\mu(B_k)}.
$$
By letting in turn $\eta\downto 0$, this yields
\beq\label{1rmr5}
R_F(\mu)\ge \varlimsup_{k\to\infty}\frac{\ov R_{T,\mu}(B_k)}{\mu(B_k)}.
\eeq
Passing to the proof of the opposite inequality, denote $\lt(\frac1{\mu(F)}+\eta\rt)n$ by $n^+$ and $\lt(\frac1{\mu(F)}+\eta\rt)^{-1}n$ by $n^-$. We have
\beq\label{4rmr5.3}
\begin{aligned}
B_{k,n}^-(\eta)
&=\bu_{j=0}^{n^+}\(B_{k,n}^-(\eta)\cap E_F^{-1}(j)\)\cup E_F^{-1}((n^+,+\infty]) \\
&=E_F^{-1}((n^+,+\infty])\cup\bu_{j=0}^{n^+}E_F^{-1}(j)
   \cap T^{-j}\((B_k)_{n^+-j}^c\).
\end{aligned}
\eeq
Now,
\beq\label{2rmr5.3}
\begin{aligned}
E_F^{-1}(j)&\cap T^{-j}\((B_k)_{n^+-j}^c\)
 =E_F^{-1}(j)\cap T^{-j}(F)\cap T^{-j}\((B_k)_{n^+-j}^c\) \\
&=E_F^{-1}(j)\cap T^{-j}\(F\cap (B_k)_{n^+-j}^c\) \\
&=E_F^{-1}(j)\cap \Big(T^{-j}\(F_{(n^+-j)^--1}(\eta)\cap (B_k)_{n^+-j}^c\)\cup 
T^{-j}\(F_{(n^+-j)^--1}^c(\eta)\cap (B_k)_{n^+-j}^c\)\Big). \\
\end{aligned}
\eeq
By \eqref{1rmr3} we have
\beq\label{1rmr5.3}
\begin{aligned}
F_{(n^+-j)^--1}^c(\eta)\cap (B_k)_{n^+-j}^c
&=F_{(n^+-j)^--1}^c(\eta)\cap (B_k)_{((n^+-j)^-)^+}^c \\
&\sbt F_{(n^+-j)^--1}^c(\eta)\cap (B_k)_{(n^+-j)^-}^c(T_F) \\
&\sbt (B_k)_{(n^+-j)^-}^c(T_F).
\end{aligned}
\eeq
Now take $p, q>1$ such that $\frac1p+\frac1q=1$. By applying H\"older inequality, $T$-invariantness of the measure $\mu$, \eqref{1rmr4.1}, and making use of Definition~\ref{d1_2016_07_14}, we get for all $0\le j\le N_k^{(1)}$, that 
$$
\begin{aligned}
\mu\Big(E_F^{-1}(j)&\cap T^{-j}\(F_{(n^+-j)^--1}(\eta)
        \cap (B_k)_{n^+-j}^c\)\Big)
\le \mu\Big(E_F^{-1}(j)\cap T^{-j}\((B_k)_{(n^+-j)^-}^c(T_F)\)\Big)= \\
&=\int_X\1_{E_F^{-1}(j)}\1_{T^{-j}\((B_k)_{(n^+-j)^-}^c(T_F)\)}\,d\mu \\
&\le \lt(\int_X\1_{E_F^{-1}(j)}\,d\mu\rt)^{1/p} 
     \lt(\int_X\1_{T^{-j}\((B_k)_{(n^+-j)^-}^c(T_F)\)}\,d\mu\rt)^{1/q} \\
&=\mu^{1/p}(E_F^{-1}(j))\mu^{1/q}\(T^{-j}\((B_k)_{(n^+-j)^-}^c(T_F)\)\) \\
&=\mu^{1/p}(E_F^{-1}(j))\mu^{1/q}\((B_k)_{(n^+-j)^-}^c(T_F)\) \\
&\le C^{1/p}e^{-\frac{\a}pj}\exp\lt(-\frac{(1-\e)^2}q
       R_F(\mu)\mu_F(B_k)(n^+-j)^-\rt) \\
&=C^{1/p}e^{-\frac{\a}pj}\exp\lt(-\frac{(1-\e)^2}q
       R_F(\mu)\mu_F(B_k)\lt(\frac1{\mu(F)}+\eta\rt)^{-1}
       \lt(\lt(\frac1{\mu(F)}+\eta\rt)n-j\rt)\rt) \\
&=C^{1/p}\exp\lt(-\frac{(1-\e)^2}q R_F(\mu)\mu_F(B_k)n\rt) \cdot  \\
& \  \  \  \  \  \   \  \  \   \   \  \  \  \  \  \  \   \  \  \cdot \exp\lt(-\lt(\frac{\a}p-\frac{(1-\e)^2}q\lt(\frac1{\mu(F)}+\eta\rt)^{-1}
      R_F(\mu)\mu_F(B_k)\rt)j\rt)
\end{aligned}
$$
Together  with the left-hand side of \eqref{1rmr4.1} this gives that
$$
\begin{aligned}
&\frac{\mu\Big(E_F^{-1}(j)\cap T^{-j}\(F_{(n^+-j)^--1}(\eta)
        \cap (B_k)_{n^+-j}^c\)\Big)}
{\mu\((B_k)_n^c(T_F)\)} \le \\
& \  \  \  \  \  \   \  \  
\le C^{1/p}\exp\lt(R_F(\mu)\mu_F(B_k)\((1+e)^2-\frac1q (1-\e)^2\)n\rt)\cdot \\
& \  \  \  \  \  \   \  \  \   \   \  \  \  \  \  \  \  \  \   \  \  \  \  \  \cdot \exp\lt(-\lt(\frac{\a}p-\frac{(1-\e)^2}q\lt(\frac1{\mu(F)}+\eta\rt)^{-1}
      R_F(\mu)\mu_F(B_k)\rt)j\rt).
\end{aligned}
$$
Taking now $q>1$ sufficiently close to $1$ and looking at \eqref{2rmr4.1} we will have for every $\e>0$ small enough that
$$
(1+e)^2-\frac1q (1-\e)^2\le 1-\frac1q+6\e\le 7\e 
\  \  \  \and \  \  \
\frac{\a}p-\frac{(1-\e)^2}q\lt(\frac1{\mu(F)}+\eta\rt)^{-1}\!\!\!R_F(\mu)\mu_F(B_k)
>\frac\a{2p}.
$$
Therefore, for all $k\ge M_1$ and for all $0\le j\le n^+-N_k^{(1)}$, we have that
\beq\label{2rmr5.4}
\frac{\mu\Big(E_F^{-1}(j)\cap T^{-j}\(F_{(n^+-j)^--1}(\eta)\cap (B_k)_{n^+-j}^c\)\Big)}
   {\mu\((B_k)_n^c(T_F)\)} 
\le  C^{1/p}e^{7\e n}e^{-\frac{\a}{2p}j}.
\eeq
For $n^+-N_k^{(1)}<j\le n^+$, using \eqref{2rmr4.1}, the left-hand side of \eqref{1rmr4.1}, and looking up at Definition~\ref{d1_2016_07_14}, we 
have the easier estimate:
\beq\label{3rmr5.4}
\begin{aligned}
&\frac{\mu\Big(E_F^{-1}(j)\cap T^{-j}\(F_{(n^+-j)^--1}(\eta)
  \cap (B_k)_{n^+-j}^c\)\Big)}{\mu\((B_k)_n^c(T_F)\)} 
\le \frac{\mu\(E_F^{-1}(j)\)}{\mu\((B_k)_n^c(T_F)\)}\le   \\
& \  \  \  \  \  \   \   \   \  \le C^{-\a j}\exp\((1+e)^2R_F(\mu)\mu_F(B_k)n\) \\
& \  \  \  \  \  \   \   \   \  \le C\exp\(-\a(n^+-N_k^{(1)})\)\exp\((1+e)^2R_F(\mu)\mu_F(B_k)n\) \\
& \  \  \  \  \  \   \   \   \  =Ce^{\a N_k^{(1)}}\exp\lt(\lt((1+e)^2R_F(\mu)\mu_F(B_k)
   -\a\lt(\frac1{\mu(F)}+\eta\rt)\rt)n\rt)  \\
& \  \  \  \  \  \   \   \   \  \le Ce^{\a N_k^{(1)}}.
\end{aligned}
\eeq
Now we can estimate the second part of \eqref{2rmr5.3}. We note that
$$
E_F^{-1}(j)\cap T^{-j}\(F_{(n^+-j)^--1}(\eta)\cap (B_k)_{n^+-j}^c\)
\sbt E_F^{-1}(j)\cap T^{-j}\(F_{(n^+-j)^--1}(\eta)\),
$$
and use again H\"older inequality, $T$-invariance of measure $\mu$, and Definitions~\ref{d1_2016_07_14} and \ref{1rmr3}, to estimate:
$$
\begin{aligned}
\mu\Big(E_F^{-1}(j)\cap T^{-j}&\(F_{(n^+-j)^--1}(\eta)
  \cap (B_k)_{n^+-j}^c\)\Big) \le  \\
&\le \mu\Big(E_F^{-1}(j)\cap T^{-j}\(F_{(n^+-j)^--1}(\eta)\Big) 
  =\int_X\1_{E_F^{-1}(j)}\1_{T^{-j}\(F_{(n^+-j)^--1}(\eta)}\,d\mu \\
&\le \mu^{1/p}(E_F^{-1}(j))\cdot \nu^{1/q}\(T^{-j}\(F_{(n^+-j)^--1}(\eta)\)\) \\
&=   \mu^{1/p}(E_F^{-1}(j))\cdot \nu^{1/q}\(F_{(n^+-j)^--1}(\eta)\) \\
&\le  C^{1/p}e^{-\frac{\a}{p}j}C_\eta^{1/q}e^{-\frac{\hat\eta}{q}\((n^+-j)^--1\)} \\
&=  C^{1/p}C_\eta^{1/q}e^{\frac{\hat\eta}{q}e^{-\frac{\a}{p}j}}
       \exp\lt(-\lt(\frac{\a}{p}-\frac{\hat\eta}{q}\lt(\frac1{\mu(F)}+\eta\rt)^{-1}\rt)j\rt).
\end{aligned}
$$
Combining this  with the left-hand side of \eqref{1rmr4.1} this gives that
\beq\label{1rmr5.5}
\begin{aligned}
&\frac{\mu\Big(E_F^{-1}(j)\cap T^{-j}\(F_{(n^+-j)^--1}(\eta)
  \cap (B_k)_{n^+-j}^c\)\Big)}{\mu\((B_k)_n^c(T_F)\)} \\
&\  \  \  \  \  \  \  \  \le  C^{1/p}C_\eta^{1/q}e^{\frac{\hat\eta}q}
    \exp\lt(-\lt(\frac{\hat\eta}q-(1+e)^2R_F(\mu)\mu_F(B_k)\rt)n\rt)\cdot \\
& \  \  \  \  \  \  \  \  \  \  \  \  \  \  \  \  \   \ \cdot \exp\lt(-\lt(\frac{\a}{p}-\frac{\hat\eta}{q}\lt(\frac1{\mu(F)}+\eta\rt)^{-1}\rt)j\rt).
\end{aligned}
\eeq
Now, first take $q>1$ so large that 
$$
\frac{\a}{p}-\frac{\hat\eta}{q}\lt(\frac1{\mu(F)}+\eta\rt)^{-1}>\frac{\a}2.
$$
Then take $k\ge M_1$, say $k\ge M_{1,q}\ge M_1$ so large that
$$
\frac{\hat\eta}q-(1+e)^2R_F(\mu)\mu_F(B_k)\ge \frac{\hat\eta}{2q}.
$$
Inserting these two inequalities into \eqref{1rmr5.5}, yields
\beq\label{1rmr5.6}
\frac{\mu\Big(E_F^{-1}(j)\cap T^{-j}\(F_{(n^+-j)^--1}(\eta)
  \cap (B_k)_{n^+-j}^c\)\Big)}{\mu\((B_k)_n^c(T_F)\)}
\le  C^{1/p}C_\eta^{1/q}e^{\frac{\hat\eta}{q}}e^{-\frac{\hat\eta}{2q}n}e^{-\frac{\a}2j}.
\eeq
Finally, by Definition~\ref{d1_2016_07_14}, the left-hand side of \eqref{1rmr4.1}, and \eqref{2rmr4.1}, 
$$
\begin{aligned}
\frac{m\(E_F^{-1}((n^+,+\infty])\)}{\mu\((B_k)_n^c(T_F)\)}
&\le Ce^{-\a n^+}\exp\((1+e)^2R_F(\mu)\mu_F(B_k)n\) \\
&=C\exp\lt(-\lt(\a\lt(\frac1{\mu(F)}+\eta\rt)-(1+e)^2R_F(\mu)\mu_F(B_k)n\rt)\rt) \\
&\le C
\end{aligned}
$$
for every $\e>0$ small enough and $n\ge M_1$. Combining this inequality, \eqref{2rmr5.4}, \eqref{3rmr5.4}, \eqref{1rmr5.3}, \eqref{2rmr5.3}, and \eqref{4rmr5.3}, we get for every $k\ge 1$ large enough, every $e>0$, and every $n\ge N_k^{(1)}$, that 
$$
\frac{\mu\(B_{k,n}^-(\eta)\)}{\mu\((B_k)_n^c(T_F)\)}
\le C\(1+e^{\a N_k^{(1)}}\)+C'\sum_{j=0}^{n^+}e^{-\frac{\a}2j}+C''e^{7\e n}\sum_{j=0}^{n^+-N_k^{(1)}}e^{-\frac{\a}{2p}j}
\le C'''e^{7\e n}
$$
with some constants $C, C', C'', C'''\in (0,+\infty)$ and $p>1$ independent of $k\ge 1$ large enough, $n\ge N_k^{(1)}$, and $\e\in (0,1)$ small enough. Hence, 
$$
-\varlimsup_{n\to\infty}\frac1n\log\mu\(B_{k,n}^-(\eta)\)
\ge \un R_{T_F,\mu}(B_k)-7\e
$$
for every $\e>0$ and every $k\ge 1$ large enough. Therefore,
$$
-\varlimsup_{n\to\infty}\frac1n\log\mu\(B_{k,n}^-(\eta)\)
\ge \un R_{T_F,\mu}(B_k)
$$
for every $k\ge 1$ large enough. Since 
$$
(B_k)_n^c\sbt B_{k,\lt[\frac{n}{\frac1{\mu(F)}+\eta}\rt]}^-\,,
$$
we get
$$
\varlimsup_{n\to\infty}\frac1n\log\mu\((B_k)_n^c\)
\le  \frac1{\frac1{\mu(F)}+\eta}\varlimsup_{n\to\infty}\frac1n\log\mu\(B_{k,n}^-(\eta)\).
$$
Therefore,
$$
\un R_{T_F,\mu}(B_k)\le \lt(\frac1{\mu(F)}+\eta\rt)\un R_{T,\mu}(B_k)
$$
for every $k\ge 1$ large enough.
Dividing both sides of this inequality by $\mu(B_k)$ and passing to the limit with $k\to\infty$, this gives
$$
R_F(\mu)\le (1+\eta\mu(F))\varliminf_{k\to\infty}\frac{\un R_{T,\mu}(B_k)}{\mu(B_k)}.
$$
By letting in turn $\eta\downto 0$, this yields
$$
R_F(\mu)\le \varliminf_{k\to\infty}\frac{\un R_{T,\mu}(B_k)}{\mu(B_k)}.
$$
Together  with \eqref{1rmr5} this finishes the proof of Theorem~\ref{t2rmr3}.
\epf

\section{First Return Maps and Escaping Rates, II}\label{FRMERII}

In this section we keep the settings of Section~\ref{FRMERI}; more specifically   
that described between  its beginning until formula \eqref{1rmr3}. In particular, we do not assume appriori that (LDP) holds. In fact our goal in this section is provide natural sufficient conditions for (LDP) to hold. Let $\vp:X\to\R$ be a Borel measurable function. We define the function $\vp_F:F\to\R$ by the formula
\beq\label{2ld3}
\vp_F(x):=\sum_{j=0}^{\tau_F(x)-1}\vp\circ T^j(x).
\eeq
It is well-known 
%(see \cite{KU} for example) that if $\vp\in L^1(\mu)$, then $\vp_F\in L^1(\mu)$ and 
\beq\label{1ld3}
\int_X\vp\,d\mu=\mu(F)\int_F\vp_F\,d\mu_F.
\eeq
In particular,
$$
\1_F=\tau_F,
$$
and, inserting this to \eqref{1ld3}, we obtain the familiar, discussed in the previous section, Kac's Formula
$$
\int_F\tau_F\, d\mu_F=\frac1{\mu(F)}.
$$
\bdfn\label{d1ld3}
We say that a pentadde $(X,T,F,\mu,\phi)$, or just $T$, is of symbol return type (SRT) if the following conditions are satisfied:

\sp\begin{itemize}
\item[(a)] $F=E_A^\infty$ for some countable alphabet $E$ and some finitely irreducible incidence matrix $A$.

\sp\item[(b)] $T_F=\sg:E_A^\infty\to E_A^\infty$.

\sp\item[(c)] $\vp_F:F\to\R$ is a H\"older continuous summable potential.

\sp\item[(d)] $\P(\vp_F)=0$.

\sp\item[(e)] $\mu=\mu_{\vp_F}$ is the Gibbs/equilibrium state for the potential $\vp_F:F\to\R$

\sp \item[(f)] There are two constants $C, \a>0$ such that
$$
\mu\(\tau_F^{-1}(n)\)\le Ce^{-\a n}
$$
for all integers $n\ge 1$. 
\end{itemize}
\edfn

\fr Since 
$$
\tau_F^{-1}(n)\sbt T^{-1}(E_F^{-1}(n-1))
$$ 
and since the measure $\mu$ is $T$-invariant, we immediately obtain the following.

\bobs\label{o2_2016_07_14}
If a pentadde $(X,T,F,\mu,\phi)$ satisfies all conditions (a)--(e) of Definition~\ref{d1ld3} and it also has exponential tail decay {\rm (ETD)}, then   $(X,T,F,\mu,\phi)$ also satisfies condition (f) of Definition~\ref{d1ld3}; thus in conclusion, the pentadde $(X,T,F,\mu,\phi)$ is then of symbol return type {\rm (SRT)}.
\eobs

\fr Given $\th\in\R$ we consider the potential
$$
\vp_\th:=\vp_F+\th\tau_F:F\to\R.
$$
We shall prove several lemmas. We start with the following.

\blem\label{l1ld4}
If $T$ is an {\rm (SRT)} system, then the potential $\vp_\th:F\to\R$ is summable for every $\th<\a$.
\elem

\bpf
Since $T$ is  SRT, we have that 
$$
\begin{aligned}
   \sum_{e\in E}\exp\(\sup\(\vp_\th|_{[e]}\)\)
&= \sum_{e\in E}\exp\(\sup\((\vp_F+\th\tau_F)|_{[e]}\)\) \\
&= \sum_{e\in E}\exp\(\sup\(\vp_F|_{[e]}\)\)\exp\(\th\tau_F(e)\) 
 \comp \sum_{e\in E}\mu([e])\exp\(\th\tau_F(e)\)  \\
&=\sum_{n=1}^\infty\sum_{\tau_F(e)=n}\mu([e])e^{\th n} 
 =\sum_{n=1}^\infty e^{\th n}\sum_{\tau_F(e)=n}\mu([e]) \\
&=\sum_{n=1}^\infty e^{\th n}\mu\(\tau_F^{-1}(n)\)
\le C\sum_{n=1}^\infty \exp\((\th-\a)n\)<+\infty,
\end{aligned}
$$
whenever $\th<\a$. The proof is complete.
\epf

\blem\label{l2ld4}
If $T$ is an {\rm (SRT)} system, then the function $(-\infty,\a)\ni\th\mapsto \P(\vp_\th)\in\R$ is real-analytic.
\elem

\bpf
In the terminology of Corollary~2.6.10 in \cite{GDMS}, condition (c) of Definition~\ref{d1ld3} says that $\vp_F\in \cK_\b$, where $\b>0$ is the H\"older exponent of $\vp_F$. Of course $\tau_F\in \cK_\b$ since $\tau_F$ is constant on cylinders of length one. Lemma~\ref{l1ld4} says that $\vp_\th\in \cK_\b$ for all $\th<\a$; in fact the proof of this lemma shows that $\vp_\th\in \cK_\b$ for all $\th\in\C$ with $\re(\th)<\a$. This now means that all hypotheses of Corollary~2.6.10 from \cite{GDMS} are satisfied. The upshot of this corollary is that the function
$$
\{\th\in\C:\re(\th)<\a\}\ni\th\mapsto \pf_{\phi_\th}\in L(\cK_\b)
$$
is holomorphic, where $\pf_{\phi_\th}$ is the Perron-Frobenius operator associated to the potential $\vp_\th$ and the shift map $\sg=T_F$. The proof is now concluded by applying Kato-Rellich perturbation Theorem and the fact that $\exp\(\P(\vp_\th)\)$ is a simple isolated eigenvalue of $\pf_{\phi_\th}$ for all real $\th<\a$ (it is not really relevant here but in fact $\exp\(\P(\vp_\th)\)$ is equal to the spectral radius of the operator $\pf_{\phi_\th}\in L(\cK_\b)$), see the paragraph of \cite{GDMS} located between Remark~2.6.11 and Theorem~2.6.12 for more details. 
\epf

\fr Because of Lemma~\ref{l1ld4},  for every $\th<\a$ there exists a unique Gibbs/equilibrium state $\mu_\th$ for the potential $\vp_\th:F\to\R$. Having the previous two lemmas, Proposition~2.6.13 in \cite{GDMS} applies to give the following.

\blem\label{l1ld5}
If $T$ is an  {\rm (SRT)} system, then
$$
\frac{{\rm d}}{{\rm d}\th}\P(\vp_\th)=\int_F\tau_F\,d\mu_\th
$$
for every $\th<\a$.
\elem

\fr Now having all the three previous lemmas along with Definition~\ref{d1ld3}, employing the standard (by now) tools of \cite{GDMS}, exactly the same proof as in \cite{DK} yields the following.

\bthm\label{t1ld6}
If $T$ is an  {\rm (SRT)} system, then for every $\th<\a$ we have that
$$
\lim_{n\to\infty}\frac1n\log
  \mu\(\big\{x\in F:\sign(\th)\tau_F^{n)}(x)\ge \sign(\th)n\int_F\tau_F\,d\mu_\th\big\}\)=-\th\int_F\tau_F\,d\mu_\th+\P(\vp_\th).
$$
\ethm

\fr In order to make use of this theorem we shall prove the following.

\blem\label{l2ld6}
If $T$ is an  {\rm (SRT)} system and the first return map function $\tau_F:F_\infty\to\N$ is unbounded, then for every non-zero $\th<\a$ we have that
$$
\P(\vp_\th)-\th\int_F\tau_F\,d\mu_\th<0.
$$
\elem

\bpf
Since $\mu_\th $ is an equilibrium state for $\vp_\th$, we have that
$$
\begin{aligned}
\P(\vp_\th)-\th\int_F\tau_F\,d\mu_\th
&=\h_{\mu_\th}(\sg)+\int_F\vp_F\,d\mu_\th+\th\int_F\tau_F\,d\mu_\th-\th\int_F\tau_F\,
  d\mu_\th \\
&=\h_{\mu_\th}(\sg)+\int_F\vp_F\,d\mu_\th \\
&\le \P(\vp_F)
=0.
\end{aligned}
$$
Hence, in order to complete the proof we only need to show that the inequality sign above is strict. In order to do this suppose for a contradiction that $\h_{\mu_\th}(\sg)+\int_F\vp_F\,d\mu_\th=\P(\vp_F)$. But then the fact that $\mu$ is the only equilibrium state for $\vp_F$, implies that $\mu_\th=\mu$. But because of Theorem~2.2.7 in \cite{GDMS} this in turn implies that the function $\vp_\th-\vp_F$ is cohomologous to a constant in the class of H\"older continuous functions defined on $E_A^\infty=F$. But $\vp_\th-\vp_F=\th\tau_F$ is, by our hypotheses, unbounded unless $\th=0$. This finishes the proof.
\epf

\fr Now we can prove the main result of this section:

\blem\label{l1ld7}
If $T$ is an  {\rm (SRT)} system and the first return map function $\tau_F:F_\infty\to\N$ is unbounded, then the pair $(T_F,F)$ satisfies the large deviation property {\text{\rm (LDP)}}.
\elem

\bpf
Fix $\eta\in(0,1)$. It follows from Lemma~\ref{l2ld4} and Lemma~\ref{l1ld5} that the function
$$
(-\infty,\a)\ni\th\mapsto \int_F\tau_F\,d\mu_\th\in [1,+\infty)
$$
is continuous. Therefore, there exists $\d\in(0,\a)$ such that
$$
\int_F\tau_F\,d\mu-\eta
\le \int_F\tau_F\,d\mu_\d, 
\int_F\tau_F\,d\mu_{-\d}
\le \int_F\tau_F\,d\mu+\eta.
$$
Equivalently:
$$
\mu(F)^{-1}-\eta
\le \int_F\tau_F\,d\mu_\d, 
\int_F\tau_F\,d\mu_{-\d}
\le \mu(F)^{-1}+\eta.
$$
Hence for every $k\ge 1$:
$$
F_{k-1}^c(\eta)
\sbt \big\{x\in F:\tau_F^{k)}(x)\ge k\int_F\tau_F\,d\mu_\d\big\}
 \cup\big\{x\in F:\tau_F^{k)}(x)\le k\int_F\tau_F\,d\mu_{-\d}\big\}.
$$
So, denoting
$$
\hat\eta
:=\frac12\min\Big\{\d\int_F\tau_F\,d\mu_\d-\P(\vp_d),
      -\d\int_F\tau_F\,d\mu_{-\d}-\P(\vp_{-\d})\Big\},
$$
which is positive by Lemma~\ref{l2ld6}, we conclude from Theorem~\ref{t1ld6}, that
$$
\mu\(F_{k-1}^c(\eta)\)\le C_\eta e^{-\hat\eta k}
$$
for all $k\ge 1$. The proof is complete.
\epf

\section{Escape Rates for Interval Maps}\label{Interval_Maps}

In this and the next sections we will reap the benefits 
of our work in the  previous sections, most notably of that on escape rates of conformal countable alphabet IFSs and of that on the first return map techniques including large deviations.
This section is devoted to the study of  the multimodal smooth maps of an interval.  
  
We start with the definition of the class of dynamical systems and potentials we consider.

\begin{dfn}\label{dami1}
Let $I = [0,1]$ be the closed interval.  Let $T:  I \to I$ be a $C^3$ differentiable map with the  following properties:

\sp\begin{enumerate}
\item[(a)] $T$ has only a finitely  many maximal closed intervals of monotonicity; or equivalently  ${\rm Crit}(T) = \{x \in I \hbox{ : } T'(x) = 0\}$, the set of all critical points of $T$ is finite.

\sp\item[(b)] The dynamical system $T: I \to I$ is topologically exact, meaning that for every non-empty subset $U$ of $I$ there exists an integer $n \geq 0$ such that $T^n(U) = I$.

\sp\item[(c)] All critical points are non-flat.

\sp\item[(d)]  $T: I \to I$  is a \emph{topological Collet-Eckmann map}, meaning that 
$$
\inf\{\left(|(T^n)'(x)| \right)^{1/n} \hbox{ : } T^n(x)=x \hbox{ for }  n \geq 1 \} > 1
$$
where the infimum is taken over all integers $n \geq 1$ and all fixed points of $T^n$.
\end{enumerate}
We then call $T: I \to I$  a \emph{topologically exact topological Collet-Eckmann map} (teTCE).
If (c) and (d) are relaxed and only (a) and (b) are assumed then $T$ is called a topologically exact multimodal map.

As in the case of rational functions we set 
$$
\PC(T) := \bu_{n=1}^\infty T^n({\rm Crit}(T))
$$ 
and call this the \emph{postcritical} set of $T$.  Again as in the case of rational functions we say that the map $T:I \to I$
is \emph{tame} if
$$
\overline {\PC(T)} \neq I.
$$
\end{dfn}

\fr The following theorem is due to many authors and a detailed and readible discussion on this topic can be found, for example, in \cite{PR-Geometric}
 
\begin{thm}[Exponential Shrinking Property]\label{t1mi2.1}
If $T: I \to I$ satisfies conditions {\rm (a)--(c)} of Definition~\ref{dami1}, then $T$ is a {\rm (te)TCE}, i.e. condition {\rm (d)} holds if and only if there exist $\delta>0$, $\gamma>0$ and $C>0$ such that if $z\in I$ and $n \geq 0$ then
$$
\hbox{\rm diam}(W) \leq C e^{-\gamma n}
$$
for each connected component $W$ of $T^{-n}(B(z, 2\delta))$.
\end{thm}
 
\fr The hard part of this theorem is its ``if'' part. The converse is easy. There are more conditions equivalent to teTCE,  but we need only the above Exponential Shrinking Property (ESP) and we do not bring them up here. We now however articulate two standard sufficient conditions for (ESP) to hold. It is implied by the Collet-Eckmann condition which requires that there exist $\lambda > 1$ and $C>0$ such that for every integer $n \geq 0$ we have that
$$
|(f^n)'(f(c))| \geq C \lambda^n.
$$

If also suffices to assume that the map $T$ is semi-hyperbolic, i.e., that no critical points $c$ in the Julia belongs to its own omega limit set $\om(c)$ for (ESP) to hold.  
This so for example, if $T$ is a classical unimodal map of the form $I \ni x \mapsto \lambda x (1-x)$, with 
$0 < \lambda \leq 4$ and the critical point $1/2$ is not in its own omega limit set, i.e., $1/2 \not\in \omega (1/2)$.

\sp We call a potential $\psi: I \to \mathbb R$ \emph{acceptable} if it is Lipschitz continuous and 
$$
\sup (\psi) - \inf (\psi) < h_{\rm top}(T).
$$
We would also like to mention that for the purposes of this section
 it would suffice that $\psi: I \to \mathbb R$ is H\"older continuous 
 (with any exponent) and of bounded variation.  We denote by ${\rm BV}_I$ the vector space of all functions in 
%${\rm Leb}_1(I)$ 
$L^1(\lambda)$, where $\lambda$ denotes Lebesgue measure on $I$, 
that have a version of bounded variation.  This vector space becomes a Banach space when endowed with the norm 
$$
\|g\|_{BV} := \|g\|_{{\rm Leb}_1} + v_I(g)
$$
where $v_I(g)$ denotes the variation of $g$ on I. For every $g \in {\rm BV}_I$ define the 
\emph{Perron-Frobenius operator} associated to $\psi$ by
$$
\mathcal L_\psi (g)(x) = \sum_{y \in T^{-1}(x)} g(y) e^{\psi(y)}.
$$
It is well known and easy to check that $\mathcal L_{\psi}({\rm BV}_I) \subset {\rm BV}_I$ and 
 $\mathcal L_{\psi}:{\rm BV}_I \to {\rm BV}_I$ is a bounded linear operator.
 
 The following theorem collects together some fundamental results of 
\cite{HK-MZ} and \cite{HK-ETDS}
 
\begin{thm}\label{t1mi3}  If $T: I \to I$ is a topologically exact multimodal map and $\psi: I \to \mathbb R$ is an acceptable potential  then
 \begin{enumerate}
\item[(a)] there exists a Borel probability eigenmeasure 
$m_\psi$ for the dual operator $\mathcal L_{\psi}^*$ whose corresponding eigenvalue is equal to $e^{\P(\psi)}$.   It then follows that ${\rm supp}(m_\psi) = I$.  
\item[(b)] there exists a unique Borel $T$-invariant probability measure $\mu_\psi$ on $I$ absolutely continuous with respect to $m_\psi$.  Furthermore, $\mu_\psi$ is equivalent to $m_{\psi}$;
\item[(c)] $h_{\mu_\psi}(T) + \int_I \psi d\mu_\psi = P(\psi)$, meaning that $\mu_\psi$ is an (ergodic) equilibrium state for 
$\psi:I \to \mathbb R$ with respect to the dynamcial system $T: I \to I$.
\item[(d)] The Perron-Frobenius $\mathcal L_\psi : {\rm BV}_I \to  {\rm BV}_I$ is quasi-compact.
\item[(e)] $r({\mathcal L}_\psi) = e^{P(\psi)}$.
\item[(f)] $\hbox{\rm sp}(\mathcal L_\psi) \cap \partial B(0, e^{P(\psi)}) = \{e^{P(\phi)}\}$
\item[(g)] The number $e^{P(\psi)}$ is a simple isolated eigenvalue (this follows from (f), (e) and (f))  of $\mathcal L_\psi:  {\rm BV}_I  \to  {\rm BV}_I$ with eigenfunction $\rho_{\psi} := \frac{d\mu_\psi}{dm_\psi}$ which is Lipschitz continuous and $\log$-bounded.
 \end{enumerate}
 \end{thm}
We shall use the commonly accepted convention, used throughout this article, that for every $r \in (0,1]$ and every bounded interval $\Delta \subset \mathbb R$ we denote by $r \Delta$ the (smaller) interval of length $r |\Delta|$ centred at the same point as $\Delta$. 
 We now consider the following version of the bounded distortion property taken from \cite{PR-Geometric} whose proof has a long history and is well documented therein.

\begin{thm}\label{t1mi4} Let $T: I \to I$ be a teTCE.  Then for every $r \in (0,1)$ there exists $K(r) \in (0, +\infty)$ such that if 
$\Delta \subset I$ is an interval, $n \geq 0$ is an integer, the map $T^n|_\Delta$ is $1$-to-$1$, and $x,y \in \Delta$ are such that $T^n(x), T^n(y) \in r T^n(\Delta)$, then 
$$
\left|
\frac{(T^n)'(y)}{(T^n)'(x)} - 1 
\right| \le K(r) |(T^n)(y) - (T^n)(x)|.
$$
\end{thm}
 
We next recall the following definition.
 
\begin{dfn}\label{d2mi4} 
An interval $V \subset I$ is called a \emph{nice set} for a multimodal map $T: I \to I$ if
$$
 \hbox{int}(V) \cap \bu_{n=0}^\infty T^n(\partial V) = \emptyset
$$
\end{dfn}
 
 The proof of the following theorem is both standard and straightforward, and has been presented  in various similar settings. 
  We provide the proof below because of the critical importance for us of the theorem it proves and the brevity of the proof,  for the sake of completeness, and for the convenience of the reader.

\begin{thm}\label{t3mi4}
If $T: I \to I$ is topologically exact multimodal map then for every point $\xi \in (0,1)$ and every $R > 0$ there exists a nice set $V \subset I$ 
such that $\xi \in V \subset B(\xi, R)$.  
\end{thm}

\begin{proof}
Since the map $T: I \to I$  is topologically  exact it has a dense set of periodic points.  Fix one periodic point $\omega$, say of prime period $p \geq 1$, such that
$\xi \not\in \cup_{k=0}^\infty T^{-k}(\{T^{j}(\omega) \hbox{ : } 0 \leq j \leq p-1\})$.
Again because of topological exactness of $T$,
$$
\xi \in 
\overline{
(0, \xi) \cap \bu_{k=0}^\infty 
 T^{-k}(\{T^{j}(\omega) \hbox{ : } 0 \leq j \leq p-1\})
}
$$
and 
$$
\xi \in 
\overline{
(\xi,1) \cap \bu_{k=0}^\infty 
 T^{-k}(\{T^{j}(\omega) \hbox{ : } 0 \leq j \leq p-1\})
}.
$$
For every $n \geq 1$, sufficiently large denote by $\xi_n^- \in I$
the point closest to $\xi$ in 
$$
\overline{
(0, \xi) \cap \bu_{k=0}^n 
 T^{-k}(\{T^{j}(\omega) \hbox{ : } 0 \leq j \leq p-1\})}
$$
and by $\xi_n^+ \in I$
the point closest to $\xi$ in 
$$
\overline{
(\xi,1) \cap \bu_{k=0}^n 
 T^{-k}(\{T^{-j}(\omega) \hbox{ : } 0 \leq j \leq p-1\})}.
$$
We then denote  
$$V_n := (\xi^-_n, \xi^+_n).$$
Then obviously 
$\xi \in V_n$, $T^k(\xi^\pm_n) \not\in (\xi_n^-, \xi_n^+)$
for all $k = 0,1, \cdots, n-1$, and 
$T^k(\xi_n^{\pm}) \in \{ T^j(w) \hbox{ : } 0 \leq j \leq n-1\}$
for all $k \geq n$.
Since $\lim_{n\to +\infty} \xi_n^\pm = \xi$ 
it then follows that $T^k(\xi_n^\pm) \not\in V_n$ for all $k \geq n$.
In conclusion, $V_n$ are  the required nice sets for all integers $n \geq 1$.
Since in addition $\lim_{n \to +\infty} \hbox{\rm diam}(V_n)=0$ the proof is complete.
\end{proof}

From their definitions, nice sets enjoy the following property.

\begin{thm}\label{t1mi6}
If $V$ is a nice set for a multimodal map, then for every integer $n \geq 0$
and every $U \in \hbox{\rm Comp}(T^{-n}(V))$ either 
$$
U \cap V = \emptyset \hbox{ or } U \subset V.
$$
\end{thm}

From now on throughout this section we assume that $T: I \to I$ is a 
tame teTCE map.  Fix a point $\xi \in I \backslash \overline {PC(T)}$.
By virtue of Theorem \ref{t1mi4} there is a nice set $V$ such that 
$$
\xi \in V
\hbox{ and  }
2V \cap
\overline {\PC(T)} = \emptyset.
$$
The nice set $V$  canonically gives rise to a countable alphabet conformal iterated function system in the sense considered in the previous sections of the present paper. Namely, put 
$$
\hbox{\rm Comp}_*(V)=\bu_{n=1}^\infty\Comp(f^{-n}(V)).
$$
For every $U\in \hbox{\rm Comp}_*(V)$ let $\tau_V(U)\ge 1$ the unique integer $n\ge 1$
such that $U\in \hbox{\rm Comp}(f^{-n}(V))$. Put further
$$
\phi_U:=f_U^{-\tau_V(U)}:V\to U
$$
and keep in mind that
$$
\phi_U(V)=U.
$$
Denote by $E_V$ the subset of all elements $U$ of $\hbox{\rm Comp}_*(V)$ such that 

\begin{enumerate}
\item[(a)] $\phi_U(V)\subset V$, 
\item[(b)] $f^k(U)\cap V=\emptyset$ \, for all \, $k=1,2,\ld,\tau_V(U)-1$.
\end{enumerate}
 The collection 
$$
\mathcal S_V:=\{\phi_U:V\to V\}
$$
of all such inverse branches forms obviously an iterated function system in the sense considered in the previous sections of the present paper. In other words the elements of $\cS_V$ are formed by all  inverse branches of the first return map $f_V:V\to V$. In particular, $\tau_V(U)$ is the first return time of all points in $U=\phi_U(V)$ to $V$. We define the function $N_V:E_V^\infty\to\N_1$ by setting
$$
N_V(\omega):=\tau_V(\omega_1).
$$
Let 
$$
\pi_V:E_V^\infty\to\mathbb R
$$
be the canonical projection induced by the iterated function system $\cS_V$. Let
$$
J_V:= \pi_V\(E_V^\infty)
$$
be the limit set of the system $\cS_V$. Clearly
$$
J_V\sbt I.
$$
It is immediate from our definitions that
$$
\tau_V(\pi(\omega))=N_V(\omega)
$$
for all $\omega\in E_V^\N$.

We shall now prove the following.

\begin{prop}\label{p2mi6}
Let $T: I \to I$ be a tame {\rm teTCE} map.
Let $\psi: I \to \mathbb R$ be an acceptable potential. Then
\begin{enumerate}
\item[(a)]
$\widetilde \psi_V:= \psi_V \circ \pi_V 
- P(\psi) N_V: E^{\mathbb N} \to \mathbb R$
is a summable H\"older continuous potential;
\item[(b)]  $P(\sigma, \widetilde \psi_V)=0$ for the pressure 
for  the shift map $\sigma: E_V^{\mathbb N} \to E_V^{\mathbb N}$;
\item[(c)] $\mu_{\phi, V} = \mu_{\widetilde \psi_V} \circ \pi_V^{-1}$,
where $\mu_{\widetilde \psi_V}$ is the equilibrium state for $\widetilde \psi_V$
and the shift map $\sigma: E_V^{\mathcal N} \to E_V^{\mathcal N}$;
\item[(d)] In addition, $\psi_V$ is the amalgamated function of a summable H\"older continuous system of functions.
\end{enumerate}
\end{prop}

\begin{proof}
H\"older continuity of $\widetilde \psi_V$ follows directly from 
Theorem~\ref{t1mi2.1} (the Exponential Shrinking Property) and the fact that the function $N_V$ is constant on cylinders of length $1$.
%\footnote{Also to be added in the section on rational maps}
H\"older continuity of $\widetilde \psi_V$ follows directly from 
Theorem \ref{t1mi2.1} (the Exponential Shrinking Property) and the fact that the function $N_V$ is constant on cylinders of length $1$.
%\footnote{Also to be added in the section on rational functions}
We define a H\"older continuous system of functions $G = \{ g^{(l)}: V \to \mathbb R\}_{e\in E}$  by putting
$$
g^{(e)} := \(\psi_V - \P(\phi)\tau_V\) \circ \phi_e, \ e \in E.
$$
Theorem \ref{t1mi3} then implies 
%%%%%%%%%%%%%%%%%%%%%%%%%%%%%%%%%%%%%%%%%%%%%%%%%%%%%%%%%%%%%%%%%%%%%%%%
the system $G$ is summable, $\P(G)=0$, and $m_{\psi,V}$ is the unique $G$-conformal measure for the IFS $\mathcal S_V$. 
According to \cite{GDMS}, $g:E_V^\N\to\R$, the amalgamated function of $G$ is defined by the formula 
$$
\begin{aligned}
g(\om)
&=g^{(\om_1)}(\pi_V(\sg(\om)))
=\psi_V\circ\phi_{\om_1}(\pi_V(\sg(\om)))-\P(\psi)\tau_V\circ\phi_{\om_1} 
    (\pi_V(\sg(\om))) \\
&=\psi_V\circ\pi_V(\om)-\P(\psi)N_V(\om) \\
&=\tilde\psi_V(\om).
\end{aligned}
$$
By Proposition~3.1.4 in \cite{GDMS} we thus have that
$$
\P\(\sg,\tilde\psi_V\)=\P(G)=0.
$$
Now, since $\pi_V\circ\sg=T_V\circ\pi_V$, i.e. since the dynamical system $T_V:J_V\to J_V$ is a factor of the shift map $\sg:E_V^\N\to E_V^\N$ via the map $\pi_V:E_V^\N\to J_V$, we see that $\mu_{\tilde\psi_V}\circ\pi_V^{-1}$ is a Borel $f_V$-invariant probability measure on $J_V$ equivalent to $m_{\tilde\psi_V}\circ\pi_V^{-1}=m_g\circ\pi^{-1}=m_G=m_{\psi,V}$. Since $m_{\psi,V}$ is equivalent to $\mu_{\psi,V}$, we thus conclude that the measures $m_{\tilde\psi_V}\circ\pi_V^{-1}$ and $\mu_{\psi,V}$ are equivalent. Since both these measures are $T_V$-invariant and $\mu_{\psi,V}$ is ergodic, they must be equal. The proof is thus complete.
\epf

Since $\pi_V:E_V^\N\to J_V=V_\infty$, where, we recall the latter is the set of points returning infinitely often to $V$, is a measurable isomorphism sending the $\sg$-invariant measure $\mu_{\tilde\psi_V}$ to the $f_V$-invariant probability measure $\mu_{\psi,V}$, by identifying the sets $E_V^\N$ and $V_\infty(=J_V)$, we can prove the following.

\blem\label{l1ma5-T}
With all the hypotheses of Proposition~\ref{p2mi6}, 
the pentade $(I,T,V,\tilde\psi_V,\mu_{\tilde\psi_V})$ is
an {\rm SRT} system having exponential tail decay {\rm(ETD)}, where we recall that $V_\infty$ is identified with $E_V^\N$, $\tilde\psi_V$ is identified with $\psi_V-\P(\psi)\tau_V$, and $\mu_{\tilde\psi_V}$ is identified with $\mu_{\psi,V}$.
\elem

\bpf
By virtue of Proposition~\ref{p2mi6} and Observation~\ref{o2_2016_07_14} we only need to prove that the pentade $(I,T,V,\tilde\psi_V,\mu_{\tilde\psi_V})$ has exponential tail decay {\rm(ETD)}. We can assume without loss of generality that $\psi:I \to\R$ is normalized so that 
$$
\P(\psi)=0 \  \  \and  \  \  m_\psi=\mu_\psi.
$$
For every $n \ge 1$ denote by $\cC_V(n)$ all the connected components of $T^{-n}(V)$. Then define
$$
\cC_V^0(n)
:=\big\{U\in \cC_V(n):\forall_{(0\le k\le n-1)} \, T^k(U)\cap V=\es\big\}
$$
and 
$$
\cC_V^*(n):=\{U\in \cC_V(n):U\sbt V\}
=\{U\in \cC_V(n):U\cap V\ne\es\}.
$$
Since the map $T:I\to I$ is topologically exact, there exists an integer $q\ge 1$ such that
$$
T^q(V)\spt I.
$$
Therefore for every $e\in \cC_V(n)$ there exists (at least one) $\hat e\in\cC_V^*(n+q)$ such that 
$$
T^q\circ\phi_{\hat e}=\phi_e.
$$
By conformality of the measure $\mu_\psi$, for every $e\in \cC_V(n)$, we have
$$
\mu_\psi\(\phi_{\hat e}(V)\)
\ge \exp(-q||\psi||_\infty)\mu_\psi(\phi_e(V)).
$$
So, since
$$
\bu_{a\in\cC_V^0(n+q)}\phi_a(V)
\sbt \bu_{{b\in\cC_V(n+q)\atop T^q\circ \phi_b\in\cC_V^0(n)}}\phi_b(V) \sms
\bu_{e\in\cC_V^0(n)}\phi_{\hat e}(V),
$$ 
we therefore get
$$
\begin{aligned}
\mu_\psi\lt(\bu_{a\in\cC_V^0(n+q)}\phi_a(V)\rt)
&\le \mu_\psi\lt(\bu_{{b\in\cC_V(n+q)\atop T^q\circ 
    \phi_b\in\cC_V^0(n)}}\phi_b(V) \sms\bu_{e\in\cC_V^0(n)}
    \phi_{\hat e}(V)\rt)\\
&= \mu_\psi\lt(\bu_{{b\in\cC_V(n+q)\atop f^q\circ 
    \phi_b\in\cC_V^0(n)}}\phi_b(V)\rt) - 
    \mu\lt(\bu_{e\in\cC_V^0(n)}\phi_{\hat e}(V)\rt) \\
&= \mu_\psi\lt(T^{-q}\lt(\bu_{c\in\cC_V^0(n)}\phi_c(V)\rt)\rt)-
   \sum_{e\in\cC_V^0(n)}\mu_\psi\(\phi_{\hat e}(V)\) \\
&= \mu_\psi\lt(\bu_{c\in\cC_V^0(n)}\phi_c(V)\rt)-
   \sum_{e\in\cC_V^0(n)}\mu_\psi\(\phi_{\hat e}(V)\) \\
&\le \mu_\psi\lt(\bu_{c\in\cC_V^0(n)}\phi_c(V)\rt)-
   \exp(-q||\psi||_\infty)\sum_{e\in\cC_V^0(n)}\mu_\psi\(\phi_e(V)\) \\
&=\g\mu_\psi\lt(\bu_{c\in\cC_V^0(n)}\phi_c(V)\rt),
\end{aligned}
$$
where $\g:=1-\exp(-q||\psi||_\infty)\in[0,1)$. An immediate induction then yields
$$
\mu_\psi\lt(\bu_{e\in\cC_V^0(n)}\phi_e(V)\rt)\le \g^{-1}\g^{n/q}
$$
for all $n\ge 0$. But, as
$$
E_V^{-1}([n,+\infty])
=E_V^{-1}(\{+\infty\})\cup\bu_{k=n}^\infty\bu_{e\in\cC_V^0(k)}\phi_e(V)
$$
and since $\mu_\psi\(E_V^{-1}(\{+\infty\})\)=0$ by ergodicty of $\mu_\psi$ and of $\mu_\psi(V)>0$, we therefore get that
\beq\label{1_2016_07_14}
\mu_\psi\(E_V^{-1}([n,+\infty])\)
\le \(\g(1-\g^{1/q})\)^{-1}\g^{n/q}
\eeq
for all $n\ge 0$. This just means that the pentade $(I,T,V,\tilde\psi_V,\mu_{\tilde\psi_V})$ has exponential tail decay {\rm(ETD)},
and the proof is complete.  
\epf

\sp\fr Denote by $I_R(T)$ the set of all recurrent points of $T$ in $I$. Formally
$$
I_R(T):=\{z\in I:\varliminf_{n\to\infty}|T^n(z)-z|=0\}.
$$
Of course $I_R(T)\sbt J_T$ and $\mu_\psi(I\sms I_R(T))=0$ because of Poincar\'e's Recurrence Theorem. The set $I_R(T)$ is significant for us since
$$
I_R(T)\cap V\sbt J_V.
$$
Now we can harvest the fruits of the work we have done. As a direct consequence of Theorem~\ref{t1fp83}, Theorem~\ref{t3_2016_05_27}, Proposition~\ref{p2mi6}, Lemma~\ref{l1ma5-T}, Lemma~\ref{l1ld7}, and Theorem~\ref{t2rmr3}, we get the following two results.

\bthm\label{t5_2016_05_31}
Let $T: I \to I$ be a tame {\rm teTCE} map.
Let $\psi:  I \to \mathbb R$ be an acceptable potential.
Let $z \in I_R(T) \backslash \overline{\PC(T)}$.  

Assume that the equilibrium state $\mu_\psi$ is {\rm(WBT)} at $z$. Then
$$
\begin{aligned}
\lim_{\ep\to 0}\frac{\un R_{\mu_\psi}(B(z,\ep))}{\mu_\psi(B(z,\ep))}
&=\lim_{\ep\to 0}\frac{\ov R_{\mu_\psi}(B(z,\ep))}{\mu_\psi(B(z,\ep))}= \\
&=\begin{cases}
  &{\rm if} \ z \ \text{{\rm is not any periodic point of }} T, \\
1-\exp\(S_p\psi(z)-p\P(f,\psi)\) &{\rm if} \ z  \ \text{{\rm is a periodic point of }} T.  
\end{cases}
\end{aligned}
$$
\ethm

\sp

\bthm\label{t1mi7}
Let $T: I \to I$ be a tame {\rm teTCE} map.
Let $\psi:  I \to \mathbb R$ be an acceptable potential. 
Then
$$
\lim_{\ep\to 0}\frac{\un R_{\mu_\psi}(B(z,\ep))}{\mu_\psi(B(z,\ep))}
=\lim_{\ep\to 0}\frac{\ov R_{\mu_\psi}(B(z,\ep))}{\mu_\psi(B(z,\ep))}
=1
$$
for $\mu_\psi$--a.e. point $z\in I$. 
\ethm

%\brem\label{r1ma10.1}
%Theorem~\ref{t1ma7}holds in fact for a larger set than $J_R(f)$. %Indeed, it holds for every point in $V\cap J_{\cS_V}$, where $V$ is %an arbitrary nice set.
%\erem

\begin{dfn}\label{d2mi7}
A multimodal map $T: I \to I$  is called \emph{subexpanding} if
$$
\hbox{\rm Crit}(T) \cap \overline{\PC(T)} = \emptyset.
$$
\end{dfn}

It is not hard to see (good references for a proof can be found in
\cite{PR-Geometric}) that the following it true.

\begin{prop}\label{p3mi7}
Any topologically exact multimodal subexpanding map of the interval $I$
is a tame teTCE map.
\end{prop}

Let us quote another well-known result which can be found, for example, in the book of de Melo and van Strien \cite{MeSt}.

\begin{thm}\label{t1mi8}
If $T: I \to I$  is a topologically exact multimodal subexpanding map,
then there exists a unique Borel probability $T$-invariant measure $\mu$
absolutely continuous with respect to Lebesgue measure $\lambda$.
In fact, 
\begin{itemize}
\item[(a)] $\mu$ is equivalent to $\lambda$ and (therefore) 

\item[(b)] has full topological support.  

\item[(c)] The Radon--Nikodym derivative $\frac{d\mu}{d\lam}$ is uniformly bounded above and separated from zero on the complement of every fixed neighborhood of $\ov{\PC(T)}$.

\item[(d)] $\mu$ is ergodic, even $K$-mixing, 

\item[(e)] $\mu$ has Rokhlin's natural extension metrically isomorphic to some two sided Bernoulli shift and

\item[(f)] $\mu$ charges with full measure both topologically transitive and radial points of $T$.
\end{itemize}
\end{thm}

\fr As an immediate consequence of this theorem, particularly of its item (c), we get the following.

\bcor\label{c1_2016_07_06}
If $T:I \to I$  is a topologically exact multimodal subexpanding map,
then the $T$-invariant measure $\mu$ absolutely continuous with respect to Lebesgue measure $\lambda$ is (\WBT) at every point of $I\sms \ov{\PC(T)}$.
\ecor

\sp Passing to escape rates, by a small obvious modification (see \cite{PR-Geometric} for details) of the proof of Theorem \ref{t3mi4} for all $c \in \hbox{\rm Crit}(T) \cup \{\xi\}$
there are arbitrarily small open intervals $V_c$, $c \in V_c$, such that
$V_c \cap \overline{PC}(T) = \emptyset$ and the 
collection $T_*^{-n}$, $n \geq 1$, of all continuous 
(equivalently smooth inverse branches of $T^n$) defined on $V_c$, $c \in \hbox{\rm Crit}(T) \cup \{ \xi\}$, and such that for some 
$c' \in \hbox{\rm Crit}(T) \cup \{ \xi\}$, 
$$
T_*^{-n}(V_c) = V_{c'}
$$
and 
$$ 
\bu_{k=1}^{n-1}  T^k(T_*^{-n}(V_c)) \cap \bu\big\{ V_z \hbox{ : } 
z \in \hbox{\rm Crit}(T) \cup \{ \xi\}\big\} = \emptyset
$$
forms a finitely primitive conformal GDS, which we will call $\mathcal S_T$,
whose limit set contains $\hbox{\rm Trans}(T)$.
Another characterization of $\mathcal S_T$ is that its elements 
are composed of continuous inverse branches of the first return map of $f$
from 
$$
V:= \bu \big\{V_z \hbox{ : } z \in \hbox{\rm Crit}(T) \cup \{ \xi\} \big\}
$$
to $V$. It has been proved in \cite{PR-Geometric} that $\HD(K(V)) < 1$.

%%%%%%%%%%%%%%%%%%%%%%%%%%%%%%%%%%%%%%%%%%%%%%%%%%%%%%%%%%%%
 So, since by Theorem~\ref{t2had30}, $\lim_{r\to 0}\HD\(K(B(\xi,r))\)=h$, we conclude that
$$
\HD(K(V))<\HD\(K(B(\xi,r))\)
$$
for all $r>0$ small enough. Therefore, since $b_{\cS_T}=1$ and since $\mu_{h,V}=\mu_{b_{\cS_f}}$, applying Theorem~\ref{t2had30}, Corollary~\ref{t2_2016_05_27} and Corollary~\ref{c1rm3}, we get the following two theorems.

\begin{thm}\label{t2mi8}
Let $T: I \to I$  be a topologically exact multimodal subexpanding map.
Fix $\xi \in I \backslash \overline{\PC(T)}$. Assume that the parameter $1$ is powering at $\xi$ with respect to the conformal \GDS \ $\cS_T$. Then the following limit exists, is finite, 
and positive:
$$
\lim_{r \to 0}  \frac{1 - \HD(K_\xi(r))}{\mu(B(\xi,r))}.
$$ 
\end{thm}

\begin{thm}\label{t2mi8+1}
If $T: I \to I$  is a topologically exact multimodal subexpanding map, then for Lebesgue--a.e. point $\xi\in I\sms \ov{\PC(T)}$ the following limit exists, is finite and positive:
$$
\lim_{r \to 0}  \frac{1 - \HD(K_\xi(r))}{\mu(B(\xi,r))}.
$$ 
\end{thm}

\section{Escape Rates for Rational Functions of the Riemann Sphere}\label{Rational Functions}

Now, we will apply the results of sections 14 and 15 to two large classes of conformal dynamical systems in te complex plane: rational functions of the Riemann sphere $\oc$ in this section and, in the next section, transcendental meromorphic functions on $\C$. This section considerably overlaps in some of its parts with the previous section on the multimodal interval maps. We provide here its full exposition for the sake of coherent completeness and convenience of the readers not necessarily interested in interval maps.

\sp As said, now we deal with rational functions. Let $f:\oc\to\oc$ be a rational function of degree $d\ge 2$. Let $J(f)$ denote the Julia sets of $f$ and let
$$
\Crit(f):=\{c\in\oc:f'(c)=0\}
$$
be the set of all critical (branching) points of $f$. Put
$$
\PC(f):=\bu_{n=1}^\infty f^n(\Crit(f))
$$
and call it the postcritical set of $f$. The best understood and the easiest (nowadays) to deal with class of rational functions is formed by expanding (also frequently called hyperbolic) maps. The rational map $f:\oc\to\oc$ is said to be expanding if the restriction  $f|_{J(f)}: J(f) \to J(f)$ satisfies 
\beq\label{5_2016_07_07}
\inf\{|f'(z)|:z\in J(f)\} > 1
\eeq
or, equivalently, 
\beq\label{6_2016_07_07}
|f'(z)|>1
\eeq
for all $z\in J(f)$. Another, topological, characterization of expandingness is this.

\bfact
A rational function $f:\oc\to\oc$ is expanding if and only if 
$$
J(f)\cap\ov{\PC(f)}=\es.
$$
\efact
It is immediate from this characterization that all the polynomials $z\mapsto z^d$, $d\ge 2$, are expanding along with their small perturbations $z\mapsto z^d+\e$; in fact expanding rational functions are commonly believed to form a vast majority amongst all rational functions. This is known at least for polynomials with real coefficients. We however do not restrict ourselves to expanding rational maps only. We start with all rational functions, no restriction whatsoever, and then make some, weaker than hyperbolicity, appropriate assumptions.

Let $\psi:\oc\to\R$ be a H\"older continuous function, referred to in the sequel as potential. We say that $\psi:\oc\to\R$ has a pressure gap if
$$
n\P(\psi)-\sup\(\psi_n\)>0  
$$
for some integer $n\ge 1$, where $\P(\psi)$ denotes the ordinary topological pressure of $\psi|_{J(f)}$ and the Birkhoff's sum $\psi_n$ is also considered as restricted to $J(f)$. 

We would like to mention that \eqref{1_2016_07_07} always holds (with all $n\ge 0$ sufficiently large) if the function $f:\oc\to\oc$ restricted to its Julia set is expanding (also frequently referred to as hyperbolic).

The probability invariant measure we are interested in comes from the following.

\bthm[\cite{DU}]\label{t1ma1}
If $f:\oc\to\oc$ is a rational function of degree $d\ge 2$ and if $\psi:\oc\to\R$ is a H\"older continuous potential with a pressure gap, then $\psi$ admits a unique equilibrium state $\mu_\psi$, i.e. a unique Borel probability $f$-invariant measure on $J(f)$ such that
$$
\P(\psi)=\h_{\mu_\psi}(f)+\int_{J(f)}\psi\,d\mu_\psi.
$$
In addition, 

\sp\begin{itemize}
\item [(a)] the measure $\mu_\psi$ is ergodic, in fact K-mixing, and (see \cite{SUZ_I}) enjoys further finer stochastic properties.

\item [(b)] The Jacobian 
$$
J(f)\ni z\longmapsto \frac{d\mu_\psi\circ T}{d\mu_\psi}(z)\in (0,+\infty)
$$ 
is a H\"older continuous function.
\end{itemize}
\ethm

\fr In \cite{PU_tame} a rational function $f:\oc\to\oc$ was called tame if
$$
J(f)\sms \ov{\PC(f)}\ne\es.
$$
Likewise, following \cite{SkU}, we adopt the same definition for (transcendental) meromorphic functions $f:\C\to\oc$.

\begin{rem}
Tameness is a very mild hypothesis and there are many classes of maps foe which these hold.  These include:
\begin{enumerate}
\item Quadratic maps $z \mapsto z^2 + c$ for which the Julia set is not contained in the real line;
\item Rational maps for which the restriction to the Julia set is expansive which includes the case of expanding rational functions; and 
\item Misiurewicz maps, where the critical point is not recurrent.  
\end{enumerate}
\end{rem}

 In this paper the main advantage of dealing with tame functions is that these admit Nice Sets. Let us define and discuss them now.

Given a set $F\sbt\oc$ and $n\ge 0$,we denote  by $\Comp(f^{-n}(F))$ 
the collection of all connected components of $f^{-n}(F)$. J. Rivera-Letelier introduced in
\cite{Riv07} the concept of Nice Sets in the realm of the dynamics of
rational maps of the Riemann sphere. In \cite{Dob11} N. Dobbs proved
their existence for tame meromorphic functions from $\C$ to $\oc$. We
quote now his theorem.
\bthm\label{prop:1}
Let $f:\C\to\oc$ be a tame meromorphic function. Fix a non-periodic point $z\in J(f)\sms
\ov{\PC(f)}$, $\kappa>1$, and $K>1$. Then for all $L>1$ and 
for all $r>0$ sufficiently small there exists
an open connected set $V=V(z,r)\sbt\C\sms\ov{\PC(f)}$ such that
  \begin{itemize}
  \item[(a)] If $U\in \Comp(f^{-n}(V))$ and $U\cap V\neq \emptyset$, then 
    $U\subseteq V$.
  \item[(b)] If $U\in \Comp(f^{-n}(V))$ and $U\cap v\neq \emptyset$,
    then, for all $w,w'\in U,$ 
    \begin{displaymath}
      |(f^n)'(w)|\ge L
\  \textrm{ and } \
       \frac{|(f^n)'(w)|}{|(f^n)'(w')|}\le K. 
    \end{displaymath}
  \item[(c)] $\overline{B(z,r)}\subset U\subset B(z,\kappa r)\sbt\C\sms\ov{\PC(f)}$.
\end{itemize}
\ethm

\fr Each nice set canonically gives rise to a countable alphabet conformal iterated function system in the sense considered in the previous sections of the present paper. Namely, put 
$$
\Comp_*(V)=\bu_{n=1}^\infty\Comp(f^{-n}(V)).
$$
For every $U\in \Comp_*(V)$ let $\tau_V(U)\ge 1$ the unique integer $n\ge 1$
such that $U\in \Comp(f^{-n}(V))$. Put further
$$
\phi_U:=f_U^{-\tau_V(U)}:V\to U
$$
and keep in mind that
$$
\phi_U(V)=U.
$$
Denote by $E_V$ the subset of all elements $U$ of $\Comp_*(V)$ such that 

\sp\begin{itemize}

\item[(a)] $\phi_U(V)\sbt V$, 

\sp\item[(b)] $f^k(U)\cap V=\es$ \, for all \, $k=1,2,\ld,\tau_V(U)-1$.
\end{itemize}

\sp\fr The collection 
$$
\cS_V:=\{\phi_U:V\to V\}
$$
of all such inverse branches forms obviously a conformal iterated function system in the sense considered in the previous sections of the present paper. In other words the elements of $\cS_V$ are formed by all holomorphic inverse branches of the first return map $f_V:V\to V$. In particular, $\tau_V(U)$ is the first return time of all points in $U=\phi_U(V)$ to $V$. We define the function $N_V:E_V^\N\to\N_1$ by setting
$$
N_V(\om):=\tau_V(\om_1).
$$
Let 
$$
\pi_V:E_V^\N\to\oc
$$
be the canonical projection induced by the iterated function system $\cS_V$. Let
$$
J_V:\pi_V\(E_V^\N\)
$$
be the limit set of the system $\cS_V$. Clearly
$$
J_V\sbt J(f).
$$
It is immediate from our definitions that
$$
\tau_V(\pi(\om))=N_V(\om)
$$
for all $\om\in E_V^\N$. 

\sp Now, having in addition a H\"older continuous potential $\psi:\oc\to\R$ with pressure gap, we already know from the previous sections that $\mu_{\psi,V}$, the conditional measure of $\mu_\psi$ on $V$ is $f_V$-invariant and ergodic. 

\bdfn\label{d_ESP_Rat_Fun}
We say that the rational function $f:\oc\to\oc$ has the Exponential Shrinking Property (ESP) if there exist $\d>0$, $\g>0$, and $C>0$ such that if $z\in J(f)$ and $n\ge 0$, then 
\beq\label{2_2016_07_07}
\diam(W)\le Ce^{-\g n}
\eeq
for each $W\in \Comp\(f^{-n}(B(z,2\d))\)$. 
\edfn

\begin{rem}
This property has been throughly explored in the papers including \cite{PrzRiv07} and the references therein. These papers provide several different characterizations of Exponential Shrinking Property, most notably the one called Topological Collet-Eckmann; one of them being uniform hyperbolicity of periodic points in the Julia set. We do not recall any more of them here as we will only need (ESP). 

We now however articulate two standard sufficient conditions for (ESP) to hold. It is implied by the Collet-Eckmann condition which requires that there exist $\lambda > 1$ and $C>0$ such that for every integer $n \geq 0$ we have that
$$
|(f^n)'(f(c))| \geq C \lambda^n.
$$

If also suffices for (ESP) to hold to assume that a rational map is semi-hyperbolic, i.e., that no critical point $c$ in the Julia belongs to its own omega limit set $\om(c)$. This so for example, if $T$ is a classical unimodal map of the form $I \ni x \mapsto \lambda x (1-x)$, with 
$0 < \lambda \leq 4$ and the critical point $1/2$ is not in its own omega limit set, i.e., $1/2 \not\in \omega (1/2)$.  

Last observation: all expanding rational functions have the Exponential Shrinking Property (ESP).
\end{rem}

\fr We shall prove the following.

\bprop\label{p1ma3}
Let $f:\oc\to\oc$ be a tame rational function satisfying {\rm (ESP)}. Let $\psi:\oc\to\R$ be a H\"older continuous potential with pressure gap. If $V$ is a nice set for $f$, then

\begin{itemize}
\item[(a)]
$$
\tilde\psi_V:=\psi_V\circ\pi_V-\P(\psi)N_V:E_V^\N\to\R
$$
is a H\"older continuous potential,

\sp\item[(b)] $\P\(\sg,\tilde\psi_V\)=0$,

\sp\item[(c)] 
$$
\mu_{\psi,V}=\mu_{\tilde\psi_V}\circ\pi_V^{-1},
$$
where $\mu_{\tilde\psi_V}$ is the equilibrium/Gibbs state for the potential $\tilde\psi_V$ and the shift map $\sg:E_V^\N\to E_V^\N$.

\sp\item[(d)] In addition, $\tilde\psi_V$ is the amalgamated function of a summable 
H\"older continuous system of functions.
\end{itemize}
\eprop

\bpf
H\"older continuity of $\tilde\psi_V$ follows directly from (ESP) i.e Definition~\ref{d_ESP_Rat_Fun}, and the fact that the function $N_V$ is constant on cylinders of length one. Now, it follows from \cite{DU} that there exists a unique $\exp(\P(\psi)-\psi)$-conformal measure on $J(f)$, i.e. a Borel 
probability measure $m_\psi$ on $J(f)$ such that
$$
m_\psi\(f(A)\)=e^{\P(\psi)}\int_Ae^{-\psi}\,dm_\psi
$$
for every Borel set $A\sbt J(f)$ such that the map $f|_A$ is 1-to-1. In addition $m_\psi$ is equivalent to $\mu_\psi$ with logarithmically bounded H\"older continuous Radon-Nikodym derivative. It immediately follows from this formula that for every $e\in E_V$ and every Borel set $A\sbt V$, we have that
\beq\label{1ma4}
m_{\psi,V}\(\vp_e(A)\)=\int_A\exp\((\psi_V-\P(\psi)\tau_V)\circ\phi_e\)\,dm_{\psi,V},
\eeq
where $m_{\psi,V}$ is the conditional measure of $m_\psi$ on $V$. Now we define a H\"older continuous system of functions $G:=\{g^{(e)}:V\to\R\}_{e\in E}$ by putting
$$
g^{(e)}:=(\psi_V-\P(\psi)\tau_V)\circ\phi_e, \  \  e\in E_V.
$$
Formula \eqref{1ma4} thus means that the system $G$ is summable, $\P(G)=0$, and $m_{\psi,V}$ is the unique $G$-conformal measure for the IFS $\cS_V$. According to \cite{GDMS}, $g:E_V^\N\to\R$, the amalgamated function of $G$ is defined by the formula 
$$
\begin{aligned}
g(\om)
&=g^{(\om_1)}(\pi_V(\sg(\om)))
=\psi_V\circ\phi_{\om_1}(\pi_V(\sg(\om)))-\P(\psi)\tau_V\circ\phi_{\om_1} 
    (\pi_V(\sg(\om))) \\
&=\psi_V\circ\pi_V(\om)-\P(\psi)N_V(\om) \\
&=\tilde\psi_V(\om).
\end{aligned}
$$
By Proposition~3.1.4 in \cite{GDMS} we thus have that
$$
\P\(\sg,\tilde\psi_V\)=\P(G)=0.
$$
Now, since $\pi_V\circ\sg=f_V\circ\pi_V$, i.e. since the dynamical system $f_V:J_V\to J_V$ is a factor of the shift map $\sg:E_V^\N\to E_V^\N$ via the map $\pi_V:E_V^\N\to J_V$, we see that $\mu_{\tilde\psi_V}\circ\pi_V^{-1}$ is a Borel $f_V$-invariant probability measure on $J_V$ equivalent to $m_{\tilde\psi_V}\circ\pi_V^{-1}=m_g\circ\pi^{-1}=m_G=m_{\psi,V}$. Since $m_{\psi,V}$ is equivalent to $\mu_{\psi,V}$, we thus conclude that the measures $m_{\tilde\psi_V}\circ\pi_V^{-1}$ and $\mu_{\psi,V}$ are equivalent. Since both these measures are $f_V$-invariant and $\mu_{\psi,V}$ is ergodic, they must be equal. The proof is thus complete.
\epf

Since $\pi_V:E_V^\N\to J_V=V_\infty$, where, we recall the latter is the set of points returning infinitely often to $V$, is a measurable isomorphism sending the $\sg$-invariant measure $\mu_{\tilde\psi_V}$ to the $f_V$-invariant probability measure $\mu_{\psi,V}$, by identifying the sets $E_V^\N$ and $V_\infty(=J_V)$, we can prove the following.

\blem\label{l1ma5}
With the hypotheses of Proposition~\ref{p1ma3}, the pentade $(J(f),f,V,\tilde\psi_V,\mu_{\tilde\psi_V})$ is an {\rm SRT} system and has exponential tail decay {\rm(ETD)}, where we recall that $V_\infty$ is identified with $E_V^\N$, $\tilde\psi_V$ is identified with $\psi_V-\P(\psi)\tau_V$, and $\mu_{\tilde\psi_V}$ is identified with $\mu_{\psi,V}$.
\elem

\bpf
By virtue of Proposition~\ref{p1ma3} and Observation~\ref{o2_2016_07_14} we only need to prove that the pentade $(J(f),f,V,\tilde\psi_V,\mu_{\tilde\psi_V})$ has exponential tail decay {\rm(ETD)}. We can assume without loss of generality that $\psi:\oc\to\R$ is normalized so that 
$$
\P(\psi)=0 \  \  \and  \  \  m_\psi=\mu_\psi.
$$
For every $\ge 1$ denote by $\cC_V(n)$ the set of all connected components of $T^{-n}(V)$. Then define
$$
\cC_V^0(n)
:=\big\{U\in \cC_V(n):\forall_{(0\le k\le n)} \, f^k(U)\cap V=\es\big\}
$$
Since the map $f:J(f)\to J(f)$ is topologically exact, there exists an integer $q\ge 1$ such that
$$
f^q(V)\spt J(f).
$$
Therefore for every $e\in \cC_V(n)$ there exists (at least one) $\hat e\in\cC_V^*(n+q)$ such that 
$$
f^q\circ\phi_{\hat e}=\phi_e.
$$
By conformality of the measure $\mu_\psi$, for every $e\in \cC_V(n)$, we have
$$
\mu_\psi\(\phi_{\hat e}(V)\)
\ge \exp(-q||\psi||_\infty)\mu_\psi(\phi_e(V)).
$$
So, since
$$
\bu_{a\in\cC_V^0(n+q)}\phi_a(V)
\sbt \bu_{{b\in\cC_V(n+q)\atop f^q\circ \phi_b\in\cC_V^0(n)}}\phi_b(V) \sms
\bu_{e\in\cC_V^0(n)}\phi_{\hat e}(V),
$$ 
we therefore get
$$
\begin{aligned}
\mu_\psi\lt(\bu_{a\in\cC_V^0(n+q)}\phi_a(V)\rt)
&\le \mu_\psi\lt(\bu_{{b\in\cC_V(n+q)\atop f^q\circ 
    \phi_b\in\cC_V^0(n)}}\phi_b(V) \sms\bu_{e\in\cC_V^0(n)}
    \phi_{\hat e}(V)\rt)\\
&= \mu_\psi\lt(\bu_{{b\in\cC_V(n+q)\atop f^q\circ 
    \phi_b\in\cC_V^0(n)}}\phi_b(V)\rt) - 
    \mu\lt(\bu_{e\in\cC_V^0(n)}\phi_{\hat e}(V)\rt) \\
&= \mu_\psi\lt(f^{-q}\lt(\bu_{c\in\cC_V^0(n)}\phi_c(V)\rt)\rt)-
   \sum_{e\in\cC_V^0(n)}\mu_\psi\(\phi_{\hat e}(V)\) \\
&= \mu_\psi\lt(\bu_{c\in\cC_V^0(n)}\phi_c(V)\rt)-
   \sum_{e\in\cC_V^0(n)}\mu_\psi\(\phi_{\hat e}(V)\) \\
&\le \mu_\psi\lt(\bu_{c\in\cC_V^0(n)}\phi_c(V)\rt)-
   \exp(-q||\psi||_\infty)\sum_{e\in\cC_V^0(n)}\mu_\psi\(\phi_e(V)\) \\
&=\g\mu_\psi\lt(\bu_{c\in\cC_V^0(n)}\phi_c(V)\rt),
\end{aligned}
$$
where $\g:=1-\exp(-q||\psi||_\infty)\in[0,1)$. An immediate induction then yields
$$
\mu_\psi\lt(\bu_{e\in\cC_V^0(qn)}\phi_e(V)\rt)\le \g^n
$$
for all $n\ge 0$.  An immediate induction then yields
$$
\mu_\psi\lt(\bu_{e\in\cC_V^0(n)}\phi_e(V)\rt)\le \g^{-1}\g^{n/q}
$$
for all $n\ge 0$. But, as
$$
E_V^{-1}([n,+\infty])
=E_V^{-1}(\{+\infty\})\cup\bu_{k=n}^\infty\bu_{e\in\cC_V^0(k)}\phi_e(V)
$$
and since $\mu_\psi\(E_V^{-1}(\{+\infty\})\)=0$ by ergodicty of $\mu_\psi$ and of $\mu_\psi(V)>0$, we therefore get that
\beq\label{1_2016_07_14}
\mu_\psi\(E_V^{-1}([n,+\infty])\)
\le \(\g(1-\g^{1/q})\)^{-1}\g^{n/q}
\eeq
for all $n\ge 0$. This just means that the pentade $(I,f,V,\tilde\psi_V,\mu_{\tilde\psi_V})$ has exponential tail decay {\rm(ETD)}, and the proof is complete.
\epf

\sp\fr Denote by $J_R(f)$ the set of all recurrent points of $f$ in $J(f)$. Formally
$$
J_R(f):=\{z\in J(f):\varliminf_{n\to\infty}|f^n(z)-z|=0\}.
$$
Of course $J_R(f)\sbt J_f$ and $\mu_\psi(J(f)\sms J_R(f))=0$ because of Poincar\'e's Recurrence Theorem. The set $J_R(f)$ is significant for us since
$$
J_R(f)\cap V\sbt J_V.
$$
Now we can now apply  the conclusions of the work done. As a direct consequence of Theorem~\ref{t1fp83}, Proposition~\ref{p1ma3}, Lemma~\ref{l1ma5}, Lemma~\ref{l1ld7}, and Theorem~\ref{t2rmr3}, we get the following.

\bthm\label{t1ma7}
Let $f:\oc\to\oc$ be a tame rational function having the exponential shrinking property (ESP). Let $\psi:\oc\to\R$ be a H\"older continuous potential with pressure gap. Let $z\in J_R(f)\sms \ov{\PC(f)}$. Assume that the equilibrium state $\mu_\psi$ is {\rm(WBT)} at $z$. Then
$$
\begin{aligned}
\lim_{\ep\to 0}\frac{\un R_{\mu_\psi}(B(z,\ep))}{\mu_\psi(B(z,\ep))} 
&=\lim_{\ep\to 0}\frac{\ov R_{\mu_\psi}(B(z,\ep))}{\mu_\psi(B(z,\ep))} \\
&=
\begin{cases}
1 &{\rm if} \ z \ \text{{\rm is not any periodic point of }} f, \\
1-\exp\(S_p\psi(z)-p\P(f,\psi)\) &{\rm if} \ z  \ \text{{\rm is a periodic point of }} f.
\end{cases}
\end{aligned}
$$
\ethm

\brem\label{r1ma10.1+1}
Theorem~\ref{t1ma7}holds in fact for a larger set than $J_R(f)$. Indeed, it holds for every point in $V\cap J_{\cS_V}$, where $V$ is an arbitrary nice set.
\erem

\fr As a fairly immediate consequence of Theorem~\ref{t1ma7} and Theorem~\ref{t1wbt6}, we get the following.

\bcor\label{t1ma7B}
Let $f:\oc\to\oc$ be a tame rational function having the exponential shrinking property (ESP) whose Julia set $J(f)$ is geometrically irreducible. If $\psi:\oc\to\R$ is a H\"older continuous potential with pressure gap, then 
$$
\lim_{\ep\to 0}\frac{\un R_{\mu_\psi}(B(z,\ep))}{\mu_\psi(B(z,\ep))} 
=\lim_{\ep\to 0}\frac{\ov R_{\mu_\psi}(B(z,\ep))}{\mu_\psi(B(z,\ep))} 
=1 
$$
for $\mu_\psi$--a.e. $z\in J(f)$.
\ecor

\fr Indeed in order to prove this corollary  it suffices to note that if the Julia set $J(f)$ is geometrically irreducible, then neither is the limit set of the iterated function system constructed in the arguments leading to Theorem~\ref{t1ma7}.

\brem 
We would like to note that if the rational function $f:\oc\to\oc$ is expanding, then it is tame, satisfies (ESP), and each H\"older continuous potential has pressure gap. In particular the two above theorems hold for it. 
\erem

\sp Now turn to  the asymptotics of Hausdorff dimension. We recall the following.

\bdfn\label{d1ma8}
Let $f:\oc\to\oc$ be a rational function of degree $d\ge 2$. We say that the map $f$ is sub-expanding if one of the following two equivalent conditions holds:

\sp
\begin{itemize}
\item[(a)]
$$
\ov{\bu_{n=0}^\infty f^n\(\Crit(f)\sms J(f)\)}\cap J(f)=\es 
\  \  \and \  \
\Crit(f)\cap\ov{\bu_{n=1}^\infty f^n\(\Crit(f)\cap J(f)\)}=\es,
$$

\sp\item[(b)]
$$
\Crit(f)\cap\ov{\bu_{n=1}^\infty f^n\(\Crit(f)\cap J(f)\)}=\es
\  \  \text{{\rm and $f$ has no rationally indifferent periodic points.}}
$$
\end{itemize}
\edfn

\fr Let
$$
h:=\HD(J(f)).
$$
It was proved in \cite{U1} and \cite{U2} that there exists a unique $h$--conformal measure $m_h$ on $J(f)$ for $f$ and a unique $f$-invariant (ergodic) measure $\mu_h$ on $J(f)$ equivalent to $m_h$. In addition $\mu_h$ is supported on the intersection of the transitive and radial points of $f$. It has been proved in \cite{U2} that any subexpanding rational function enjoys ESP. It therefore follows from \cite{PrzRiv07} that there are arbitrarily small open
connected sets $V_c$, $c\in J(f)\cap \Crit(f)$, and $V_\xi$, respectively containing points $c$ and $\xi$ such that the collection of all holomorphic inverse branches $f_*^{-n}$ of $f^n$, $n\ge 0$, defined on $V_z$, $z\in (J(f)\cap \Crit(f))\cup \{\xi\}$, and such that for some $z'\in (J(f)\cap \Crit(f))\cup \{\xi\}$,
$$
f_*^{-n}(V_z)\sbt V_{z'} 
$$
and 
$$
\bu_{k=1}^{n-1}f^k\(f_*^{-n}(V_z)\)\ \cap \  \bu\big\{V_w:w\in (J(f)\cap \Crit(f))\cup \{\xi\}\big\}=\es. 
$$
forms a finitely primitive conformal GDS, call it $\cS_f$. Another characterization of $\cS_f$ is that its elements are composed of analytic inverse branches of the first return map of $f$ from 
$$
V:=\bu\big\{V_w:w\in (J(f)\cap \Crit(f))\cup \{\xi\}\big\}
$$
$V$. It has been proved in \cite{StU1} and \cite{StU1} that the system $\cS_f$ is strongly regular. It follows from Lemma~6.2 in \cite{PrzRiv07} that $\HD(K(V))<h$. So, as by Theorem~\ref{t2had30}, $\lim_{r\to 0}\HD\(K(B(\xi,r))\)=h$, we conclude that
$$
\HD(K(V))<\HD\(K(B(\xi,r))\)
$$
for all $r>0$ small enough. Therefore, since $h=b_{\cS_f}$ and since $\mu_{h,V}=\mu_{b_{\cS_f}}$, applying Theorem~\ref{t2had30}, Corollary~\ref{t2_2016_05_27}, and Corollary~\ref{c1rm3}, we get the following two theorems.

\bthm\label{t2ma8}
Let $f:\oc\to\oc$ be a subexpanding rational function of degree $d\ge 2$. Fix $\xi\in J(f)\sms \ov{\PC(f)}$. Assume that the measure $\mu_h$ is {\rm (WBT)} at $\xi$ and the parameter $h$ is powering at $\xi$ with respect to the conformal \GDS \ $\cS_f$. Then the following limit exists, is finite and positive:
$$
\lim_{r\to 0}\frac{\HD(J(f))-\HD(K_\xi(r))}{\mu_h(B(\xi,r))}.
$$
\ethm

\bthm\label{t2ma8B}
If $f:\oc\to\oc$ be a subexpanding rational function of degree $d\ge 2$ whose Julia set $J(f)$ is geometrically irreducible, then for $\mu_h$--a.e. point $\xi\in J(f)\sms \ov{\PC(f)}$ the following limit exists, is finite and positive:
$$
\lim_{r\to 0}\frac{\HD(J(f))-\HD(K_\xi(r))}{\mu_h(B(\xi,r))}.
$$
\ethm

\brem 
We would like to note that if the rational function $f:\oc\to\oc$ is expanding, then it is automatically subexpanding and the two above theorems apply. 
\erem

\sp   

\section{Escape Rates for Meromorphic Functions on the Complex Plane}\label{Transcendental_Functions}

\sp We deal in this final  section with transcendental meromorphic functions. We also apply here the results on escape rates for conformal GDMS and the techniques of first return maps. Let $f:\C\to\oc$ be a meromorphic function. Let $\Sing(f^{-1})$ be the set of all singular points of $f^{-1}$, i. e. the set of all points $w\in\oc$ such that if $W$ is any open connected neighborhood of $w$, then there exists a connected component $U$ of $f^{-1}(W)$ such that the map $f:U\to W$ is not bijective. Of course if $f$ is a rational function, then $\Sing(f^{-1})=f(\Crit(f))$. As in the case of rational functions, we define
$$
\PS(f):=\bu_{n=0}^\infty f^n(\Sing(f^{-1})).
$$
The function $f$ is called \emph{topologically hyperbolic} if
$$
\dist_{\text{Euclid}}(J_f ,\PS(f)) >0,
$$
and it is called \emph{expanding} if there exist $c>0$ and $ \lam>1$
such that
$$
|(f^n)'(z)|\ge c\lam^n 
$$
for all integers $n\ge 1$ and all points $z\in J_f\sms
f^{-n}(\infty)$. Note that every topologically hyperbolic meromorphic
function is tame.  A meromorphic function that is both topologically
hyperbolic and expanding 
is called \emph{hyperbolic}. The meromorphic function $f:\C\to\oc$ is
called dynamically {\it semi-regular} if it is of finite order, commonly
denoted by $\rho_f$, and satisfies the following rapid
growth condition for its derivative.  
\beq \label{eq intro}
|f'(z)|\geq \kappa ^{-1} (1+|z|)^{\al _1} (1+|f(z)|)^{\al_2} \; ,
\quad z\in J_f, 
\eeq 
with some constant $\kappa >0$ and $\a_1,\a_2$ such that $\al_2 > \max\{-\al _1 ,0\}$.  Set $\a:=\a_1+\a_2$.

\ 
\begin{rem}
A particularly simple example of such maps are meromorphic functions $f_\lambda(z) = \lambda e^{z}$
where $\lambda \in (0, 1/e)$ since these maps have an attracting periodic point.  A good reference is \cite{MayUrbETDS}.
\end{rem}
\
\fr Let $h:J_f\to\R$ be a weakly
H\"older continuous function 
in the sense of \cite{MayUrb10}. The definition, introduced in
\cite{MayUrb10} is somewhat technical and we will not provided it in
the current paper. What is important is that each bounded, uniformly
locally H\"older function $h:J_f\to\R$ is weakly H\"older. Fix
$\tau>\a_2$ as required in \cite{MayUrb10}. For $t\in\R$, let
\begin{equation}
  \label{eq:7}
  \psi_{t,h}=-t\log|f'|_\tau+h
\end{equation}
where $|f'(z)|_\tau$ is the norm, or, equivalently, the scaling
factor, of the derivative of $f$ evaluated at a point $z\in J_f$ with
respect to the Riemannian metric
$$
|d\tau(z)|=(1+|z|)^{-\tau}|dz|.
$$
Following \cite{MayUrb10} functions
of the form (\ref{eq:7})(frequently referred to as potentials) are
called \emph{loosely tame}. Let $\cL_{t,h}:C_b(J_f)\to C_b(J_f)$ be the
corresponding \emph{Perron-Frobenius operator} given by the formula
$$
\cL_{t,h}g(z):=\sum_{w\in f^{-1}(z)}g(w)e^{\psi_{t,h}(w)}.
$$
It was shown in \cite{MayUrb10} that, for every $z\in J_f$ and for the
function $\1:z\mapsto 1$, the limit
$$
\lim_{n\to\infty}\frac1n\log\cL_{t,h}\1(z)
$$
exists and takes on the same common value, which we denote by $\P(t)$
and call \emph{the topological pressure} of the potential
$\psi_t$. The following theorem was proved in \cite{MayUrb10}.

\

\bthm\label{t1dns111}
If $f:\C\to\oc$ is a dynamically semi-regular meromorphic function and
$h:J_f\to\R$ is a weakly H\"older continuous potential, then for every
$t>\rho_f/\a$ there exist uniquely determined Borel probability measures $m_{t,h}$ and $\mu_{t,h}$ on $J_f$ with the following properties.

\sp\begin{itemize}
\item[{\rm(a)}] \ $\cL_{t,h}^*m_{t,h}=m_{t,h}$.

\sp\item[{\rm(b)}] \ $\P\(\psi_{t,h}\)=\sup\big\{\hmu(f)+\int\psi_{t,h}\, d\mu:\mu\circ f^{-1}=\mu \  \
  \text{{\rm and }}  \ \int\psi_{t,h}\, d\mu>-\infty\big\}$.

\sp\item[{\rm(c)}] \ $\mu_{t,h}\circ f^{-1}
    =\mu_{t,h}$, $\int\psi_{t,h}\, d\mu_{t,h}>-\infty$, \  and  \
$
\h_{\mu_{t,h}}(f)+\int\psi_{t,h}\,d\mu_{t,h}=\P\(\psi_{t,h}\).
$

\sp\item[{\rm(d)}] \ The measures $\mu_{t,h}$ and $m_{t,h}$ are equivalent and the
  Radon--Nikodym derivative $\frac{d\mu_{t,h}}{dm_{t,h}}$ has a
  nowhere-vanishing H\"older continuous version which is bounded
  above.
\end{itemize}
\ethm  

\fr The exact analogue of Theorem~\ref{prop:1} holds, with the same references, for all hyperbolic meromorphic functions; we will refer to this theorem as  Theorem~\ref{prop:1}(M). Also, for the system $\cS_V$ and the projection $\pi_V:E_V^\N\to J_V$ have the same meaning. As in the case of rational functions denote by $J_R(f)$ the set of all recurrent points of $f$ in $J(f)$. Formally
$$
J_R(f):=\{z\in J(f):\varliminf_{n\to\infty}|f^n(z)-z|=0\}.
$$
Of course $J_R(f)\sbt J_f$ and $\mu_\psi(J(f)\sms J_R(f))=0$ because of Poincar\'e's Recurrence Theorem. The set $J_R(f)$ is significant for us since
$$
J_R(f)\cap V\sbt J_V.
$$
The Exponential Shrinking Property (ESP) holds since now the function $f:\C\to\oc$ is expanding. The proof of Proposition~\ref{p1ma3} goes through unchanged except that instead of using \cite{DU} we now invoke Theorem~\ref{t1dns111} (a). We also will refer this proposition (\ref{p1ma3}) as Proposition~\ref{p1ma3} (M). Lemma~\ref{l1ma5} also carries on to the meromorphic case (we refer to it as Lemma~\ref{l1ma5} (M); the proof of items (a)--(e) Definition~\ref{d1ld3} required by this lemma to hold, follows as in the case of rational functions, from proposition~\ref{p1ma3} (M), while the proof of item (f) of this definition is now a direct consequence of Lemma~4.1 in \cite{SkU_PUZ}. Now, in exactly the same way as in the case of rational functions, as a direct consequence of Theorem~\ref{t1fp83}, Theorem~\ref{t3_2016_05_27}, Proposition~\ref{p1ma3} (M), Lemma~\ref{l1ma5} (M), Lemma~\ref{l1ld7}, and Theorem~\ref{t2rmr3}, we get the following two theorems.

\bthm\label{t1ma10}
Let $f:\C\to\oc$ be a dynamically semi-regular meromorphic function. Let $t>\rho_f/\a$ and let $h:J(f)\to\R$ be a weakly H\"older continuous function. Let $z\in J_R(f)$. Assume that the corresponding equilibrium state $\mu_{t,h}$ is {\rm(WBT)} at $z$. Then
$$
\begin{aligned}
\lim_{\ep\to 0}\frac{\un R_{\mu_{t,h}}(B(z,\ep))}{\mu_{t,h}(B(z,\ep))}
&=\lim_{\ep\to 0}\frac{\ov R_{\mu_{t,h}}(B(z,\ep))}{\mu_{t,h}(B(z,\ep))}
= \\
&=
\begin{cases}
1 &{\rm if} \ z \ \text{{\rm is not any periodic point of }} f, \\
1-\exp\(S_p\psi_{t,h}(z)-p\P(\psi_{t,h})\) &{\rm if} \, z  \ \text{{\rm is a periodic point of }} f.
\end{cases}
\end{aligned}
$$
\ethm

\bthm\label{t1ma10B}
Let $f:\C\to\oc$ be a dynamically semi-regular meromorphic function whose Julia set is geometrically irreducible. Let $t>\rho_f/a$ and let $h:J(f)\to\R$ be a weakly H\"older continuous function. Then
$$
\lim_{\ep\to 0}\frac{\un R_{\mu_{t,h}}(B(z,\ep))}{\mu_{t,h}(B(z,\ep))}
=\lim_{\ep\to 0}\frac{\ov R_{\mu_{t,h}}(B(z,\ep))}{\mu_{t,h}(B(z,\ep))}
=1 
$$
for $\mu_{t,h}$--a.e. $z\in J(f)$. 
\ethm

\brem\label{r1ma10.1+2} 
Theorem~\ref{t1ma10} holds in fact for a larger set than $J_R(f)$. Indeed, it holds for every point in $V\cap J_{\cS_V}$, where $V$ is an arbitrary nice set.
\erem

Turning  to the asymptotics of Hausdorff dimension, let $J_r(f)$ be the set of radial (or conical) points in $J(f)$, i. e.  the set of all those points in $J(f)$ that do not escape to infinity under the action of the map $f:\C\to\oc$. 
Assume now more, namely that $f:\C\to\oc$ is dynamically regular in the sense of \cite{MayUrbETDS} and \cite{MayUrb10}. What at the moment is important for us is that $\P(h_r)=0$, where 
$$
h_r:=\HD(J_r(f)).
$$
We already know that there exists a nice set $V$ containing $\xi$ and the 
elements of the corresponding conformal IFS $\cS_f$ are composed of analytic inverse branches of the the first return map from $V$ to $V$. Since $\xi\in J_R(f)$, we have that $\xi\in J_V$. Corollary~6.4 in \cite{SkU} tells us that $\HD(K(V))<h_r$. So, since by Theorem~\ref{t2had30}, $\lim_{r\to 0}\HD(K(B(\xi,r)))=h_r$, we conclude that
$$
\HD(K(V))<\HD(K(B(\xi,r)))
$$
for all $r>0$ small enough. Therefore, since $h_r=b_{\cS_V}$ and since $\mu_{h,V}=\mu_{b_{\cS_V}}$, applying Theorem~\ref{t2had30}, Corollary~\ref{t2_2016_05_27}, and Corollary~\ref{c1rm3}, we get the following two theorems.

\bthm\label{t1ma11}
Let $f:\C\to\oc$ be a dynamically regular meromorphic function. Fix $\xi\in J_R(f)$. Assume that the measure $\mu_{h_r}$ (i.e. $\mu_{h_r,0}$ with the weakly H\"older function $h$ identically equal to $0$) is {\rm (WBT)} at $\xi$ and the parameter $h_r$ is powering at $\xi$ with respect to the conformal \IFS \ $\cS_f$. Then the following limit exists and is finite and positive:
$$
\lim_{r\to 0}\frac{\HD(J_r(f))-\HD(K_z(r))}{\mu_{h_r}(B(z,r))}.
$$
\ethm

\bthm\label{t1ma11B}
Let $f:\C\to\oc$ be a dynamically regular meromorphic function whose Julia set is geometrically irreducible. Then the following limit exists and is finite and positive for $\mu_{h_r}$--a.e. $z\in J(f)$:
$$
\lim_{r\to 0}\frac{\HD(J_r(f))-\HD(K_z(r))}{\mu_{h_r}(B(z,r))}.
$$
\ethm

\fr Note that the conclusion of Remark~\ref{r1ma10.1+2} holds in the case of Theorem~\ref{t1ma11} too.

\

\

\end{document}